\documentclass[10pt]{amsart}

\usepackage{latexsym,exscale,enumerate}
\usepackage{amsmath,amsthm,amsfonts,amssymb,amscd, stmaryrd}
\usepackage{hyperref}

\usepackage[all]{xy}
\SelectTips{cm}{}

\usepackage{graphicx}

\usepackage{tikz}
\usetikzlibrary{decorations.markings}
\usetikzlibrary{decorations.pathreplacing}

\newcommand{\ep}{\underline{\epsilon}}
\newcommand{\onen}{{\mathbf 1}_{n}}
\newcommand{\onenn}[1]{{\mathbf 1}_{#1}}
\newcommand{\onenp}{{\mathbf 1}_{n'}}
\newcommand{\onenpp}{{\mathbf 1}_{n''}}
\newcommand{\onem}{{\mathbf 1}_{m}}
\newcommand{\onel}{{\mathbf 1}_{\lambda}}
\newcommand{\onell}[1]{{\mathbf 1}_{#1}}
\newcommand{\onelp}{{\mathbf 1}_{\lambda'}}
\newcommand\rE{{\sf{E}}}
\newcommand\rF{{\sf{F}}}
\newcommand\sE{{\cal{E}}}
\newcommand\sF{{\cal{F}}}
\def\cal#1{\mathcal{#1}}%

\newcommand{\bV}{\textstyle{\bigwedge_q}}
\newcommand{\Alt}[2]{\textstyle{\bigwedge_{#1}^{#2}}}
\newcommand{\Alts}[2]{ \textstyle{\scriptstyle\bigwedge_{#1}^{#2}}}

\def\1{\mathbbm{1}}%
\newcommand\E{{\sf{E}}}
\newcommand\F{{\sf{F}}}
\newcommand\T{{\sf{T}}}
\def\l{\lambda}

\newcommand{\END}{{\rm END}}
\newcommand{\Gr}{\cat{Flag}_{N}}
\newcommand{\Grn}[1]{\cat{Flag}_{#1}}

\newcommand{\Uq}{{\bf U}_q(\mathfrak{sl}_2)}
\newcommand{\U}{\dot{{\bf U}}}
\newcommand{\Ucat}{\cal{U}}
\newcommand{\Ucatc}{\check{\cal{U}}}
\newcommand{\Ucatq}{\cal{U}_q}
\newcommand{\UcatD}{\dot{\cal{U}}}
\newcommand{\UcatDq}{\dot{\cal{U}}_q}
\newcommand{\B}{\dot{\mathbb{B}}}
\newcommand{\Bnm}{{_m\dot{\cal{B}}_n}}
\newcommand{\UA}{{_{\cal{A}}\dot{{\bf U}}}}
\newcommand{\sln}{\mf{sl}_n}
\newcommand{\slm}{\mf{sl}_m}

\newcommand{\und}[1]{\underline{#1}}

\newcommand{\foam}[3][N]{#2\cat{Foam}_{#3}(#1)}
\newcommand{\Bfoam}[3][N]{#2\cat{BFoam}_{#3}(#1)}

\newcommand{\qbin}[2]{
\left[
 \begin{array}{c}
 #1 \\
 #2 \\
 \end{array}
 \right]_{q^2}
}
\newcommand{\qbins}[2]{
\left[
 \begin{array}{c}
 \scs #1 \\
 \scs #2 \\
 \end{array}
 \right]
}

\newcommand{\xsum}[2]{
  \xy
  (0,.4)*{\sum};
  (0,3.7)*{\scs #2};
  (0,-2.9)*{\scs #1};
  \endxy
}

\newcommand{\refequal}[1]{\xy {\ar@{=}^{#1}
(-1,0)*{};(1,0)*{}};
\endxy}

\newcommand{\cat}[1]{\ensuremath{\mbox{\bfseries {\upshape {#1}}}}}
\newcommand{\numroman}{\renewcommand{\labelenumi}{\roman{enumi})}}
\newcommand{\numarabic}{\renewcommand{\labelenumi}{\arabic{enumi})}}
\newcommand{\numAlph}{\renewcommand{\labelenumi}{\Alph{enumi}.}}

\newcommand{\To}{\Rightarrow}
\newcommand{\TO}{\Rrightarrow}
\newcommand{\Hom}{{\rm Hom}}
\newcommand{\HOM}{{\rm HOM}}
\renewcommand{\to}{\rightarrow}
\newcommand{\maps}{\colon}
\newcommand{\op}{{\rm op}}
\newcommand{\co}{{\rm co}}
\newcommand{\iso}{\cong}
\newcommand{\id}{{\rm id}}
\newcommand{\bigb}[1]{
\begin{pspicture}(0,0)
 \rput(0,0){\psframebox[framearc=.5,fillstyle=solid]{\small $#1$}}
\end{pspicture}}
\newcommand{\del}{\partial}
\newcommand{\Res}{{\rm Res}}
\newcommand{\End}{{\rm End}}
\newcommand{\Aut}{{\rm Aut}}
\newcommand{\im}{{\rm im\ }}
\newcommand{\coim}{{\rm coim\ }}
\newcommand{\chr}{{\rm char\ }}
\newcommand{\coker}{{\rm coker\ }}
\newcommand{\spann}{{\rm span}}
\newcommand{\rk}{{\rm rk\ }}
\def\bigboxtimes{\mathop{\boxtimes}\limits}

\newcommand{\scs}{\scriptstyle}

\def\Res{{\mathrm{Res}}}
\def\Ind{{\mathrm{Ind}}}
\def\lra{{\longrightarrow}}
\def\dmod{{\mathrm{-mod}}}   
\def\fmod{{\mathrm{-fmod}}}   
\def\pmod{{\mathrm{-pmod}}}  
\def\rk{{\mathrm{rk}}}
\def\NH{{\mathrm{NH}}}
\def\pseq{{\mathrm{Seqd}}}
\def\Id{\mathrm{Id}}
\def\mc{\mathcal}
\def\mf{\mathfrak}
\def\Af{{_{\mc{A}}\mathbf{f}}}    
\def\primef{{'\mathbf{f}}}    
\def\shuffle{\,\raise 1pt\hbox{$\scriptscriptstyle\cup{\mskip
               -4mu}\cup$}\,}
\newcommand{\define}{\stackrel{\mbox{\scriptsize{def}}}{=}}

\theoremstyle{definition}
\newtheorem{thm}{Theorem}[section]

\newtheorem{lem}[thm]{Lemma}
\newtheorem{rem}[thm]{Remark}
\newtheorem{prop}[thm]{Proposition}
\newtheorem{defn}[thm]{Definition}

\numberwithin{equation}{section}

\def\AL#1{\textcolor[rgb]{1.00,0.00,0.00}{[AL: #1]}}%
\def\HQ#1{\textcolor[rgb]{0.00,1.00,0.00}{[HQ: #1]}}%
\def\DR#1{\textcolor[rgb]{0.00,0.00,1.00}{[DR: #1]}}%
\def\MK#1{[MK: #1]}%
\def\AB#1{[AB: #1]}%
\def\b{$\blacktriangleright$}
\def\e{$\blacktriangleleft$}
\def\new#1{\b #1\e}%

\def\emph#1{{\sl #1\/}}
\def\ie{{\sl i.e.\/}}
\def\eg{{\sl e.g.\/}}
\def\Eg{{\sl E.g.\/}}
\def\etc{{\sl etc.\/}}
\def\cf{{\sl c.f.\/}}
\def\etal{\sl{et al.\/}}%
\def\vs{\sl{vs.\/}}%

\let\hat=\widehat
\let\tilde=\widetilde

\let\phi=\varphi
\let\theta=\vartheta
\let\epsilon=\varepsilon

\usepackage{bbm}
\def\C{{\mathbbm C}}
\def\N{{\mathbbm N}}
\def\R{{\mathbbm R}}
\def\Z{{\mathbbm Z}}
\def\Q{{\mathbbm Q}}
\def\H{{\mathbbm H}}

\def\cal#1{\mathcal{#1}}%
\def\1{\mathbbm{1}}%
\def\ev{\mathrm{ev}}%
\def\coev{\mathrm{coev}}%
\def\tr{\mathrm{tr}}%
\def\st{\mathrm{st}}%
\def\pullback#1#2#3{%
  \,\mbox{\raisebox{-.8ex}{$\scriptstyle #1$}}%
  \!\prod\!
  \mbox{\raisebox{-.8ex}{$\scriptstyle #3$}}\,}%
\def\nn{\notag}
\newcommand{\ontop}[2]{\genfrac{}{}{0pt}{2}{\scriptstyle #1}{\scriptstyle #2}}

\def\la{\langle}
\def\ra{\rangle}

\newcommand{\sdotu}[1]{\xybox{%
  (-2,0)*{};
  (2,0)*{};
  (0,0)*{}; (0,-9)*{} **\dir{-}?(.5)*{\bullet} ?(0)*\dir{<};
  (0,-11)*{\scs  #1};
}}
\newcommand{\urcurve}[1]{\xybox{
(-8,0)*{};
  (8,0)*{};
  (0,0);(0,18) **\crv{(8,9)}?(1)*\dir{>};
  (0,-2)*{ \scs #1}
}}
\newcommand{\ulcurve}[1]{\xybox{
(-8,0)*{};
  (8,)*{};
  (0,0);(0,18) **\crv{(-8,9)}?(1)*\dir{>};
  (0,-2)*{ \scs #1}
}}
\newcommand{\slineu}[1]{\xybox{%
  (-2,0)*{};
  (2,0)*{};
  (0,0)*{}; (0,-9)*{} **\dir{-}; ?(0)*\dir{<};
  (0,-11)*{\scs  #1};
}}

\newcommand{\drcurve}[1]{\xybox{
(-8,0)*{};
  (8,0)*{};
  (0,0);(0,18) **\crv{(10,9)}?(.5)*\dir{<}+(2.5,0)*{ \scs #1};
}}

\newcommand{\dlcurve}[1]{\xybox{
(-8,0)*{};
  (8,)*{};
  (0,0);(0,18) **\crv{(-10,9)}?(.5)*\dir{<}+(-2.5,0)*{ \scs #1};
}}
\newcommand{\lineu}[1]{\xybox{%
  (-2,0)*{};
  (2,0)*{};
  (0,0)*{}; (0,-18)*{} **\dir{-}; ?(.5)*\dir{<}+(1.7,-7)*{\scs #1};
}}
\newcommand{\lined}[1]{\xybox{%
  (-2,0)*{};
  (2,0)*{};
  (0,0)*{}; (0,-18)*{} **\dir{-}; ?(.5)*\dir{>}+(1.7,-7)*{\scs #1};
}}
\newcommand{\dotu}[1]{\xybox{%
  (-2,0)*{};
  (2,0)*{};
  (0,0)*{}; (0,-18)*{} **\dir{-}
  ?(.2)*{\bullet} ?(.5)*\dir{<}+(1.7,-7)*{\scs #1};
}}

\newcommand{\twod}{\xybox{%
  (-6,0)*{};
  (6,0)*{};
  (0,6)*{}="f";
  (-3,0)*{}="t1";
  (3,0)*{}="t2";
  (0,12)*{}="b";
  "t1";"f" **\crv{(-3,4)};
  "t2";"f" **\crv{(3,4)};
  "f"+(.5,0);"b"+(.5,0) **\dir{-};
  "f"+(-.5,0);"b"+(-.5,0) **\dir{-};
}}

\newcommand{\twom}{\xybox{%
  (-6,0)*{};
  (6,0)*{};
  (0,6)*{}="f";
  (-3,12)*{}="t1";
  (3,12)*{}="t2";
  (0,0)*{}="b";
  "t1";"f" **\crv{(-3,8)};
  "t2";"f" **\crv{(3,8)};
  "f"+(.5,0);"b"+(.5,0) **\dir{-};
  "f"+(-.5,0);"b"+(-.5,0) **\dir{-};
}}
\newcommand{\twoI}{\xybox{%
  (-6,0)*{};
  (6,0)*{};
  (0,6)*{}="f'";
  (0,12)*{}="f";
  (-3,18)*{}="t1";
  (3,18)*{}="t2";
    (-3,0)*{}="t1'";
  (3,0)*{}="t2'";
  (0,0)*{}="b";
  "t1";"f" **\crv{(-3,14)};
  "t2";"f" **\crv{(3,14)};
  "f"+(.5,0);"f'"+(.5,0) **\dir{-};
  "f"+(-.5,0);"f'"+(-.5,0) **\dir{-};
  "t1'";"f'" **\crv{(-3,4)};
  "t2'";"f'" **\crv{(3,4)};
}}
\newcommand{\twoIu}[1]{\xybox{%
  (-6,0)*{};
  (6,0)*{};
  (0,6)*{}="f'";
  (0,12)*{}="f";
  (-3,18)*{}="t1";
  (3,18)*{}="t2";
    (-3,0)*{}="t1'";
  (3,0)*{}="t2'";
  (0,0)*{}="b";
  "t1";"f" **\crv{(-3,14)};?(.1)*\dir{<};
  "t2";"f" **\crv{(3,14)};?(.1)*\dir{<};
  "f"+(.5,0);"f'"+(.5,0) **\dir{-};
  "f"+(-.5,0);"f'"+(-.5,0) **\dir{-};
  "t1'";"f'" **\crv{(-3,4)};
  "t2'";"f'" **\crv{(3,4)} ?(.15)*\dir{ }+(2,0)*{\scs #1};
}}

\newcommand{\lowrru}[1]{\xybox{%
  (-8,0)*{};
  (8,0)*{};
  (-6,-18)*{};(6,-9)*{} **\crv{(-6,-13) & (6,-15)} ?(1)*\dir{>};
  (6,-9)*{};(6,0)*{}  **\dir{-} ?(.3)*\dir{ }+(2,0)*{\scs {\bf j}};
}}

\newcommand{\lowllu}[1]{\xybox{%
  (-8,0)*{};
  (8,0)*{};
  (6,-18)*{};(-6,-9)*{} **\crv{(6,-13) & (-6,-15)} ?(1)*\dir{>};
  (-6,-9)*{};(-6,0)*{}  **\dir{-} ?(.3)*\dir{ }+(-2,0)*{\scs {\bf j}};
}}
\newcommand{\highrru}[1]{\xybox{%
  (-8,0)*{};
  (8,0)*{};
  (6,0)*{};(-6,-9)*{} **\crv{(6,-5) & (-6,-3)} ?(1)*\dir{<};
  (-6,-9)*{};(-6,-18)*{}  **\dir{-} ?(.3)*\dir{ }+(-2,0)*{\scs #1};
}}

\newcommand{\highllu}[1]{\xybox{%
  (-8,0)*{};
  (8,0)*{};
  (-6,0)*{};(6,-9)*{} **\crv{(-6,-5) & (6,-3)} ?(1)*\dir{<};
  (6,-9)*{};(6,-18)*{}  **\dir{-} ?(.3)*\dir{ }+(2,0)*{\scs #1};
}}

\newcommand{\bbmf}[3]{\xybox{%
  (-6,0)*{};
  (6,0)*{};
  (0,-8)*{}="f";
  (-4,0)*{}="t1";
  (4,0)*{}="t2";
  (0,-16)*{}="b";
  "t1";"f" **\crv{(-4,-4)}; ?(.35)*\dir{>}+(-2.3,0)*{\scriptstyle{#1}};
  "t2";"f" **\crv{(4,-4)}; ?(.35)*\dir{>}+(2.3,0)*{\scriptstyle{#2}};
  "f";"b" **\dir{-}; ?(.75)*\dir{>}+(3.8,0)*{\scriptstyle{#3}};;
}}

\newcommand{\bbme}[3]{\xybox{%
  (-6,0)*{};
  (6,0)*{};
  (0,-8)*{}="f";
  (-4,0)*{}="t1";
  (4,0)*{}="t2";
  (0,-16)*{}="b";
  "t1";"f" **\crv{(-4,-4)}; ?(.35)*\dir{<}+(-2.3,0)*{\scriptstyle{#1}};
  "t2";"f" **\crv{(4,-4)}; ?(.35)*\dir{<}+(2.3,0)*{\scriptstyle{#2}};
  "f";"b" **\dir{-}; ?(.6)*\dir{<}+(3.8,0)*{\scriptstyle{#3}};
}}

\newcommand{\bbdf}[3]{\xybox{%
  (-6,0)*{};
  (6,0)*{};
  (0,-8)*{}="f";
  (-4,-16)*{}="b1";
  (4,-16)*{}="b2";
  (0,-0)*{}="t";
  "f";"b1" **\crv{(-4,-12)}; ?(.75)*\dir{>}+(-2.3,0)*{\scriptstyle{#1}};
  "f";"b2" **\crv{(4,-12)}; ?(.75)*\dir{>}+(2.3,0)*{\scriptstyle{#2}};
  "t";"f" **\dir{-}; ?(.35)*\dir{>}+(3.8,0)*{\scriptstyle{#3}};
}}

\newcommand{\bbde}[3]{\xybox{%
  (-6,0)*{};
  (6,0)*{};
  (0,-8)*{}="f";
  (-4,-16)*{}="b1";
  (4,-16)*{}="b2";
  (0,-0)*{}="t";
  "f";"b1" **\crv{(-4,-12)}; ?(.6)*\dir{<}+(-2.3,0)*{\scriptstyle{#1}};
  "f";"b2" **\crv{(4,-12)}; ?(.6)*\dir{<}+(2.3,0)*{\scriptstyle{#2}};
  "t";"f" **\dir{-}; ?(.3)*\dir{<}+(3.8,0)*{\scriptstyle{#3}};
}}

\newcommand{\ccbub}[1]{
\xybox{%
 (-6,0)*{};
  (6,0)*{};
  (-4,0)*{}="t1";
  (4,0)*{}="t2";
  "t2";"t1" **\crv{(4,6) & (-4,6)};
  ?(.05)*\dir{>} ?(1)*\dir{>};
  "t2";"t1" **\crv{(4,-6) & (-4,-6)};
   ?(.25)*\dir{}+(0,0)*{\bullet}+(0,-3)*{\scs {#1}};
}}

\newcommand{\iccbub}[2]{
\xybox{%
 (-6,0)*{};
  (6,0)*{};
  (-4,0)*{}="t1";
  (4,0)*{}="t2";
  "t2";"t1" **\crv{(4,6) & (-4,6)}; ?(.7)*\dir{}+(-2,0)*{\scs #2}
  ?(.05)*\dir{>} ?(1)*\dir{>};
  "t2";"t1" **\crv{(4,-6) & (-4,-6)};
   ?(.3)*\dir{}+(0,0)*{\bullet}+(0,-3)*{\scs {#1}};
}}
\newcommand{\icbub}[2]{
\xybox{%
 (-6,0)*{};
  (6,0)*{};
  (-4,0)*{}="t1";
  (4,0)*{}="t2";
  "t2";"t1" **\crv{(4,6) & (-4,6)};?(.7)*\dir{}+(-2,0)*{\scs #2};
   ?(0)*\dir{<} ?(.95)*\dir{<};
  "t2";"t1" **\crv{(4,-6) & (-4,-6)};
   ?(.3)*\dir{}+(0,0)*{\bullet}+(0,-3)*{\scs {#1}};
}}

\newcommand{\cbub}[1]{
\xybox{%
 (-6,0)*{};
  (6,0)*{};
  (-4,0)*{}="t1";
  (4,0)*{}="t2";
  "t2";"t1" **\crv{(4,6) & (-4,6)};
   ?(0)*\dir{<} ?(.95)*\dir{<};
  "t2";"t1" **\crv{(4,-6) & (-4,-6)};
   ?(.75)*\dir{}+(0,0)*{\bullet}+(0,-3)*{\scs {#1}};
}}
\newcommand{\bbe}[1]{\xybox{%
  (-2,0)*{};
  (2,0)*{};
  (0,0);(0,-18) **\dir{-}; ?(.5)*\dir{<}+(2.3,0)*{\scriptstyle{#1}};
}}

\newcommand{\bbf}[1]{\xybox{%
  (-2,0)*{};
  (2,0)*{};
  (0,0);(0,-18) **\dir{-}; ?(.5)*\dir{>}+(2.3,0)*{\scriptstyle{#1}};
}}

\newcommand{\bbelong}[1]{\xybox{%
  (-2,0)*{};
  (2,0)*{};
  (0,0);(0,-22) **\dir{-}; ?(.5)*\dir{<}+(2.3,0)*{\scriptstyle{#1}};
}}

\newcommand{\bbflong}[1]{\xybox{%
  (-2,0)*{};
  (2,0)*{};
  (0,0);(0,-22) **\dir{-}; ?(.5)*\dir{>}+(2.3,0)*{\scriptstyle{#1}};
}}

\newcommand{\bbsid}{\xybox{%
  (-2,0)*{};
  (2,0)*{};
  (0,10);(0,4) **\dir{-};
}}
\newcommand{\bbpef}[1]{\xybox{%
  (-6,0)*{};
  (6,0)*{};
  (-4,0)*{}="t1";
  (4,0)*{}="t2";
  "t1";"t2" **\crv{(-4,-6) & (4,-6)}; ?(.15)*\dir{>} ?(.9)*\dir{>}
    ?(.5)*\dir{}+(0,-2)*{\scriptstyle{#1}};
}}
\newcommand{\bbpfe}[1]{\xybox{%
  (-6,0)*{};
  (6,0)*{};
  (-4,0)*{}="t1";
  (4,0)*{}="t2";
  "t2";"t1" **\crv{(4,-6) & (-4,-6)}; ?(.15)*\dir{>} ?(.9)*\dir{>} ?(.5)*\dir{}+(0,-2)*{\scriptstyle{#1}};
}}

\newcommand{\bbcfe}[1]{\xybox{%
  (-6,0)*{};
  (6,0)*{};
  (-4,0)*{}="t1";
  (4,0)*{}="t2";
  "t1";"t2" **\crv{(-4,6) & (4,6)}; ?(.15)*\dir{>} ?(.9)*\dir{>}
  ?(.5)*\dir{}+(0,2)*{\scriptstyle{#1}};
}}
\newcommand{\bbcef}[1]{\xybox{%
  (-6,0)*{};
  (6,0)*{};
  (-4,0)*{}="t1";
  (4,0)*{}="t2";
  "t2";"t1" **\crv{(4,6) & (-4,6)}; ?(.15)*\dir{>}
  ?(.9)*\dir{>} ?(.5)*\dir{}+(0,2)*{\scriptstyle{#1}};
}}
\newcommand{\lbbcef}[1]{\xybox{%
  (-8,0)*{};
  (8,0)*{};
  (-4,0)*{}="t1";
  (4,0)*{}="t2";
  "t2";"t1" **\crv{(4,6) & (-4,6)}; ?(.15)*\dir{>}
  ?(.9)*\dir{>} ?(.5)*\dir{}+(0,2)*{\scriptstyle{#1}};
}}

\newcommand{\sccbub}[1]{%
\xybox{%
 (-6,0)*{};
  (6,0)*{};
  (-4,0)*{}="t1";
  (4,0)*{}="t2";
  "t2";"t1" **\crv{(4,6) & (-4,6)}; ?(.05)*\dir{>} ?(1)*\dir{>};
  "t2";"t1" **\crv{(4,-6) & (-4,-6)}; ?(.3)*\dir{}+(2,-1)*{\scs #1};
}}
\newcommand{\scbub}[1]{%
\xybox{%
 (-6,0)*{};
  (6,0)*{};
  (-4,0)*{}="t1";
  (4,0)*{}="t2";
  "t2";"t1" **\crv{(4,6) & (-4,6)}; ?(0)*\dir{<} ?(.95)*\dir{<};
  "t2";"t1" **\crv{(4,-6) & (-4,-6)}; ?(.3)*\dir{}+(2,-1)*{\scs #1};
}}
\newcommand{\mcbub}[1]{%
\xybox{%
 (-12,0)*{};
  (12,0)*{};
  (-8,0)*{}="t1";
  (8,0)*{}="t2";
  "t2";"t1" **\crv{(8,10) & (-8,10)}; ?(0)*\dir{<} ?(1)*\dir{<};
  "t2";"t1" **\crv{(8,-10) & (-8,-10)}; ?(.3)*\dir{}+(2,-1)*{\scs #1};
}}
\newcommand{\mccbub}[1]{%
\xybox{%
 (-12,0)*{};
  (12,0)*{};
  (-8,0)*{}="t1";
  (8,0)*{}="t2";
  "t2";"t1" **\crv{(8,10) & (-8,10)}; ?(0)*\dir{>} ?(1)*\dir{>};
  "t2";"t1" **\crv{(8,-10) & (-8,-10)}; ?(.3)*\dir{}+(2,-1)*{\scs #1};
}}
\newcommand{\lccbub}[1]{
\xybox{%
 (-12,0)*{};
  (12,0)*{};
  (-12,0)*{}="t1";
  (12,0)*{}="t2";
  "t2";"t1" **\crv{(12,14) & (-12,14)}; ?(0)*\dir{>} ?(1)*\dir{>};
  "t2";"t1" **\crv{(12,-14) & (-12,-14)}; ?(.3)*\dir{}+(2,-1)*{\scs #1};
}}
\newcommand{\lcbub}[1]{
\xybox{%
 (-12,0)*{};
  (12,0)*{};
  (-12,0)*{}="t1";
  (12,0)*{}="t2";
  "t2";"t1" **\crv{(12,14) & (-12,14)}; ?(0)*\dir{<} ?(1)*\dir{<};
  "t2";"t1" **\crv{(12,-14) & (-12,-14)}; ?(.3)*\dir{}+(2,-1)*{\scs #1};
}}
\newcommand{\xlccbub}[1]{
\xybox{%
 (-12,0)*{};
  (12,0)*{};
  (-20,0)*{}="t1";
  (20,0)*{}="t2";
  "t2";"t1" **\crv{(20,22) & (-20,22)}; ?(0)*\dir{>} ?(1)*\dir{>};
  "t2";"t1" **\crv{(20,-22) & (-20,-22)}; ?(.3)*\dir{}+(2,-1)*{\scs #1};
}}
\newcommand{\xlcbub}[1]{
\xybox{%
 (-12,0)*{};
  (12,0)*{};
  (-20,0)*{}="t1";
  (20,0)*{}="t2";
  "t2";"t1" **\crv{(20,24) & (-20,24)}; ?(0)*\dir{<} ?(1)*\dir{<};
  "t2";"t1" **\crv{(20,-24) & (-20,-24)}; ?(.3)*\dir{}+(2,-1)*{\scs #1};
}}
\newcommand{\scap}{
\xybox{%
(-6,0)*{};
  (6,0)*{};
 (-4,0)*{};(4,0)*{}; **\crv{(4,5) & (-4,5)};
 }}
\newcommand{\mcap}{
\xybox{%
 (0,6)*{};(0,-6)*{};
 (-8,0)*{};(8,0)*{}; **\crv{(8,10) & (-8,10)};
 }}
\newcommand{\lcap}{
\xybox{%
 (-12,2)*{};(12,2)*{}; **\crv{(12,14) & (-12,14)};
 }}
\newcommand{\xlcap}{
\xybox{%
 (-20,0)*{};(20,0)*{}; **\crv{(20,22) & (-20,22)};
 }}

\newcommand{\scupfe}{\xybox{%
 (-4,0)*{};(4,0)*{}; **\crv{(4,-5) & (-4,-5)} ?(0)*\dir{<} ?(.95)*\dir{<};
 }}
\newcommand{\mcupfe}{\xybox{%
 (-8,0)*{};(8,0)*{}; **\crv{(8,-10) & (-8,-10)} ?(0)*\dir{<} ?(.95)*\dir{<};
 }}
\newcommand{\lcupfe}{\xybox{%
 (-12,0)*{};(12,0)*{}; **\crv{(12,-14) & (-12,-14)} ?(0)*\dir{<} ?(.95)*\dir{<};
 }}
\newcommand{\xlcupfe}{\xybox{%
 (-20,0)*{};(20,0)*{}; **\crv{(20,-22) & (-20,-22)} ?(0)*\dir{<} ?(.95)*\dir{<};
 }}
\newcommand{\scupef}{\xybox{%
 (-4,0)*{};(4,0)*{}; **\crv{(4,-5) & (-4,-5)} ?(0)*\dir{>} ?(1)*\dir{>};
 }}
\newcommand{\mcupef}{\xybox{%
 (-8,0)*{};(8,0)*{}; **\crv{(8,-10) & (-8,-10)} ?(0)*\dir{>} ?(1)*\dir{>};
 }}
\newcommand{\lcupef}{\xybox{%
 (12,0)*{};(-12,0)*{} **\crv{(12,-12) & (-12,-12)} ?(0)*\dir{>} ?(1)*\dir{>};
 }}
\newcommand{\xlcupef}{\xybox{%
 (-20,0)*{};(20,0)*{}; **\crv{(20,-22) & (-20,-22)} ?(0)*\dir{>} ?(1)*\dir{>};
 }}
\newcommand{\ecross}{\xybox{%
(-6,0)*{};
  (6,0)*{};
 (-4,-4)*{};(4,4)*{} **\crv{(-4,-1) & (4,1)}?(1)*\dir{>};
 (4,-4)*{};(-4,4)*{} **\crv{(4,-1) & (-4,1)}?(1)*\dir{>};
 }}
\newcommand{\fcross}{\xybox{%
 (-4,-4)*{};(4,4)*{} **\crv{(-4,-1) & (4,1)}?(1)*\dir{<};
 (4,-4)*{};(-4,4)*{} **\crv{(4,-1) & (-4,1)}?(1)*\dir{<};
 }}
\newcommand{\fecross}{\xybox{%
 (-4,-4)*{};(4,4)*{} **\crv{(-4,-1) & (4,1)}?(1)*\dir{>};
 (4,-4)*{};(-4,4)*{} **\crv{(4,-1) & (-4,1)}?(1)*\dir{<};
 }}
\newcommand{\efcross}{\xybox{%
 (-4,-4)*{};(4,4)*{} **\crv{(-4,-1) & (4,1)}?(1)*\dir{<};
 (4,-4)*{};(-4,4)*{} **\crv{(4,-1) & (-4,1)}?(1)*\dir{>};
 }}
\newcommand{\seline}{\xybox{%
 (0,-4)*{};(0,4)*{} **\dir{-} ?(1)*\dir{>};
}}
\newcommand{\sfline}{\xybox{%
 (0,-4)*{};(0,4)*{} **\dir{-} ?(1)*\dir{<};
}}
\newcommand{\meline}{\xybox{%
 (0,-8)*{};(0,8)*{} **\dir{-} ?(1)*\dir{>};
}}
\newcommand{\mfline}{\xybox{%
 (0,-8)*{};(0,8)*{} **\dir{-} ?(1)*\dir{<};
}}
\newcommand{\leline}{\xybox{%
 (0,-12)*{};(0,12)*{} **\dir{-} ?(1)*\dir{>};
}}
\newcommand{\lfline}{\xybox{%
 (0,-12)*{};(0,12)*{} **\dir{-} ?(1)*\dir{<};
}}

\newcommand{\chsheet}[1][1]{
\xy
(0,0)*{
\begin{tikzpicture} [fill opacity=0.2,decoration={markings, mark=at position 0.6 with {\arrow{>}}; }, scale=#1]
	\path [fill=red] (1,1) -- (-1,2) -- (-1,-1) -- (1,-2) -- cycle;
	\draw (-1,.58) -- (-1,2) -- (1,1) -- (1,-2) -- (-1,-1) -- (-1,-.32);
	\draw [fill=red] (0,.75) to [out=225, in=0] (-2,.75) to [out=180, in=90] (-3,0) to [out=270, in=180] (-2,-.5) to [out=0, in=135] (0,-.5);
	\filldraw [fill=white, opacity=1] (-2.5,.15) arc (-120:-60:1) -- (-1.65,.1) arc (70:110:1);
	\draw [thick, red, postaction = {decorate}] (-1,-.3) [out=0,in=270] to (-.75,.13) to [out=90,in=0] (-1,.54);
	\draw [thick, red, dashed] (-1,-.3) to [out=180,in=270] (-1.25,.13) to [out=90,in=180] (-1,.54);
	\path [fill=blue] (-1,.54) to [out=0,in=90] (-.75,.13) to [out=270,in=0] (-1,-.3) to [out=180,in=270] (-1.25,.13) to [out=90,in=180] (-1,.54);
\end{tikzpicture}};
\endxy}

\newcommand{\handlesheet}[1][1]{
\xy
(0,0)*{
\begin{tikzpicture} [fill opacity=0.2,decoration={markings, mark=at position 0.6 with {\arrow{>}}; }, scale=#1]
	\path [fill=red] (1,1) -- (-1,2) -- (-1,-1) -- (1,-2) -- cycle;
	\draw (-1,.58) -- (-1,2) -- (1,1) -- (1,-2) -- (-1,-1) -- (-1,-.32);
	\draw [fill=red] (0,.75) to [out=225, in=0] (-2,.75) to [out=180, in=90] (-3,0) to [out=270, in=180] (-2,-.5) to [out=0, in=135] (0,-.5);
	\filldraw [fill=white, opacity=1] (-2.5,.15) arc (-120:-60:1) -- (-1.65,.1) arc (70:110:1);
\end{tikzpicture}};
\endxy}

\newcommand{\diagsheet}[1][1]{
\xy
(0,0)*{
\begin{tikzpicture} [fill opacity=0.2,decoration={markings, mark=at position 0.6 with {\arrow{>}}; }, scale=#1]
	\draw [fill=red] (1,1) -- (-1,2) -- (-1,-1) -- (1,-2) -- cycle;
\end{tikzpicture}};
\endxy}

\newcommand{\dotsheet}[2][1]{
\xy
(0,0)*{
\begin{tikzpicture} [fill opacity=0.2,decoration={markings, mark=at position 0.6 with {\arrow{>}}; }, scale=#1]
	\draw [fill=red] (1,1) -- (-1,2) -- (-1,-1) -- (1,-2) -- cycle;
	\node [opacity=1] at (0,0) {$\bullet$};
	\node [opacity=1] at (-.3,0){$#2$};
\end{tikzpicture}};
\endxy}

\newcommand{\sphere}[1][1]{
\xy
(0,0)*{
\begin{tikzpicture} [fill opacity=0.2, scale=#1]
	\filldraw [fill=red] (0,0) circle (1);
	\draw (-1,0) .. controls (-1,-.4) and (1,-.4) .. (1,0);
	\draw[dashed] (-1,0) .. controls (-1,.4) and (1,.4) .. (1,0);
\end{tikzpicture}};
\endxy}

\newcommand{\dottedsphere}[2][1]{
\xy
(0,0)*{
\begin{tikzpicture} [fill opacity=0.2, scale=#1]
	\filldraw [fill=red] (0,0) circle (1);
	\draw (-1,0) .. controls (-1,-.4) and (1,-.4) .. (1,0);
	\draw[dashed] (-1,0) .. controls (-1,.4) and (1,.4) .. (1,0);
	\node [opacity=1]  at (0,0.6) {$\bullet$};
	\node [opacity=1] at (-0.3,0.6) {$#2$};
\end{tikzpicture}};
\endxy}

\newcommand{\twodottedsphere}[1][1]{
\xy
(0,0)*{
\begin{tikzpicture} [fill opacity=0.2, scale=#1]
	\filldraw [fill=red] (0,0) circle (1);
	\draw (-1,0) .. controls (-1,-.4) and (1,-.4) .. (1,0);
	\draw[dashed] (-1,0) .. controls (-1,.4) and (1,.4) .. (1,0);
	\node [opacity=1]  at (0.25,0.6) {$\bullet$};
	\node [opacity=1] at (-0.25,0.6) {$\bullet$};
\end{tikzpicture}};
\endxy}

\newcommand{\thetafoam}[4][1]{
\xy
(0,0)*{
\begin{tikzpicture} [fill opacity=0.2,decoration={markings, mark=at position 0.6 with {\arrow{>}}; }, scale=#1]
	\filldraw [fill=red] (0,0) circle (1);
	\path [fill=blue] (0,0) ellipse (1 and 0.3); 	
	\draw [very thick, red, postaction={decorate}] (1,0) .. controls (1,-.4) and (-1,-.4) .. (-1,0);
	\draw[very thick, red, dashed] (-1,0) .. controls (-1,.4) and (1,.4) .. (1,0);
	\node [opacity=1]  at (0,0.7) {$\bullet$};
	\node [opacity=1]  at (0,0) {$\bullet$};
	\node [opacity=1]  at (0,-0.7) {$\bullet$};
	\node [opacity=1] at (-0.3,0.7) {$#2$};
	\node [opacity=1] at (-0.3,0) {$#3$};
	\node [opacity=1] at (-0.3,-0.7) {$#4$};
\end{tikzpicture}};
\endxy}

\newcommand{\cylinder}[1][1]{
\xy
(0,0)*{
\begin{tikzpicture} [fill opacity=0.2,  decoration={markings, mark=at position 0.6 with {\arrow{>}};  }, scale=#1]
	\draw [fill=red] (0,4) ellipse (1 and 0.5);
	\path [fill=red] (0,0) ellipse (1 and 0.5);
	\draw (1,0) .. controls (1,-.66) and (-1,-.66) .. (-1,0);
	\draw[dashed] (1,0) .. controls (1,.66) and (-1,.66) .. (-1,0);
	\draw (1,4) -- (1,0);
	\draw (-1,4) -- (-1,0);
	\path[fill=red, opacity=.3] (-1,4) .. controls (-1,3.34) and (1,3.34) .. (1,4) --
		(1,0) .. controls (1,.66) and (-1,.66) .. (-1,0) -- cycle;
\end{tikzpicture}};
\endxy
}

\newcommand{\bamboo}[1][1]{
\xy
(0,0)*{
\begin{tikzpicture} [fill opacity=0.2,  decoration={markings, mark=at position 0.6 with {\arrow{>}};  }, scale=#1]
	\draw [fill=red] (0,4) ellipse (1 and 0.5);
	\path [fill=red] (0,0) ellipse (1 and 0.5);
	\draw (1,0) .. controls (1,-.66) and (-1,-.66) .. (-1,0);
	\draw[dashed] (1,0) .. controls (1,.66) and (-1,.66) .. (-1,0);
	\draw (1,4) -- (1,0);
	\draw (-1,4) -- (-1,0);
	\path[fill=red, opacity=.3] (-1,4) .. controls (-1,3.34) and (1,3.34) .. (1,4) --
		(1,0) .. controls (1,.66) and (-1,.66) .. (-1,0) -- cycle;
	\draw [ultra thick, red, postaction={decorate}] (1,2) .. controls (1,1.34) and (-1,1.34) .. (-1,2);
	\draw[ultra thick, dashed, red] (1,2) .. controls (1,2.66) and (-1,2.66) .. (-1,2);
	\path [fill=blue] (0,2) ellipse (1 and 0.5);
\end{tikzpicture}};
\endxy
}

\newcommand{\bambooRHSone}[1][1]{
\xy
(0,0)*{
\begin{tikzpicture} [fill opacity=0.2,  decoration={markings, mark=at position 0.6 with {\arrow{>}};  }, scale=#1]
	\draw [fill=red] (0,4) ellipse (1 and 0.5);
	\draw (-1,4) .. controls (-1,2) and (1,2) .. (1,4);
	\path [fill=red, opacity=.3] (1,4) .. controls (1,3.34) and (-1,3.34) .. (-1,4) --
		(-1,4) .. controls (-1,2) and (1,2) .. (1,4);
	\node[opacity=1] at (0,3) {$\bullet$};
	\path [fill=red] (0,0) ellipse (1 and 0.5);
	\draw (1,0) .. controls (1,-.66) and (-1,-.66) .. (-1,0);
	\draw[dashed] (1,0) .. controls (1,.66) and (-1,.66) .. (-1,0);
	\draw (-1,0) .. controls (-1,2) and (1,2) .. (1,0);
	\path [fill=red, opacity=.3] (1,0) .. controls (1,.66) and (-1,.66) .. (-1,0) --
		(-1,0) .. controls (-1,2) and (1,2) .. (1,0);
\end{tikzpicture}};
\endxy
}

\newcommand{\bambooRHStwo}[1][1]{
\xy
(0,0)*{
\begin{tikzpicture} [fill opacity=0.2,  decoration={markings, mark=at position 0.6 with {\arrow{>}};  }, scale=#1]
	\draw [fill=red] (0,4) ellipse (1 and 0.5);
	\draw (-1,4) .. controls (-1,2) and (1,2) .. (1,4);
	\path [fill=red, opacity=.3] (1,4) .. controls (1,3.34) and (-1,3.34) .. (-1,4) --
		(-1,4) .. controls (-1,2) and (1,2) .. (1,4);
	\path [fill=red] (0,0) ellipse (1 and 0.5);
	\draw (1,0) .. controls (1,-.66) and (-1,-.66) .. (-1,0);
	\draw[dashed] (1,0) .. controls (1,.66) and (-1,.66) .. (-1,0);
	\draw (-1,0) .. controls (-1,2) and (1,2) .. (1,0);
	\path [fill=red, opacity=.3] (1,0) .. controls (1,.66) and (-1,.66) .. (-1,0) --
		(-1,0) .. controls (-1,2) and (1,2) .. (1,0);
	\node[opacity=1] at (0,1) {$\bullet$};
\end{tikzpicture}};
\endxy
}

\newcommand{\airlock}[1][1]{
\xy
(0,0)*{
\begin{tikzpicture} [fill opacity=0.2,  decoration={markings, mark=at position 0.6 with {\arrow{>}};  }, scale=#1]
	\draw [fill=red] (0,4) ellipse (1 and 0.5);
	\path [fill=red] (0,0) ellipse (1 and 0.5);
	\draw (1,0) .. controls (1,-.66) and (-1,-.66) .. (-1,0);
	\draw[dashed] (1,0) .. controls (1,.66) and (-1,.66) .. (-1,0);
	\draw (1,4) -- (1,0);
	\draw (-1,4) -- (-1,0);
	\path[fill=red, opacity=.3] (-1,4) .. controls (-1,3.34) and (1,3.34) .. (1,4) --
		(1,0) .. controls (1,.66) and (-1,.66) .. (-1,0) -- cycle;
	\draw [ultra thick, red, postaction={decorate}] (-1,1.25) .. controls (-1,.59) and (1,.59) .. (1,1.25);
	\draw[ultra thick, dashed, red] (1,1.25) .. controls (1,1.91) and (-1,1.91) .. (-1,1.25);
	\path [fill=blue] (0,1.25) ellipse (1 and 0.5);
	\draw [ultra thick, red, postaction={decorate}] (1,2.75) .. controls (1,2.09) and (-1,2.09) .. (-1,2.75);
	\draw[ultra thick, dashed, red] (1,2.75) .. controls (1,3.41) and (-1,3.41) .. (-1,2.75);
	\path [fill=blue] (0,2.75) ellipse (1 and 0.5);
\end{tikzpicture}};
\endxy
}

\newcommand{\capcup}[1][1]{
\xy
(0,0)*{
\begin{tikzpicture} [fill opacity=0.2,  decoration={markings, mark=at position 0.6 with {\arrow{>}};  }, scale=#1]
	\draw [fill=red] (0,4) ellipse (1 and 0.5);
	\draw (-1,4) .. controls (-1,2) and (1,2) .. (1,4);
	\path [fill=red, opacity=.3] (1,4) .. controls (1,3.34) and (-1,3.34) .. (-1,4) --
		(-1,4) .. controls (-1,2) and (1,2) .. (1,4);
	\path [fill=red] (0,0) ellipse (1 and 0.5);
	\draw (1,0) .. controls (1,-.66) and (-1,-.66) .. (-1,0);
	\draw[dashed] (1,0) .. controls (1,.66) and (-1,.66) .. (-1,0);
	\draw (-1,0) .. controls (-1,2) and (1,2) .. (1,0);
	\path [fill=red, opacity=.3] (1,0) .. controls (1,.66) and (-1,.66) .. (-1,0) --
		(-1,0) .. controls (-1,2) and (1,2) .. (1,0);
\end{tikzpicture}};
\endxy
}

\newcommand{\slthncone}[1][1]{
\xy
(0,0)*{
\begin{tikzpicture} [fill opacity=0.2,  decoration={markings, mark=at position 0.6 with {\arrow{>}};  }, scale=#1]
	\draw [fill=red] (0,4) ellipse (1 and 0.5);
	\draw (-1,4) .. controls (-1,2) and (1,2) .. (1,4);
	\path [fill=red, opacity=.3] (1,4) .. controls (1,3.34) and (-1,3.34) .. (-1,4) --
		(-1,4) .. controls (-1,2) and (1,2) .. (1,4);
	\node[opacity=1] at (.25,3) {$\bullet$};
	\node[opacity=1] at (-.25,3) {$\bullet$};
	\path [fill=red] (0,0) ellipse (1 and 0.5);
	\draw (1,0) .. controls (1,-.66) and (-1,-.66) .. (-1,0);
	\draw[dashed] (1,0) .. controls (1,.66) and (-1,.66) .. (-1,0);
	\draw (-1,0) .. controls (-1,2) and (1,2) .. (1,0);
	\path [fill=red, opacity=.3] (1,0) .. controls (1,.66) and (-1,.66) .. (-1,0) --
		(-1,0) .. controls (-1,2) and (1,2) .. (1,0);
\end{tikzpicture}};
\endxy
}

\newcommand{\slthnctwo}[1][1]{
\xy
(0,0)*{
\begin{tikzpicture} [fill opacity=0.2,  decoration={markings, mark=at position 0.6 with {\arrow{>}};  }, scale=#1]
	\draw [fill=red] (0,4) ellipse (1 and 0.5);
	\draw (-1,4) .. controls (-1,2) and (1,2) .. (1,4);
	\path [fill=red, opacity=.3] (1,4) .. controls (1,3.34) and (-1,3.34) .. (-1,4) --
		(-1,4) .. controls (-1,2) and (1,2) .. (1,4);
	\node[opacity=1] at (0,3) {$\bullet$};
	\path [fill=red] (0,0) ellipse (1 and 0.5);
	\draw (1,0) .. controls (1,-.66) and (-1,-.66) .. (-1,0);
	\draw[dashed] (1,0) .. controls (1,.66) and (-1,.66) .. (-1,0);
	\draw (-1,0) .. controls (-1,2) and (1,2) .. (1,0);
	\path [fill=red, opacity=.3] (1,0) .. controls (1,.66) and (-1,.66) .. (-1,0) --
		(-1,0) .. controls (-1,2) and (1,2) .. (1,0);
		\node[opacity=1] at (0,1) {$\bullet$};
\end{tikzpicture}};
\endxy
}

\newcommand{\slthncthree}[1][1]{
\xy
(0,0)*{
\begin{tikzpicture} [fill opacity=0.2,  decoration={markings, mark=at position 0.6 with {\arrow{>}};  }, scale=#1]
	\draw [fill=red] (0,4) ellipse (1 and 0.5);
	\draw (-1,4) .. controls (-1,2) and (1,2) .. (1,4);
	\path [fill=red, opacity=.3] (1,4) .. controls (1,3.34) and (-1,3.34) .. (-1,4) --
		(-1,4) .. controls (-1,2) and (1,2) .. (1,4);
	\path [fill=red] (0,0) ellipse (1 and 0.5);
	\draw (1,0) .. controls (1,-.66) and (-1,-.66) .. (-1,0);
	\draw[dashed] (1,0) .. controls (1,.66) and (-1,.66) .. (-1,0);
	\draw (-1,0) .. controls (-1,2) and (1,2) .. (1,0);
	\path [fill=red, opacity=.3] (1,0) .. controls (1,.66) and (-1,.66) .. (-1,0) --
		(-1,0) .. controls (-1,2) and (1,2) .. (1,0);
	\node[opacity=1] at (.25,1) {$\bullet$};
	\node[opacity=1] at (-.25,1) {$\bullet$};
\end{tikzpicture}};
\endxy
}

\newcommand{\slthncfour}[1][1]{
\xy
(0,0)*{
\begin{tikzpicture} [fill opacity=0.2,  decoration={markings, mark=at position 0.6 with {\arrow{>}};  }, scale=#1]
	\draw [fill=red] (0,4) ellipse (1 and 0.5);
	\draw (-1,4) .. controls (-1,2) and (1,2) .. (1,4);
	\path [fill=red, opacity=.3] (1,4) .. controls (1,3.34) and (-1,3.34) .. (-1,4) --
		(-1,4) .. controls (-1,2) and (1,2) .. (1,4);
	\path [fill=red] (0,0) ellipse (1 and 0.5);
	\draw (1,0) .. controls (1,-.66) and (-1,-.66) .. (-1,0);
	\draw[dashed] (1,0) .. controls (1,.66) and (-1,.66) .. (-1,0);
	\draw (-1,0) .. controls (-1,2) and (1,2) .. (1,0);
	\path [fill=red, opacity=.3] (1,0) .. controls (1,.66) and (-1,.66) .. (-1,0) --
		(-1,0) .. controls (-1,2) and (1,2) .. (1,0);
		\node[opacity=1] at (0,1) {$\bullet$};
\end{tikzpicture}};
\endxy
}

\newcommand{\slthncfive}[1][1]{
\xy
(0,0)*{
\begin{tikzpicture} [fill opacity=0.2,  decoration={markings, mark=at position 0.6 with {\arrow{>}};  }, scale=#1]
	\draw [fill=red] (0,4) ellipse (1 and 0.5);
	\draw (-1,4) .. controls (-1,2) and (1,2) .. (1,4);
	\path [fill=red, opacity=.3] (1,4) .. controls (1,3.34) and (-1,3.34) .. (-1,4) --
		(-1,4) .. controls (-1,2) and (1,2) .. (1,4);
	\node[opacity=1] at (0,3) {$\bullet$};
	\path [fill=red] (0,0) ellipse (1 and 0.5);
	\draw (1,0) .. controls (1,-.66) and (-1,-.66) .. (-1,0);
	\draw[dashed] (1,0) .. controls (1,.66) and (-1,.66) .. (-1,0);
	\draw (-1,0) .. controls (-1,2) and (1,2) .. (1,0);
	\path [fill=red, opacity=.3] (1,0) .. controls (1,.66) and (-1,.66) .. (-1,0) --
		(-1,0) .. controls (-1,2) and (1,2) .. (1,0);
\end{tikzpicture}};
\endxy
}

\newcommand{\tubeLHS}[1][1]{
\xy
(0,0)*{
\begin{tikzpicture} [fill opacity=0.2,decoration={markings, mark=at position 0.5 with {\arrow{>}}; }, scale=#1]
	\draw[very thick, postaction={decorate}] (2,0) -- (.75,0);
	\draw[very thick, postaction={decorate}] (-.75,0) -- (-2,0);
	\draw[very thick, postaction={decorate}] (-.75,0) .. controls (-.5,-.5) and (.5,-.5) .. (.75,0);
	\draw[very thick, dashed, postaction={decorate}] (-.75,0) .. controls (-.5,.5) and (.5,.5) .. (.75,0);
	\draw (-2,0) -- (-2,4);
	\draw (2,0) -- (2,4);
	\path [fill=red] (-2,4) -- (-.75,4) -- (-.75,0) -- (-2,0) -- cycle;
	\path [fill=red] (2,4) -- (.75,4) -- (.75,0) -- (2,0) -- cycle;
	\path [fill=blue] (-.75,4) .. controls (-.5,4.5) and (.5,4.5) .. (.75,4) --
		(.75,0) .. controls (.5,.5) and (-.5,.5) .. (-.75,0);
	\path [fill=blue] (-.75,4) .. controls (-.5,3.5) and (.5,3.5) .. (.75,4) --
		(.75,0) .. controls (.5,-.5) and (-.5,-.5) .. (-.75,0);
	\draw [very thick, red, postaction={decorate}] (-.75,0) -- (-.75,4);
	\draw [very thick, red, postaction={decorate}] (.75, 4) -- (.75,0);
	\draw[very thick, postaction={decorate}] (2,4) -- (.75,4);
	\draw[very thick, postaction={decorate}] (-.75,4) -- (-2,4);
	\draw[very thick, postaction={decorate}] (-.75,4) .. controls (-.5,3.5) and (.5,3.5) .. (.75,4);
	\draw[very thick, postaction={decorate}] (-.75,4) .. controls (-.5,4.5) and (.5,4.5) .. (.75,4);
\end{tikzpicture}};
\endxy
}

\newcommand{\tubeRHStopdot}[1][1]{
\xy
(0,0)*{
\begin{tikzpicture} [fill opacity=0.2,decoration={markings, mark=at position 0.5 with {\arrow{>}}; }, scale=#1]
	\draw[very thick, postaction={decorate}] (2,0) -- (.75,0);
	\draw[very thick, postaction={decorate}] (-.75,0) -- (-2,0);
	\draw[very thick, postaction={decorate}] (-.75,0) .. controls (-.5,-.5) and (.5,-.5) .. (.75,0);
	\draw[very thick, dashed, postaction={decorate}] (-.75,0) .. controls (-.5,.5) and (.5,.5) .. (.75,0);
	\draw (-2,0) -- (-2,4);
	\draw (2,0) -- (2,4);
	\path [fill=red] (-2,4) -- (-.75,4) -- (-.75,0) -- (-2,0) -- cycle;
	\path [fill=red] (2,4) -- (.75,4) -- (.75,0) -- (2,0) -- cycle;
	\path [fill=red] (-.75,4) .. controls (-.75,2) and (.75,2) .. (.75,4) --
			(.75, 0) .. controls (.75,2) and (-.75,2) .. (-.75,0);
	\path [fill=blue] (-.75,4) .. controls (-.5,4.5) and (.5,4.5) .. (.75,4) --
			(.75,4) .. controls (.75,2) and (-.75,2) .. (-.75,4);
	\path [fill=blue] (-.75,4) .. controls (-.5,3.5) and (.5,3.5) .. (.75,4) --
			(.75,4) .. controls (.75,2) and (-.75,2) .. (-.75,4);
	\path [fill=blue] (-.75, 0) .. controls (-.75,2) and (.75,2) .. (.75,0) --
			(.75,0) .. controls (.5,-.5) and (-.5,-.5) .. (-.75,0);
	\path [fill=blue] (-.75, 0) .. controls (-.75,2) and (.75,2) .. (.75,0) --
			(.75,0) .. controls (.5,.5) and (-.5,.5) .. (-.75,0);
	\node [opacity=1] at (0,4) {$\bullet$};
	\draw [very thick, red, postaction={decorate}] (.75,4) .. controls (.75,2) and (-.75,2) .. (-.75,4);
	\draw [very thick, red, postaction={decorate}] (-.75, 0) .. controls (-.75,2) and (.75,2) .. (.75,0);
	\draw[very thick, postaction={decorate}] (2,4) -- (.75,4);
	\draw[very thick, postaction={decorate}] (-.75,4) -- (-2,4);
	\draw[very thick, postaction={decorate}] (-.75,4) .. controls (-.5,3.5) and (.5,3.5) .. (.75,4);
	\draw[very thick, postaction={decorate}] (-.75,4) .. controls (-.5,4.5) and (.5,4.5) .. (.75,4);
\end{tikzpicture}};
\endxy
}

\newcommand{\tubeRHSbottomdot}[1][1]{
\xy
(0,0)*{
\begin{tikzpicture} [fill opacity=0.2,decoration={markings, mark=at position 0.5 with {\arrow{>}}; }, scale=#1]
	\draw[very thick, postaction={decorate}] (2,0) -- (.75,0);
	\draw[very thick, postaction={decorate}] (-.75,0) -- (-2,0);
	\draw[very thick, postaction={decorate}] (-.75,0) .. controls (-.5,-.5) and (.5,-.5) .. (.75,0);
	\draw[very thick, dashed, postaction={decorate}] (-.75,0) .. controls (-.5,.5) and (.5,.5) .. (.75,0);
	\draw (-2,0) -- (-2,4);
	\draw (2,0) -- (2,4);
	\path [fill=red] (-2,4) -- (-.75,4) -- (-.75,0) -- (-2,0) -- cycle;
	\path [fill=red] (2,4) -- (.75,4) -- (.75,0) -- (2,0) -- cycle;
	\path [fill=red] (-.75,4) .. controls (-.75,2) and (.75,2) .. (.75,4) --
			(.75, 0) .. controls (.75,2) and (-.75,2) .. (-.75,0);
	\path [fill=blue] (-.75,4) .. controls (-.5,4.5) and (.5,4.5) .. (.75,4) --
			(.75,4) .. controls (.75,2) and (-.75,2) .. (-.75,4);
	\path [fill=blue] (-.75,4) .. controls (-.5,3.5) and (.5,3.5) .. (.75,4) --
			(.75,4) .. controls (.75,2) and (-.75,2) .. (-.75,4);
	\path [fill=blue] (-.75, 0) .. controls (-.75,2) and (.75,2) .. (.75,0) --
			(.75,0) .. controls (.5,-.5) and (-.5,-.5) .. (-.75,0);
	\path [fill=blue] (-.75, 0) .. controls (-.75,2) and (.75,2) .. (.75,0) --
			(.75,0) .. controls (.5,.5) and (-.5,.5) .. (-.75,0);
	\node [opacity=1] at (0,0) {$\bullet$};
	\draw [very thick, red, postaction={decorate}] (.75,4) .. controls (.75,2) and (-.75,2) .. (-.75,4);
	\draw [very thick, red, postaction={decorate}] (-.75, 0) .. controls (-.75,2) and (.75,2) .. (.75,0);
	\draw[very thick, postaction={decorate}] (2,4) -- (.75,4);
	\draw[very thick, postaction={decorate}] (-.75,4) -- (-2,4);
	\draw[very thick, postaction={decorate}] (-.75,4) .. controls (-.5,3.5) and (.5,3.5) .. (.75,4);
	\draw[very thick, postaction={decorate}] (-.75,4) .. controls (-.5,4.5) and (.5,4.5) .. (.75,4);
\end{tikzpicture}};
\endxy
}

\newcommand{\Rfoamone}[1][1]{
\xy
(0,0)*{
\begin{tikzpicture} [fill opacity=0.2,  decoration={markings,
                        mark=at position 0.5 with {\arrow{>}};    }, scale=#1]
     	\filldraw [fill=red] (2,2) rectangle (-2,-1);
      	\draw[ultra thick, red, postaction={decorate}] (-.75,-1) -- (-.75,2);
	\draw[ultra thick, red, postaction={decorate}] (.75,2) -- (.75,-1);
	\path [fill=blue] (-.75,2) -- (.25,1) -- (.25,-2) -- (-.75,-1) -- cycle;
	\path [fill=blue] (.75,2) -- (1.75,1) -- (1.75,-2) -- (.75,-1) -- cycle;
	\draw[very thick, postaction={decorate}] (-.75,-1) -- (.25,-2);
     	\draw[very thick, postaction={decorate}] (1.75,-2) --(.75,-1);
	\draw[very thick, postaction={decorate}] (-1,-2) --(.25,-2);
	\draw[very thick, postaction={decorate}] (1.75,-2) --(3,-2);
	\draw[very thick, postaction={decorate}] (1.75,-2) --(.25,-2);
	\draw[very thick, postaction={decorate}] (-.75,-1) -- (-2,-1);
	\draw[very thick, postaction={decorate}] (2,-1) -- (.75,-1);
	\draw[very thick, postaction={decorate}] (-.75,-1) -- (.75,-1);
     	\filldraw [fill=red] (-1,-2) rectangle (3,1);
     	\draw[ultra thick, red, postaction={decorate}] (.25,1) -- (.25,-2);
	\draw[ultra thick, red, postaction={decorate}] (1.75,-2) -- (1.75,1);
	\draw[very thick, postaction={decorate}] (-.75,2) -- (.25,1);
     	\draw[very thick, postaction={decorate}] (1.75,1) --(.75,2);
	\draw[very thick, postaction={decorate}] (-1,1) --(.25,1);
	\draw[very thick, postaction={decorate}] (1.75,1) --(3,1);
	\draw[very thick, postaction={decorate}] (1.75,1) --(.25,1);
	\draw[very thick, postaction={decorate}] (-.75,2) -- (-2,2);
	\draw[very thick, postaction={decorate}] (2,2) -- (.75,2);
	\draw[very thick, postaction={decorate}] (-.75,2) -- (.75,2);
\end{tikzpicture}};
\endxy
}

\newcommand{\Rfoamtwo}[1][1]{
\xy
(0,0)*{
\begin{tikzpicture} [fill opacity=0.2,  decoration={markings,
                        mark=at position 0.5 with {\arrow{>}};    }, scale=#1]
     	\filldraw [fill=red] (2,2) rectangle (-2,-1);
      	\draw[ultra thick, red, postaction={decorate}] (.25,1) arc (180:360:.75);
     	\draw[ultra thick, red, postaction={decorate}] (.75,2) arc (360:180:.75);
	\path [fill=blue] (1.75,-2) arc (0:60:0.75) -- (0.375,-0.35) arc (60:0:0.75);
      	\path [fill=blue] (0.375,-0.35) arc (60:180:0.75) -- (0.25,-2) arc (180:60:0.75);
	\draw[very thick, postaction={decorate}] (-.75,-1) -- (.25,-2);
     	\draw[very thick, postaction={decorate}] (1.75,-2) --(.75,-1);
	\draw[very thick, postaction={decorate}] (-1,-2) --(.25,-2);
	\draw[very thick, postaction={decorate}] (1.75,-2) --(3,-2);
	\draw[very thick, postaction={decorate}] (1.75,-2) --(.25,-2);
	\draw[very thick, postaction={decorate}] (-.75,-1) -- (-2,-1);
	\draw[very thick, postaction={decorate}] (2,-1) -- (.75,-1);
	\draw[very thick, postaction={decorate}] (-.75,-1) -- (.75,-1);
	\path [fill=blue] (.25,1) arc (180:240:0.75) -- (-0.375,1.35) arc (240:180:0.75);
	\path [fill=blue] (0.625,0.35) arc (240:360:0.75) -- (.75,2) arc (0:-120:0.75);
     	\filldraw [fill=red] (-1,-2) rectangle (3,1);
     	\draw[ultra thick, red, postaction={decorate}] (1.75,-2) arc (0:180:.75);
	\draw[ultra thick, red, postaction={decorate}] (-.75,-1) arc (180:0:.75);
	\draw[very thick, postaction={decorate}] (-.75,2) -- (.25,1);
     	\draw[very thick, postaction={decorate}] (1.75,1) --(.75,2);
	\draw[very thick, postaction={decorate}] (-1,1) --(.25,1);
	\draw[very thick, postaction={decorate}] (1.75,1) --(3,1);
	\draw[very thick, postaction={decorate}] (1.75,1) --(.25,1);
	\draw[very thick, postaction={decorate}] (-.75,2) -- (-2,2);
	\draw[very thick, postaction={decorate}] (2,2) -- (.75,2);
	\draw[very thick, postaction={decorate}] (-.75,2) -- (.75,2);
\end{tikzpicture}};
\endxy
}

\newcommand{\Rfoamthree}[1][1]{
\xy
(0,0)*{
\begin{tikzpicture} [fill opacity=0.2,  decoration={markings,
                        mark=at position 0.5 with {\arrow{>}};    }, scale=#1]
     	\draw [fill=red] (2,2) rectangle (.75,-1);
	\draw [fill=red] (-2,2) rectangle (-.75,-1);
	\draw[very thick, postaction={decorate}] (-.75,-1) -- (.25,-2);
     	\draw[very thick, postaction={decorate}] (1.75,-2) --(.75,-1);
	\draw[very thick, postaction={decorate}] (-1,-2) --(.25,-2);
	\draw[very thick, postaction={decorate}] (1.75,-2) --(3,-2);
	\draw[very thick, postaction={decorate}] (1.75,-2) --(.25,-2);
	\draw[very thick, postaction={decorate}] (-.75,-1) -- (-2,-1);
	\draw[very thick, postaction={decorate}] (2,-1) -- (.75,-1);
	\draw[very thick, postaction={decorate}] (-.75,-1) -- (.75,-1);
     	\draw [fill=red] (-1,-2) rectangle (.25,1);
	\draw [fill=red] (3,-2) rectangle (1.75,1);
	\path [fill=red] (-.75,2) -- (.25,1) -- (.25,-2) -- (-.75,-1) -- cycle;
	\path [fill=red] (.75,2) -- (1.75,1) -- (1.75,-2) -- (.75,-1) -- cycle;
	\draw[ultra thick, red, postaction={decorate}] (-.75,-1) .. controls (-.75,0) and (.25,0) .. (.25,-2);
	\draw[ultra thick, red, postaction={decorate}] (1.75,-2) .. controls (1.75,0) and (.75,0) .. (.75,-1);
	\draw[ultra thick, red, postaction={decorate}] (.25,1) .. controls (.25,0) and (-.75,0) .. (-.75,2);
	\draw[ultra thick, red, postaction={decorate}] (.75,2) .. controls (.75,0) and (1.75,0) .. (1.75,1);
	\path [fill=blue] (.25,-2) .. controls (.25,-.8) and (0,-.55) .. (-.25,-.35) --
			(1.25,-.35) .. controls (1.5,-.55) and (1.75,-.8) .. (1.75,-2);
	\path [fill=blue] (-.75,-1) .. controls (-.75,0) and (0,-.5) .. (-.25,-.35) --
			(1.25,-.35) .. controls (1.5,-.5) and (.75,0) .. (.75,-1);
	\path [fill=blue] (.25,1) .. controls (.25,0) and (-.5,.5) .. (-.25,.35) --
			(1.25,.35) .. controls (1,.5) and (1.75,0) .. (1.75,1);
	\path [fill=blue] (-.25,.35) .. controls (-.5,.55) and (-.75,.8) .. (-.75,2) --
			(.75,2) .. controls (.75,.8) and (1,.55) .. (1.25,.35);
	\draw[very thick, postaction={decorate}] (-.75,2) -- (.25,1);
     	\draw[very thick, postaction={decorate}] (1.75,1) --(.75,2);
	\draw[very thick, postaction={decorate}] (-1,1) --(.25,1);
	\draw[very thick, postaction={decorate}] (1.75,1) --(3,1);
	\draw[very thick, postaction={decorate}] (1.75,1) --(.25,1);
	\draw[very thick, postaction={decorate}] (-.75,2) -- (-2,2);
	\draw[very thick, postaction={decorate}] (2,2) -- (.75,2);
	\draw[very thick, postaction={decorate}] (-.75,2) -- (.75,2);
\end{tikzpicture}};
\endxy
}

\newcommand{\Efoam}[1][1]{
\xy
(0,0)*{
\begin{tikzpicture} [fill opacity=0.2,  decoration={markings,
                        mark=at position 0.6 with {\arrow{>}};    }, scale=#1]
     	\filldraw [fill=red] (-2,0) rectangle (2,3);
      	\draw[very thick] (-2,3) -- (2,3);
      	\draw[very thick] (-1,-1) -- (3,-1);
      	\draw[ultra thick, red] (0,0) -- (0,3);
     	\path [fill=blue] (1,-1) -- (1,2) -- (0,3) -- (0,0) -- cycle;
    	\draw[very thick, postaction={decorate}] (1,-1) -- (0,0);
    	\draw[very thick, postaction={decorate}] (1,2) -- (0,3);
     	\filldraw [fill=red] (-1,-1) rectangle (3,2);
     	\draw[very thick] (-1,2) -- (3,2);
     	\draw[very thick] (-2,0) -- (2,0);
     	\draw[ultra thick, red] (1,-1) -- (1,2);
\end{tikzpicture}};
\endxy
}

\newcommand{\dotEfoam}[1][1]{
\xy
(0,0)*{
\begin{tikzpicture} [fill opacity=0.2,  decoration={markings,
                        mark=at position 0.6 with {\arrow{>}};    }, scale=#1]
     	\filldraw [fill=red] (-2,0) rectangle (2,3);
      	\draw[very thick] (-2,3) -- (2,3);
      	\draw[very thick] (-1,-1) -- (3,-1);
      	\draw[ultra thick, red] (0,0) -- (0,3);
	\node [opacity=1]  at (0.5,1) {$\bullet$};
	\path [fill=blue] (1,-1) -- (1,2) -- (0,3) -- (0,0) -- cycle;
      	\draw[very thick, postaction={decorate}] (1,-1) -- (0,0);
      	\draw[very thick, postaction={decorate}] (1,2) -- (0,3);
     	\filldraw [fill=red] (-1,-1) rectangle (3,2);
     	\draw[very thick] (-1,2) -- (3,2);
     	\draw[very thick] (-2,0) -- (2,0);
     	\draw[ultra thick, red] (1,-1) -- (1,2);
\end{tikzpicture}};
\endxy
}

\newcommand{\crossingEEfoam}[1][1]{
\xy
(0,0)*{
\begin{tikzpicture} [fill opacity=0.2,  decoration={markings,
                        mark=at position 0.6 with {\arrow{>}};    }, scale=#1]
     	\filldraw [fill=red] (2,2) rectangle (-2,-1);
     	\draw[very thick](-2,2)--(2,2);
      	\draw[ultra thick, red] (-.75,2) arc (180:360:.75);
	\draw[very thick] (-2,-1)--(2,-1);
     	\draw[ultra thick, red] (-.75,-1) arc (180:0:.75);
	\path [fill=blue] (1.75,-2) arc (0:60:0.75) -- (0.375,-0.35) arc (60:0:0.75);
      	\path [fill=blue] (0.375,-0.35) arc (60:180:0.75) -- (0.25,-2) arc (180:60:0.75);
	\draw[very thick, postaction={decorate}] (.25,-2) -- (-.75,-1);
     	\draw[very thick, postaction={decorate}] (1.75,-2) --(.75,-1);
	\filldraw[ultra thick, red] [fill=blue] (1,-1.25) -- (0,-0.25) -- (0,1.25) -- (1,0.25) -- cycle;
	\path [fill=blue] (.25,1) arc (180:240:0.75) -- (-0.375,1.35) arc (240:180:0.75);
	\path [fill=blue] (0.625,0.35) arc (240:360:0.75) -- (.75,2) arc (0:-120:0.75);
     	\filldraw [fill=red] (-1,-2) rectangle (3,1);
     	\draw[very thick] (-1,1)--(3,1);
     	\draw[very thick] (-1,-2)--(3,-2);
     	\draw[ultra thick, red] (.25,-2) arc (180:0:.75);
	\draw[ultra thick, red] (.25,1) arc (180:360:.75);
     	\draw[very thick, postaction={decorate}] (.25,1) -- (-.75,2);
     	\draw[very thick, postaction={decorate}] (1.75,1) --(.75,2);
\end{tikzpicture}};
\endxy
}

\newcommand{\crossingEoneEtwofoam}[1][1]{
\xy
(0,0)*{
\begin{tikzpicture}  [fill opacity=0.2,  decoration={markings, mark=at position 0.6 with {\arrow{>}};}, scale=#1]
	\filldraw [fill=red] (-3,1) rectangle (1,4);
	\draw[very thick] (-3,1) -- (1,1);
	\draw[very thick] (-3,4) -- (1,4);
	\draw [very thick, red] (-1.75,4) .. controls (-1.75,3) and (-.25,2) .. (-.25,1);
	\draw[very thick, postaction={decorate}] (-.75,3) -- (-1.75,4);
	\draw[very thick, postaction={decorate}] (.75,0) -- (-.25,1);
	\path [fill=blue] (-1.75,4) .. controls (-1.75,3) and (-.25,2) .. (-.25,1) --
			(.75,0) .. controls (.75,1) and (-.75,2) .. (-.75,3);
	\filldraw [fill=red] (-2,0) rectangle (2,3);
	\draw[very thick] (-2,0) -- (2,0);
	\draw[very thick] (-2,3) -- (2,3);
	\draw [very thick, red] (-.75,3) .. controls (-.75,2) and (.75,1) .. (.75,0);
	\draw [very thick, red] (.75,3) .. controls (.75,2) and (-.75,1) .. (-.75,0);
	\draw[very thick, postaction={decorate}] (1.75,2) -- (.75,3);
	\draw[very thick, postaction={decorate}] (.25,-1) -- (-.75,0);
	\path [fill=blue] (.75,3) .. controls (.75,2) and (-.75,1) .. (-.75,0) --
			(.25,-1) .. controls (.25,0) and (1.75,1) .. (1.75,2);
	\filldraw [fill=red] (-1,-1) rectangle (3,2);
	\draw[very thick] (-1,-1) -- (3,-1);
	\draw[very thick] (-1,2) -- (3,2);
	\draw[very thick, red] (.25,-1) .. controls (.25,0) and (1.75,1) .. (1.75,2);
\end{tikzpicture}}
\endxy
}

\newcommand{\crossingEtwoEonefoam}[1][1]{
\xy
(0,0)*{
\begin{tikzpicture}  [fill opacity=0.2,  decoration={markings, mark=at position 0.6 with {\arrow{>}};}, scale=#1]
	\filldraw [fill=red] (-3,1) rectangle (1,4);
	\draw[very thick] (-3,1) -- (1,1);
	\draw[very thick] (-3,4) -- (1,4);
	\draw [very thick, red] (-.25,4) .. controls (-.25,3) and (-1.75,2) .. (-1.75,1);
	\draw[very thick, postaction={decorate}] (.75,3) -- (-.25,4);
	\draw[very thick, postaction={decorate}] (-.75,0) -- (-1.75,1);
	\path [fill=blue] (-.25,4) .. controls (-.25,3) and (-1.75,2) .. (-1.75,1) --
			(-.75,0) .. controls (-.75,1) and (.75,2) .. (.75,3);
	\filldraw [fill=red] (-2,0) rectangle (2,3);
	\draw[very thick] (-2,0) -- (2,0);
	\draw[very thick] (-2,3) -- (2,3);
	\draw [very thick, red] (-.75,3) .. controls (-.75,2) and (.75,1) .. (.75,0);
	\draw [very thick, red] (.75,3) .. controls (.75,2) and (-.75,1) .. (-.75,0);
	\draw[very thick, postaction={decorate}] (.25,2) -- (-.75,3);
	\draw[very thick, postaction={decorate}] (1.75,-1) -- (.75,0);
	\path [fill=blue] (-.75,3) .. controls (-.75,2) and (.75,1) .. (.75,0) --
			(1.75,-1) .. controls (1.75,0) and (.25,1) .. (.25,2);
	\filldraw [fill=red] (-1,-1) rectangle (3,2);
	\draw[very thick] (-1,-1) -- (3,-1);
	\draw[very thick] (-1,2) -- (3,2);
	\draw[very thick, red] (1.75,-1) .. controls (1.75,0) and (.25,1) .. (.25,2);
\end{tikzpicture}}
\endxy
}

\newcommand{\crossingEthreeEonefoam}[1][1]{
\xy
(0,0)*{
\begin{tikzpicture} [fill opacity=0.2,decoration={markings, mark=at position 0.6 with {\arrow{>}}; }, scale=#1]
	\filldraw [fill=red] (-4,2) rectangle (0,6);
	\draw[very thick] (-4,2) -- (0,2);
	\draw[very thick] (-4,6) -- (0,6);
	\draw[very thick, red] (-1.25,6) .. controls (-1.25,5) and (-2.75,3) .. (-2.75,2);
	\draw[very thick, postaction={decorate}] (-1.75,1) -- (-2.75,2);
	\draw[very thick, postaction={decorate}] (-.25,5) -- (-1.25,6);
	\path [fill=blue] (-1.25,6) .. controls (-1.25,5) and (-2.75,3) .. (-2.75,2) --
		 (-1.75,1) .. controls (-1.75,2) and (-.25,4) .. (-.25,5);
	\filldraw [fill=red] (-3,1) rectangle (1,5);
	\draw[very thick] (-3,1) -- (1,1);
	\draw[very thick] (-3,5) -- (1,5);
	\draw[very thick, red] (-1.75,1) .. controls (-1.75,2) and (-.25,4) .. (-.25,5);
	\filldraw [fill=red] (-2,0) rectangle (2,4);
	\draw[very thick] (-2,0) -- (2,0);
	\draw[very thick] (-2,4) -- (2,4);
	\draw[very thick, red] (-.75,4) .. controls (-.75,3) and (.75,1) .. (.75,0);
	\draw[very thick, postaction={decorate}] (.25,3) -- (-.75,4);
	\draw[very thick, postaction={decorate}] (1.75,-1) -- (.75,0);
	\path [fill=blue] (-.75,4) .. controls (-.75,3) and (.75,1) .. (.75,0) --
		(1.75,-1) .. controls (1.75,0) and (.25,2) .. (.25,3);
	\filldraw [fill=red] (-1,-1) rectangle (3,3);
	\draw[very thick] (-1,-1) -- (3,-1);
	\draw[very thick] (-1,3) -- (3,3);
	\draw[very thick, red] (1.75,-1) .. controls (1.75,0) and (.25,2) .. (.25,3);
\end{tikzpicture}}
\endxy
}

\newcommand{\crossingEoneEthreefoam}[1][1]{
\xy
(0,0)*{
\begin{tikzpicture} [fill opacity=0.2,decoration={markings, mark=at position 0.6 with {\arrow{>}}; }, scale=#1]
	\filldraw [fill=red] (-4,2) rectangle (0,6);
	\draw[very thick] (-4,2) -- (0,2);
	\draw[very thick] (-4,6) -- (0,6);
	\draw[very thick, red] (-2.75,6) .. controls (-2.75,5) and (-1.25,3) .. (-1.25,2);
	\draw[very thick, postaction={decorate}] (-.25,1) -- (-1.25,2);
	\draw[very thick, postaction={decorate}] (-1.75,5) -- (-2.75,6);
	\path [fill=blue] (-2.75,6) .. controls (-2.75,5) and (-1.25,3) .. (-1.25,2) --
		 (-.25,1) .. controls (-.25,2) and (-1.75,4) .. (-1.75,5);
	\filldraw [fill=red] (-3,1) rectangle (1,5);
	\draw[very thick] (-3,1) -- (1,1);
	\draw[very thick] (-3,5) -- (1,5);
	\draw[very thick, red] (-.25,1) .. controls (-.25,2) and (-1.75,4) .. (-1.75,5);
	\filldraw [fill=red] (-2,0) rectangle (2,4);
	\draw[very thick] (-2,0) -- (2,0);
	\draw[very thick] (-2,4) -- (2,4);
	\draw[very thick, red] (.75,4) .. controls (.75,3) and (-.75,1) .. (-.75,0);
	\draw[very thick, postaction={decorate}] (1.75,3) -- (.75,4);
	\draw[very thick, postaction={decorate}] (.25,-1) -- (-.75,0);
	\path [fill=blue] (.75,4) .. controls (.75,3) and (-.75,1) .. (-.75,0) --
		(.25,-1) .. controls (.25,0) and (1.75,2) .. (1.75,3);
	\filldraw [fill=red] (-1,-1) rectangle (3,3);
	\draw[very thick] (-1,-1) -- (3,-1);
	\draw[very thick] (-1,3) -- (3,3);
	\draw[very thick, red] (.25,-1) .. controls (.25,0) and (1.75,2) .. (1.75,3);
\end{tikzpicture}}
\endxy
}

\newcommand{\cupFEfoam}[1][1]{
\xy
(0,0)*{
\begin{tikzpicture} [fill opacity=0.2,  decoration={markings,
                        mark=at position 0.6 with {\arrow{>}};    }, scale=#1]
     	\filldraw [fill=red] (-2,0) rectangle (2,3);
     	\draw[very thick](-2,3)--(2,3);
     	\draw[very thick] (-2,0)--(2,0);
      	\draw[ultra thick, red] (-.75,3) arc (180:360:.75);
	\path [fill=blue] (.25,2) arc (180:240:0.75) -- (-0.375,2.35) arc (240:180:0.75);
	\path [fill=blue] (0.625,1.35) arc (240:360:0.75) -- (.75,3) arc (0:-120:0.75);
     	\filldraw [fill=red] (-1,-1) rectangle (3,2);
     	\draw[very thick] (-1,-1)--(3,-1);
     	\draw[very thick] (-1,2)--(3,2);
     	\draw[very thick, postaction={decorate}] (-.75,3) -- (.25,2);
     	\draw[very thick, postaction={decorate}] (1.75,2) --(.75,3);
       	\draw[ultra thick, red] (.25,2) arc (180:360:.75);
\end{tikzpicture}};
\endxy
}

\newcommand{\cupEFfoam}[1][1]{
\xy
(0,0)*{
\begin{tikzpicture} [fill opacity=0.2,  decoration={markings,
                        mark=at position 0.6 with {\arrow{>}};    }, scale=#1]
     	\filldraw [fill=red] (-2,0) rectangle (2,3);
     	\draw[very thick](-2,3)--(2,3);
     	\draw[very thick] (-2,0)--(2,0);
      	\draw[ultra thick, red] (-.75,3) arc (180:360:.75);
	\path [fill=blue] (.25,2) arc (180:240:0.75) -- (-0.375,2.35) arc (240:180:0.75);
	\path [fill=blue] (0.625,1.35) arc (240:360:0.75) -- (.75,3) arc (0:-120:0.75);
     	\filldraw [fill=red] (-1,-1) rectangle (3,2);
     	\draw[very thick] (-1,-1)--(3,-1);
     	\draw[very thick] (-1,2)--(3,2);
     	\draw[very thick, postaction={decorate}] (.25,2) -- (-.75,3) ;
     	\draw[very thick, postaction={decorate}] (.75,3) -- (1.75,2);
       	\draw[ultra thick, red] (.25,2) arc (180:360:.75);
\end{tikzpicture}};
\endxy
}

\newcommand{\capFEfoam}[1][1]{
\xy
(0,0)*{
\begin{tikzpicture} [fill opacity=0.2,  decoration={markings,
                        mark=at position 0.6 with {\arrow{>}};    }, scale=#1]
     	\filldraw [fill=red] (-2,0) rectangle (2,3);
     	\draw[very thick](-2,3)--(2,3);
	\draw[very thick] (-2,0)--(2,0);
     	\draw[ultra thick, red] (-.75,0) arc (180:0:.75);
	\path [fill=blue] (1.75,-1) arc (0:60:0.75) -- (0.375,0.65) arc (60:0:0.75);
      	\path [fill=blue] (0.375,0.65) arc (60:180:0.75) -- (0.25,-1) arc (180:60:0.75);
     	\filldraw [fill=red] (-1,-1) rectangle (3,2);
     	\draw[very thick] (-1,-1)--(3,-1);
     	\draw[very thick] (-1,2)--(3,2);
     	\draw[ultra thick, red] (.25,-1) arc (180:0:.75);
     	\draw[very thick, postaction={decorate}] (-.75,0) -- (.25,-1);
     	\draw[very thick, postaction={decorate}] (1.75,-1) --(.75,0);
	\end{tikzpicture}};
\endxy
}

\newcommand{\capEFfoam}[1][1]{
\xy
(0,0)*{
 \begin{tikzpicture} [fill opacity=0.2,  decoration={markings,
                        mark=at position 0.6 with {\arrow{>}};    }, scale=#1]
     	\filldraw [fill=red] (-2,0) rectangle (2,3);
     	\draw[very thick](-2,3)--(2,3);
     	\draw[very thick] (-2,0)--(2,0);
     	\draw[ultra thick, red] (-.75,0) arc (180:0:.75);
	\path [fill=blue] (1.75,-1) arc (0:60:0.75) -- (0.375,0.65) arc (60:0:0.75);
       	\path [fill=blue] (0.375,0.65) arc (60:180:0.75) -- (0.25,-1) arc (180:60:0.75);
     	\filldraw [fill=red] (-1,-1) rectangle (3,2);
     	\draw[very thick] (-1,-1)--(3,-1);
     	\draw[very thick] (-1,2)--(3,2);
     	\draw[ultra thick, red] (.25,-1) arc (180:0:.75);
     	\draw[very thick, postaction={decorate}] (.25,-1) -- (-.75,0) ;
     	\draw[very thick, postaction={decorate}] (.75,0) -- (1.75,-1);
\end{tikzpicture}};
\endxy
}

\begin{document}
%

\title[Khovanov homology is a $2$-representation]
{Khovanov homology is a skew Howe $2$-representation of
categorified quantum $\mathfrak{sl}_m$}

\author{Aaron D. Lauda}
\address{Department of Mathematics, University of Southern California, Los Angeles, CA 90089, USA}
\email{lauda@usc.edu}

\author{Hoel Queffelec}
\address{Institut de Math\'ematiques de Jussieu, Universit\'e Paris Diderot, 175 rue du Chevaleret, 75013 Paris, France}
\email{queffelec@math.jussieu.fr}

\author{David E. V. Rose}
\address{Department of Mathematics, University of Southern California, Los Angeles, CA 90089, USA}
\email{davidero@usc.edu}

\begin{abstract}
We show that Khovanov homology (and its $\mathfrak{sl}_3$ variant) can be
understood in the context of higher representation theory. Specifically, we show that the combinatorially defined
foam constructions of these theories arise as a family of $2$-representations of categorified quantum
$\mathfrak{sl}_m$ via categorical skew Howe duality.
Utilizing Cautis-Rozansky categorified clasps we also obtain a unified construction of foam-based categorifications of
Jones-Wenzl projectors and their $\mf{sl}_3$ analogs purely from the higher representation theory of
categorified quantum groups.  In the $\mf{sl}_2$ case, this work reveals the importance of a modified class of
foams introduced by Christian Blanchet which in turn suggest a similar modified version of the
$\mf{sl}_3$ foam category introduced here.
\end{abstract}

\maketitle
\setcounter{tocdepth}{2}
\tableofcontents

%
\section{Introduction}
%


\subsection{Categorified knot invariants and quantum groups}

One of the original motivations for categorifying quantum groups was to provide a representation theoretic explanation for the existence of Khovanov homology and other link homologies categorifying quantum link invariants. Just as the Jones polynomial is described representation theoretically by the quantum group $U_q(\mf{sl}_2)$ and tensor powers of its two dimensional representation, the categorification of the Jones polynomial via Khovanov homology should be described in terms of the 2-representation theory of the categorified quantum group associated to $U_q(\mf{sl}_2)$.

Currently, the link between categorified quantum groups and Khovanov homology follows the indirect path through Webster's work on categorified tensor products~\cite{Web,Web2}. This connection utilizes an isomorphism relating Webster's categorifications of tensor products with categories associated to blocks of parabolic graded category $\cal{O}$.  Categorifications associated with category $\cal{O}$ were initiated by Bernstein, Frenkel, and Khovanov~\cite{BFK} and were further developed in \cite{Strop2,FKS}.  The relation to the familiar picture-world~\cite{BN1,BN2} of Khovanov homology then relies on several technical results of Stroppel~\cite{Strop1,Strop2} relating the knot homologies constructed using category $\cal{O}$ to Khovanov's more elementary construction~\cite{Kh1,Kh2}.  More generally, for link homology theories associated with fundamental $\mf{sl}_n$ representations, Webster describes an isomorphism relating his construction to Sussan's category $\cal{O}$ based link homology theory~\cite{Sussan}
,
which is related via Koszul duality to a theory defined by Mazorchuk and Stroppel~\cite{MS2}.  When $n=3$, the latter of these link homologies can then be identified~\cite{MS2} with Khovanov's more elementary construction~\cite{Kh5} of $\mf{sl}_3$ link homology defined using singular cobordisms called foams.

Alternatively, there is an algebro-geometric construction of Khovanov homology and related $\mf{sl}_n$ link homologies due to Cautis and Kamnitzer~\cite{CK01,CK02}.  These knot homologies arise from derived categories of coherent sheaves on algebraic varieties associated to orbits in the affine Grassmannian.  In the $\mf{sl}_2$ case this knot homology agrees with Khovanov homology~\cite[Theorem 8.2]{CK01} and these geometric categories can be understood as 2-representations of categorified quantum groups~\cite{CLau,CKL2}.  These link homologies are related to those of Seidel-Smith~\cite{SeSm} and Manolescu~\cite{Man} by mirror symmetry.

In this article, we provide a direct construction of foam based $\mf{sl}_n$ link homology theories for $n=2$ or $n=3$ intrinsically in terms of categorified quantum groups.  We show that all of the components involved in these knot homologies are already present within the structure of categorified quantum groups including the relations in foam categories and the complexes defining the braiding.  Utilizing Cautis-Rozansky categorified
clasps~\cite{Cautis,Roz} we also obtain categorified projectors lifting Jones-Wenzl idempotents and their $\mf{sl}_3$ analogs purely from the higher relations of categorified quantum groups.  In the $\mf{sl}_2$ case this work reveals the importance of a modified class of foams introduced by Christian Blanchet~\cite{Blan}, suggesting that this version of the foam category is most natural from the perspective of categorified quantum groups.  In the $\mf{sl}_3$ case these results suggest a similar modified version of the $\mf{sl}_3$ foam category.

\subsection{Categorified representation theory}
Recall that in categorified representation theory,  $\C(q)$-vector spaces $V$ with decompositions into weight spaces $V = \oplus_{\lambda}V_{\lambda}$, are replaced by graded categories $\mathcal{V} = \oplus_{\lambda} \mathcal{V}_{\lambda}$, and instead of linear maps between spaces, Chevalley generators act by functors  $\mathcal{E}_i\onel \maps \mathcal{V}_{\lambda} \to \mathcal{V}_{\lambda+\alpha_i}$, $\mathcal{F}_i\onel \maps \mathcal{V}_{\lambda} \to \mathcal{V}_{\lambda-\alpha_i}$ satisfying quantum Serre relations up to isomorphism of functors.  The higher structure of categorified representation theory appears at the level of natural transformations between these functors.  In most instances when $U_q(\mf{sl}_n)$ admits a categorical action of this form, the natural transformations that appear between functors are predictable and can be systematically described.  A key part of this structure is that $\cal{F}$ is a left and right adjoint for $\cal{E}$ and that the endomorphisms of $\cal{E}^a$ are
acted upon by the so called KLR algebras developed in ~\cite{CR,KL,KL2,Rou2}.

In \cite{Lau1,KL3} it was suggested that the full structure of categorical representations of $U_q(\sln)$ is described by a 2-functor from an additive 2-category $\UcatD_Q(\mf{sl}_n)$.  This 2-category categorifies Lusztig's modified version $\U_q(\mf{sl}_n)$~\cite{Lus1} of the quantum group $U_q(\sln)$.
The objects of $\UcatD_Q(\sln)$ are indexed by the weight lattice of $\U_q(\sln)$, 1-morphisms correspond to the elements of $\U_q(\mf{sl}_n)$, and the 2-morphisms govern the natural transformations that appear in categorical representations.  However, the 2-category $\UcatD_Q(\sln)$ has additional relations on 2-morphisms beyond specified adjoints and KLR relations.  We refer to the collection of relations on 2-morphisms as ``higher relations" because they can be viewed as replacements for the quantum Serre relations.  Indeed, these higher relations give rise to explicit isomorphisms lifting the defining relations in $\U_q(\mf{sl}_n)$, while simultaneously controlling the Grothendieck group of $\UcatD_Q(\mf{sl}_n)$, allowing for a $\Z[q,q^{-1}]$-algebra isomorphism between its split Grothendieck ring $K_0(\UcatD_Q(\mf{sl}_n))$ and the integral version of $\U_q(\mf{sl}_n)$.  Under this isomorphism, the images of indecomposable 1-morphisms from $\UcatD_Q(\sln)$ map to the canonical basis of
$\U_q(\mf{sl}_n)$~\cite{Lau1,Web4}.

Here we show that these higher relations also encode the information needed to construct {\em all} $\mf{sl}_2$ and $\mf{sl}_3$ knot homology theories in a framework where computations are accessible.

\subsection{Braidings via skew Howe duality.}
The key insight for our elementary construction of knot homologies from categorified quantum groups is the fundamental observation of Cautis, Kamnitzer and Licata that the $R$-matrix describing the braiding in an $m$-fold tensor product of fundamental representations of $U_q(\mf{sl}_n)$ in Reshetikhin-Turaev link invariants can be obtained from a deformed Weyl group action associated with $U_q(\slm)$~\cite{CKL}.

Recall that the Weyl group $W$ of a simply-laced Kac-Moody algebra $\mathfrak{g}$ is a finite Coxeter group associated to the root system of $\mathfrak{g}$.   Passing from $U(\mathfrak{g})$ to $U_q(\mathfrak{g})$, the Weyl group deforms to a braid group of type $\mf{g}$, which acts on $U_q(\mathfrak{g})$-modules.
In the simplest case of $\mathfrak{g}=\mathfrak{sl}_2$, the Weyl group $W=\mathfrak{S}_2$
deforms to the braid group $B_2$ giving a reflection isomorphism $T \maps V_{\lambda} \to V_{-\lambda}$ between weight spaces of a $U_q(\mathfrak{sl}_2)$-module. This action  can be expressed in a completion of the idempotented quantum algebra $\U_q(\mf{sl}_2)$ by the power series
\begin{equation} \label{eq_def_tau}
 T 1_{\lambda} = \sum_{s \geq 0} (-q)^s F^{(\lambda+s)}E^{(s)}1_{\lambda} \quad \text{$\lambda\geq 0$}, \qquad \quad
 T 1_{\lambda} =\sum_{s\geq 0} (-q)^s  E^{(-\lambda+s)}F^{(s)}1_{\lambda} \quad \text{$\lambda\leq 0$}.
\end{equation}
On any finite-dimensional representation, $T 1_{\lambda}$ can be expressed as a finite sum.  When $\mf{g}= \slm$ and $W=\mathfrak{S}_m$, there are analogous maps $T_i1_{\lambda}$ for each $1 \leq i \leq m-1$ satisfying
the braid relations.

Cautis, Kamnitzer, and Licata related the braiding of fundamental $U_q(\sln)$ representations to the Weyl group action using a version of Howe duality for exterior algebras they called skew Howe duality~\cite{CKL}.    The key idea is to study quantum exterior powers.  Denote by $\C^n_q$ the standard $\U_q(\sln)$-module with basis denoted $x_1, \dots, x_n$.  The quantum exterior algebra is the $\U_q(\sln)$-module defined as
\[
\Alt{q}{\bullet}(\C^n_q)=\C(q)\langle x_1,\dots,x_n\rangle/(x_i^2,\ x_ix_j+qx_jx_i \; \; \text{for  $i<j$}).
\]
By assigning degree one to each $x_i$ the quantum exterior algebra is a graded $\U_q(\sln)$-module whose homogenous subspace of degree $N$ is denoted by $\Alt{q}{N}(\C^n_q)$.

The space $\Alt{q}{N}(\C^n_q\otimes \C^m_q)$ admits commuting actions of $\U_q(\slm)$ and $\U_q(\sln)$
which constitute a Howe pair. For example, when $m=2$ the space $\Alt{q}{N}(\C^n_q \otimes \C^2_q)$ decomposes into $\U_q(\mf{sl}_2)$ weight spaces as
\[
\Alt{q}{N}(\C^n_q \otimes \C^2_q) \cong \Alt{q}{N}(\C^n_q \oplus \C^n_q) \cong
\bigoplus_{a+b=N} \Alt{q}{a}(\C^n_q) \otimes \Alt{q}{b}(\C^n_q),
\]
where the weight of a summand $\Alt{q}{a}(\C^n_q) \otimes \Alt{q}{b}(\C^n_q)$ is $\lambda=b-a$. The action of $\U_q(\mf{sl}_2)$ is given by maps
\begin{align} \label{eqn_Esplit}
 E1_{\lambda} &\maps \Alt{q}{a}(\C^n_q) \otimes \Alt{q}{b}(\C^n_q) \to
 \Alt{q}{a-1}(\C^n_q) \otimes \Alt{q}{b+1}(\C^n_q) , \\
 F1_{\lambda} &\maps \Alt{q}{a}(\C^n_q) \otimes \Alt{q}{b}(\C^n_q) \to
 \Alt{q}{a+1}(\C^n_q) \otimes \Alt{q}{b-1}(\C^n_q). \nn
 \end{align}
For more details on quantum skew Howe duality see ~\cite{CKM,Cautis}.

The Weyl group action gives an isomorphism between the $\lambda$th and $-\lambda$th weight spaces of  $\Alt{q}{N}(\C^n_q \otimes \C^2_q)$.
 \[
 \xy (0,10)*{};
  (16,0)*+{\scs \Alts{q}{a}\scs(\C^n_q) \otimes \Alts{q}{b}\scs(\C^n_q)}="3";
  (48,0)*+{\scs \Alts{q}{a-1}\scs(\C^n_q) \otimes \Alts{q}{b+1}\scs(\C^n_q)}="4";
  (68,0)*{\cdots};
  (0,0)*{\cdots};
    {\ar@/^1.2pc/^{\scs E} "3";"4"};
    {\ar@/^1.2pc/^{\scs F} "4";"3"};
  (-16,0)*+{\scs \Alts{q}{b}\scs(\C^n_q) \otimes \Alts{q}{a}\scs(\C^n_q)}="3'";
  (-48,0)*+{\scs \Alts{q}{b+1}\scs(\C^n_q) \otimes \Alts{q}{a-1}\scs(\C^n_q)}="4'";
  (-68,0)*{\cdots};
    {\ar@/^1.2pc/^{\scs F} "3'";"4'"};
    {\ar@/^1.2pc/^{\scs E} "4'";"3'"};
    (-18,-8)*{ \textcolor[rgb]{0.60,0.00,0.00}{-\lambda}};
    (18,-8)*{ \textcolor[rgb]{0.60,0.00,0.00}{\lambda}};
    (-48,-8)*{ \textcolor[rgb]{0.60,0.00,0.00}{-\lambda-2}};
    (48,-8)*{ \textcolor[rgb]{0.60,0.00,0.00}{\lambda+2}};
    \textcolor[rgb]{0.60,0.00,0.00}{{\ar@{<->}@/_2.3pc/^{T} "3"; "3'"}}
 \endxy
\]

Since $\C^n_q$ is the defining representation of $\U_q(\mf{sl}_n)$, the quantum exterior powers $\Alt{q}{a}(\C^n_q) = V_{\omega_a}$ correspond to fundamental $\U_q(\mf{sl}_n)$-representations, where $\omega_a$ for $1 \leq a \leq n-1$ are the fundamental weights of $\mf{sl}_n$.  The deformed reflection isomorphism
\[
 \xy
  (-40,0)*+{V_{\omega_a} \otimes V_{\omega_b} \cong \Alt{q}{a}(\C^n_q) \otimes \Alt{q}{b}(\C^n_q)}="1";
  (40,0)*+{\Alt{q}{b}(\C^n_q) \otimes \Alt{q}{a}(\C^n_q)\cong V_{\omega_b}\otimes V_{\omega_a}.}="2";
  {\ar^{T} "1";"2"}
 \endxy
\]
gives a braiding of fundamental representations that agrees with the $R$-matrix in the Reshetikhin-Turaev construction~\cite{CKL} (up to a power of $\pm q$).  The key advantage of this realization of the $R$-matrix in terms of skew Howe duality is that it suggests a procedure for categorification.

\subsection{Knot homology from categorical skew Howe duality}

Following the ideas of Chuang and Rouquier~\cite{CR} (see also ~\cite{CK}), one can define a categorification of the reflection isomorphism $T1_{\lambda} \maps V_{\lambda} \to V_{-\lambda}$ using the 2-category $\UcatD_Q(\mf{sl}_2)$ categorifying $\U_q(\mf{sl}_2)$.   Passing to the category of complexes $Kom(\UcatD_Q(\mf{sl}_2))$, it is possible to define a complex $\cal{T}\onel$ of 1-morphisms
\begin{equation}
 \xy
 (-46,0)*+{\scs  \cal{E}^{(-\lambda)}\onel}="1";
 (-22,0)*+{\scs  \cal{E}^{(-\lambda+1)}\cal{F}\onel\{ 1 \}}="2";
 (15,0)*+{\scs \cal{E}^{(-\lambda+2)}\cal{F}^{(2)}\onel \{ 2\}}="3";
 (68,0)*+{\scs \cal{E}^{(-\lambda+k)}\cal{F}^{(k)}\onel \{ k \}}="4";
 {\ar^-{d_1} "1";"2"};
 {\ar^-{d_2} "2";"3"};
 {\ar^-{} "3";(36,0)*{}};
  (42,0)*{\cdots};
 {\ar (48,0)*{};"4"};
 {\ar^-{d_{k+1}} "4"; (92,0)*{}};
 (98,0)*{\cdots};
 \endxy
\end{equation}
for $\lambda\leq 0$ and a similar complex for $\lambda\geq 0$ (compare with \eqref{eq_def_tau}).
The differentials in this complex can be explicitly defined using the 2-morphisms in $\UcatD_Q(\mathfrak{sl}_2)$. Verification that $d^2=0$ follows from the relations in the 2-category $\UcatD_Q(\mf{sl}_2)$; the enhanced graphical calculus from \cite{KLMS} is useful for this computation.

Given a 2-representation $\mathcal{V}$ of the 2-category $\UcatD_Q(\mathfrak{sl}_2)$ with weight decomposition into abelian categories $\mathcal{V}_{\lambda}$, the functor of tensoring with the complex $\cal{T}\onel$ gives rise to derived equivalences $\cal{T}\onel \maps D(\mathcal{V}_{\lambda}) \to D(\mathcal{V}_{-\l})$.
The resulting derived equivalences are highly non-trivial and have led to the resolution of several important conjectures~\cite{CR,CKL2,CKL3}.  Our interest in these equivalences stems from their application to knot homology theory. Given a categorification of $\Alt{q}{N}(\C^n_q\otimes \C^2_q)$ with commuting categorical actions of $\UcatD_Q(\mf{sl}_n)$ and $\UcatD_Q(\mf{sl}_2)$,  the categorified braid group action gives a categorification of the $R$-matrix. More generally, one can categorify the braid group action on an $m$-fold tensor product of $U_q(\sln)$
representation using the categorified braid group action coming from the deformed
Weyl group action of $\U_q(\slm)$~\cite{CK}.

In fact, Cautis and Kamnitzer's  algebro-geometric construction of Khovanov homology~\cite{CK01} and $\mf{sl}_n$ link homology~\cite{CK02} can be understood in this framework.  Their invariants arise from a categorification of $\Alt{q}{N}(\C^n_q\otimes \C^m_q)$ defined from derived categories of coherent sheaves on varieties related to orbits in the affine Grassmannian~\cite[Theorem 2.6]{Cautis}.

\subsection{Reinterpreting $\mf{sl}_n$ skein theory using skew Howe duality.}

While categorifications of $\Alt{q}{N}(\C^n_q\otimes \C^m_q)$ defined via derived categories of coherent sheaves are far from elementary,
it turns out that this story has a more combinatorial description.
In the decategorified case, the usual skein theory description of $\mf{sl}_n$ link invariants in terms of
MOY-calculus~\cite{MOY} can also be understood in terms of skew Howe duality.

Recall that an $\mf{sl}_n$ web is a graphical presentation of intertwiners between tensor products of
fundamental representations of $U_q(\mf{sl}_n)$.
When $n=2$, the calculus of $\mf{sl}_2$ webs is described by the Temperley-Lieb algebra; Kuperberg
described the $n=3$ case using a graphical calculus of oriented trivalent graphs~\cite{Kup} which depict the
morphisms in a combinatorially defined pivotal category called the $\mf{sl}_3$ spider. These descriptions have
recently been generalized by Cautis, Kamnitzer, and Morrison~\cite{CKM} to general $n$, building on
earlier work of Kim~\cite{Kim} and Morrison~\cite{Mor}. We briefly summarize this construction,
referring the reader to their work for the details.

The category $n\cat{Web}$ is the pivotal category whose objects are sequences in the symbols
$\{1^{\pm},\ldots,(n-1)^{\pm}\}$. Morphisms are oriented graphs with edges labeled by $\{1,\ldots,n-1\}$
generated by the following:
\begin{equation}\label{webgen}
\xy
(0,0)*{
\begin{tikzpicture} [decoration={markings,
                        mark=at position 0.6 with {\arrow{>}};    }]
\draw[very thick, postaction={decorate}] (0,0) -- (0,1);
\node at (0,1.2) {$k+l$};
\draw[very thick, postaction={decorate}] (-.875,-.5) -- (0,0);
\node at (-1.2,-.5) {$k$};
\draw[very thick, postaction={decorate}] (.875,-.5) -- (0,0);
\node at (1.2,-.5) {$l$};
\end{tikzpicture}};
\endxy
\quad , \quad
\xy
(0,0)*{
\begin{tikzpicture} [decoration={markings,
                        mark=at position 0.6 with {\arrow{>}};    }]
\draw[very thick, postaction={decorate}] (0,1) -- (0,0);
\node at (0,1.2) {$k+l$};
\draw[very thick, postaction={decorate}] (0,0) -- (.875,-.5);
\node at (-1.2,-.5) {$k$};
\draw[very thick, postaction={decorate}] (0,0) -- (-.875,-.5);
\node at (1.2,-.5) {$l$};
\end{tikzpicture}};
\endxy
\quad , \quad
\xy
(0,0)*{
\begin{tikzpicture} [decoration={markings,
                        mark=at position 0.6 with {\arrow{>}};    }]
\draw[very thick, postaction={decorate}] (0,-.5) -- (0,.25);
\node at (0,-.7) {$k$};
\draw[very thick, postaction={decorate}] (0,1) -- (0,.25);
\node at (0,1.2) {$n-k$};
\draw[very thick] (0,.25) -- (.2,.25);
\end{tikzpicture}};
\endxy
\quad , \quad
\xy
(0,0)*{
\begin{tikzpicture} [decoration={markings,
                        mark=at position 0.6 with {\arrow{>}};    }]
\draw[very thick, postaction={decorate}] (0,.25) -- (0,-.5);
\node at (0,-.7) {$k$};
\draw[very thick, postaction={decorate}] (0,.25) -- (0,1);
\node at (0,1.2) {$n-k$};
\draw[very thick] (0,.25) -- (.2,.25);
\end{tikzpicture}};
\endxy
\end{equation}
where a strand labeled by $k$ is directed out from the label $k^{+}$ and into the label $k^{-}$ in the domain, and
vice versa in the codomain. These graphs, called $\mf{sl}_n$ webs, are considered up to isotopy (relative to
their boundary) and local relations. The category $n\cat{Web}$ can be identified with the full subcategory of
$U_q(\mf{sl}_n)$ representations generated (as a pivotal category) by the fundamental representations by
identifying the symbol $k^{+}$ with $\Alt{q}{k}(\C^n_q)$ and identifying $k^{-}$ with its dual. Sequences
correspond to tensor products of the relevant representations.

The connection to skew Howe duality is given by considering a related family of $m$-sheeted web categories.
Let $n\cat{Web}_{m}(N)$ denote the category whose objects are sequences $\mathbf{a}=(a_1, a_2, \ldots, a_m)$
with $0\leq a_i \leq n$ and $\displaystyle \sum_{i=1}^m a_i = N$.
Note that here we allow the symbols $0$ and $n$ in the object sequences, but none of the
dual symbols $k^{-}$.
As above, these labels should be interpreted as representations $\Alt{q}{k}(\C^n_q)$ for $0 \leq k \leq n$ with
$\Alt{q}{0}(\C^n_q) = \Alt{q}{n}(\C^n_q)=\C(q)$ corresponding to the trivial representation.
Morphisms in $n\cat{Web}_{m}(N)$ are given by $\mf{sl}_n$ webs mapping between the symbols
$a_i \neq 0,n$ in each sequence.

Via skew Howe duality, the action of $\U_q(\mf{sl}_m)$ on $\Alt{q}{N}(\C^n_q\otimes \C^m_q)$ gives morphisms
between tensor products of fundamental representations. This map has a graphical interpretation
described in \cite{CKM} using ``ladder diagrams'' to represent webs:
\[
1_\l \mapsto
\xy
(0,0)*{
\begin{tikzpicture} [decoration={markings, mark=at position 0.6 with {\arrow{>}}; },scale=.75]
\draw [very thick, postaction={decorate}] (3,0) -- (0,0);
\draw [very thick, postaction={decorate}] (3,1) -- (0,1);
\node at (3.6,0) {$a_1$};
\node at (3.6,1) {$a_m$};
\node at (1.5,.7) {$\vdots$};
\end{tikzpicture}};
\endxy
\]
\[
E_i 1_\l \mapsto
\xy
(0,0)*{
\begin{tikzpicture} [decoration={markings, mark=at position 0.6 with {\arrow{>}}; },scale=.75]
\draw [very thick, postaction={decorate}] (3,0) -- (2,0);
\draw[very thick, postaction={decorate}] (2,0) -- (0,0);
\draw [very thick, postaction={decorate}] (3,1) -- (1,1);
\draw[very thick, postaction={decorate}] (1,1) -- (0,1);
\draw [very thick, postaction={decorate}] (2,0) -- (1,1);
\node at (3.6,0) {$a_i$};
\node at (3.6,1) {$a_{i+1}$};
\node at (-1,0) {$a_{i} - 1$};
\node at (-1,1) {$a_{i+1}+1$};
\end{tikzpicture}};
\endxy
\]
\[
F_i 1_\l \mapsto
\xy
(0,0)*{
\begin{tikzpicture} [decoration={markings, mark=at position 0.6 with {\arrow{>}}; },scale=.75]
\draw [very thick, postaction={decorate}] (3,0) -- (1,0);
\draw[very thick, postaction={decorate}] (1,0) -- (0,0);
\draw [very thick, postaction={decorate}] (3,1) -- (2,1);
\draw[very thick, postaction={decorate}] (2,1) -- (0,1);
\draw [very thick, postaction={decorate}] (2,1) -- (1,0);
\node at (3.6,0) {$a_i$};
\node at (3.6,1) {$a_{i+1}$};
\node at (-1,0) {$a_{i} + 1$};
\node at (-1,1) {$a_{i+1} - 1$};
\end{tikzpicture}};
\endxy
\]
where these diagrams should be read from right to left and we omit $m-2$ lines in each of the latter two diagrams (compare with \eqref{eqn_Esplit}).
The sequences on the right are
determined by the $\mf{sl}_m$ weight $\l=(\l_1, \ldots , \l_{m-1})$ by $\l_i = a_{i+1}-a_i$;
edges connected to the label $0$ should be deleted and those
connected to the label $n$ should be truncated to the ``tags''  depicted in the last two diagrams in
equation~\eqref{webgen}.

In this paper, we categorify Cautis, Kamnitzer, and Morrison's construction for the cases $n=2$ and $n=3$.
In fact, for the $\mf{sl}_2$ case we work with related categories $2\cat{BWeb}_m(N)$ when we allow strands
labeled by $2$ and no longer require the tag morphisms\footnote{Here the ``B'' stands for Blanchet; this
category is related to a decategorification of his work \cite{Blan}.}.
For example in $2\cat{BWeb}_2(2)$ we have the
morphism:
\[
\xy
(0,0)*{
};
\endxy \; \;
$
depicts a $1$-labeled edge and
$ \; \;
\xy
(0,0)*{
%
};
\endxy \; \;
$
depicts a $2$-labeled edge.

In the $\mf{sl}_3$ case, we continue to work with $3\cat{Web}_m(N)$, although this category has a simpler
description than the one given above. Since $\Alt{q}{2}(\C^3_q)$ can be canonically identified with the dual of
$\Alt{q}{1}(\C^3_q)$ we can replace $2$-labeled edges with $1$-labeled edges oriented in the opposite direction
and do away with the tag morphisms. For example, the diagram:
\[
\xy
(0,0)*{
%
};
\endxy
\]
depicts a morphism in $3\cat{Web}_4(6)$. We will later also consider a (categorified) version of this category
in which we retain $3$-labeled edges.

We will now exhibit the power of the skew Howe approach to diagrammatic representation theory (which hints
at the utility of its categorified counterpart) in an example. The decomposition of
$\Alt{q}{3}(\C^3_q \otimes \C^2_q)$ into $\U_q(\mf{sl}_2)$ weight spaces
gives:
\[
\xymatrix{
\Alt{q}{3}(\C^3_q) \otimes \Alt{q}{0}(\C^3_q) \ar@/^1pc/[r]^{E 1_{-3}}
& \Alt{q}{2}(\C^3_q) \otimes \Alt{q}{1}(\C^3_q) \ar@/^1pc/[r]^{E 1_{-1}} \ar@/^1pc/[l]^{F 1_{-1}}
& \Alt{q}{1}(\C^3_q) \otimes \Alt{q}{2}(\C^3_q)\ar@/^1pc/[r]^{E 1_1} \ar@/^1pc/[l]^{F 1_{1}}
& \Alt{q}{0}(\C^3_q) \otimes \Alt{q}{3}(\C^3_q) \ar@/^1pc/[l]^{F 1_3}
}
\]
or diagrammatically:
\[
\xymatrix{
(3,0) \ar@/^1pc/[rr]^{
\xy
(0,0)*{
};
\endxy
}
}
\]
where we again read the webs from right to left in the above.

The local relations for $\mf{sl}_3$ webs from~\cite{Kup} can be deduced from the fact that the
above is an $\mf{sl}_2$ representation.
Indeed, the $\U_q(\mf{sl}_2)$ relation $EF 1_3 = FE 1_3 + [3] 1_3$ gives the ``circle'' relation:
\[
\xy
(0,0)*{
%
};
\endxy
\]
follows from the relation $EF 1_1 = FE 1_1 + [1] 1_1$.
The above diagrammatics extends to a description of the action by the integral version $_{\cal{A}}\U(\mf{sl}_2)$
on $\Alt{q}{3}(\C^3_q \otimes \C^2_q)$
where divided powers $E^{(k)}:= E^k/[k]!$ act by ladder web diagrams with diagonal lines labeled by $k$.
In the above example, the divided power relation $E^21_{-1} = [2] E^{(2)}1_{-1}$ gives the remaining ``bigon''
relation
\[
\xy
(0,0)*{
%
};
\endxy
\]
from~\cite{Kup}.

The skein theoretic definition of the braiding can also be constructed from skew Howe duality
using the deformed Weyl group action.
For example, the braiding for edges labeled by the standard $\mf{sl}_3$
representation can be recovered from the action of $\U_q(\mf{sl}_2)$ on $\Alt{q}{2}(\C^3_q \otimes \C^2_q)$.
The decomposition into weight spaces is given diagrammatically by:
\[
\xymatrix{
(2,0) \ar@/^1pc/[rr]^{
\xy
(0,0)*{
%
};
\endxy
}
}
\]
and the Weyl group action \eqref{eq_def_tau} on the 0-weight space gives the braiding:
\[
\xy
(0,0)*{
%
};
\endxy
\]
since $T 1_0 = 1_0 - qFE 1_0$. Up to a power of $q$, this recovers the formula for the
positive crossing from \cite{Kup}; the negative crossing can be recovered by considering
$T^{-1} 1_0 = 1_0 - q^{-1}EF 1_0$.

In a similar manner, one can recover the $\mf{sl}_2$ skein theory (i.e. the Kauffman bracket) from the
action of $_{\cal{A}}\U_q(\slm)$ on $2\cat{Web}_m(N)$.
In fact, Cautis, Kamnitzer, and Morrison use this approach to deduce the $\sln$ web relations for $n \geq 4$.
One can use their setup to give a combinatorial description of
$\sln$ link invariants labeled by any fundamental representation of $U_q(\sln)$.

Moreover, one may realize the invariant of a link (or tangle) as the image of an element in
$\U_q(\slm)$ under the (appropriate) skew Howe map.
For example, the $\mf{sl}_3$ invariant of the Hopf link
\[
\xy
(0,0)*{
\begin{tikzpicture} [decoration={markings, mark=at position 0.6 with {\arrow{<}}; },scale=.75]
\draw [very thick, postaction={decorate}] (0,0) -- (1,1);
\draw [very thick, postaction={decorate}] (1,1) -- (4,1);
\draw [very thick, postaction={decorate}] (6,1) -- (7,1);
\draw [very thick, postaction={decorate}] (7,1) -- (8,0);
\draw [very thick, postaction={decorate}] (8,0) -- (0,0);
\draw [very thick] (4,2) -- (5,1);
\draw [very thick] (4,1) -- (4.4,1.4);
\draw [very thick] (4.6,1.6) -- (5,2);
\draw [very thick] (5,2) -- (6,1);
\draw [very thick] (5,1) -- (5.4,1.4);
\draw [very thick] (5.6,1.6) -- (6,2);
\draw [very thick, postaction={decorate}] (2,3) -- (3,2);
\draw [very thick, postaction={decorate}] (3,2) -- (4,2);
\draw [very thick, postaction={decorate}] (6,2) -- (9,2);
\draw [very thick, postaction={decorate}] (9,2) -- (10,3);
\draw [very thick, postaction={decorate}] (10,3) -- (2,3);
\node at (10.3,3) {\small $3$};
\node at (10.3,2) {\small $0$};
\node at (10.3,1) {\small $0$};
\node at (10.3,0) {\small $3$};
\node at (-.3,3) {\small $3$};
\node at (-.3,2) {\small $0$};
\node at (-.3,1) {\small $0$};
\node at (-.3,0) {\small $3$};
\end{tikzpicture}};
\endxy
\]
is the element in
\[
\End_{\mf{sl}_3}(\Alt{q}{3}(\C^3_q) \otimes \Alt{q}{0}(\C^3_q) \otimes \Alt{q}{0}(\C^3_q) \otimes \Alt{q}{3}(\C^3_q))
\cong \C(q)
\]
given by the action of the element
$F_1 E_3 T_2^2 E_1 F_3 1_{(-3,0,3)} \in \U_q(\mf{sl}_4)$
on $\Alt{q}{6}(\C^3_q \otimes \C^4_q)$.

\subsection{Foamation functors for knot homologies}

The observations from the previous section suggest an approach to obtaining diagrammatic $\sln$ link homologies using categorical skew Howe duality. In his work categorifying the $\mf{sl}_3$  polynomial, Khovanov utilized
certain singular web cobordisms called foams~\cite{Kh5}.  In \cite{MSV} these singular surfaces were generalized to the $\mf{sl}_n$ case to supply a diagrammatic counterpart of Khovanov-Rozansky homology~\cite{KhR,KhR2,MV}.  These foams also appear to be connected with category $\mathcal{O}$~\cite{MS2} and with Soergel bimodules~\cite{Vaz}.  However, unlike Khovanov's construction for $\mf{sl}_3$, there is no known finite set of relations on $\mf{sl}_n$ foams for $n>3$ that guarantee any closed foam can be evaluated to an element of the ground ring.  For general $\mf{sl}_n$, matrix factorizations become the primary computation tool~\cite{Yon, Yon2,Wu}, and the only way to evaluate a closed foam is through the mysterious Kapustin-Li formula~\cite{MSV,DM}.  For foams this formula was discovered by Khovanov and Rozansky~\cite{KhR3} generalizing work of the physicists Vafa~\cite{Vafa},  Kapustin, and Li~\cite{KapLi}.  It arises from the topological Landau-Ginzburg model associated to components of the foam.  A
purely combinatorial foam construction of $\mf{sl}_n$ link homology remains an important open problem.

Foams can be viewed as a categorification of webs.  Indeed, this point of view motivates our approach to constructing $\mf{sl}_n$ link homologies for $n=2$ and $n=3$. In section~\ref{sec:foams} we describe 2-categories of $m$-sheeted $\sln$ foams categorifying the above web categories.  We define 2-functors $\Phi_n \maps \UcatD_Q(\slm) \to \foam{n}{m}$ for $n=2$ and $n=3$.  The existence of such functors was predicted by Khovanov and previously defined by Mackaay in the $n=3$ case working in the restrictive setting of $\Z/ 2\Z$ coefficients in \cite{Mac} where he called them ``foamation'' 2-functors.

Here we reinterpret Mackaay's work (and extend it to the $\mf{sl}_2$ case)
using skew Howe duality, defining foamation functors for $n=2,3$ with integer coefficients. Working over $\Z$,
it is not obvious for $n=2$ how to connect categorified quantum groups with the Bar-Natan's foam description of Khovanov homology.  For example, one of the most basic relations for categorified quantum groups
$\UcatD_Q(\slm)$ is the nilHecke relation
\[
\xy 0;/r.18pc/:
  (4,4);(4,-4) **\dir{-}?(0)*\dir{}+(2.3,0)*{};
  (-4,4);(-4,-4) **\dir{-}?(0)*\dir{}+(2.3,0)*{};
 \endxy
 \quad =
\xy 0;/r.18pc/:
  (0,0)*{\xybox{
    (-4,-4)*{};(4,4)*{} **\crv{(-4,-1) & (4,1)}?(1)*\dir{}?(.25)*{\bullet};
    (4,-4)*{};(-4,4)*{} **\crv{(4,-1) & (-4,1)}?(1)*\dir{};
     (-10,0)*{};(10,0)*{};
     }};
  \endxy
 \;\; -
\xy 0;/r.18pc/:
  (0,0)*{\xybox{
    (-4,-4)*{};(4,4)*{} **\crv{(-4,-1) & (4,1)}?(1)*\dir{}?(.75)*{\bullet};
    (4,-4)*{};(-4,4)*{} **\crv{(4,-1) & (-4,1)}?(1)*\dir{};
     (-10,0)*{};(10,0)*{};
     }};
  \endxy
 \;\; =
\xy 0;/r.18pc/:
  (0,0)*{\xybox{
    (-4,-4)*{};(4,4)*{} **\crv{(-4,-1) & (4,1)}?(1)*\dir{};
    (4,-4)*{};(-4,4)*{} **\crv{(4,-1) & (-4,1)}?(1)*\dir{}?(.75)*{\bullet};
     (-10,0)*{};(10,0)*{};
     }};
  \endxy
 \;\; -
  \xy 0;/r.18pc/:
  (0,0)*{\xybox{
    (-4,-4)*{};(4,4)*{} **\crv{(-4,-1) & (4,1)}?(1)*\dir{} ;
    (4,-4)*{};(-4,4)*{} **\crv{(4,-1) & (-4,1)}?(1)*\dir{}?(.25)*{\bullet};
     (-10,0)*{};(10,0)*{};
     }};
  \endxy
\]
which should correspond to the so called ``neck-cutting'' relation
\[
\xy
 (0,0)*{
\begin{tikzpicture} [scale=0.6,fill opacity=0.2,  decoration={markings, mark=at position 0.6 with {\arrow{>}};  }]
	\draw [fill=red] (0,4) ellipse (1 and 0.5);
	\path [fill=red] (0,0) ellipse (1 and 0.5);
	\draw (1,0) .. controls (1,-.66) and (-1,-.66) .. (-1,0);
	\draw[dashed] (1,0) .. controls (1,.66) and (-1,.66) .. (-1,0);
	\draw (1,4) -- (1,0);
	\draw (-1,4) -- (-1,0);
	\path[fill=red, opacity=.3] (-1,4) .. controls (-1,3.34) and (1,3.34) .. (1,4) --
		(1,0) .. controls (1,.66) and (-1,.66) .. (-1,0) -- cycle;
\end{tikzpicture}};
\endxy
\quad = \quad
\xy
(0,0)*{
\begin{tikzpicture} [scale=0.6,fill opacity=0.2,  decoration={markings, mark=at position 0.6 with {\arrow{>}};  }]
	\draw [fill=red] (0,4) ellipse (1 and 0.5);
	\draw (-1,4) .. controls (-1,2) and (1,2) .. (1,4);
	\path [fill=red, opacity=.3] (1,4) .. controls (1,3.34) and (-1,3.34) .. (-1,4) --
		(-1,4) .. controls (-1,2) and (1,2) .. (1,4);
	\node [opacity=1] at (0,3) {$\bullet$};
	\path [fill=red] (0,0) ellipse (1 and 0.5);
	\draw (1,0) .. controls (1,-.66) and (-1,-.66) .. (-1,0);
	\draw[dashed] (1,0) .. controls (1,.66) and (-1,.66) .. (-1,0);
	\draw (-1,0) .. controls (-1,2) and (1,2) .. (1,0);
	\path [fill=red, opacity=.3] (1,0) .. controls (1,.66) and (-1,.66) .. (-1,0) --
		(-1,0) .. controls (-1,2) and (1,2) .. (1,0);
\end{tikzpicture}};
\endxy
\quad + \quad
\xy
(0,0)*{\begin{tikzpicture} [scale=0.6,fill opacity=0.2,  decoration={markings, mark=at position 0.6 with {\arrow{>}};  }]
	\draw [fill=red] (0,4) ellipse (1 and 0.5);
	\draw (-1,4) .. controls (-1,2) and (1,2) .. (1,4);
	\path [fill=red, opacity=.3] (1,4) .. controls (1,3.34) and (-1,3.34) .. (-1,4) --
		(-1,4) .. controls (-1,2) and (1,2) .. (1,4);
	\path [fill=red] (0,0) ellipse (1 and 0.5);
	\draw (1,0) .. controls (1,-.66) and (-1,-.66) .. (-1,0);
	\draw[dashed] (1,0) .. controls (1,.66) and (-1,.66) .. (-1,0);
	\draw (-1,0) .. controls (-1,2) and (1,2) .. (1,0);
	\path [fill=red, opacity=.3] (1,0) .. controls (1,.66) and (-1,.66) .. (-1,0) --
		(-1,0) .. controls (-1,2) and (1,2) .. (1,0);
	\node [opacity=1] at (0,0) {$\bullet$};
\end{tikzpicture}};
\endxy
\]
via the foamation $2$-functor.
However, the signs under this assignment do not match.  One can try to rescale the foamation functors, but
one quickly finds that there is no way to fix the signs under this assignment.

The difficulty in matching the neck-cutting relation with the nilHecke relation is closely related to the solution of another famous problem related to Khovanov homology. As originally defined,
Khovanov homology is a projective functorial invariant, meaning that to a cobordism
$f \maps T \to T'$ between two tangles one can assign a map $Kh(f) \maps Kh(T) \to Kh(T')$ between the respective homologies well defined only up to a $\pm1$ sign~\cite{Kh6,Kh2,BN2,Jac}.

Clark, Morrison and Walker~\cite{CMW}, and independently Caprau~\cite{Cap3,Cap4}, showed that the functoriality of Khovanov homology could be fixed by considering modified foam categories.  From the representation-theoretic point of view, these foam categories keep track of the fact that the defining representation of $U_q(\mf{sl}_2)$ is non-canonically isomorphic to its dual.  Keeping track of this information gives rise to a fully functorial tangle invariant.  For both of these fixes to Khovanov homology one must work with foams defined over the Gaussian integers $\Z[i]$.

Christian Blanchet proposed yet another construction fixing the functoriality of Khovanov homology~\cite{Blan}.  He works with an enhanced version of the foam category where one labels facets by elements of the set $\{1,2\}$.  The 2-labeled facets are the primary difference from the previous two constructions.  The presence of these 2-labeled facets introduces additional signs that are not present in the CMW or Caprau approaches to functoriality.   Blanchet's approach gives rise to a functorial version of Khovanov homology defined over the
integers~\cite{Blan}.

These modified foam categories are quite natural from the representation theoretic viewpoint.  In the skew Howe framework, foams naturally provide a representation of $\Ucat_Q(\mf{gl}_n)$.  Seen from this perspective, the foams introduced by Blanchet keep track of the difference between the trivial representation $\Alt{q}{0}(\C^2_q)$ and the determinant representation $\Alt{q}{2}(\C^2_q)$.  As $U_q(\mf{sl}_2)$ representations there is of course no difference between these two representations, but it appears that Blanchet's approach has additional information that contributes additional signs coming from the 2-labeled facets corresponding to the determinant representation $\Alt{q}{2}(\C^2_q)$.

In this article, we construct foamation functors into both the CMW foam categories as well as the foam categories of Blanchet.  To define the functors into the CMW foam categories one must continue working  with
complex coefficients, while Blanchet's foam categories naturally admit foamation functors defined over
the integers. This suggests that Blanchet's approach is the most natural from the perspective of categorified representation theory.  It is also interesting to note that the $\mf{sl}_2$ knot homology most closely related to categorified quantum groups is integral and functorial.

It turns out that in the $n=3$ case it is possible to modify Mackaay's definition of the foamation functors to work over $\Z$, although this requires rather complicated and unnatural sign assignments.  Motivated by the $\mf{sl}_2$ case, we consider a modified $\mf{sl}_3$ foam category that incorporates additional 3-labeled facets.  To distinguish these foams from the usual $\mf{sl}_3$ foams we call them Blanchet $\mf{sl}_3$ foams. We show that there are 2-functors into Blanchet $\mf{sl}_3$ foams with much more natural sign assignments for the generating 2-morphisms in $\Ucat_Q(\slm)$.  There is also a natural construction of a forgetful 2-functor into the usual $\mf{sl}_3$ foams defined intrinsically in terms of the topology of the Blanchet foams.  Taking the composite of these 2-functors provides an explanation for the complicated signs occurring in the standard $\mf{sl}_3$ foamation functors.

Checking the relations for the 2-category $\Ucat_Q(\slm)$ needed to define foamation functors is a laborious task.  Here we utilize recent results of the first author with Cautis showing that in a 2-representation with finitely many nonzero weight spaces many of the relations come for free~\cite{CLau}.

An independent construction of the integral foamation functors into the usual $\mf{sl}_3$ foam $2$-category was given by Mackaay, Pan, and Tubbenhauer in a recent update to their work in~\cite{MPT}. They utilize the foamation functors for a different application related to a generalization of Khovanov's arc algebra to the $\mf{sl}_3$ setting.

\subsection{Comparing knot homologies}

A careful analysis of Cautis' arguments in \cite{Cautis} reveals that the skew Howe duality approach also supplies a mechanism for equating different constructions of $\mf{sl}_n$ link homologies.  Indeed, given any 2-representation of $\UcatD_Q(\slm)$ whose
objects are indexed by the nonzero weights in $\bigwedge^{N}_q(\C^n_q \otimes \C^m_q)$, and whose endomorphisms of the highest weight object are one dimensional in degree zero and zero dimensional otherwise, one obtains a unique knot homology theory that is formally determined by the relations imposed by the 2-representation.  In Section~\ref{sec_compare}, we show that Khovanov's $\mf{sl}_2$ and $\mf{sl}_3$ link homology theories fit into this framework. We also sketch a proof of the $\mf{sl}_3$ case of Conjecture~$6.4$
from \cite{CK02} relating Khovanov-Rozansky link homology to the geometrically defined Cautis-Kamnitzer
link homology, contingent on results to appear in \cite{Ca}.

\subsection{Cautis-Rozansky categorified clasps}

Categorifying $\sln$ link invariants labeled by arbitrary (non-fundamental) representations appears to be a much more difficult problem~\cite{Web2}.  In the $n=2$ case, there are several approaches to defining categorifications of the coloured Jones polynomial by categorifying Jones-Wenzl projectors.  The approach of Cooper-Krushkal uses foam based methods~\cite{CoopKrush}, while another approach of Frenkel, Stroppel, and Sussan uses Lie theoretic methods~\cite{FSS} based on category $ \cal{O}$ for $\mf{gl}_n$. These two approaches are compared and related via Koszul duality in~\cite{SS}.  Rozansky defined yet another approach to categorifying Jones-Wenzl projectors using complexes in Bar-Natan's foam category~\cite{Roz}.    These complexes are presented as the stable limit of the complexes assigned to $k$-twist torus braids as $k \to \infty$, or ``infinite twists''.  This construction also agrees with the Cooper-Krushkal $\mf{sl}_2$ projectors.

There are analogs of Jones-Wenzl projectors for $\sln$.  Given a tensor product of fundamental $U_q(\sln)$ representations, there is a corresponding idempotent
\[
P \maps V_{\omega_{i_1}} \otimes V_{\omega_{i_2}} \otimes \dots \otimes V_{\omega_{i_m}} \to V_{\sum_k i_k},
\]
called a {\em clasp} following Kuperberg's terminology from the $\mf{sl}_3$ case.  For $n=3$ these clasps were categorified by the third author using an $\mf{sl}_3$ foam based construction and a generalization of Rozansky's infinite twist approach to projectors~\cite{Rose}.

A related, but more general, approach using infinite twists was independently considered by Cautis who showed that $\sln$ clasps can be categorified explicitly using the higher structure of categorified quantum groups~\cite{Cautis}.  His approach utilizes an infinite twist construction together with the categorified braid group action described above. Given a reduced decomposition of $w=s_{i_1} \dots s_{i_k}$ of the longest braid word $w$ in the Weyl group for $\slm$, Cautis defines a complex $\cal{T}_{w} \onel:= \cal{T}_{i_1} \dots \cal{T}_{i_k}\onel$ in $Kom(\UcatD_Q(\slm))$.
He shows that the infinite twist $\lim_{\ell \to \infty}\cal{T}_{w}^{2\ell} \onel$ converges and categorifies the clasp $P$ in any appropriate $2$-representation.

Cautis' categorified clasps are formulated explicitly using the 2-morphisms in $\UcatD_Q(\slm)$.  This is
advantageous in that it allows for explicit computations not accessible within Webster's formalism.  Given appropriate families of 2-representations of $\UcatD_Q(\slm)$ with nonzero weight spaces matching the vector space $\Alt{q}{N}(\C^n_q\otimes \C^m_q)$ where $N$ and $m$ vary,  Cautis' framework gives rise to $\sln$ knot homology theories and categorifications of $\sln$ clasps.  Cautis describes such 2-representations using derived categories of coherent sheaves.
In section~\ref{SectionKnotHom} we show that foamation functors allow Cautis' categorified clasps to be utilized in the foam setting.  In the $\mf{sl}_2$ case this gives categorified projectors which can be viewed as an extension of the Cooper-Krushkal and Rozansky projectors to the functorial foam categories of Clark-Morrison-Walker and Blanchet.  In the $\mf{sl}_3$ case the resulting projectors agree with those from~\cite{Rose}.

\subsection{Recovering relations from categorified quantum groups}

Foams can be thought of as a categorification of webs.  This perspective suggests that new insights into foam
categories can be achieved through categorical skew Howe duality. In~\cite{CKM}, the authors use skew Howe
duality to deduce the $\sln$ web relations. In section~\ref{sec_deriving_rels}, we show that this holds at the
categorified level as well, namely that relations for $\mf{sl}_2$ and $\mf{sl}_3$ foams can be deduced
from the categorified quantum group.

This suggests that one may gain further insight to $\sln$ foams for $n\geq 4$ using categorical skew Howe
duality. In a follow-up paper, we will give a foam-based construction of $\sln$ link homologies for $n\geq 4$
which avoids the use of the Kapustin-Li formula \cite{LQR2}.

Note that the relations we derive use graded parameters that are usually set to zero in the literature. These relations are similar to the ones of \cite{MV2} in the $\mathfrak{sl}_3$ case, but in the $\mathfrak{sl}_2$ case we obtain relations that slightly extend both Blanchet's \cite{Blan} and Clark-Morrison-Walker \cite{CMW} models.

\medskip

\noindent {\bf Acknowledgments:}
The authors would like to thank Christian Blanchet, Sabin Cautis, Joel Kamnitzer and Mikhail Khovanov for helpful discussions. A.L. was supported by a Zumberge Fellowship and the Alfred P. Sloan foundation.

%
\section{Categorified quantum groups}
%


In this section we recall the relevant background information on categorified quantum groups and
higher representation theory.

%
\subsection{The 2-category $\UcatD_Q(\mathfrak{sl}_n)$}
%

Fix a base field $\Bbbk$. We will always work over this field which is not assumed to be of characteristic 0, nor algebraically closed.

%
\subsubsection{The Cartan datum } \label{sec:datum}
%

Let $I=\{1,2,\dots,m-1\}$ consist of the set of vertices of the Dynkin diagram of type $A_{m-1}$
\begin{equation} \label{eq_dynkin_sln}
    \xy
  (-15,0)*{\circ}="1";
  (-5, 0)*{\circ}="2";
  (5,  0)*{\circ}="3";
  (35,  0)*{\circ}="4";
  "1";"2" **\dir{-}?(.55)*\dir{};
  "2";"3" **\dir{-}?(.55)*\dir{};
  "3";(15,0) **\dir{-}?(.55)*\dir{};
  (25,0);"4" **\dir{-}?(.55)*\dir{};
  (-15,2.2)*{\scs 1};
  (-5,2.2)*{\scs 2};
  (5,2.2)*{\scs 3};
  (35,2.2)*{\scs m-1};
  (20,0)*{\cdots };
  \endxy
  \nn
\end{equation}
enumerated from left to right. Let $X=\Z^{m-1}$ denote the weight lattice for $\mathfrak{sl}_m$ and
$\{\alpha_i\}_{i\ \in I} \subset X$ and $\{\Lambda_i\}_{i \in I} \subset X$ denote the collection of simple roots
and fundamental weights, respectively.
There is a symmetric bilinear form on $X$ defined by $(\alpha_i, \alpha_j) = a_{ij}$ where
\[ a_{ij} =
\left\{
\begin{array}{ll}
  2 & \text{if $i=j$}\\
  -1&  \text{if $|i-j|=1$} \\
  0 & \text{if $|i-j|>1$}
\end{array}
\right.
\]
is the (symmetric) Cartan matrix associated to $\mathfrak{sl}_m$. For $i \in I$ denote the simple coroots by
$h_i \in X^\vee = \Hom_{\Z}(X,\Z)$.  Write $\langle \cdot, \cdot \rangle \maps X^{\vee} \times X
\to \Z$ for the canonical pairing
$\la i,\lambda\ra :=\langle h_i, \lambda \rangle = 2 \frac{(\alpha_i,\lambda)}{(\alpha_i,\alpha_i)}$
for $i \in I$ and $\lambda \in X$ that satisfies $\la h_i, \Lambda_i \ra = \delta_{i,j}$. Any weight $\lambda \in X$
can be written as $\lambda = (\lambda_1,\lambda_2, \dots, \lambda_{m-1})$, where
$\lambda_i=\la h_i, \lambda \ra$.

We let $X^+ \subset X$ denote the dominant weights, which are those of the form
$\sum_i \lambda_i \Lambda_i$
with $\lambda_i \ge 0$.  Finally, let $[n]=\frac{q^n-q^{-n}}{q-q^{-1}}$ and $[n]!=[n][n-1]\dots [1]$.

%
\subsubsection{The algebra $\mathbf{U}_q(\mathfrak{sl}_m)$}
%

The algebra $\mathbf{U}_q(\mathfrak{sl}_m)$ is the $\Q(q)$-algebra with unit generated by the elements $E_i$,
$F_i$ and $K_i^{\pm 1}$ for $i = 1, 2, \dots , m-1$, with  the defining relations
\begin{equation}
K_iK_i^{-1} = K_i^{-1}K_i = 1, \quad  K_iK_j = K_jK_i,
\end{equation}
\begin{equation}
K_iE_jK_i^{-1} = q^{a_{ij}} E_j, \quad  K_iF_jK_i^{-1} = q^{-a_{ij}} F_j,
\end{equation}
\begin{equation}
E_iF_j - F_jE_i = \delta_{ij} \frac{K_i-K_{i}^{-1}}{q-q^{-1}},
\end{equation}
\begin{equation}
E_i^2E_j-(q+q^{-1})E_iE_jE_i+E_jE_i^2 =0\;\; \text{if $j=i\pm 1$},
\end{equation}\begin{equation}
F_i^2F_j-(q+q^{-1})F_iF_jF_i+F_jF_i^2=0 \;\; \text{if $j=i\pm 1$},
\end{equation}\begin{equation}
E_iE_j=E_jE_i, \quad F_iF_j = F_jF_i \;\; \text{if $|i-j|>1$.}
\end{equation}

Recall that $\U(\mathfrak{sl}_m)$ is the modified version of $\mathbf{U}_q(\mathfrak{sl}_m)$
where the unit is replaced by a collection of orthogonal idempotents $1_{\lambda}$ indexed by the weight lattice $X$ of $\mathfrak{sl}_m$,
\begin{equation}
  1_{\lambda}1_{\lambda'} = \delta_{\lambda\lambda'} 1_{\lambda},
\end{equation}
such that if $\lambda = (\lambda_1,\lambda_2, \dots, \lambda_{m-1})$, then
\begin{equation} \label{eq_onesubn}
K_i1_{\lambda} =1_{\lambda}K_i= q^{\lambda_i} 1_{\lambda}, \quad
E_i^{}1_{\lambda} = 1_{\lambda+\alpha_i}E_i, \quad F_i1_{\lambda} = 1_{\lambda-\alpha_i}F_i
,
\end{equation}
where
\begin{eqnarray} \label{eq_weight_action1}
 \lambda +\alpha_i &=& \quad \left\{
\begin{array}{ccl}
   (\lambda_1+2, \lambda_2-1,\lambda_3,\dots,\lambda_{m-2}, \lambda_{m-1}) & \quad & \text{if $i=1$} \\
   (\lambda_1, \lambda_2,\dots,\lambda_{m-2},\lambda_{m-1}-1, \lambda_{m-1}+2) & \quad & \text{if $i=m-1$} \\
  (\lambda_1, \dots, \lambda_{i-1}-1, \lambda_i+2, \lambda_{i+1}-1, \dots,
\lambda_{m-1}) & \quad & \text{otherwise.}
\end{array}
 \right.
\end{eqnarray}

Let $\cal{A} := \Z[q,q^{-1}]$; the $\cal{A}$-algebra $\UA(\mathfrak{sl}_m)$ is the integral form of
$\U(\mathfrak{sl}_m)$ generated by products of divided powers $E^{(a)}_i1_{\lambda}:=
\frac{E^{a}_i}{[a]!}1_{\lambda}$, $F^{(a)}_i1_{\lambda}:=
\frac{F^{a}_i}{[a]!}1_{\lambda}$ for $\lambda \in X$ and $i = 1, 2, \dots , m-1$.

%
\subsubsection{Choice of scalars $Q$}
%

Associated to the Cartan datum for $\mathfrak{sl}_m$ we also fix a choice of scalars $Q$ consisting of:
\begin{itemize}
  \item $t_{ij}$ for all $i,j \in I$,
\end{itemize}
such that
\begin{itemize}
\item $t_{ii}=1$ for all $i \in I$ and $t_{ij} \in \Bbbk^{\times}$ for $i\neq j$,
 \item $t_{ij}=t_{ji}$ when $a_{ij}=0$.
\end{itemize}

%
\subsubsection{The definition}
%

We now recall the general version of the 2-category categorifying $\U(\mf{sl}_m)$ given in \cite{CLau}.  There a 2-category $\Ucat_Q(\mf{g})$ was defined associated to any root datum and choice of scalars $Q$.  This 2-category is a modest generalization of the 2-category originally defined in \cite{KL3} for the choice of scalars $Q$ where all $t_{ij}=1$.  It follows from \cite[pg. 15]{KL2} and \cite{KL3,KL3p} that $\cal{U}_Q(\mf{sl}_m)$
is independent of the choice of scalars $Q$ up to isomorphism.
Here we present the general definition; in later sections we will choose a convenient choice of scalars.

\begin{defn} \label{defU_cat}
The 2-category $\Ucat_Q(\mathfrak{sl}_m)$ is the graded additive $\Bbbk$-linear 2-category consisting of:
\begin{itemize}
\item objects $\lambda$ for $\lambda \in X$.
\item 1-morphisms are formal direct sums of (shifts of) compositions of
$$\onel, \quad \onenn{\l+\alpha_i} \sE_i = \onenn{\l+\alpha_i} \sE_i\onel = \sE_i \onel, \quad \text{ and }\quad \onenn{\lambda-\alpha_i} \sF_i = \onenn{\lambda-\alpha_i} \sF_i\onel = \sF_i\onel$$
for $i \in I$ and $\l \in X$.
\item 2-morphisms are $\Bbbk$-vector spaces spanned by compositions of (decorated) tangle-like diagrams illustrated below.
\begin{align}
  \xy 0;/r.17pc/:
 (0,7);(0,-7); **\dir{-} ?(.75)*\dir{>};
 (0,0)*{\bullet};
 (7,3)*{ \scs \lambda};
 (-9,3)*{\scs  \lambda+\alpha_i};
 (-2.5,-6)*{\scs i};
 (-10,0)*{};(10,0)*{};
 \endxy &\maps \cal{E}_i\onel \to \cal{E}_i\onel\{ (\alpha_i,\alpha_i) \}  & \quad
 &
    \xy 0;/r.17pc/:
 (0,7);(0,-7); **\dir{-} ?(.75)*\dir{<};
 (0,0)*{\bullet};
 (7,3)*{ \scs \lambda};
 (-9,3)*{\scs  \lambda-\alpha_i};
 (-2.5,-6)*{\scs i};
 (-10,0)*{};(10,0)*{};
 \endxy\maps \cal{F}_i\onel \to \cal{F}_i\onel\{ (\alpha_i,\alpha_i) \}  \nn \\
   & & & \nn \\
   \xy 0;/r.17pc/:
  (0,0)*{\xybox{
    (-4,-4)*{};(4,4)*{} **\crv{(-4,-1) & (4,1)}?(1)*\dir{>} ;
    (4,-4)*{};(-4,4)*{} **\crv{(4,-1) & (-4,1)}?(1)*\dir{>};
    (-5.5,-3)*{\scs i};
     (5.5,-3)*{\scs j};
     (9,1)*{\scs  \lambda};
     (-10,0)*{};(10,0)*{};
     }};
  \endxy \;\;&\maps \cal{E}_i\cal{E}_j\onel  \to \cal{E}_j\cal{E}_i\onel\{ - (\alpha_i,\alpha_j) \}  &
  &
   \xy 0;/r.17pc/:
  (0,0)*{\xybox{
    (-4,4)*{};(4,-4)*{} **\crv{(-4,1) & (4,-1)}?(1)*\dir{>} ;
    (4,4)*{};(-4,-4)*{} **\crv{(4,1) & (-4,-1)}?(1)*\dir{>};
    (-6.5,-3)*{\scs i};
     (6.5,-3)*{\scs j};
     (9,1)*{\scs  \lambda};
     (-10,0)*{};(10,0)*{};
     }};
  \endxy\;\; \maps \cal{F}_i\cal{F}_j\onel  \to \cal{F}_j\cal{F}_i\onel\{ - (\alpha_i,\alpha_j) \}  \nn \\
  & & & \nn \\
     \xy 0;/r.17pc/:
    (0,-3)*{\bbpef{i}};
    (8,-5)*{\scs  \lambda};
    (-10,0)*{};(10,0)*{};
    \endxy &\maps \onel  \to \cal{F}_i\cal{E}_i\onel\{ 1 + (\l, \alpha_i) \}   &
    &
   \xy 0;/r.17pc/:
    (0,-3)*{\bbpfe{i}};
    (8,-5)*{\scs \lambda};
    (-10,0)*{};(10,0)*{};
    \endxy \maps \onel  \to\cal{E}_i\cal{F}_i\onel\{ 1 - (\l, \alpha_i) \}  \nn \\
      & & & \nn \\
  \xy 0;/r.17pc/:
    (0,0)*{\bbcef{i}};
    (8,4)*{\scs  \lambda};
    (-10,0)*{};(10,0)*{};
    \endxy & \maps \cal{F}_i\cal{E}_i\onel \to\onel\{ 1 + (\l, \alpha_i) \}  &
    &
 \xy 0;/r.17pc/:
    (0,0)*{\bbcfe{i}};
    (8,4)*{\scs  \lambda};
    (-10,0)*{};(10,0)*{};
    \endxy \maps\cal{E}_i\cal{F}_i\onel  \to\onel\{ 1 - (\l, \alpha_i) \} \nn
\end{align}
\end{itemize}
\end{defn}
Here we follow the grading conventions in \cite{CLau}
which are opposite to those from \cite{KL3} but line up nicely
with the gradings on foams used later in the paper.
In this $2$-category (and those throughout the paper) we
read diagrams from right to left and bottom to top.  The identity 2-morphism of the 1-morphism
$\cal{E}_i \onel$ is
represented by an upward oriented line labeled by $i$ and the identity 2-morphism of $\cal{F}_i \onel$ is
represented by a downward such line.

The 2-morphisms satisfy the following relations:
\begin{enumerate}
\item \label{item_cycbiadjoint} The 1-morphisms $\cal{E}_i \onel$ and $\cal{F}_i \onel$ are biadjoint (up to a specified degree shift). These conditions are expressed diagrammatically as
    \begin{equation} \label{eq_biadjoint1}
 \xy   0;/r.17pc/:
    (-8,0)*{}="1";
    (0,0)*{}="2";
    (8,0)*{}="3";
    (-8,-10);"1" **\dir{-};
    "1";"2" **\crv{(-8,8) & (0,8)} ?(0)*\dir{>} ?(1)*\dir{>};
    "2";"3" **\crv{(0,-8) & (8,-8)}?(1)*\dir{>};
    "3"; (8,10) **\dir{-};
    (12,-9)*{\lambda};
    (-6,9)*{\lambda+\alpha_i};
    \endxy
    \; =
    \;
\xy   0;/r.17pc/:
    (-8,0)*{}="1";
    (0,0)*{}="2";
    (8,0)*{}="3";
    (0,-10);(0,10)**\dir{-} ?(.5)*\dir{>};
    (5,8)*{\lambda};
    (-9,8)*{\lambda+\alpha_i};
    \endxy
\qquad \quad \xy  0;/r.17pc/:
    (8,0)*{}="1";
    (0,0)*{}="2";
    (-8,0)*{}="3";
    (8,-10);"1" **\dir{-};
    "1";"2" **\crv{(8,8) & (0,8)} ?(0)*\dir{<} ?(1)*\dir{<};
    "2";"3" **\crv{(0,-8) & (-8,-8)}?(1)*\dir{<};
    "3"; (-8,10) **\dir{-};
    (12,9)*{\lambda+\alpha_i};
    (-6,-9)*{\lambda};
    \endxy
    \; =
    \;
\xy  0;/r.17pc/:
    (8,0)*{}="1";
    (0,0)*{}="2";
    (-8,0)*{}="3";
    (0,-10);(0,10)**\dir{-} ?(.5)*\dir{<};
    (9,-8)*{\lambda+\alpha_i};
    (-6,-8)*{\lambda};
    \endxy
\end{equation}

\begin{equation}\label{eq_biadjoint2}
 \xy   0;/r.17pc/:
    (8,0)*{}="1";
    (0,0)*{}="2";
    (-8,0)*{}="3";
    (8,-10);"1" **\dir{-};
    "1";"2" **\crv{(8,8) & (0,8)} ?(0)*\dir{>} ?(1)*\dir{>};
    "2";"3" **\crv{(0,-8) & (-8,-8)}?(1)*\dir{>};
    "3"; (-8,10) **\dir{-};
    (12,9)*{\lambda};
    (-5,-9)*{\lambda+\alpha_i};
    \endxy
    \; =
    \;
    \xy 0;/r.17pc/:
    (8,0)*{}="1";
    (0,0)*{}="2";
    (-8,0)*{}="3";
    (0,-10);(0,10)**\dir{-} ?(.5)*\dir{>};
    (5,-8)*{\lambda};
    (-9,-8)*{\lambda+\alpha_i};
    \endxy
\qquad \quad \xy   0;/r.17pc/:
    (-8,0)*{}="1";
    (0,0)*{}="2";
    (8,0)*{}="3";
    (-8,-10);"1" **\dir{-};
    "1";"2" **\crv{(-8,8) & (0,8)} ?(0)*\dir{<} ?(1)*\dir{<};
    "2";"3" **\crv{(0,-8) & (8,-8)}?(1)*\dir{<};
    "3"; (8,10) **\dir{-};
    (12,-9)*{\lambda+\alpha_i};
    (-6,9)*{\lambda};
    \endxy
    \; =
    \;
\xy   0;/r.17pc/:
    (-8,0)*{}="1";
    (0,0)*{}="2";
    (8,0)*{}="3";
    (0,-10);(0,10)**\dir{-} ?(.5)*\dir{<};
   (9,8)*{\lambda+\alpha_i};
    (-6,8)*{\lambda};
    \endxy
\end{equation}

  \item The 2-morphisms are $Q$-cyclic with respect to this biadjoint structure.
\begin{equation} \label{eq_cyclic_dot}
 \xy 0;/r.17pc/:
    (-8,5)*{}="1";
    (0,5)*{}="2";
    (0,-5)*{}="2'";
    (8,-5)*{}="3";
    (-8,-10);"1" **\dir{-};
    "2";"2'" **\dir{-} ?(.5)*\dir{<};
    "1";"2" **\crv{(-8,12) & (0,12)} ?(0)*\dir{<};
    "2'";"3" **\crv{(0,-12) & (8,-12)}?(1)*\dir{<};
    "3"; (8,10) **\dir{-};
    (17,-9)*{\lambda+\alpha_i};
    (-12,9)*{\lambda};
    (0,4)*{\bullet};
    (10,8)*{\scs };
    (-10,-8)*{\scs };
    \endxy
    \quad = \quad
      \xy 0;/r.17pc/:
 (0,10);(0,-10); **\dir{-} ?(.75)*\dir{<}+(2.3,0)*{\scriptstyle{}}
 ?(.1)*\dir{ }+(2,0)*{\scs };
 (0,0)*{\bullet};
 (-6,5)*{\lambda};
 (10,5)*{\lambda+\alpha_i};
 (-10,0)*{};(10,0)*{};(-2,-8)*{\scs };
 \endxy
    \quad = \quad
   \xy 0;/r.17pc/:
    (8,5)*{}="1";
    (0,5)*{}="2";
    (0,-5)*{}="2'";
    (-8,-5)*{}="3";
    (8,-10);"1" **\dir{-};
    "2";"2'" **\dir{-} ?(.5)*\dir{<};
    "1";"2" **\crv{(8,12) & (0,12)} ?(0)*\dir{<};
    "2'";"3" **\crv{(0,-12) & (-8,-12)}?(1)*\dir{<};
    "3"; (-8,10) **\dir{-};
    (17,9)*{\lambda+\alpha_i};
    (-12,-9)*{\lambda};
    (0,4)*{\bullet};
    (-10,8)*{\scs };
    (10,-8)*{\scs };
    \endxy
\end{equation}
The $Q$-cyclic relations for crossings are given by
\begin{equation} \label{eq_almost_cyclic}
   \xy 0;/r.17pc/:
  (0,0)*{\xybox{
    (-4,4)*{};(4,-4)*{} **\crv{(-4,1) & (4,-1)}?(1)*\dir{>} ;
    (4,4)*{};(-4,-4)*{} **\crv{(4,1) & (-4,-1)}?(1)*\dir{>};
    (-6.5,-3)*{\scs i};
     (6.5,-3)*{\scs j};
     (9,1)*{\scs  \lambda};
     (-10,0)*{};(10,0)*{};
     }};
  \endxy \quad = \quad
  t_{ij}^{-1}\xy 0;/r.17pc/:
  (0,0)*{\xybox{
    (4,-4)*{};(-4,4)*{} **\crv{(4,-1) & (-4,1)}?(1)*\dir{>};
    (-4,-4)*{};(4,4)*{} **\crv{(-4,-1) & (4,1)};
     (-4,4)*{};(18,4)*{} **\crv{(-4,16) & (18,16)} ?(1)*\dir{>};
     (4,-4)*{};(-18,-4)*{} **\crv{(4,-16) & (-18,-16)} ?(1)*\dir{<}?(0)*\dir{<};
     (-18,-4);(-18,12) **\dir{-};(-12,-4);(-12,12) **\dir{-};
     (18,4);(18,-12) **\dir{-};(12,4);(12,-12) **\dir{-};
     (8,1)*{ \lambda};
     (-10,0)*{};(10,0)*{};
     (-4,-4)*{};(-12,-4)*{} **\crv{(-4,-10) & (-12,-10)}?(1)*\dir{<}?(0)*\dir{<};
      (4,4)*{};(12,4)*{} **\crv{(4,10) & (12,10)}?(1)*\dir{>}?(0)*\dir{>};
      (-20,11)*{\scs j};(-10,11)*{\scs i};
      (20,-11)*{\scs j};(10,-11)*{\scs i};
     }};
  \endxy
\quad =  \quad t_{ji}^{-1}
\xy 0;/r.17pc/:
  (0,0)*{\xybox{
    (-4,-4)*{};(4,4)*{} **\crv{(-4,-1) & (4,1)}?(1)*\dir{>};
    (4,-4)*{};(-4,4)*{} **\crv{(4,-1) & (-4,1)};
     (4,4)*{};(-18,4)*{} **\crv{(4,16) & (-18,16)} ?(1)*\dir{>};
     (-4,-4)*{};(18,-4)*{} **\crv{(-4,-16) & (18,-16)} ?(1)*\dir{<}?(0)*\dir{<};
     (18,-4);(18,12) **\dir{-};(12,-4);(12,12) **\dir{-};
     (-18,4);(-18,-12) **\dir{-};(-12,4);(-12,-12) **\dir{-};
     (8,1)*{ \lambda};
     (-10,0)*{};(10,0)*{};
      (4,-4)*{};(12,-4)*{} **\crv{(4,-10) & (12,-10)}?(1)*\dir{<}?(0)*\dir{<};
      (-4,4)*{};(-12,4)*{} **\crv{(-4,10) & (-12,10)}?(1)*\dir{>}?(0)*\dir{>};
      (20,11)*{\scs i};(10,11)*{\scs j};
      (-20,-11)*{\scs i};(-10,-11)*{\scs j};
     }};
  \endxy
\end{equation}

The $Q$-cyclic condition for sideways crossings is given by the equalities:
\begin{equation} \label{eq_crossl-gen}
  \xy 0;/r.18pc/:
  (0,0)*{\xybox{
    (-4,-4)*{};(4,4)*{} **\crv{(-4,-1) & (4,1)}?(1)*\dir{>} ;
    (4,-4)*{};(-4,4)*{} **\crv{(4,-1) & (-4,1)}?(0)*\dir{<};
    (-5,-3)*{\scs j};
     (6.5,-3)*{\scs i};
     (9,2)*{ \lambda};
     (-12,0)*{};(12,0)*{};
     }};
  \endxy
\quad = \quad
 \xy 0;/r.17pc/:
  (0,0)*{\xybox{
    (4,-4)*{};(-4,4)*{} **\crv{(4,-1) & (-4,1)}?(1)*\dir{>};
    (-4,-4)*{};(4,4)*{} **\crv{(-4,-1) & (4,1)};
     (-4,4);(-4,12) **\dir{-};
     (-12,-4);(-12,12) **\dir{-};
     (4,-4);(4,-12) **\dir{-};(12,4);(12,-12) **\dir{-};
     (16,1)*{\lambda};
     (-10,0)*{};(10,0)*{};
     (-4,-4)*{};(-12,-4)*{} **\crv{(-4,-10) & (-12,-10)}?(1)*\dir{<}?(0)*\dir{<};
      (4,4)*{};(12,4)*{} **\crv{(4,10) & (12,10)}?(1)*\dir{>}?(0)*\dir{>};
      (-14,11)*{\scs i};(-2,11)*{\scs j};
      (14,-11)*{\scs i};(2,-11)*{\scs j};
     }};
  \endxy
  \quad = \quad t_{ij} \;\;
 \xy 0;/r.17pc/:
  (0,0)*{\xybox{
    (-4,-4)*{};(4,4)*{} **\crv{(-4,-1) & (4,1)}?(1)*\dir{<};
    (4,-4)*{};(-4,4)*{} **\crv{(4,-1) & (-4,1)};
     (4,4);(4,12) **\dir{-};
     (12,-4);(12,12) **\dir{-};
     (-4,-4);(-4,-12) **\dir{-};(-12,4);(-12,-12) **\dir{-};
     (16,1)*{\lambda};
     (10,0)*{};(-10,0)*{};
     (4,-4)*{};(12,-4)*{} **\crv{(4,-10) & (12,-10)}?(1)*\dir{>}?(0)*\dir{>};
      (-4,4)*{};(-12,4)*{} **\crv{(-4,10) & (-12,10)}?(1)*\dir{<}?(0)*\dir{<};
     }};
     (12,11)*{\scs j};(0,11)*{\scs i};
      (-17,-11)*{\scs j};(-5,-11)*{\scs i};
  \endxy
\end{equation}
\begin{equation} \label{eq_crossr-gen}
  \xy 0;/r.18pc/:
  (0,0)*{\xybox{
    (-4,-4)*{};(4,4)*{} **\crv{(-4,-1) & (4,1)}?(0)*\dir{<} ;
    (4,-4)*{};(-4,4)*{} **\crv{(4,-1) & (-4,1)}?(1)*\dir{>};
    (5.1,-3)*{\scs i};
     (-6.5,-3)*{\scs j};
     (9,2)*{ \lambda};
     (-12,0)*{};(12,0)*{};
     }};
  \endxy
\quad = \quad
 \xy 0;/r.17pc/:
  (0,0)*{\xybox{
    (-4,-4)*{};(4,4)*{} **\crv{(-4,-1) & (4,1)}?(1)*\dir{>};
    (4,-4)*{};(-4,4)*{} **\crv{(4,-1) & (-4,1)};
     (4,4);(4,12) **\dir{-};
     (12,-4);(12,12) **\dir{-};
     (-4,-4);(-4,-12) **\dir{-};(-12,4);(-12,-12) **\dir{-};
     (16,-6)*{\lambda};
     (10,0)*{};(-10,0)*{};
     (4,-4)*{};(12,-4)*{} **\crv{(4,-10) & (12,-10)}?(1)*\dir{<}?(0)*\dir{<};
      (-4,4)*{};(-12,4)*{} **\crv{(-4,10) & (-12,10)}?(1)*\dir{>}?(0)*\dir{>};
      (14,11)*{\scs j};(2,11)*{\scs i};
      (-14,-11)*{\scs j};(-2,-11)*{\scs i};
     }};
  \endxy
  \quad = \quad t_{ji} \;\;
  \xy 0;/r.17pc/:
  (0,0)*{\xybox{
    (4,-4)*{};(-4,4)*{} **\crv{(4,-1) & (-4,1)}?(1)*\dir{<};
    (-4,-4)*{};(4,4)*{} **\crv{(-4,-1) & (4,1)};
     (-4,4);(-4,12) **\dir{-};
     (-12,-4);(-12,12) **\dir{-};
     (4,-4);(4,-12) **\dir{-};(12,4);(12,-12) **\dir{-};
     (16,6)*{\lambda};
     (-10,0)*{};(10,0)*{};
     (-4,-4)*{};(-12,-4)*{} **\crv{(-4,-10) & (-12,-10)}?(1)*\dir{>}?(0)*\dir{>};
      (4,4)*{};(12,4)*{} **\crv{(4,10) & (12,10)}?(1)*\dir{<}?(0)*\dir{<};
      (-14,11)*{\scs i};(-2,11)*{\scs j};(14,-11)*{\scs i};(2,-11)*{\scs j};
     }};
  \endxy
\end{equation}
where the second equality in \eqref{eq_crossl-gen} and \eqref{eq_crossr-gen}
follow from \eqref{eq_almost_cyclic}.

\item The $\cal{E}$'s carry an action of the KLR algebra associated to $Q$. The KLR algebra $R=R_Q$ associated to $Q$ is defined by finite $\Bbbk$-linear combinations of braid--like diagrams in the plane, where each strand is labeled by a vertex $i \in I$.  Strands can intersect and can carry dots but triple intersections are not allowed.  Diagrams are considered up to planar isotopy that do not change the combinatorial type of the diagram. We recall the local relations:
\begin{enumerate}[i)]
\item
If all strands are labeled by the same $i \in I$ then the  NilHecke algebra axioms hold
 \begin{equation}
\vcenter{
\xy 0;/r.17pc/:
	(-4,-4)*{};(4,4)*{} **\crv{(-4,-1) & (4,1)}?(1)*\dir{};
	(4,-4)*{};(-4,4)*{} **\crv{(4,-1) & (-4,1)}?(1)*\dir{};
	(-4,4)*{};(4,12)*{} **\crv{(-4,7) & (4,9)}?(1)*\dir{};
	(4,4)*{};(-4,12)*{} **\crv{(4,7) & (-4,9)}?(1)*\dir{};
	(-4,12); (-4,13) **\dir{-}?(1)*\dir{>};
	(4,12); (4,13) **\dir{-}?(1)*\dir{>};
	(9,8)*{\lambda};
\endxy}
 \;\; =\;\; 0, \qquad \quad
\vcenter{\xy 0;/r.17pc/:
    (-4,-4)*{};(4,4)*{} **\crv{(-4,-1) & (4,1)}?(1)*\dir{};
    (4,-4)*{};(-4,4)*{} **\crv{(4,-1) & (-4,1)}?(1)*\dir{};
    (4,4)*{};(12,12)*{} **\crv{(4,7) & (12,9)}?(1)*\dir{};
    (12,4)*{};(4,12)*{} **\crv{(12,7) & (4,9)}?(1)*\dir{};
    (-4,12)*{};(4,20)*{} **\crv{(-4,15) & (4,17)}?(1)*\dir{};
    (4,12)*{};(-4,20)*{} **\crv{(4,15) & (-4,17)}?(1)*\dir{};
    (-4,4)*{}; (-4,12) **\dir{-};
    (12,-4)*{}; (12,4) **\dir{-};
    (12,12)*{}; (12,20) **\dir{-};
    (4,20); (4,21) **\dir{-}?(1)*\dir{>};
    (-4,20); (-4,21) **\dir{-}?(1)*\dir{>};
    (12,20); (12,21) **\dir{-}?(1)*\dir{>};
   (18,8)*{\lambda};
\endxy}
 \;\; =\;\;
\vcenter{\xy 0;/r.17pc/:
    (4,-4)*{};(-4,4)*{} **\crv{(4,-1) & (-4,1)}?(1)*\dir{};
    (-4,-4)*{};(4,4)*{} **\crv{(-4,-1) & (4,1)}?(1)*\dir{};
    (-4,4)*{};(-12,12)*{} **\crv{(-4,7) & (-12,9)}?(1)*\dir{};
    (-12,4)*{};(-4,12)*{} **\crv{(-12,7) & (-4,9)}?(1)*\dir{};
    (4,12)*{};(-4,20)*{} **\crv{(4,15) & (-4,17)}?(1)*\dir{};
    (-4,12)*{};(4,20)*{} **\crv{(-4,15) & (4,17)}?(1)*\dir{};
    (4,4)*{}; (4,12) **\dir{-};
    (-12,-4)*{}; (-12,4) **\dir{-};
    (-12,12)*{}; (-12,20) **\dir{-};
    (4,20); (4,21) **\dir{-}?(1)*\dir{>};
    (-4,20); (-4,21) **\dir{-}?(1)*\dir{>};
    (-12,20); (-12,21) **\dir{-}?(1)*\dir{>};
  (10,8)*{\lambda};
\endxy}
 \label{eq_nil_rels}
  \end{equation}

\begin{equation}
 \xy 0;/r.18pc/:
  (4,6);(4,-4) **\dir{-}?(0)*\dir{<}+(2.3,0)*{};
  (-4,6);(-4,-4) **\dir{-}?(0)*\dir{<}+(2.3,0)*{};
 \endxy
 \quad =
\xy 0;/r.18pc/:
  (0,0)*{\xybox{
    (-4,-4)*{};(4,6)*{} **\crv{(-4,-1) & (4,1)}?(1)*\dir{>}?(.25)*{\bullet};
    (4,-4)*{};(-4,6)*{} **\crv{(4,-1) & (-4,1)}?(1)*\dir{>};
     (-10,0)*{};(10,0)*{};
     }};
  \endxy
 \;\; -
\xy 0;/r.18pc/:
  (0,0)*{\xybox{
    (-4,-4)*{};(4,6)*{} **\crv{(-4,-1) & (4,1)}?(1)*\dir{>}?(.75)*{\bullet};
    (4,-4)*{};(-4,6)*{} **\crv{(4,-1) & (-4,1)}?(1)*\dir{>};
     (-10,0)*{};(10,0)*{};
     }};
  \endxy
 \;\; =
\xy 0;/r.18pc/:
  (0,0)*{\xybox{
    (-4,-4)*{};(4,6)*{} **\crv{(-4,-1) & (4,1)}?(1)*\dir{>};
    (4,-4)*{};(-4,6)*{} **\crv{(4,-1) & (-4,1)}?(1)*\dir{>}?(.75)*{\bullet};
     (-10,0)*{};(10,0)*{};
     }};
  \endxy
 \;\; -
  \xy 0;/r.18pc/:
  (0,0)*{\xybox{
    (-4,-4)*{};(4,6)*{} **\crv{(-4,-1) & (4,1)}?(1)*\dir{>} ;
    (4,-4)*{};(-4,6)*{} **\crv{(4,-1) & (-4,1)}?(1)*\dir{>}?(.25)*{\bullet};
     (-10,0)*{};(10,0)*{};
     }};
  \endxy \label{eq_nil_dotslide}
\end{equation}

\item For $i \neq j$
\begin{equation}
 \vcenter{\xy 0;/r.17pc/:
    (-4,-4)*{};(4,4)*{} **\crv{(-4,-1) & (4,1)}?(1)*\dir{};
    (4,-4)*{};(-4,4)*{} **\crv{(4,-1) & (-4,1)}?(1)*\dir{};
    (-4,4)*{};(4,12)*{} **\crv{(-4,7) & (4,9)}?(1)*\dir{};
    (4,4)*{};(-4,12)*{} **\crv{(4,7) & (-4,9)}?(1)*\dir{};
    (8,8)*{\lambda};
    (4,12); (4,13) **\dir{-}?(1)*\dir{>};
    (-4,12); (-4,13) **\dir{-}?(1)*\dir{>};
  (-5.5,-3)*{\scs i};
     (5.5,-3)*{\scs j};
 \endxy}
 \qquad = \qquad
 \left\{
 \begin{array}{ccc}
     t_{ij}\;\xy 0;/r.17pc/:
  (3,9);(3,-9) **\dir{-}?(0)*\dir{<}+(2.3,0)*{};
  (-3,9);(-3,-9) **\dir{-}?(0)*\dir{<}+(2.3,0)*{};
  (-5,-6)*{\scs i};     (5.1,-6)*{\scs j};
 \endxy &  &  \text{if $(\alpha_i, \alpha_j)=0$,}\\ \\
 t_{ij} \vcenter{\xy 0;/r.17pc/:
  (3,9);(3,-9) **\dir{-}?(0)*\dir{<}+(2.3,0)*{};
  (-3,9);(-3,-9) **\dir{-}?(0)*\dir{<}+(2.3,0)*{};
  (-3,4)*{\bullet};(-6.5,5)*{};
  (-5,-6)*{\scs i};     (5.1,-6)*{\scs j};
 \endxy} \;\; + \;\; t_{ji}
  \vcenter{\xy 0;/r.17pc/:
  (3,9);(3,-9) **\dir{-}?(0)*\dir{<}+(2.3,0)*{};
  (-3,9);(-3,-9) **\dir{-}?(0)*\dir{<}+(2.3,0)*{};
  (3,4)*{\bullet};(7,5)*{};
  (-5,-6)*{\scs i};     (5.1,-6)*{\scs j};
 \endxy}
   &  & \text{if $(\alpha_i, \alpha_j) \neq 0$,}
 \end{array}
 \right. \label{eq_r2_ij-gen}
\end{equation}

\item For $i \neq j$ the dot sliding relations
\begin{equation} \label{eq_dot_slide_ij-gen}
\xy 0;/r.18pc/:
  (0,0)*{\xybox{
    (-4,-4)*{};(4,6)*{} **\crv{(-4,-1) & (4,1)}?(1)*\dir{>}?(.75)*{\bullet};
    (4,-4)*{};(-4,6)*{} **\crv{(4,-1) & (-4,1)}?(1)*\dir{>};
    (-5,-3)*{\scs i};
     (5.1,-3)*{\scs j};
     (-10,0)*{};(10,0)*{};
     }};
  \endxy
 \;\; =
\xy 0;/r.18pc/:
  (0,0)*{\xybox{
    (-4,-4)*{};(4,6)*{} **\crv{(-4,-1) & (4,1)}?(1)*\dir{>}?(.25)*{\bullet};
    (4,-4)*{};(-4,6)*{} **\crv{(4,-1) & (-4,1)}?(1)*\dir{>};
    (-5,-3)*{\scs i};
     (5.1,-3)*{\scs j};
     (-10,0)*{};(10,0)*{};
     }};
  \endxy
\qquad  \xy 0;/r.18pc/:
  (0,0)*{\xybox{
    (-4,-4)*{};(4,6)*{} **\crv{(-4,-1) & (4,1)}?(1)*\dir{>};
    (4,-4)*{};(-4,6)*{} **\crv{(4,-1) & (-4,1)}?(1)*\dir{>}?(.75)*{\bullet};
    (-5,-3)*{\scs i};
     (5.1,-3)*{\scs j};
     (-10,0)*{};(10,0)*{};
     }};
  \endxy
\;\;  =
  \xy 0;/r.18pc/:
  (0,0)*{\xybox{
    (-4,-4)*{};(4,6)*{} **\crv{(-4,-1) & (4,1)}?(1)*\dir{>} ;
    (4,-4)*{};(-4,6)*{} **\crv{(4,-1) & (-4,1)}?(1)*\dir{>}?(.25)*{\bullet};
    (-5,-3)*{\scs i};
     (5.1,-3)*{\scs j};
     (-10,0)*{};(12,0)*{};
     }};
  \endxy
\end{equation}
hold.

\item Unless $i = k$ and $(\alpha_i, \alpha_j) < 0$ the relation
\begin{equation}
\vcenter{\xy 0;/r.17pc/:
    (-4,-4)*{};(4,4)*{} **\crv{(-4,-1) & (4,1)}?(1)*\dir{};
    (4,-4)*{};(-4,4)*{} **\crv{(4,-1) & (-4,1)}?(1)*\dir{};
    (4,4)*{};(12,12)*{} **\crv{(4,7) & (12,9)}?(1)*\dir{};
    (12,4)*{};(4,12)*{} **\crv{(12,7) & (4,9)}?(1)*\dir{};
    (-4,12)*{};(4,20)*{} **\crv{(-4,15) & (4,17)}?(1)*\dir{};
    (4,12)*{};(-4,20)*{} **\crv{(4,15) & (-4,17)}?(1)*\dir{};
    (-4,4)*{}; (-4,12) **\dir{-};
    (12,-4)*{}; (12,4) **\dir{-};
    (12,12)*{}; (12,20) **\dir{-};
    (4,20); (4,21) **\dir{-}?(1)*\dir{>};
    (-4,20); (-4,21) **\dir{-}?(1)*\dir{>};
    (12,20); (12,21) **\dir{-}?(1)*\dir{>};
   (18,8)*{\lambda};  (-6,-3)*{\scs i};
  (6,-3)*{\scs j};
  (15,-3)*{\scs k};
\endxy}
 \;\; =\;\;
\vcenter{\xy 0;/r.17pc/:
    (4,-4)*{};(-4,4)*{} **\crv{(4,-1) & (-4,1)}?(1)*\dir{};
    (-4,-4)*{};(4,4)*{} **\crv{(-4,-1) & (4,1)}?(1)*\dir{};
    (-4,4)*{};(-12,12)*{} **\crv{(-4,7) & (-12,9)}?(1)*\dir{};
    (-12,4)*{};(-4,12)*{} **\crv{(-12,7) & (-4,9)}?(1)*\dir{};
    (4,12)*{};(-4,20)*{} **\crv{(4,15) & (-4,17)}?(1)*\dir{};
    (-4,12)*{};(4,20)*{} **\crv{(-4,15) & (4,17)}?(1)*\dir{};
    (4,4)*{}; (4,12) **\dir{-};
    (-12,-4)*{}; (-12,4) **\dir{-};
    (-12,12)*{}; (-12,20) **\dir{-};
    (4,20); (4,21) **\dir{-}?(1)*\dir{>};
    (-4,20); (-4,21) **\dir{-}?(1)*\dir{>};
    (-12,20); (-12,21) **\dir{-}?(1)*\dir{>};
  (10,8)*{\lambda};
  (-14,-3)*{\scs i};
  (-6,-3)*{\scs j};
  (6,-3)*{\scs k};
\endxy}
 \label{eq_r3_easy-gen}
\end{equation}
holds. Otherwise, $(\alpha_i, \alpha_j) =-1$ and
\begin{equation}
\vcenter{\xy 0;/r.17pc/:
    (-4,-4)*{};(4,4)*{} **\crv{(-4,-1) & (4,1)}?(1)*\dir{};
    (4,-4)*{};(-4,4)*{} **\crv{(4,-1) & (-4,1)}?(1)*\dir{};
    (4,4)*{};(12,12)*{} **\crv{(4,7) & (12,9)}?(1)*\dir{};
    (12,4)*{};(4,12)*{} **\crv{(12,7) & (4,9)}?(1)*\dir{};
    (-4,12)*{};(4,20)*{} **\crv{(-4,15) & (4,17)}?(1)*\dir{};
    (4,12)*{};(-4,20)*{} **\crv{(4,15) & (-4,17)}?(1)*\dir{};
    (-4,4)*{}; (-4,12) **\dir{-};
    (12,-4)*{}; (12,4) **\dir{-};
    (12,12)*{}; (12,20) **\dir{-};
    (4,20); (4,21) **\dir{-}?(1)*\dir{>};
    (-4,20); (-4,21) **\dir{-}?(1)*\dir{>};
    (12,20); (12,21) **\dir{-}?(1)*\dir{>};
   (18,8)*{\lambda};  (-6,-3)*{\scs i};
  (6,-3)*{\scs j};
  (15,-3)*{\scs k};
\endxy}
\quad - \quad
\vcenter{\xy 0;/r.17pc/:
    (4,-4)*{};(-4,4)*{} **\crv{(4,-1) & (-4,1)}?(1)*\dir{};
    (-4,-4)*{};(4,4)*{} **\crv{(-4,-1) & (4,1)}?(1)*\dir{};
    (-4,4)*{};(-12,12)*{} **\crv{(-4,7) & (-12,9)}?(1)*\dir{};
    (-12,4)*{};(-4,12)*{} **\crv{(-12,7) & (-4,9)}?(1)*\dir{};
    (4,12)*{};(-4,20)*{} **\crv{(4,15) & (-4,17)}?(1)*\dir{};
    (-4,12)*{};(4,20)*{} **\crv{(-4,15) & (4,17)}?(1)*\dir{};
    (4,4)*{}; (4,12) **\dir{-};
    (-12,-4)*{}; (-12,4) **\dir{-};
    (-12,12)*{}; (-12,20) **\dir{-};
    (4,20); (4,21) **\dir{-}?(1)*\dir{>};
    (-4,20); (-4,21) **\dir{-}?(1)*\dir{>};
    (-12,20); (-12,21) **\dir{-}?(1)*\dir{>};
  (10,8)*{\lambda};
  (-14,-3)*{\scs i};
  (-6,-3)*{\scs j};
  (6,-3)*{\scs i};
\endxy}
 \;\; =\;\;
 t_{ij} \;\;
\xy 0;/r.17pc/:
  (4,12);(4,-12) **\dir{-}?(0)*\dir{<};
  (-4,12);(-4,-12) **\dir{-}?(0)*\dir{<}?(.25)*\dir{};
  (12,12);(12,-12) **\dir{-}?(0)*\dir{<}?(.25)*\dir{};
  (-6,-9)*{\scs i};     (6.1,-9)*{\scs j};
  (14,-9)*{\scs i};
 \endxy
 \label{eq_r3_hard-gen}
\end{equation}
\end{enumerate}

\item When $i \ne j$ one has the mixed relations  relating $\cal{E}_i \cal{F}_j$ and $\cal{F}_j \cal{E}_i$:
\begin{equation}
 \vcenter{   \xy 0;/r.18pc/:
    (-4,-4)*{};(4,4)*{} **\crv{(-4,-1) & (4,1)}?(1)*\dir{>};
    (4,-4)*{};(-4,4)*{} **\crv{(4,-1) & (-4,1)}?(1)*\dir{<};?(0)*\dir{<};
    (-4,4)*{};(4,12)*{} **\crv{(-4,7) & (4,9)};
    (4,4)*{};(-4,12)*{} **\crv{(4,7) & (-4,9)}?(1)*\dir{>};
  (8,8)*{\lambda};(-6,-3)*{\scs i};
     (6,-3)*{\scs j};
 \endxy}
 \;\; = \;\; t_{ji}\;\;
\xy 0;/r.18pc/:
  (3,9);(3,-9) **\dir{-}?(.55)*\dir{>}+(2.3,0)*{};
  (-3,9);(-3,-9) **\dir{-}?(.5)*\dir{<}+(2.3,0)*{};
  (8,2)*{\lambda};(-5,-6)*{\scs i};     (5.1,-6)*{\scs j};
 \endxy
\qquad \quad
    \vcenter{\xy 0;/r.18pc/:
    (-4,-4)*{};(4,4)*{} **\crv{(-4,-1) & (4,1)}?(1)*\dir{<};?(0)*\dir{<};
    (4,-4)*{};(-4,4)*{} **\crv{(4,-1) & (-4,1)}?(1)*\dir{>};
    (-4,4)*{};(4,12)*{} **\crv{(-4,7) & (4,9)}?(1)*\dir{>};
    (4,4)*{};(-4,12)*{} **\crv{(4,7) & (-4,9)};
  (8,8)*{\lambda};(-6,-3)*{\scs i};
     (6,-3)*{\scs j};
 \endxy}
 \;\;=\;\; t_{ij}\;\;
\xy 0;/r.18pc/:
  (3,9);(3,-9) **\dir{-}?(.5)*\dir{<}+(2.3,0)*{};
  (-3,9);(-3,-9) **\dir{-}?(.55)*\dir{>}+(2.3,0)*{};
  (8,2)*{\lambda};(-5,-6)*{\scs i};     (5.1,-6)*{\scs j};
 \endxy
\end{equation}

\item \label{item_positivity} Negative degree bubbles are zero. That is, for all $m \in \Z_+$ one has
\begin{equation} \label{eq_positivity_bubbles}
\xy 0;/r.18pc/:
 (-12,0)*{\icbub{m}{i}};
 (-8,8)*{\lambda};
 \endxy
  = 0
 \qquad  \text{if $m<\lambda_i-1$,} \qquad \xy 0;/r.18pc/: (-12,0)*{\iccbub{m}{i}};
 (-8,8)*{\lambda};
 \endxy = 0\quad
  \text{if $m< -\lambda_i-1$.}
\end{equation}
On the other hand, a dotted bubble of degree zero is just  the identity 2-morphism\footnote{One can define the 2-category so that degree zero bubbles are multiplication by arbitrary scalars at the cost of modifying some of the other relations, see for example~\cite{Lau4,MSV2}.  However, it is shown in \cite{CLau} that the resulting 2-categories are all isomorphic.}:
\[
\xy 0;/r.18pc/:
 (0,0)*{\icbub{\lambda_i-1}{i}};
  (4,8)*{\lambda};
 \endxy
  =  \Id_{\onenn{\lambda}} \quad \text{for $\lambda_i \geq 1$,}
  \qquad \quad
  \xy 0;/r.18pc/:
 (0,0)*{\iccbub{-\lambda_i-1}{i}};
  (4,8)*{\lambda};
 \endxy  =  \Id_{\onenn{\lambda}} \quad \text{for $\lambda_i \leq -1$.}\]

\item \label{item_highersl2} For any $i \in I$ one has the extended ${\mathfrak{sl}}_2$-relations. In order to describe certain extended ${\mathfrak{sl}}_2$ relations it is convenient to use a shorthand notation from \cite{Lau1} called fake bubbles. These are diagrams for dotted bubbles where the labels of the number of dots is negative, but the total degree of the dotted bubble taken with these negative dots is still positive. They allow us to write these extended ${\mathfrak{sl}}_2$ relations more uniformly (i.e. independent on whether the weight $\lambda_i$ is positive or negative).
\begin{itemize}
 \item Degree zero fake bubbles are equal to the identity 2-morphisms
\[
 \xy 0;/r.18pc/:
    (2,0)*{\icbub{\l_i-1}{i}};
  (12,8)*{\lambda};
 \endxy
  =  \Id_{\onenn{\lambda}} \quad \text{if $\lambda_i \leq 0$,}
  \qquad \quad
\xy 0;/r.18pc/:
    (2,0)*{\iccbub{-\lambda_i-1}{i}};
  (12,8)*{\lambda};
 \endxy =  \Id_{\onenn{\lambda}} \quad  \text{if $\lambda_i \geq 0$}.\]

  \item Higher degree fake bubbles for $\lambda_i<0$ are defined inductively as
  \begin{equation} \label{eq_fake_nleqz}
  \vcenter{\xy 0;/r.18pc/:
    (2,-11)*{\icbub{\l_i-1+j}{i}};
  (12,-2)*{\l};
 \endxy} \;\; =
 \left\{
 \begin{array}{cl}
  \;\; -\;\;
\xsum{\xy (0,6)*{};  (0,1)*{\scs a+b=j}; (0,-2)*{\scs b\geq 1}; \endxy}
\;\; \vcenter{\xy 0;/r.18pc/:
    (2,0)*{\cbub{\l_i-1+a}{}};
    (20,0)*{\ccbub{-\l-1+b}{}};
  (12,8)*{\lambda};
 \endxy}  & \text{if $0 \leq j < -\l_i+1$} \\ & \\
   0 & \text{if $j < 0$. }
 \end{array}
\right.
 \end{equation}

  \item Higher degree fake bubbles for $\lambda_i>0$ are defined inductively as
   \begin{equation} \label{eq_fake_ngeqz}
  \vcenter{\xy 0;/r.18pc/:
    (2,-11)*{\iccbub{-\l_i-1+j}{i}};
  (12,-2)*{\l};
 \endxy} \;\; =
 \left\{
 \begin{array}{cl}
  \;\; -\;\;
\xsum{\xy (0,6)*{}; (0,1)*{\scs a+b=j}; (0,-2)*{\scs a\geq 1}; \endxy}
\;\; \vcenter{\xy 0;/r.18pc/:
    (2,0)*{\cbub{\l_i-1+a}{}};
    (20,0)*{\ccbub{-\l-1+b}{}};
  (12,8)*{\lambda};
 \endxy}  & \text{if $0 \leq j < \l_i+1$} \\ & \\
   0 & \text{if $j < 0$. }
 \end{array}
\right.
\end{equation}
\end{itemize}
These equations arise from the homogeneous terms in $t$ of the `infinite Grassmannian' equation
\begin{center}
\begin{eqnarray}
 \makebox[0pt]{ $
\left( \xy 0;/r.15pc/:
 (0,0)*{\iccbub{-\l_i-1}{i}};
  (4,8)*{\l};
 \endxy
 +
 \xy 0;/r.15pc/:
 (0,0)*{\iccbub{-\l_i-1+1}{i}};
  (4,8)*{\l};
 \endxy t
 + \cdots +
\xy 0;/r.15pc/:
 (0,0)*{\iccbub{-\l_i-1+\alpha}{i}};
  (4,8)*{\l};
 \endxy t^{\alpha}
 + \cdots
\right)
\left(\xy 0;/r.15pc/:
 (0,0)*{\icbub{\l_i-1}{i}};
  (4,8)*{\l};
 \endxy
 + \xy 0;/r.15pc/:
 (0,0)*{\icbub{\l_i-1+1}{i}};
  (4,8)*{\l};
 \endxy t
 +\cdots +
\xy 0;/r.15pc/:
 (0,0)*{\icbub{\l_i-1+\alpha}{i}};
 (4,8)*{\l};
 \endxy t^{\alpha}
 + \cdots
\right) = \Id_{\onel}.$ } \nn \\ \label{eq_infinite_Grass}
\end{eqnarray}
\end{center}
Now we can define the extended ${\mathfrak{sl}}_2$ relations.  Note that in \cite{CLau} additional curl relations were provided that can be derived from those above.  Here we provide a minimal set of relations.

If $\l_i > 0$ then we have:
\begin{equation} \label{eq_reduction-ngeqz}
  \xy 0;/r.17pc/:
  (14,8)*{\l};
  (-3,-10)*{};(3,5)*{} **\crv{(-3,-2) & (2,1)}?(1)*\dir{>};?(.15)*\dir{>};
    (3,-5)*{};(-3,10)*{} **\crv{(2,-1) & (-3,2)}?(.85)*\dir{>} ?(.1)*\dir{>};
  (3,5)*{}="t1";  (9,5)*{}="t2";
  (3,-5)*{}="t1'";  (9,-5)*{}="t2'";
   "t1";"t2" **\crv{(4,8) & (9, 8)};
   "t1'";"t2'" **\crv{(4,-8) & (9, -8)};
   "t2'";"t2" **\crv{(10,0)} ;
   (-6,-8)*{\scs i};
 \endxy\;\; = \;\; 0
   \qquad \quad
 \vcenter{\xy 0;/r.17pc/:
  (-8,0)*{};(-6,-8)*{\scs i};(6,-8)*{\scs i};
  (8,0)*{};
  (-4,10)*{}="t1";
  (4,10)*{}="t2";
  (-4,-10)*{}="b1";
  (4,-10)*{}="b2";
  "t1";"b1" **\dir{-} ?(.5)*\dir{>};
  "t2";"b2" **\dir{-} ?(.5)*\dir{<};
  (10,2)*{\l};
  \endxy}
\;\; = \;\; -\;\;
   \vcenter{\xy 0;/r.17pc/:
    (-4,-4)*{};(4,4)*{} **\crv{(-4,-1) & (4,1)}?(1)*\dir{<};?(0)*\dir{<};
    (4,-4)*{};(-4,4)*{} **\crv{(4,-1) & (-4,1)}?(1)*\dir{>};
    (-4,4)*{};(4,12)*{} **\crv{(-4,7) & (4,9)}?(1)*\dir{>};
    (4,4)*{};(-4,12)*{} **\crv{(4,7) & (-4,9)};
  (8,8)*{\l};(-6.5,-3)*{\scs i};  (6,-3)*{\scs i};
 \endxy}
\end{equation}
\begin{equation}
 \vcenter{\xy 0;/r.17pc/:
  (-8,0)*{};
  (8,0)*{};
  (-4,10)*{}="t1";
  (4,10)*{}="t2";
  (-4,-10)*{}="b1";
  (4,-10)*{}="b2";(-6,-8)*{\scs i};(6,-8)*{\scs i};
  "t1";"b1" **\dir{-} ?(.5)*\dir{<};
  "t2";"b2" **\dir{-} ?(.5)*\dir{>};
  (10,2)*{\l};
  \endxy}
\;\; = \;\; -\;\;
 \vcenter{   \xy 0;/r.17pc/:
    (-4,-4)*{};(4,4)*{} **\crv{(-4,-1) & (4,1)}?(1)*\dir{>};
    (4,-4)*{};(-4,4)*{} **\crv{(4,-1) & (-4,1)}?(1)*\dir{<};?(0)*\dir{<};
    (-4,4)*{};(4,12)*{} **\crv{(-4,7) & (4,9)};
    (4,4)*{};(-4,12)*{} **\crv{(4,7) & (-4,9)}?(1)*\dir{>};
  (8,8)*{\l};
     (-6,-3)*{\scs i};
     (6.5,-3)*{\scs i};
 \endxy}
  \;\; + \;\;
   \sum_{ \xy  (0,3)*{\scs f_1+f_2+f_3}; (0,0)*{\scs =\lambda_i-1};\endxy}
    \vcenter{\xy 0;/r.17pc/:
    (-12,10)*{\l};
    (-8,0)*{};
  (8,0)*{};
  (-4,-15)*{}="b1";
  (4,-15)*{}="b2";
  "b2";"b1" **\crv{(5,-8) & (-5,-8)}; ?(.05)*\dir{<} ?(.93)*\dir{<}
  ?(.8)*\dir{}+(0,-.1)*{\bullet}+(-3,2)*{\scs f_3};
  (-4,15)*{}="t1";
  (4,15)*{}="t2";
  "t2";"t1" **\crv{(5,8) & (-5,8)}; ?(.15)*\dir{>} ?(.95)*\dir{>}
  ?(.4)*\dir{}+(0,-.2)*{\bullet}+(3,-2)*{\scs \; f_1};
  (0,0)*{\iccbub{\scs \quad -\l_i-1+f_2}{i}};
  (7,-13)*{\scs i};
  (-7,13)*{\scs i};
  \endxy} \label{eq_ident_decomp-ngeqz}
\end{equation}

If $\lambda_i < 0$ then we have:
\begin{equation} \label{eq_reduction-nleqz}
  \xy 0;/r.17pc/:
  (-14,8)*{\l};
  (3,-10)*{};(-3,5)*{} **\crv{(3,-2) & (-2,1)}?(1)*\dir{>};?(.15)*\dir{>};
    (-3,-5)*{};(3,10)*{} **\crv{(-2,-1) & (3,2)}?(.85)*\dir{>} ?(.1)*\dir{>};
  (-3,5)*{}="t1";  (-9,5)*{}="t2";
  (-3,-5)*{}="t1'";  (-9,-5)*{}="t2'";
   "t1";"t2" **\crv{(-4,8) & (-9, 8)};
   "t1'";"t2'" **\crv{(-4,-8) & (-9, -8)};
   "t2'";"t2" **\crv{(-10,0)} ;
   (6,-8)*{\scs i};
 \endxy\;\; = \;\;
0
\qquad \qquad
\vcenter{\xy 0;/r.17pc/:
  (-8,0)*{};
  (8,0)*{};
  (-4,10)*{}="t1";
  (4,10)*{}="t2";
  (-4,-10)*{}="b1";
  (4,-10)*{}="b2";(-6,-8)*{\scs i};(6,-8)*{\scs i};
  "t1";"b1" **\dir{-} ?(.5)*\dir{<};
  "t2";"b2" **\dir{-} ?(.5)*\dir{>};
  (10,2)*{\l};
  \endxy}
\;\; = \;\; -\;\;
\vcenter{   \xy 0;/r.17pc/:
    (-4,-4)*{};(4,4)*{} **\crv{(-4,-1) & (4,1)}?(1)*\dir{>};
    (4,-4)*{};(-4,4)*{} **\crv{(4,-1) & (-4,1)}?(1)*\dir{<};?(0)*\dir{<};
    (-4,4)*{};(4,12)*{} **\crv{(-4,7) & (4,9)};
    (4,4)*{};(-4,12)*{} **\crv{(4,7) & (-4,9)}?(1)*\dir{>};
  (8,8)*{\l};(-6,-3)*{\scs i};
     (6.5,-3)*{\scs i};
 \endxy}
\end{equation}
\begin{equation}
 \vcenter{\xy 0;/r.17pc/:
  (-8,0)*{};(-6,-8)*{\scs i};(6,-8)*{\scs i};
  (8,0)*{};
  (-4,10)*{}="t1";
  (4,10)*{}="t2";
  (-4,-10)*{}="b1";
  (4,-10)*{}="b2";
  "t1";"b1" **\dir{-} ?(.5)*\dir{>};
  "t2";"b2" **\dir{-} ?(.5)*\dir{<};
  (10,2)*{\l};
  (-10,2)*{\l};
  \endxy}
\;\; = \;\;
  -\;\;\vcenter{\xy 0;/r.17pc/:
    (-4,-4)*{};(4,4)*{} **\crv{(-4,-1) & (4,1)}?(1)*\dir{<};?(0)*\dir{<};
    (4,-4)*{};(-4,4)*{} **\crv{(4,-1) & (-4,1)}?(1)*\dir{>};
    (-4,4)*{};(4,12)*{} **\crv{(-4,7) & (4,9)}?(1)*\dir{>};
    (4,4)*{};(-4,12)*{} **\crv{(4,7) & (-4,9)};
  (8,8)*{\l};(-6.5,-3)*{\scs i};  (6,-3)*{\scs i};
 \endxy}
  \;\; + \;\;
    \sum_{ \xy  (0,3)*{\scs g_1+g_2+g_3}; (0,0)*{\scs =-\l_i-1};\endxy}
    \vcenter{\xy 0;/r.17pc/:
    (-8,0)*{};
  (8,0)*{};
  (-4,-15)*{}="b1";
  (4,-15)*{}="b2";
  "b2";"b1" **\crv{(5,-8) & (-5,-8)}; ?(.1)*\dir{>} ?(.95)*\dir{>}
  ?(.8)*\dir{}+(0,-.1)*{\bullet}+(-3,2)*{\scs g_3};
  (-4,15)*{}="t1";
  (4,15)*{}="t2";
  "t2";"t1" **\crv{(5,8) & (-5,8)}; ?(.15)*\dir{<} ?(.9)*\dir{<}
  ?(.4)*\dir{}+(0,-.2)*{\bullet}+(3,-2)*{\scs g_1};
  (0,0)*{\icbub{\scs \quad\; \l_i-1 + g_2}{i}};
    (7,-13)*{\scs i};
  (-7,13)*{\scs i};
  (-10,10)*{\l};
  \endxy} \label{eq_ident_decomp-nleqz}
\end{equation}

If $\lambda_i =0$ then we have:
\begin{equation}\label{eq_reduction-neqz}
  \xy 0;/r.17pc/:
  (14,8)*{\l};
  (-3,-10)*{};(3,5)*{} **\crv{(-3,-2) & (2,1)}?(1)*\dir{>};?(.15)*\dir{>};
    (3,-5)*{};(-3,10)*{} **\crv{(2,-1) & (-3,2)}?(.85)*\dir{>} ?(.1)*\dir{>};
  (3,5)*{}="t1";  (9,5)*{}="t2";
  (3,-5)*{}="t1'";  (9,-5)*{}="t2'";
   "t1";"t2" **\crv{(4,8) & (9, 8)};
   "t1'";"t2'" **\crv{(4,-8) & (9, -8)};
   "t2'";"t2" **\crv{(10,0)} ;
   (-6,-8)*{\scs i};
 \endxy\;\; = \;\;-
   \xy 0;/r.17pc/:
  (-8,8)*{\l};
  (0,0)*{\bbe{}};
  (-3,-8)*{\scs i};
 \endxy
\qquad \qquad
  \xy 0;/r.17pc/:
  (-14,8)*{\l};
  (3,-10)*{};(-3,5)*{} **\crv{(3,-2) & (-2,1)}?(1)*\dir{>};?(.15)*\dir{>};
    (-3,-5)*{};(3,10)*{} **\crv{(-2,-1) & (3,2)}?(.85)*\dir{>} ?(.1)*\dir{>};
  (-3,5)*{}="t1";  (-9,5)*{}="t2";
  (-3,-5)*{}="t1'";  (-9,-5)*{}="t2'";
   "t1";"t2" **\crv{(-4,8) & (-9, 8)};
   "t1'";"t2'" **\crv{(-4,-8) & (-9, -8)};
   "t2'";"t2" **\crv{(-10,0)} ;
   (6,-8)*{\scs i};
 \endxy \;\; = \;\;
  \xy 0;/r.17pc/:
  (-8,8)*{\l};
  (0,0)*{\bbe{}};
  (-3,-8)*{\scs i};
 \endxy
\end{equation}
\begin{equation}\label{eq_reduction-neqz_2}
 \vcenter{\xy 0;/r.17pc/:
  (-8,0)*{};
  (8,0)*{};
  (-4,10)*{}="t1";
  (4,10)*{}="t2";
  (-4,-10)*{}="b1";
  (4,-10)*{}="b2";(-6,-8)*{\scs i};(6,-8)*{\scs i};
  "t1";"b1" **\dir{-} ?(.5)*\dir{<};
  "t2";"b2" **\dir{-} ?(.5)*\dir{>};
  (10,2)*{\l};
  \endxy}
\;\; = \;\;
 \;\; - \;\;
 \vcenter{   \xy 0;/r.17pc/:
    (-4,-4)*{};(4,4)*{} **\crv{(-4,-1) & (4,1)}?(1)*\dir{>};
    (4,-4)*{};(-4,4)*{} **\crv{(4,-1) & (-4,1)}?(1)*\dir{<};?(0)*\dir{<};
    (-4,4)*{};(4,12)*{} **\crv{(-4,7) & (4,9)};
    (4,4)*{};(-4,12)*{} **\crv{(4,7) & (-4,9)}?(1)*\dir{>};
  (8,8)*{l};(-6,-3)*{\scs i};
     (6.5,-3)*{\scs i};
 \endxy}
   \qquad \quad
 \vcenter{\xy 0;/r.17pc/:
  (-8,0)*{};(-6,-8)*{\scs i};(6,-8)*{\scs i};
  (8,0)*{};
  (-4,10)*{}="t1";
  (4,10)*{}="t2";
  (-4,-10)*{}="b1";
  (4,-10)*{}="b2";
  "t1";"b1" **\dir{-} ?(.5)*\dir{>};
  "t2";"b2" **\dir{-} ?(.5)*\dir{<};
  (10,2)*{\l};
  \endxy}
\;\; = \;\;
 \;\; - \;\;
   \vcenter{\xy 0;/r.17pc/:
    (-4,-4)*{};(4,4)*{} **\crv{(-4,-1) & (4,1)}?(1)*\dir{<};?(0)*\dir{<};
    (4,-4)*{};(-4,4)*{} **\crv{(4,-1) & (-4,1)}?(1)*\dir{>};
    (-4,4)*{};(4,12)*{} **\crv{(-4,7) & (4,9)}?(1)*\dir{>};
    (4,4)*{};(-4,12)*{} **\crv{(4,7) & (-4,9)};
  (8,8)*{\l};(-6,-3)*{\scs i};  (6,-3)*{\scs i};
 \endxy}
\end{equation}
\end{enumerate}

%
\subsubsection{Karoubi completions}
%

Recall that an idempotent $e \maps b\to b$ in a category $\cal{C}$ is a morphism such that
$e^2 = e$.  The idempotent is said to split if there exist morphisms
$
 \xymatrix{ b \ar[r]^g & b' \ar[r]^h &b}
$
such that $e=h  g$ and $gh = \id_{b'}$.    The Karoubi envelope $Kar(\cal{C})$ (also called the idempotent completion or Cauchy completion) of a category $\cal{C}$ is a minimal enlargement of the category $\cal{C}$ in which all idempotents split.  More precisely, the category $Kar(\cal{C})$ has
\begin{itemize}
  \item objects of $Kar(\cal{C})$:  pairs $(b,e)$
where $e \maps b \to b$ is an idempotent of $\cal{C}$.
\item morphisms: $(e,f,e') \maps (b,e) \to (b',e')$
where $f \maps b \to b'$ in $\cal{C}$ making the diagram
\begin{equation} \label{eq_Kar_morph}
 \xymatrix{
 b \ar[r]^f \ar[d]_e \ar[dr]^{f} & b' \ar[d]^{e'} \\ b \ar[r]_f & b'
 }
\end{equation}
commute, i.e. $ef=f=fe'$.
 \item identity 1-morphisms: $(e,e,e) \maps (b,e) \to (b,e)$.
\end{itemize}
When $\cal{C}$ is an additive category we write $(b,e)\in Kar(\cal{C})$ as $\im e$ and we have $b \cong \im e \oplus \im (1-e)$ in $Kar(\cal{C})$.

The Karoubi envelope $\UcatD_Q(\slm) := Kar(\Ucat_Q(\mf{sl}_m))$ of the 2-category $\Ucat_Q(\mf{sl}_m)$ is the 2-category with the same objects as $\Ucat_Q(\mf{sl}_m)$  whose Hom categories are given by
\[
 \UcatD_Q(\onel,\onelp) := Kar\big(\Ucat_Q(\onel,\onelp)\big).
\]
In particular, all idempotent 2-morphisms split in $\UcatD_Q(\onel,\onelp)$. It was shown in \cite{KL3} that there is an isomorphism of
$\mathcal{A}$-algebras
\begin{equation} \label{eq_gamma}
  \gamma \maps K_0(\UcatD_Q(\mf{sl}_m)) \longrightarrow _{\mathcal{A}}\U(\mf{sl}_m)
\end{equation}
between the split Grothendieck ring $K_0(\UcatD_Q(\mf{sl}_m))$ and the integral form $_{\mathcal{A}}\U(\mf{sl}_m)$ of the idempotent modified quantum enveloping algebra.
Recent results of Webster have generalized this statement to arbitrary type~\cite{Web}. Furthermore, the images of the indecomposable 1-morphisms in $\UcatD_Q(\mf{sl}_m)$ in $K_0(\Ucat_Q(\mf{sl}_m))$ agree with the Lusztig canonical basis in $_{\mathcal{A}}\U(\mf{sl}_m)$~\cite{Web4}.

Typically the passage from a diagrammatically defined category to its Karoubi envelope results in the loss of a completely diagrammatic description of the resulting category.  However, the Karoubi envelope $\UcatD_Q(\mf{sl}_2)$ of the 2-category $\Ucat_Q(\mf{sl}_2)$ still admits a completely diagrammatic description~\cite{KLMS}.  In this case, one defines idempotent 2-morphisms $e_a \maps \cal{E}^a\onel \to \cal{E}^a\onel$ given by the composite of any reduced presentation of the longest braid word on $a$ strands together with a specific pattern of dots starting with $a-1$ dots on the left-most strand, $a-2$ on the next strand, and ending with no dots on the last of the $a$ strands. An example is shown below for $a=4$.
\[
\xy 0;/r.15pc/:
 (-12,-20)*{}; (12,20) **\crv{(-12,-8) & (12,8)}?(1)*\dir{>};
 (-4,-20)*{}; (4,20) **\crv{(-4,-13) & (12,2) & (12,8)&(4,13)}?(1)*\dir{>};?(.88)*\dir{}+(0.1,0)*{\bullet};
 (4,-20)*{}; (-4,20) **\crv{(4,-13) & (12,-8) & (12,-2)&(-4,13)}?(1)*\dir{>}?(.86)*\dir{}+(0.1,0)*{\bullet};
 ?(.92)*\dir{}+(0.1,0)*{\bullet};
 (12,-20)*{}; (-12,20) **\crv{(12,-8) & (-12,8)}?(1)*\dir{>}?(.70)*\dir{}+(0.1,0)*{\bullet};
 ?(.90)*\dir{}+(0.1,0)*{\bullet};?(.80)*\dir{}+(0.1,0)*{\bullet};
 \endxy
 \qquad =: \qquad   \xy
 (0,0)*{\includegraphics[scale=0.4]{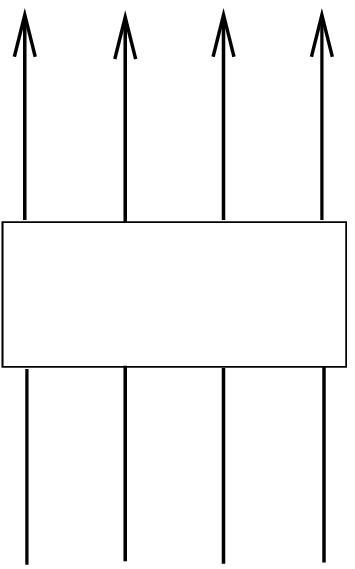}};
 (0,-0.5)*{e_a};
  \endxy
\]
It is convenient to introduce a box notation for this composite 2-morphism.

The divided power $\cal{E}^{(a)}\onel$ is defined in the Karoubi envelope $\UcatD_Q(\mf{sl}_2)$ as the pair
\[
\cal{E}^{(a)}\onel:= (\cal{E}^a\onel \{ \frac{a(a-1)}{2}\} , e_a)
\]
where the grading shift is necessary to get an isomorphism $\cal{E}^a\onel \cong \oplus_{[a]!}\cal{E}^{(a)}\onel$.
The divided power $\onel\cal{F}^{(a)}$ is then defined as the adjoint of $\cal{E}^{(a)}\onel$.  It was shown in \cite{KLMS} that splitting the idempotents $e_a$ by adding $\cal{E}^{(a)}\onel$ and $\cal{F}^{(b)}\onel$  gives rise to explicit decompositions of arbitrary 1-morphisms into indecomposable 1-morphisms using only the relations from $\Ucat_Q(\mf{sl}_2)$.  This allows for a strengthening of the categorification result to the case when we define $\Ucat_Q(\mf{sl}_2)$ by taking $\Z$-linear combinations of 2-morphisms, rather than $\Bbbk$-linear combinations for a field $\Bbbk$.

It is possible to represent the 1-morphisms $\cal{E}^{(a)}\onel$ in $\UcatD_Q(\mf{sl}_2)$ by introducing an augmented graphical calculus of thickened strands.  For
example, the identity 2-morphism for $\cal{E}^{(a)}\onel$ is given
by the triple
\begin{equation}
 (e_a,e_a,e_a)
 \quad
 =
 \quad
 \left( e_a,\;
 \xy
 (0,0)*{\includegraphics[scale=0.5]{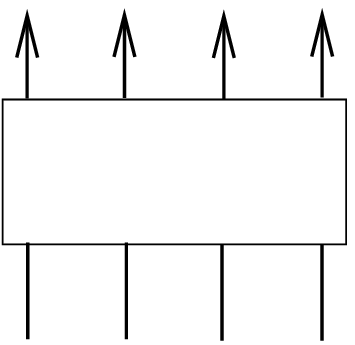}};
 (0,0)*{e_a}; (11,-5)*{\lambda};
  \endxy
 , e_a \right)
  \;\; =: \;\;     \xy
 (0,0)*{\includegraphics[scale=0.5]{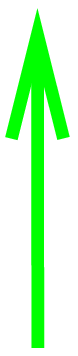}};
 (-2.5,-6)*{a}; (5,5)*{\lambda};
  \endxy
\end{equation}
where we think of the label $a$ on the right as describing the thickness of the strand.
A downward oriented line of thickness $b$ conveniently describes the 1-morphism $\cal{F}^{(b)}\onel$ in $\UcatD_Q(\mf{sl}_2)$.

One can introduce further notation to describe natural 2-morphisms in $\UcatD_Q(\mf{sl}_2)$.  For example, using the shorthand
\[ 
  \xy
 (0,0)*{\includegraphics[scale=0.5]{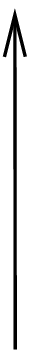}};
 (-3,-3)*{a};
  \endxy
  \quad : = \quad
  \xy
 (0,0)*{\includegraphics[scale=0.5]{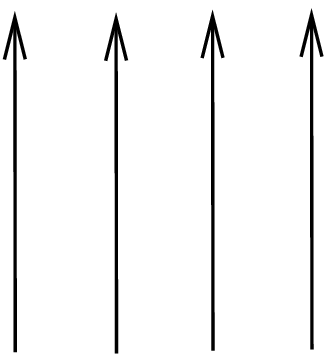}};
 (0,-11)*{\underbrace{\hspace{0.7in}}};  (0,-14)*{a};
  \endxy
 \qquad
 \qquad
   \xy
 (0,0)*{\includegraphics[scale=0.5]{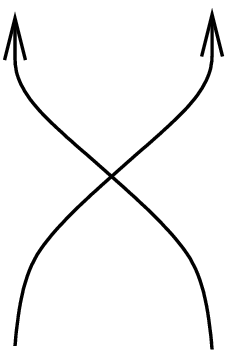}};
 (-7,-6)*{a};(7,-6)*{b};
  \endxy
  \quad := \quad
  \xy
 (0,0)*{\includegraphics[scale=0.5]{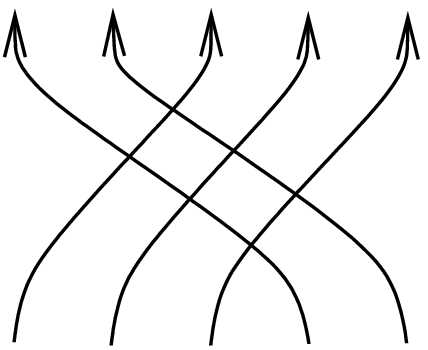}};
 (-5.5,-11)*{\underbrace{\hspace{0.45in}}};  (-5.5,-14)*{a};
 (7.5,-11)*{\underbrace{\hspace{0.25in}}};  (7.5,-14)*{b};
  \endxy
\]
there are 2-morphisms in $\UcatD_Q(\mf{sl}_2)$ given by
\begin{eqnarray}
    \xy
 (0,0)*{\includegraphics[scale=0.5]{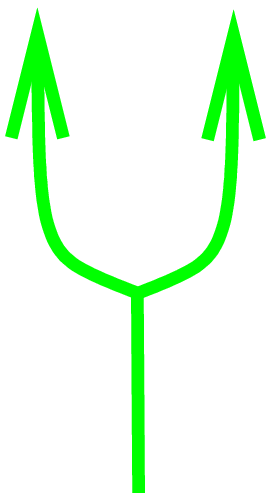}};
 (-5,-10)*{a+b};(-8,4)*{a};(8,4)*{b}; (7,-6)*{\l};
  \endxy
& :=&
     \left(e_{a+b}, \;\xy
 (0,0)*{\includegraphics[scale=0.5]{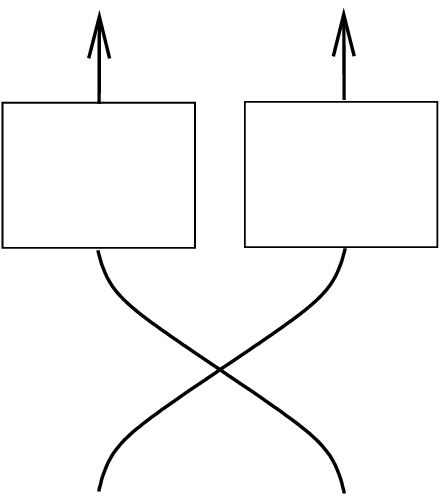}};
 (-7,-10)*{b};(7,-10)*{a};(-6,4)*{e_a};(6,4)*{e_b}; (14,-2)*{\l};
  \endxy \;, e_ae_b \right) \maps \cal{E}^{(a+b)}\onel \{t\} \to
  \cal{E}^{(a)}\cal{E}^{(b)}\onel\{t-ab\} \nn
  \\
\xy
 (0,0)*{\includegraphics[scale=0.5,angle=180]{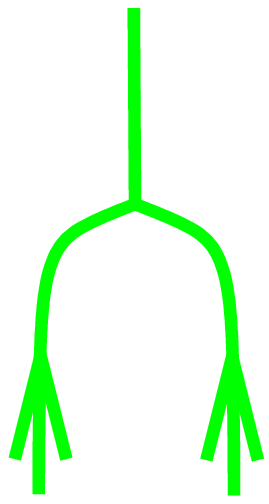}};
 (-6,10)*{a+b};(-8,-4)*{a};(8,-4)*{b}; (7,6)*{\l};
  \endxy
  &:= &
  \left( e_a e_b,\;
     \xy
 (0,0)*{\includegraphics[scale=0.5,angle=180]{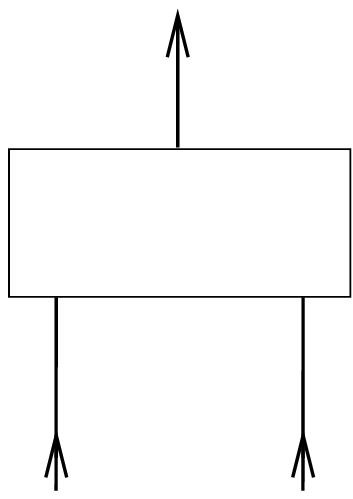}};
 (-8,-7)*{a};(8,-7)*{b};(0,1)*{e_{a+b}};(-5,10)*{a+b}; (11,6)*{\l};
  \endxy\;, e_{a+b} \right) \maps \cal{E}^{(a)}\cal{E}^{(b)}\onel \{t\} \to
  \cal{E}^{(a+b)}\onel\{t-ab\}. \nn
\end{eqnarray}
To compute the degree of the above diagrams one must account for the shift in the definition of divided powers.   For example, in the first diagram the degree shift in
the divided power for $\cal{E}^{(a+b)}\onel$ is
$\frac{(a+b)(a+b-1)}{2}$, while the degree shift in the composite
$\cal{E}^{(a)}\cal{E}^{(b)}\onel$ is
$\frac{a(a-1)}{2}+\frac{b(b-1)}{2}$, so that the net difference is
$\frac{2ab}{2} = ab$.  Both of the above diagrams in the thick calculus have
degree $-ab$.

For general $m$ there is no completely diagrammatic description of the Karoubi envelope of $\Ucat_Q(\slm)$.  In this case one lacks a set of diagrammatic relations needed to decompose arbitrary 1-morphisms into indecomposables, though explicit isomorphisms giving higher Serre relations were defined by Sto{\v{s}}i{\'c}~\cite{Stosic}. It will nevertheless be convenient to introduce a version of the 2-category $\Ucat_Q(\slm)$ where we have split the idempotents needed to define divided powers, but where we have not passed to the full Karoubi completion.  Diagrammatically this 2-category can be defined using thick strands carrying two labels, one indicating the thickness of the strand, and one indicated the label $i \in I$ of the a strands.  Since the thick strands are defined in terms of idempotents in thin strands, all the 2-morphisms can be studied using only the relations from $\Ucat_Q(\slm)$.

\begin{defn}
Let $\Ucatc_Q(\slm)$ denote the full sub-2-category of  $\UcatD_Q(\slm)$ with the same objects $\lambda \in X$ as $\UcatD_Q(\slm)$ and with 1-morphisms generated as a graded additive $\Bbbk$-linear category by  the 1-morphisms $\cal{E}_i\onel:=(\cal{E}_i \onel,\id_{\cal{E}_i \onel})$ and $\cal{E}^{(a)}_i\onel:= (\cal{E}_i^a\onel
\{\frac{a(a-1)}{2}\} , e_a)$ and their adjoints.
\end{defn}

%
\subsubsection{2-representations}
%

Let $\Ucat_Q$ denote any one of the 2-categories $\Ucat_Q(\mf{sl}_m)$, $\Ucatc_Q(\mf{sl}_m)$, or $\UcatD_Q(\mf{sl}_m)$.

\begin{defn}
A 2-representation of $\Ucat_Q$  is a graded additive $\Bbbk$-linear 2-functor $\Ucat_Q \to \cal{K}$ for some graded, additive 2-category $\cal{K}$.
\end{defn}

When all of the Hom categories $\cal{K}(x,y)$ between objects $x$ and $y$ of $\cal{K}$ are idempotent complete, in other words $Kar(\cal{K}) \cong \cal{K}$, then any graded additive $\Bbbk$-linear 2-functor $\Ucat_{Q}(\mf{g}) \to \cal{K}$ extends uniquely to a 2-representation of $\UcatD_Q(\mf{g})$.

\begin{rem}
For each $i \in I$ there is a sub 2-category $\Ucat_Q(\mf{sl}_2)_i$ of $\Ucat_Q(\slm)$ where we restrict to diagrams where all strands are labeled $i$.  For general 2-representations $\cal{F} \maps \Ucat_Q(\slm) \to \cal{K}$ it may happen that $\cal{K}$ is not Karoubi complete.  However, there are many instances when the images of divided powers $\cal{E}_i^{(a)}\onel$ and $\cal{F}_i^{(b)}\onel$ exist in $\cal{K}$.  In this case, the composite 2-functors $\cal{F}_i \maps \Ucat_Q(\mf{sl}_2)_i \to \Ucat_Q(\slm) \to \cal{K}$ extend to give 2-representations from the Karoubi envelope of the $\mf{sl}_2$ subcategories $\UcatD_Q(\mf{sl}_2)_i \to \cal{K}$. In this case, the 2-representation $\cal{F}$ extends to a 2-representation $\cal{F} \maps \Ucatc_Q(\slm) \to \cal{K}$.
\end{rem}

%
\subsubsection{Minimal relations and defining 2-functors}
%

In \cite{CLau}, it is shown that a 2-representation of $\Ucat_Q(\mf{sl}_m)$ can be specified by
defining a $2$-category
satisfying a small number of axioms. The following is a slightly stronger statement of the main theorem
from that work.

\begin{thm}[\cite{CLau} Theorem 1.1] \label{thm_CL}
A map $\mathcal{R}$ from the set of weights $X$ of $\mf{sl}_m$ to the objects of
graded additive $\Bbbk$-linear $2$-category $\cal{K}$ extends to a
2-representation $\Ucat_Q(\mf{sl}_m) \to \cal{K}$ provided the following conditions are
satisfied:
\begin{enumerate}

\item\label{co:int} The object $\cal{R}(\l+r\alpha_i)$ is
(isomorphic to) the zero object  for $r \gg 0$ or $r \ll 0$.

\item \label{co:hom}
$\Hom_{\cal{K}}(\1_{\l}, \1_{\l} \{ l \})$ is zero if $l < 0$
and one-dimensional if $l=0$, where $\1_\l$ denotes the
identity endomorphism of $\cal{R}(\l)$. Moreover, the space of 2-morphisms between any two
1-morphisms in $\cal{K}$ is finite dimensional.

\item\label{co:E} There exist 1-morphisms
$\E_i \1_{\l}: \cal{R}(\lambda) \rightarrow \cal{R}(\l + \alpha_i)$ in
$\cal{K}$ which possess both right and left adjoints.

\item \label{co:EF}
Defining 1-morphisms $\F_i \1_\l: \cal{R}(\l) \rightarrow \cal{R}(\l-\alpha_i)$ for all $\lambda \in X$ via
\begin{equation*} 
  \F_i \1_{\l+\alpha_i} := (\E_i \1_\l)_R \{ - \lambda_i - 1\}
\end{equation*}
we have the following isomorphisms in $\cal{K}$:
$$\F_i \1_{\lambda+\alpha_i} \E_i \1_{\l} \cong
\E_i \1_{\lambda-\alpha_i} \F_i \1_{\l} \oplus \left(\bigoplus_{[- \la i,\l \ra]} \1_\l \right)
\text{ if } \la i,\l \ra \le 0$$
$$\E_i \1_{\lambda-\alpha_i} \F_i \1_{\l} \cong
\F_i \1_{\lambda+\alpha_i} \E_i \1_{\l} \oplus \left( \bigoplus_{[\la i,\l \ra]} \1_\l \right)
\text{ if } \la i,\l \ra \ge 0.$$

\item \label{co:KLR} The $\E$'s carry an action of the KLR algebra associated to $Q$.

\item \label{co:EiFj} If $i \ne j \in I$ then $\F_j \1_{\l + \alpha_i} \E_i \1_{\l} \cong
\E_i \1_{\l - \alpha_j} \F_j \1_{\l}$ in $\cal{K}$.
\end{enumerate}
\end{thm}

In the above, we set
\[
\bigoplus_{f(q)} M := \bigoplus_{i= -l}^{k} (M\{i\})^{\oplus r_i}.
\]
when $\displaystyle f(q)=\sum_{i=-l}^k r_i q^i$ is a Laurent polynomial with $r_i\geq0$.

%
\subsection{Categorified Weyl group action}
%

The Weyl group for $\mathfrak{sl}_m$ is the symmetric group $\mathfrak{S}_m$  generated by transpositions $s_i$ associated to the roots $\alpha_1,\dots,\alpha_{m-1}$. The action of the Weyl group on the weights
lifts to a braid group action on representations of the associated quantum group $U_q(\mathfrak{sl}_m)$ (see for example \cite{Lus4,KamnTin,CKL}).

The action of a simple transposition is described by an element of the completion $\tilde{U_q(\slm)}$ of $U_q(\slm)$. This ring is defined as a quotient of the ring of series $\sum_{k=1}^{\infty} X_k$ of elements of $U_q(\slm)$ acting on each irreducible representation $V_{\lambda}$ of highest weight $\lambda$ by zero but for finitely many terms $X_k$ (see \cite{KamnTin}).
To $s_i$, we associate the braiding map $T_i\in \tilde{U_q(\slm)}$ :
\begin{equation} \label{qWeylAction_1}
T_{i}1_{\lambda} := \sum_{s\geq 0} (-q)^s E_{i}^{(-\lambda_i+s)}F_{i}^{(s)}1_{\lambda}
\end{equation}
if $\lambda_i\leq 0$,
\begin{equation}\label{qWeylAction_2}
T_{i}1_{\lambda} := \sum_{s\geq 0} (-q)^s F_{i}^{(\lambda_i+s)}E_{i}^{(s)}1_{\lambda}
\end{equation}
if $\lambda_i\geq 0$.
This definition differs from the one given in \cite[Section 5.2.1]{Lus4} but is equivalent up to rescaling, see \cite[Remark 6.4]{Cautis}. With this definition, $T_i=\sum_{\lambda \in X}T_i1_{\lambda}$ gives an endomorphism of any finite-dimensional representation. Note that if $v$ is a weight vector of weight $\lambda$, $T_i(v)$ is a weight vector of weight $s_i(\lambda)$. 

For $\mf{sl}_2$ the deformed Weyl group action on a $U_q(\mf{sl}_2)$-representation $V$ gives a reflection isomorphism from the $\lambda$ weight space of $V$ to the $-\lambda$ weight space.  This reflection isomorphism was categorified by Chuang and Rouquier in the context of abelian categories~\cite{CR}.  Their work is closely related to a variant of the 2-category $\Ucat_Q(\mf{sl}_2)$ where the nilHecke algebra is replaced by the affine Hecke algebra and there is no grading. Cautis, Kamnitzer and Licata later
developed analogous complexes in the context of $\Ucat_Q(\mf{sl}_2)$ and generalized Chuang and Rouquier's results to triangulated categories~\cite{CKL3}.

To categorify the reflection isomorphism $T_i1_{\lambda}$ it is clear from \eqref{qWeylAction_1} and \eqref{qWeylAction_2} that we will need to work in $\Ucatc_Q(\slm)$ so that we have lifts of divided powers. Also, the minus signs in the definition of the braid group generators suggests that we will have to pass to the 2-category $Kom(\Ucatc_Q(\slm))$ of bounded complexes over the 2-category $\Ucatc_Q(\slm)$ whose objects are weights $\lambda \in X$, 1-morphisms are bounded complexes of 1-morphisms in $\Ucatc_Q(\slm)$, and 2-morphisms are chain maps constructed from the 2-morphisms in $\Ucatc_Q(\slm)$.

 The braid group generator $T_i1_{\lambda}$ lifts to a complex $\cal{T}_i\onel$ in $Kom(\Ucatc_Q(\slm))$ of the form\footnote{Note that we take the mirror of Cautis' definition of these complexes in \cite{Cautis}, in order to better fit with usual definition of Khovanov homology. This also reverses the decategorification process so that a shift by $k$ will decategorify to $q^k$, while it decategorifies to $q^{-k}$ in \cite{Cautis}.}:
\begin{equation} \label{Rickardn}
\cal{T}_i \onel =
\xymatrix{ \cal{E}_i^{(-\l_i)} \onel \ar[r]^-{d_1} & \cal{E}_i^{(-\l_i+1)} \cal{F}_i \onel \{1\} \ar[r]^-{d_2} & \cdots
\ar[r]^-{d_s} & \cal{E}_i^{(-\l_i+s)} \cal{F}_i^{(s)} \onel \{s\} \ar[r]^-{d_{s+1}} &\cdots}
\end{equation}
when $\l_i \leq 0$ and
\begin{equation} \label{Rickardp}
\cal{T}_i \onel =
\xymatrix{ \cal{F}_i^{(\l_i)} \onel \ar[r]^-{d_1} & \cal{F}_i^{(\l_i+1)} \cal{E}_i \onel \{1\} \ar[r]^-{d_2} & \cdots
\ar[r]^-{d_s} & \cal{F}_i^{(\l_i+s)} \cal{E}_i^{(s)} \onel \{s\} \ar[r]^-{d_{s+1}} &\cdots}
\end{equation}
when $\l_i \geq 0$, where in the above formulae the leftmost term is in homological degree zero.
The differential $d_k$ that appears in the first complex is conveniently expressed in the extended graphical calculus from \cite{KLMS} as
\[
d_k =
\xy
 (0,0)*{\includegraphics[scale=0.35]{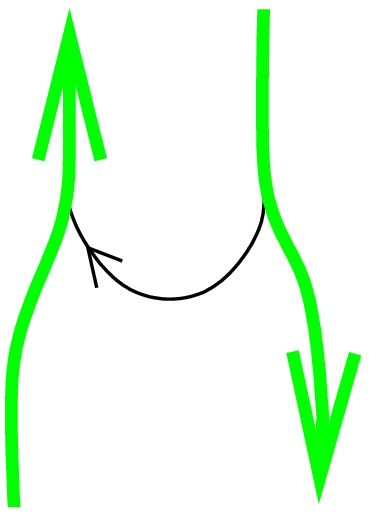}};
 (8,-7)*{\scs k};
 (-11,-7)*{\scs -\lambda_i+k};
 (8,7)*{\scs k+1};
 (-12,7)*{\scs -\lambda_i+k+1};
 (14,2)*{\lambda};
\endxy
\]
where all strands are colored by the index $i \in I$ and the labels indicate the thickness of strands.  The differential in the second complex is defined similarly.  Using the extended calculus it is easy to see that $d^2=0$.  Results of Cautis and Kamnitzer show the images of the complexes $\cal{T}_i\onel$ under any integrable 2-representation $\Ucatc(\slm) \to \cal{K}$ satisfy braid relations up to homotopy in $Kom(\cal{K})$~\cite[Section 6]{CK}.

The complexes $\cal{T}_i\onel$ are invertible, up to homotopy, with inverses given by taking the left adjoint of the complex $\cal{T}_i\onel$ in the 2-category $Kom(\Ucatc_Q(\slm)$. More explicitly, the inverses are given by
\[
\onel \cal{T}_i^{-1} =
\xymatrix{ \cdots \ar[r]^-{d_{s+1}^*} & \onel \cal{E}_i^{(s)} \cal{F}_i^{(-\l_i+s)} \{-s\} \ar[r]^-{d_s^*}
& \cdots \ar[r]^-{d_2^*}
& \onel \cal{E}_i \cal{F}_i^{(-\l_i+1)} \{-1\} \ar[r]^-{d_1^*} & \onel \cal{F}_i^{(-\l_i)}}
\]
when $\l_i \leq 0$ and
\[
\onel \cal{T}_i^{-1} =
\xymatrix{ \cdots \ar[r]^-{d_{s+1}^*} & \onel \cal{F}_i^{(s)} \cal{E}_i^{(\l_i+s)} \{-s\} \ar[r]^-{d_s^*}
& \cdots \ar[r]^-{d_2^*}
& \onel \cal{F}_i \cal{E}_i^{(\l_i+1)} \{-1\} \ar[r]^-{d_1^*} & \onel \cal{E}_i^{(\l_i)}}
\]
when $\l_i \geq 0$, where in these formulae the rightmost term is in homological degree zero.

Given a 2-representation $\cal{F} \maps \Ucatc_Q(\slm) \to \cal{K}$, the braiding $T_i1_{\lambda}$ is lifted to a complex $\cal{T}_i\onel$ that gives an equivalence between $\cal{F}(\lambda)$ and $\cal{F}(s_i(\lambda))$ in the 2-category of complexes $Kom(\cal{K})$ over $\cal{K}$.
For $\mf{sl}_2$ there are no interesting braid relations to check.  The content of a categorification of the reflection isomorphism is that the complex $\cal{T}_i\onel$ has a homotopy inverse, so that a 2-representation $\Ucatc_Q(\mf{sl}_2) \to \cal{K}$ induces an equivalence in the category of complexes over $\cal{K}$ \cite[Theorem 6.4]{CR}.  The resulting equivalences are highly nontrivial and have been applied to a variety of contexts ranging from the representation theory of the symmetric group~\cite{CR} to coherent sheaves on cotangent bundles~\cite{CKL2,CKL3}. Cautis and Kamnitzer later showed that given an integrable 2-representation  $\Ucat_Q(\mf{sl}_m) \to \cal{K}$
the complexes $\cal{T}_i\onel$ defined for each $i \in I$ satisfy the braid relations~\cite[Section 6]{CK}, see also ~\cite[Section 4.1]{Cautis}.  This is a crucial observation for Cautis's construction of knot homology theories from the 2-category $\Ucat_Q(\slm)$.

%
\section{Foams and foamation} \label{sec:foams}
%


We now aim to define families of foamation $2$-functors from $\Ucat_Q(\mathfrak{sl_m})$ to certain
$2$-categories of $\mathfrak{sl}_2$ and $\mathfrak{sl}_3$ foams. We use the particular choice of scalars $Q$
given by $t_{i,i+1}=1$, $t_{i,i-1}=-1$, and $t_{i,j}=1$ when $a_{i,j}=0$.

%
\subsection{$\mathfrak{sl}_2$ foam 2-categories}
%

In this section, we define a family of $2$-functors from $\Ucat_Q(\mathfrak{sl_m})$ to suitable
categories of $\mathfrak{sl}_2$ foams.
We first review Bar-Natan's cobordism-based construction of ($\mathfrak{sl}_2$) Khovanov homology
\cite{BN2} as well as a functorial enhancement of this theory due to Blanchet \cite{Blan}
which encodes additional representation-theoretic information. We will define our foamation $2$-functors into
a family of related $2$-categories which are natural to consider from the perspective of skew Howe duality.
We also construct such $2$-functors into the Clark-Morrison-Walker functorial
formulation of Khovanov homology \cite{CMW}.

%
\subsubsection{Standard $\mathfrak{sl}_2$ foams}
%

In \cite{BN2}, Bar Natan gave a construction of Khovanov homology as a quotient of the cobordism category of
planar tangles and surfaces. This work gives a categorification of (a version of) the
category $2\cat{Web}$.
We summarize this construction, which can be understood as a
$2$-category defined as follows:
\begin{itemize}
 \item Objects are sequences of points in the interval $[0,1]$, together with a zero object.
\item $1$-morphisms are formal direct sums of $\Z$-graded planar tangles with boundary corresponding to the
sequences of points in the domain and codomain.
\item $2$-morphisms are formal matrices of $\Bbbk$-linear combinations of degree-zero dotted
cobordisms between such
planar curves, modulo isotopy (relative to the boundary) and local relations.
\end{itemize}
If we denote the $\Z$-grading of a planar tangle by the monomial $q^t$ for $t \in \Z$, then the degree of a cobordism
$C: q^{t_1}T_1 \to q^{t_2}T_2$ is given by the formula
\begin{equation}\label{sl2foamdeg}
\deg(C) = \chi(C) -2\#D - \frac{\# \partial}{2} + t_2 - t_1
\end{equation}
where $\# D$ is the number of dots and
$\# \partial$ is the number of boundary points in either $T_1$ or $T_2$ (they agree!).
The local relations are then given as follows:
\begin{equation}\label{sl2closedfoam}
 \sphere[.65] \quad = \quad 0 \qquad , \qquad
\dottedsphere[.65]{}  \quad = \quad 1
\end{equation}
\begin{equation}\label{sl2neckcutting}
\cylinder[.6] \quad = \quad \slthncfour[.6] \quad + \quad \slthncfive[.6] \quad .
\end{equation}
The neck-cutting relation \eqref{sl2neckcutting} gives the formula:
\begin{equation} \label{sl2handle}
2 \quad \dotsheet[.65]{} \quad = \quad
\handlesheet[.65] \quad
\end{equation}
which allows for a completely topological description of the $2$-category
when $2$ is invertible in $\Bbbk$.

As mentioned in the introduction, the $+$ sign in the neck-cutting relation prevents us
from defining a $2$-functor from $\Ucat_Q(\mathfrak{sl_m})$ to this $2$-category since it is incompatible with
the sign in the nilHecke relation. We hence consider related versions of this construction.

%
\subsubsection{Enhanced foams}
%
Bar-Natan formulates Khovanov homology in the homotopy category of complexes in the above
$2$-category, giving an invariant which is functorial only up to a $(\pm 1)$-sign under tangle cobordism.
This functoriality issue was fixed by Clark, Morrison, and Walker \cite{CMW}
working in a related $2$-category of disoriented curves and cobordisms defined over the
Gaussian integers\footnote{Actually, they work over the ring $\Z[\frac{1}{2},i]$.} (see also the work of
Caprau \cite{Cap3}, \cite{Cap4}, \cite{Cap2} for a related construction).
Blanchet \cite{Blan} later gave another functorial construction of Khovanov homology in a
related $2$-category defined over the integers.

It turns out that in addition to fixing functoriality, these later constructions also fix the incompatibility of the
neck-cutting and nilHecke relations. We will work in Blanchet's enhanced foam model since it is more
natural to consider from the perspective of skew Howe duality and it avoids the introduction of complex
coefficients. We return to the Clark-Morrison-Walker (CMW) construction in the following section.

We begin by defining a family of $2$-categories related to Blanchet's construction which should be viewed as
categorifications of the categories $2\cat{BWeb}_m(N)$.

\begin{defn}
$\Bfoam{2}{m}$ is the $2$-category defined as follows:
\begin{itemize}
\item Objects are sequences $(a_1,\ldots,a_m)$ labeling points in the interval $[0,1]$
with $a_i \in \{0,1,2\}$ and $N = \sum_{i=1}^m a_i$, together with a zero object.
\item $1$-morphisms are formal direct sums of $\Z$-graded enhanced $\mathfrak{sl}_2$ webs -
directed planar graphs with boundary with two types of edges - $1$-labeled edges
$ \; \;
\xy
(0,0)*{
\begin{tikzpicture} [fill opacity=0.2,  decoration={markings,
                        mark=at position 0.6 with {\arrow{>}};    }]
\draw[very thick, postaction={decorate}] (0,-.5) -- (0,.5);
\end{tikzpicture}};
\endxy
\; \;
$
and $2$-labeled edges
$ \; \;
\xy
(0,0)*{
\begin{tikzpicture} [fill opacity=0.2,  decoration={markings,
                        mark=at position 0.6 with {\arrow{>}};    }]
\draw[double, postaction={decorate}] (0,-.5) -- (0,.5);
\end{tikzpicture}};
\endxy
\; \;
$
- where all vertices are trivalent and take the following two forms:
\begin{equation}\label{sl2vertices}
\xy
(0,0)*{
\begin{tikzpicture} [fill opacity=0.2,  decoration={markings,
                        mark=at position 0.6 with {\arrow{>}};    }]
\draw[double, postaction={decorate}] (0,0) -- (0,1);
\draw[very thick, postaction={decorate}] (-.875,-.5) -- (0,0);
\draw[very thick, postaction={decorate}] (.875,-.5) -- (0,0);
\end{tikzpicture}};
\endxy
\qquad
\text{ or }
\qquad
\xy
(0,0)*{
\begin{tikzpicture} [fill opacity=0.2,  decoration={markings,
                        mark=at position 0.6 with {\arrow{>}};    }]
\draw[double, postaction={decorate}] (0,1) -- (0,0);
\draw[very thick, postaction={decorate}] (0,0) -- (.875,-.5);
\draw[very thick, postaction={decorate}] (0,0) -- (-.875,-.5);
\end{tikzpicture}};
\endxy
\quad .
\end{equation}
$1$- (respectively $2$-) labeled edges are directed out from points labeled by $1$ (respectively $2$) in the
domain and directed into such labeled points in the codomain. No edges are attached to points labeled by $0$.
\item $2$-morphisms are formal matrices of $\Bbbk$-linear combinations of degree-zero $\mf{sl}_2$ foams -
surfaces with oriented singular seams which locally look like the product of the letter $Y$ with an
interval - considered up to isotopy (relative to the boundary) and local relations.
\end{itemize}
\end{defn}

There are two types of facets of an $\mathfrak{sl}_2$ foam, $1$-labeled and $2$-labeled, depending on
which type of edge they are incident upon when the foam is expressed as a composition
of elementary foams. The degree of a foam $F:q^{t_1}W_1 \to q^{t_2}W_2$ is given by the degree of the cobordism
resulting from deleting all the $2$-labeled facets and edges and forgetting the orientation of the $1$-labeled edges.

As in $\Ucat_Q(\mathfrak{sl_m})$, we shall read diagrammatic depictions of webs and foams from right to left
and from bottom to top. The orientation of a singular seam gives a cyclic ordering of the facets incident upon the
seam via the right hand rule.
By convention, a seam travels down through the first vertex in \eqref{sl2vertices} and up through the second; this
corresponds to the cyclic orientation of web vertices from \cite{Blan}.

The relations for $\mathfrak{sl}_2$ foams come from a non-local, universal construction detailed in
\cite{Blan}; however, we can exhibit a complete set of local relations giving an equivalent description of
Blanchet's work. In what follows, the $2$-labeled facets are depicted in yellow and $1$-labeled facets are drawn
in red.

We impose the relations \eqref{sl2closedfoam} and \eqref{sl2neckcutting} for $1$-labeled facets, as well as the following relations involving $2$-labeled facets:

\begin{equation} \label{sl2neckcutting_enh_2lab}
\xy
(0,0)*{
};
\endxy
\end{equation}

Relations \eqref{sl2closedfoam}, \eqref{sl2neckcutting}, \eqref{sl2neckcutting_enh_2lab},
\eqref{2labeledsphere}, and
\eqref{sl2thetafoam_enh} allow for the evaluation of any closed $\mathfrak{sl}_2$ foam. Additionally,
note that these relations imply that we can reverse the direction of any closed, singular seam at the
cost of multiplying by $-1$.
Equations \eqref{sl2Fig5Blanchet_1},  \eqref{sl2Fig5Blanchet_2}, and \eqref{sl2NH_enh} guarantee
that if a linear combination of foams evaluates to zero whenever it is ``closed off'' to give a closed foam, then
that linear combination is zero; the latter is a non-local relation from \cite{Blan}.
The equivalence of this relation to the
collection of local relations given above can be proved in a manner similar to the proof of
\cite[Lemma 3.5]{MorrisonNieh} using the above local relations.  The proof utilizes several relations that follow from the above local relations, allowing a web with a digon or square face to be expressed in terms of webs with fewer faces:
\begin{equation} \label{sl2bubble12}
\xy
{0,0}*{
};
\endxy \quad .
\end{equation}

The neck-cutting relation \eqref{sl2neckcutting_enh_2lab} implies that the topology of the $2$-labeled
facets plays a limited role. One may hence ask if there is a way to coherently deleted such facets and
obtain a forgetful $2$-functor from $\Bfoam{2}{m}$ to (an appropriate version of) the Bar-Natan
$2$-category.

Such a $2$-functor would act via
\[
\xy
(0,0)*{
%
};
\endxy \quad \mapsto \quad
\alpha \quad \cylinder[.6]
\]
for some scalar $\alpha$;
equation \eqref{sl2bubble12} shows that composing the left-hand foam with a cap produces a cap, while
pre-composing with a cup gives $-1$ multiplied by a cup.
It is therefore impossible to define such a $2$-functor which acts as the identity on foams which
contain no $2$-labeled facets.

%
\subsubsection{Foamation}
%

We now define $\mf{sl}_2$ foamation $2$-functors $\mathcal{U}_Q(\mf{sl}_m) \to \Bfoam{2}{m}$
categorifying the skew Howe map to webs discussed in the introduction.
As in the decategorified case, we define the $2$-functor on objects by sending an $\mf{sl}_m$ weight
$\lambda = (\lambda_1,\ldots,\lambda_{m-1})$ to the sequence $(a_1,\ldots,a_m)$ with
$a_i \in \{0,1,2\}$,
$\lambda_i = a_{i+1}-a_i$, and $\sum_{i=1}^{m}a_i = N$ provided it exists
and to the zero object otherwise.

The map is given on $1$-morphisms by
\[
\onel \{t\} \mapsto q^t
\xy
(0,0)*{
\begin{tikzpicture} [decoration={markings, mark=at position 0.6 with {\arrow{>}}; },scale=.75]
\draw [very thick, postaction={decorate}] (3,0) -- (0,0);
\draw [very thick, postaction={decorate}] (3,1) -- (0,1);
\node at (3.6,0) {$a_1$};
\node at (3.6,1) {$a_m$};
\node at (1.5,.7) {$\vdots$};
\end{tikzpicture}};
\endxy
\]
\[
\cal{E}_i \onel \{t\} \mapsto q^t
\xy
(0,0)*{
\begin{tikzpicture} [decoration={markings, mark=at position 0.6 with {\arrow{>}}; },scale=.75]
\draw [very thick, postaction={decorate}] (3,0) -- (2,0);
\draw[very thick, postaction={decorate}] (2,0) -- (0,0);
\draw [very thick, postaction={decorate}] (3,1) -- (1,1);
\draw[very thick, postaction={decorate}] (1,1) -- (0,1);
\draw [very thick, postaction={decorate}] (2,0) -- (1,1);
\node at (3.6,0) {$a_i$};
\node at (3.6,1) {$a_{i+1}$};
\node at (-1,0) {$a_{i} - 1$};
\node at (-1,1) {$a_{i+1}+1$};
\end{tikzpicture}};
\endxy
\]
and
\[
\cal{F}_i \onel \{t\} \mapsto q^t
\xy
(0,0)*{
\begin{tikzpicture} [decoration={markings, mark=at position 0.6 with {\arrow{>}}; },scale=.75]
\draw [very thick, postaction={decorate}] (3,0) -- (1,0);
\draw[very thick, postaction={decorate}] (1,0) -- (0,0);
\draw [very thick, postaction={decorate}] (3,1) -- (2,1);
\draw[very thick, postaction={decorate}] (2,1) -- (0,1);
\draw [very thick, postaction={decorate}] (2,1) -- (1,0);
\node at (3.6,0) {$a_i$};
\node at (3.6,1) {$a_{i+1}$};
\node at (-1,0) {$a_{i} + 1$};
\node at (-1,1) {$a_{i+1} - 1$};
\end{tikzpicture}};
\endxy
\]
when the boundary values lie in $\{0,1,2\}$ and to the zero $1$-morphism otherwise. The labelings
of the edges incident
upon the boundary are given by the boundary labels; edges incident upon boundary points labeled by zero
should be deleted. Note that we have not depicted $m-2$ horizontal strands in each of the latter two formulae.

We will make use of a preparatory lemma to deduce the existence of the foamation functors. Let the
images of $\onel$, $\cal{E}_i \onel$,  and $\cal{F}_i \onel$ given above be denoted
$\1_{\lambda}$, $\E_i \1_{\lambda} $, and $\F_i \1_{\lambda} $.

\begin{lem} \label{lem_abstract2}
There are  isomorphisms
$$\F_i \E_i \1_{\l} \cong \E_i \F_i \1_{\l} \oplus \left( \bigoplus_{[- \la i,\l \ra]} \1_\l \right) \text{ if } \la i,\l \ra \le 0$$
$$\E_i \F_i \1_{\l} \cong \F_i \E_i \1_{\l} \oplus \left( \bigoplus_{[\la i,\l \ra]} \1_\l \right) \text{ if } \la i,\l \ra \ge 0,$$
and $\F_j \E_i \1_{\l} \cong \E_i \F_j \1_{\l}$ for $i \ne j \in I$ in $\foam{2}{m}$.
\end{lem}

\begin{proof}
We'll prove only the first relation since the proof of the second is analogous and the third is straightforward.
The condition on weights implies that $\lambda$ maps to a sequence with $a_{i+1}\leq a_i$.

If $a_{i}=0$, $a_{i+1}=0$, both sides of the equation map to the zero foam.
If $a_{i}=1$, $a_{i+1}=0$, then the web:
\[
\xy
(0,0)*{
};
\endxy
\]
is isomorphic to $\1_\l$ via the foam realizing the web isotopy.

If $a_{i}=1=a_{i+1}$, then the relevant webs are isotopic:

\[
\xy
(0,0)*{
%
};
\endxy
\]
hence isomorphic.

If $a_i=2$, $a_{i+1}=0$, then the web:
\[
\xy
(0,0)*{
%
};
\endxy
\]
is isomorphic to $q^{-1} \1_\l\oplus q \1_\l$ using equations \eqref{sl2NH_enh}, \eqref{sl2bubble11nodot},
and \eqref{sl2bubble11dot}.
This confirms the relation since $\E_i \F_i \1_{\l} \mapsto 0$ for such a weight.

If $a_i=2$, $a_{i+1}=1$, then the web:
\[
\xy
(0,0)*{
%
};
\endxy
\]
is isomorphic to $\1_\l$ using equations \eqref{sl2Tube_2} and \eqref{sl2sq2}.
Finally, if $a_{i}=a_{i+1}=2$, both sides of the equation map to zero.
\end{proof}

\begin{prop} \label{prop_sl2}
For each $N>0$ there is a $2$-representation
$\Phi_2  \maps \Ucat_Q(\mf{sl}_m)  \to  \Bfoam{2}{m}$ defined on single strand $2$-morphisms by:
\[
\Phi_2 \left(
 \xy 0;/r.17pc/:
 (0,7);(0,-7); **\dir{-} ?(.75)*\dir{>};
 (7,3)*{ \scs \lambda};
 (-2.5,-6)*{\scs i};
 (-10,0)*{};(10,0)*{};
 \endxy
\right) = \quad \Efoam[.5] \quad , \quad
\Phi_2 \left(
 \xy 0;/r.17pc/:
 (0,7);(0,-7); **\dir{-} ?(.75)*\dir{>};
 (0,0)*{\bullet};
 (7,3)*{ \scs \lambda};
 (-2.5,-6)*{\scs i};
 (-10,0)*{};(10,0)*{};
 \endxy
\right) = \quad \dotEfoam[.5]
\]
on crossings by:
\[
\qquad
\Phi_2 \left(
\xy 0;/r.20pc/:
	(0,0)*{\xybox{
    	(-4,-4)*{};(4,4)*{} **\crv{(-4,-1) & (4,1)}?(1)*\dir{>} ;
    	(4,-4)*{};(-4,4)*{} **\crv{(4,-1) & (-4,1)}?(1)*\dir{>};
    	(-5.5,-3)*{\scs i};
     	(5.5,-3)*{\scs i};
     	(9,1)*{\scs  \lambda};
     	(-10,0)*{};(10,0)*{};}};
\endxy
\right) = \quad \crossingEEfoam[.5]
\]
\[
\Phi_2 \left(
\xy 0;/r.20pc/:
	(0,0)*{\xybox{
    	(-4,-4)*{};(4,4)*{} **\crv{(-4,-1) & (4,1)}?(1)*\dir{>} ;
    	(4,-4)*{};(-4,4)*{} **\crv{(4,-1) & (-4,1)}?(1)*\dir{>};
    	(-5.5,-3)*{\scs i};
     	(5.5,-3)*{\scs \ \ \ i+1};
     	(9,1)*{\scs  \lambda};
     	(-10,0)*{};(10,0)*{};}};
\endxy
\right) = \quad \crossingEoneEtwofoam[.5] \quad , \qquad
\Phi_2 \left(
\xy 0;/r.20pc/:
	(0,0)*{\xybox{
    	(-4,-4)*{};(4,4)*{} **\crv{(-4,-1) & (4,1)}?(1)*\dir{>} ;
    	(4,-4)*{};(-4,4)*{} **\crv{(4,-1) & (-4,1)}?(1)*\dir{>};
    	(-5.5,-3)*{\scs i+1 \ \ \ };
     	(5.5,-3)*{\scs i};
     	(9,1)*{\scs  \lambda};
     	(-10,0)*{};(10,0)*{};}};
\endxy
\right) = \quad \crossingEtwoEonefoam[.5]
\]

\[
\Phi_2 \left(
\xy 0;/r.20pc/:
	(0,0)*{\xybox{
    	(-4,-4)*{};(4,4)*{} **\crv{(-4,-1) & (4,1)}?(1)*\dir{>} ;
    	(4,-4)*{};(-4,4)*{} **\crv{(4,-1) & (-4,1)}?(1)*\dir{>};
    	(-5.5,-3)*{\scs j};
     	(5.5,-3)*{\scs i};
     	(9,1)*{\scs  \lambda};
     	(-10,0)*{};(10,0)*{};}};
\endxy
\right) = \quad \crossingEthreeEonefoam[.5] \quad , \qquad
\Phi_2 \left(
\xy 0;/r.20pc/:
	(0,0)*{\xybox{
    	(-4,-4)*{};(4,4)*{} **\crv{(-4,-1) & (4,1)}?(1)*\dir{>} ;
    	(4,-4)*{};(-4,4)*{} **\crv{(4,-1) & (-4,1)}?(1)*\dir{>};
    	(-5.5,-3)*{\scs i};
     	(5.5,-3)*{\scs j};
     	(9,1)*{\scs  \lambda};
     	(-10,0)*{};(10,0)*{};}};
\endxy
\right) = \quad \crossingEoneEthreefoam[.5]
\]
where $j-i > 1$, and on caps and cups by:
\[
\Phi_2 \left(
\xy 0;/r.20pc/:
	(0,0)*{\bbcef{i}};
	(8,4)*{\scs  \lambda};
	(-10,0)*{};(10,0)*{};
\endxy
\right) =  \quad \capFEfoam[.5] \quad , \qquad
\Phi_2 \left(
\xy 0;/r.20pc/:
	(0,0)*{\bbcfe{i}};
	(8,4)*{\scs  \lambda};
	(-10,0)*{};(10,0)*{};
\endxy
\right) = (-1)^{a_i} \quad \capEFfoam[.5]
\]

\[
\Phi_2 \left(
\xy 0;/r.20pc/:
	(0,0)*{\bbpef{i}};
	(8,-4)*{\scs \lambda};
	(-10,0)*{};(10,0)*{};
\endxy
\right) = (-1)^{a_i + 1} \quad \cupFEfoam[.5] \quad , \qquad
\Phi_2 \left(
\xy 0;/r.20pc/:
	(0,0)*{\bbpfe{i}};
	(8,-4)*{\scs \lambda};
	(-10,0)*{};(10,0)*{};
\endxy
\right) = \quad \cupEFfoam[.5]
\]
where in the above diagrams the $i^{th}$ sheet is always in the front.
\end{prop}

The foams drawn above are general depictions of the images. To obtain the specific image foam we delete any
facets incident upon deleted web edges and re-color facets appropriately (in particular, the blue colored facets
in the above are used only to make the pictures more readable).
The singular seams may degenerate in such examples, e.g.
\[
\Phi_2 \left(
\xy 0;/r.20pc/:
	(0,0)*{\xybox{
    	(-4,-4)*{};(4,4)*{} **\crv{(-4,-1) & (4,1)}?(1)*\dir{>} ;
    	(4,-4)*{};(-4,4)*{} **\crv{(4,-1) & (-4,1)}?(1)*\dir{>};
    	(-5.5,-3)*{\scs i};
     	(5.5,-3)*{\scs i};
     	(9,1)*{\scs  \lambda};
     	(-10,0)*{};(10,0)*{};}};
\endxy
\right) = \quad
\xy
(0,0)*{
};
\endxy
\]
when $\lambda$ maps to a sequence with $a_i = 2$ and $a_{i+1}=0$.

\begin{proof}
While it is not difficult to verify all relations by hand, we apply Theorem~\ref{thm_CL} to $\Bfoam{2}{m}$ to reduce the number of relations that need to be verified. For each $m$ and $N$, the
non-zero objects of this $2$-category are indexed by the non-zero $\mathfrak{sl}_m$ weight spaces of the
finite-dimensional $U_q(\mathfrak{sl}_m)$-module $\wedge_q^N(\C^2 \otimes \C^m)$,
so condition \eqref{co:int} is satisfied.
Furthermore, it is clear from the definitions that $\cal{E}_i\1_\l$ has $\cal{F}_i\1_{\l+\alpha_i}$ as a left and
right adjoint, up to a grading shift.
Lemma~\ref{lem_abstract2} establishes conditions \eqref{co:EF} and \eqref{co:EiFj},
thus, it suffices to show that conditions \eqref{co:hom} and \eqref{co:KLR} are satisfied.

We first check condition \eqref{co:hom}.
Given a foam in $\Hom(\1_\l, q^t \1_\l)$, we can apply the neck-cutting relations \eqref{sl2neckcutting}
and \eqref{sl2neckcutting_enh_2lab} in the neighborhood of each boundary component to express the
foam as a linear combination of foams which are the disjoint unions of closed foams,
$2$-labeled sheets, and $1$-labeled sheets, which may carry dots. Relations \eqref{sl2closedfoam},
\eqref{sl2neckcutting}, \eqref{sl2neckcutting_enh_2lab}, \eqref{2labeledsphere}, and
\eqref{sl2thetafoam_enh} give that any closed foam is equal to an element of
$\Bbbk[\; \dottedsphere[.5]{3 \ } \;]$.
Equation \eqref{sl2foamdeg} then shows that $\Hom(\1_\l, q^t \1_\l)$ is zero for $t<0$ and $1$-dimensional
for $t=0$.

Using the neck-cutting relations, we can express any foam mapping between fixed webs $W_1$ and $W_2$
as a linear combination of foams in which every
$2$-labeled facet is a disk incident upon the boundary; such facets
are determined by the collection of singular seams incident upon the web vertices. The union of the
$1$-labeled facets gives a (dotted) cobordism between the $1$-labeled edges of $W_1$ and $W_2$.
Using the neck-cutting relations and \eqref{sl2NH_enh} we can assume that this cobordism
consists of (dotted) disks.
Since there are only finitely many ways to connect the vertices of the boundary webs with singular seams
lying on the cobordism (up to isotopy), it follows from \eqref{sl2foamdeg} that
$\Hom(q^{t_1} W_1, q^{t_2} W_2)$ is finite dimensional for all values of $t_1$ and $t_2$.

Finally, we check condition~\eqref{co:KLR}, i.e. that the KLR relations are satisfied.
\begin{itemize}
 \item Relation \eqref{eq_nil_rels}: the $2$-morphism
$
\vcenter{
\xy 0;/r.11pc/:
	(-4,-4)*{};(4,4)*{} **\crv{(-4,-1) & (4,1)}?(1)*\dir{};
	(4,-4)*{};(-4,4)*{} **\crv{(4,-1) & (-4,1)}?(1)*\dir{};
	(-4,4)*{};(4,12)*{} **\crv{(-4,7) & (4,9)}?(1)*\dir{};
	(4,4)*{};(-4,12)*{} **\crv{(4,7) & (-4,9)}?(1)*\dir{};
	(-4,12); (-4,13) **\dir{-}?(1)*\dir{>};
	(4,12); (4,13) **\dir{-}?(1)*\dir{>};
	(6,8)*{\scs \lambda};
\endxy}
$
automatically maps to the zero foam unless $\l$ maps to a sequence with
$a_i=2$ and $a_{i+1}=0$. In this case, we compute the image:
\[
\xy
(0,-9)*{
};
\endxy \quad
= 0
\]
which follows from equation \eqref{sl2bubble11nodot}.
The images of the $2$-morphisms
$
\xy
(0,0)*{
\xy 0;/r.11pc/:
    (-4,-4)*{};(4,4)*{} **\crv{(-4,-1) & (4,1)}?(1)*\dir{};
    (4,-4)*{};(-4,4)*{} **\crv{(4,-1) & (-4,1)}?(1)*\dir{};
    (4,4)*{};(12,12)*{} **\crv{(4,7) & (12,9)}?(1)*\dir{};
    (12,4)*{};(4,12)*{} **\crv{(12,7) & (4,9)}?(1)*\dir{};
    (-4,12)*{};(4,20)*{} **\crv{(-4,15) & (4,17)}?(1)*\dir{};
    (4,12)*{};(-4,20)*{} **\crv{(4,15) & (-4,17)}?(1)*\dir{};
    (-4,4)*{}; (-4,12) **\dir{-};
    (12,-4)*{}; (12,4) **\dir{-};
    (12,12)*{}; (12,20) **\dir{-};
    (4,20); (4,21) **\dir{-}?(1)*\dir{>};
    (-4,20); (-4,21) **\dir{-}?(1)*\dir{>};
    (12,20); (12,21) **\dir{-}?(1)*\dir{>};
   (18,8)*{\scs \lambda};
\endxy}
\endxy
$
and
$
\xy
(0,0)*{
\xy 0;/r.11pc/:
    (4,-4)*{};(-4,4)*{} **\crv{(4,-1) & (-4,1)}?(1)*\dir{};
    (-4,-4)*{};(4,4)*{} **\crv{(-4,-1) & (4,1)}?(1)*\dir{};
    (-4,4)*{};(-12,12)*{} **\crv{(-4,7) & (-12,9)}?(1)*\dir{};
    (-12,4)*{};(-4,12)*{} **\crv{(-12,7) & (-4,9)}?(1)*\dir{};
    (4,12)*{};(-4,20)*{} **\crv{(4,15) & (-4,17)}?(1)*\dir{};
    (-4,12)*{};(4,20)*{} **\crv{(-4,15) & (4,17)}?(1)*\dir{};
    (4,4)*{}; (4,12) **\dir{-};
    (-12,-4)*{}; (-12,4) **\dir{-};
    (-12,12)*{}; (-12,20) **\dir{-};
    (4,20); (4,21) **\dir{-}?(1)*\dir{>};
    (-4,20); (-4,21) **\dir{-}?(1)*\dir{>};
    (-12,20); (-12,21) **\dir{-}?(1)*\dir{>};
  (10,8)*{\scs \lambda};
\endxy}
\endxy
$
are both zero since either $\l$ or $\l+3\alpha_i$ maps to the zero object, confirming the relation. \\

\item Relation \eqref{eq_nil_dotslide}: As before, the only non-trivial case is when $\l$ maps to
a sequence with $a_i=2$ and $a_{i+1}=0$. In this case, we must have the equalities:
\begin{align*}
\xy
(0,0)*{
};
\endxy
\end{align*}
both of which follow from \eqref{sl2NH_enh}. \\

\item Relation \eqref{eq_r2_ij-gen}: The equality $(\alpha_i,\alpha_j) =0$ corresponds to $|i-j| \geq 2$ in which
case the image of the relation is realized via an isotopy. For example, when $i<j$ we see that
\[
\xy
(0,-10)*{\crossingEoneEthreefoam[.5]};
(0,10)*{\crossingEthreeEonefoam[.5]};
\endxy
\]
is isotopic to the identity foam for any values of the $a_l$'s.

If $i\neq j$ and $(\alpha_i,\alpha_j)\neq 0$ we must have $j=i\pm1$. We begin with the case
$j=i+1$. The image of the left-hand side is zero
unless $a_{i+1} = 1$ since the intermediate objects in the relation map to the zero object in the image;
similarly, the right-hand image is zero unless $a_{i+1}=1,2$.

When $a_{i+1}=1$ we have
\[
\vcenter{\xy 0;/r.17pc/:
    (-4,-4)*{};(4,4)*{} **\crv{(-4,-1) & (4,1)}?(1)*\dir{};
    (4,-4)*{};(-4,4)*{} **\crv{(4,-1) & (-4,1)}?(1)*\dir{};
    (-4,4)*{};(4,12)*{} **\crv{(-4,7) & (4,9)}?(1)*\dir{};
    (4,4)*{};(-4,12)*{} **\crv{(4,7) & (-4,9)}?(1)*\dir{};
    (8,8)*{\lambda};
    (4,12); (4,13) **\dir{-}?(1)*\dir{>};
    (-4,12); (-4,13) **\dir{-}?(1)*\dir{>};
  (-5.5,-3)*{\scs i};
     (5.5,-3)*{\scs \ \ \ i+1};
 \endxy} \quad
\mapsto \quad
\xy
(0,-6.45)*{
\begin{tikzpicture}  [fill opacity=0.2,  decoration={markings, mark=at position 0.6 with {\arrow{>}};}, scale=.6]
	\draw (-3,1) rectangle (1,4);
	\draw[very thick] (-3,1) -- (1,1);
	\draw[very thick] (-3,4) -- (1,4);
	\draw [very thick, red] (-1.75,4) .. controls (-1.75,3) and (-.25,2) .. (-.25,1);
	\draw[very thick, postaction={decorate}] (-.75,3) -- (-1.75,4);
	\draw[very thick, postaction={decorate}] (.75,0) -- (-.25,1);
	\path [fill=red] (-1.75,4) .. controls (-1.75,3) and (-.25,2) .. (-.25,1) --
			(.75,0) .. controls (.75,1) and (-.75,2) .. (-.75,3);
	\path[fill=red] (2,0) -- (2,3) -- (.75,3) .. controls (.75,2.35) and (0,1.5) .. (0,1.5) ..
		controls (0,1.5) and (.75,.65) .. (.75,0) -- cycle;
	\draw (2,0) -- (2,3);
	\path[fill=red] (-2,0) -- (-2,3) -- (-.75,3) .. controls (-.75,2.35) and (0,1.5) .. (0,1.5) ..
		controls (0,1.5) and (-.75,.65) .. (-.75,0) -- cycle;
	\draw (-2,0) -- (-2,3);
	\path[fill=yellow , fill opacity=0.3] (-.75,3) .. controls (-.75,2.35) and (0,1.5) .. (0,1.5) .. controls (0,1.5) and (.75,2.35) .. (.75,3);
	\draw (.75,0) .. controls (.75,.65) and (0,1.5) .. (0,1.5);
	\draw (-.75,0) .. controls (-.75,.65) and (0,1.5) .. (0,1.5);
	\draw[very thick, red, postaction={decorate}] (.75,3) .. controls (.75,2.35) and (0,1.5) .. (0,1.5);
	\draw[very thick, red] (-.75,3) .. controls (-.75,2.35) and (0,1.5) .. (0,1.5);
	\draw[very thick, postaction={decorate}] (2,3) -- (.75,3);
	\draw[very thick, postaction={decorate}] (-.75,3) -- (-2,3);
	\draw[very thick, postaction={decorate}] (2,0) -- (.75,0);
	\draw[very thick, postaction={decorate}] (-.75,0) -- (-2,0);
	\draw[double, postaction={decorate}] (.75,3) -- (-.75,3);
	\draw[very thick, postaction={decorate}] (1.75,2) -- (.75,3);
	\draw[very thick, postaction={decorate}] (.25,-1) -- (-.75,0);
	\path [fill=red] (.75,3) .. controls (.75,2) and (-.75,1) .. (-.75,0) --
			(.25,-1) .. controls (.25,0) and (1.75,1) .. (1.75,2);
	\draw (-1,-1) rectangle (3,2);
	\draw[very thick] (-1,-1) -- (3,-1);
	\draw[very thick] (-1,2) -- (3,2);
	\draw[very thick, red] (.25,-1) .. controls (.25,0) and (1.75,1) .. (1.75,2);
\end{tikzpicture}
};
(0,11.5)*{
\begin{tikzpicture}  [fill opacity=0.2,  decoration={markings, mark=at position 0.6 with {\arrow{>}};}, scale=.6]
	\draw (-3,1) rectangle (1,4);
	\draw[very thick] (-3,1) -- (1,1);
	\draw[very thick] (-3,4) -- (1,4);
	\draw [very thick, red] (-.25,4) .. controls (-.25,3) and (-1.75,2) .. (-1.75,1);
	\draw[very thick, postaction={decorate}] (.75,3) -- (-.25,4);
	\draw[very thick, postaction={decorate}] (-.75,0) -- (-1.75,1);
	\path [fill=red] (-.25,4) .. controls (-.25,3) and (-1.75,2) .. (-1.75,1) --
			(-.75,0) .. controls (-.75,1) and (.75,2) .. (.75,3);
	\path[fill=red] (2,0) -- (2,3) -- (.75,3) .. controls (.75,2.35) and (0,1.5) .. (0,1.5) ..
		controls (0,1.5) and (.75,.65) .. (.75,0) -- cycle;
	\draw (2,0) -- (2,3);
	\path[fill=red] (-2,0) -- (-2,3) -- (-.75,3) .. controls (-.75,2.35) and (0,1.5) .. (0,1.5) ..
		controls (0,1.5) and (-.75,.65) .. (-.75,0) -- cycle;
	\draw (-2,0) -- (-2,3);
	\path[fill=yellow , fill opacity=0.3] (-.75,0) .. controls (-.75,.65) and (0,1.5) .. (0,1.5) .. controls (0,1.5) and (.75,.65) .. (.75,0);
	\draw (.75,3) .. controls (.75,2.35) and (0,1.5) .. (0,1.5);
	\draw (-.75,3) .. controls (-.75,2.35) and (0,1.5) .. (0,1.5);
	\draw[very thick, red] (.75,0) .. controls (.75,.65) and (0,1.5) .. (0,1.5);
	\draw[very thick, red, postaction={decorate}] (-.75,0) .. controls (-.75,.65) and (0,1.5) .. (0,1.5);
	\draw[very thick, postaction={decorate}] (2,3) -- (.75,3);
	\draw[very thick, postaction={decorate}] (-.75,3) -- (-2,3);
	\draw[very thick, postaction={decorate}] (2,0) -- (.75,0);
	\draw[very thick, postaction={decorate}] (-.75,0) -- (-2,0);
	\draw[double, postaction={decorate}] (.75,0) -- (-.75,0);
	\draw[very thick, postaction={decorate}] (.25,2) -- (-.75,3);
	\draw[very thick, postaction={decorate}] (1.75,-1) -- (.75,0);
	\path [fill=red] (-.75,3) .. controls (-.75,2) and (.75,1) .. (.75,0) --
			(1.75,-1) .. controls (1.75,0) and (.25,1) .. (.25,2);
	\draw (-1,-1) rectangle (3,2);
	\draw[very thick] (-1,-1) -- (3,-1);
	\draw[very thick] (-1,2) -- (3,2);
	\draw[very thick, red] (1.75,-1) .. controls (1.75,0) and (.25,1) .. (.25,2);
\end{tikzpicture}
};
\endxy \quad
\]
where we have omitted the shading on the front and back sheets for clarity.
Applying \eqref{sl2neckcutting} on both sides of the singular seam and evaluating the resulting
theta-foams using \eqref{sl2thetafoam_enh}, this gives
\[
\xy
(0,0)*{
\begin{tikzpicture}  [fill opacity=0.2,  decoration={markings, mark=at position 0.6 with {\arrow{>}};}, scale=.6]
	\draw (-3,1) rectangle (1,4);
	\draw[very thick] (-3,1) -- (1,1);
	\draw[very thick] (-3,4) -- (1,4);
	\draw [very thick, red] (-.25,4) -- (-.25,1);
	\draw[very thick, postaction={decorate}] (.75,3) -- (-.25,4);
	\draw[very thick, postaction={decorate}] (.75,0) -- (-.25,1);
	\draw[very thick, postaction={decorate}] (2,0) -- (.75,0);
	\draw[very thick, postaction={decorate}] (2,3) -- (.75,3);
	\draw[very thick, postaction={decorate}] (-.75,0) -- (-2,0);
	\draw[very thick, postaction={decorate}] (-.75,3) -- (-2,3);
	\draw (.75,3) -- (.75,0);
	\draw (-.75,3) -- (-.75,0);
	\draw (2,3) -- (2,0);
	\draw (-2,3) -- (-2,0);
	\path [fill=red] (-2,3) rectangle (-.75,0);
	\path [fill=red] (2,3) rectangle (.75,0);
	\path [fill=red] (-.75,3) -- (.25,2) -- (.25,-1) -- (-.75,0) -- cycle;
	\path [fill=red] (.75,3) -- (-.25,4) -- (-.25,1) -- (.75,0) -- cycle;
	\node[opacity=1] at (-.25, 0) {$\bullet$};
	\draw[very thick, postaction={decorate}] (.25,2) -- (-.75,3);
	\draw[very thick, postaction={decorate}] (.25,-1) -- (-.75,0);
	\draw (-1,-1) rectangle (3,2);
	\draw[very thick] (-1,-1) -- (3,-1);
	\draw[very thick] (-1,2) -- (3,2);
	\draw[very thick, red] (.25,-1) -- (.25,2);
\end{tikzpicture}
};
\endxy \quad - \quad
\xy
(0,0)*{
\begin{tikzpicture}  [fill opacity=0.2,  decoration={markings, mark=at position 0.6 with {\arrow{>}};}, scale=.6]
	\draw (-3,1) rectangle (1,4);
	\draw[very thick] (-3,1) -- (1,1);
	\draw[very thick] (-3,4) -- (1,4);
	\draw [very thick, red] (-.25,4) -- (-.25,1);
	\draw[very thick, postaction={decorate}] (.75,3) -- (-.25,4);
	\draw[very thick, postaction={decorate}] (.75,0) -- (-.25,1);
	\draw[very thick, postaction={decorate}] (2,0) -- (.75,0);
	\draw[very thick, postaction={decorate}] (2,3) -- (.75,3);
	\draw[very thick, postaction={decorate}] (-.75,0) -- (-2,0);
	\draw[very thick, postaction={decorate}] (-.75,3) -- (-2,3);
	\draw (.75,3) -- (.75,0);
	\draw (-.75,3) -- (-.75,0);
	\draw (2,3) -- (2,0);
	\draw (-2,3) -- (-2,0);
	\path [fill=red] (-2,3) rectangle (-.75,0);
	\path [fill=red] (2,3) rectangle (.75,0);
	\path [fill=red] (-.75,3) -- (.25,2) -- (.25,-1) -- (-.75,0) -- cycle;
	\path [fill=red] (.75,3) -- (-.25,4) -- (-.25,1) -- (.75,0) -- cycle;
	\node[opacity=1] at (.25, 3) {$\bullet$};
	\draw[very thick, postaction={decorate}] (.25,2) -- (-.75,3);
	\draw[very thick, postaction={decorate}] (.25,-1) -- (-.75,0);
	\draw (-1,-1) rectangle (3,2);
	\draw[very thick] (-1,-1) -- (3,-1);
	\draw[very thick] (-1,2) -- (3,2);
	\draw[very thick, red] (.25,-1) -- (.25,2);
\end{tikzpicture}
};
\endxy
\]
which is the image of
$\vcenter{\xy 0;/r.11pc/:
  (3,-9);(3,9) **\dir{-}?(1)*\dir{>}+(2.3,0)*{};
  (-3,-9);(-3,9) **\dir{-}?(1)*\dir{>}+(2.3,0)*{};
  (-3,0)*{\bullet};(-6.5,5)*{};
  (-5,-6)*{\scs i};     (5.1,-6)*{\scs \ \ \ i+1};
 \endxy} \;\; - \;\;
  \vcenter{\xy 0;/r.11pc/:
  (3,-9);(3,9) **\dir{-}?(1)*\dir{>}+(2.3,0)*{};
  (-3,-9);(-3,9) **\dir{-}?(1)*\dir{>}+(2.3,0)*{};
  (3,0)*{\bullet};(7,5)*{};
  (-5,-6)*{\scs i};     (5.1,-6)*{\scs \ \ \ i+1};
 \endxy}
 $.

When $a_{i+1}=2$, we must confirm that
\[
\xy
(0,0)*{
\begin{tikzpicture}  [fill opacity=0.2,  decoration={markings, mark=at position 0.6 with {\arrow{>}};}, scale=.6]
	\draw (-3,1) rectangle (1,4);
	\draw[very thick] (-3,1) -- (1,1);
	\draw[very thick] (-3,4) -- (1,4);
	\draw [very thick, red] (-.25,4) -- (-.25,1);
	\draw[very thick, postaction={decorate}] (.75,3) -- (-.25,4);
	\draw[very thick, postaction={decorate}] (.75,0) -- (-.25,1);
	\path[fill=red] (-.25,4) -- (.75,3) -- (.75,0) -- (-.25,1) -- cycle;
	\draw (2,0) -- (2,3);
	\draw (-2,0) -- (-2,3);
	\draw[double, postaction={decorate}] (2,3) -- (.75,3);
	\draw[double, postaction={decorate}] (-.75,3) -- (-2,3);
	\draw[very thick, postaction={decorate}] (.75,3) -- (-.75,3);
	\draw[double, postaction={decorate}] (2,0) -- (.75,0);
	\draw[double, postaction={decorate}] (-.75,0) -- (-2,0);
	\draw[very thick, postaction={decorate}] (.75,0) -- (-.75,0);
	\draw[very thick, red, postaction={decorate}] (.75,0) -- (.75,3);
	\draw[very thick, red, postaction={decorate}] (-.75,3) -- (-.75,0);
	\path[fill=red] (-.75,3) rectangle (.75,0);
	\path[fill=yellow , fill opacity=0.3] (-2,3) rectangle (-.75,0);
	\path[fill=yellow , fill opacity=0.3] (2,3) rectangle (.75,0);
	\draw[very thick, postaction={decorate}] (.25,2) -- (-.75,3);
	\draw[very thick, postaction={decorate}] (.25,-1) -- (-.75,0);
	\path[fill=red] (-.75,3) -- (.25,2) -- (.25,-1) -- (-.75,0) -- cycle;
	\node[opacity=1] at (-.15,-.2) {$\bullet$};
	\draw (-1,-1) rectangle (3,2);
	\draw[very thick] (-1,-1) -- (3,-1);
	\draw[very thick] (-1,2) -- (3,2);
	\draw[very thick, red] (.25,-1) -- (.25,2);
\end{tikzpicture}
};
\endxy \quad - \quad
\xy
(0,0)*{
\begin{tikzpicture}  [fill opacity=0.2,  decoration={markings, mark=at position 0.6 with {\arrow{>}};}, scale=.6]
	\draw (-3,1) rectangle (1,4);
	\draw[very thick] (-3,1) -- (1,1);
	\draw[very thick] (-3,4) -- (1,4);
	\draw [very thick, red] (-.25,4) -- (-.25,1);
	\draw[very thick, postaction={decorate}] (.75,3) -- (-.25,4);
	\draw[very thick, postaction={decorate}] (.75,0) -- (-.25,1);
	\path[fill=red] (-.25,4) -- (.75,3) -- (.75,0) -- (-.25,1) -- cycle;
	\node[opacity=1] at (.15,3.2) {$\bullet$};
	\draw (2,0) -- (2,3);
	\draw (-2,0) -- (-2,3);
	\draw[double, postaction={decorate}] (2,3) -- (.75,3);
	\draw[double, postaction={decorate}] (-.75,3) -- (-2,3);
	\draw[very thick, postaction={decorate}] (.75,3) -- (-.75,3);
	\draw[double, postaction={decorate}] (2,0) -- (.75,0);
	\draw[double, postaction={decorate}] (-.75,0) -- (-2,0);
	\draw[very thick, postaction={decorate}] (.75,0) -- (-.75,0);
	\draw[very thick, red, postaction={decorate}] (.75,0) -- (.75,3);
	\draw[very thick, red, postaction={decorate}] (-.75,3) -- (-.75,0);
	\path[fill=red] (-.75,3) rectangle (.75,0);
	\path[fill=yellow , fill opacity=0.3] (-2,3) rectangle (-.75,0);
	\path[fill=yellow , fill opacity=0.3] (2,3) rectangle (.75,0);
	\draw[very thick, postaction={decorate}] (.25,2) -- (-.75,3);
	\draw[very thick, postaction={decorate}] (.25,-1) -- (-.75,0);
	\path[fill=red] (-.75,3) -- (.25,2) -- (.25,-1) -- (-.75,0) -- cycle;
	\draw (-1,-1) rectangle (3,2);
	\draw[very thick] (-1,-1) -- (3,-1);
	\draw[very thick] (-1,2) -- (3,2);
	\draw[very thick, red] (.25,-1) -- (.25,2);
\end{tikzpicture}
};
\endxy\]
is the zero foam. This follows from the dot-sliding relation \eqref{sl2DotSliding}.

When $j=i-1$, it similarly suffices to confirm the relation when $a_i = 0,1$. For $a_i=0$, the images of
$\vcenter{\xy 0;/r.11pc/:
  (3,-9);(3,9) **\dir{-}?(1)*\dir{>}+(2.3,0)*{};
  (-3,-9);(-3,9) **\dir{-}?(1)*\dir{>}+(2.3,0)*{};
  (-3,0)*{\bullet};(-6.5,5)*{};
  (-5,-6)*{\scs i};     (5.1,-6)*{\scs \ \ \ i-1};
 \endxy}
 $ and
 $  \vcenter{\xy 0;/r.11pc/:
  (3,-9);(3,9) **\dir{-}?(1)*\dir{>}+(2.3,0)*{};
  (-3,-9);(-3,9) **\dir{-}?(1)*\dir{>}+(2.3,0)*{};
  (3,0)*{\bullet};(7,5)*{};
  (-5,-6)*{\scs i};     (5.1,-6)*{\scs \ \ \ i-1};
 \endxy}
 $ are isotopic, so both sides of the relation map to zero.

For $a_i = 1$, we compute
\[
\vcenter{\xy 0;/r.17pc/:
    (-4,-4)*{};(4,4)*{} **\crv{(-4,-1) & (4,1)}?(1)*\dir{};
    (4,-4)*{};(-4,4)*{} **\crv{(4,-1) & (-4,1)}?(1)*\dir{};
    (-4,4)*{};(4,12)*{} **\crv{(-4,7) & (4,9)}?(1)*\dir{};
    (4,4)*{};(-4,12)*{} **\crv{(4,7) & (-4,9)}?(1)*\dir{};
    (8,8)*{\lambda};
    (4,12); (4,13) **\dir{-}?(1)*\dir{>};
    (-4,12); (-4,13) **\dir{-}?(1)*\dir{>};
  (-5.5,-3)*{\scs i};
     (5.5,-3)*{\scs \ \ \ i-1};
 \endxy} \quad
\mapsto \quad
\xy
(0,-8.3)*{
}
\endxy
\]
by equations \eqref{sl2NH_enh} (turned sideways!) and \eqref{sl2DotSliding}. \\

\item Relation \eqref{eq_dot_slide_ij-gen} follows by sliding a dot along a facet, i.e. via isotopy. \\

\item Relation \eqref{eq_r3_easy-gen}:
For all choices of $i$, $j$, and $k$ this relation holds via isotopy (or since both sides map to zero).
This is obvious in the case that two of the three
values $(\alpha_i,\alpha_j)$, $(\alpha_i,\alpha_k)$ and $(\alpha_j,\alpha_k)$ are zero. In the other cases a
computation is necessary;
note that we can assume $i\neq j \neq k$ since otherwise both sides of the equation automatically
map to zero (an intermediate weight must map to the zero object).

Suppose that $j=i+1$ and $k=i+2$,
then we compute both sides of the relation to be
\begin{equation}\label{ex220}
\xy
(0,-20)*{
\begin{tikzpicture} [fill opacity=0.2,  decoration={markings,
                        mark=at position 0.6 with {\arrow{>}};    }, scale=.5]
	\draw[very thick] (-4,1) -- (1.5,1);
	\draw[very thick] (-3,0) -- (2.5,0);
	\draw[very thick] (-2,-1) -- (3.5,-1);
	\draw[very thick] (-1,-2) -- (4.5,-2);
	\draw[very thick, postaction={decorate}] (.25,-2) -- (-.75,-1);
	\draw[very thick, postaction={decorate}] (.75,-1) -- (-.25,0);
	\draw[very thick, postaction={decorate}] (1.25,0) -- (.25,1);
	\draw (-4,5) -- (-4,1);
	\draw (1.5,5) -- (1.5,1);
	\path[fill=red] (-4,5) rectangle (1.5,1);
	\draw[very thick, red] (.25,5) -- (.25,1);
	\path[fill=blue] (.25,5) -- (1.25,4) -- (1.25,0) -- (.25,1);
	\draw (-3,4) -- (-3,0);
	\draw (2.5,4) -- (2.5,0);
	\path[fill=red] (-3,4) rectangle (2.5,0);
	\draw[very thick, red] (1.25,4) -- (1.25,0);
	\draw[very thick, red] (-1.75, 4) .. controls (-1.75,3) and (-.25,1) .. (-.25,0);
	\path[fill=blue] (-.75,3) .. controls (-.75,2) and (.75,0) .. (.75,-1) -- (-.25,0) ..
		controls (-.25,1) and (-1.75,3) .. (-1.75, 4);
	\draw (-2,3) -- (-2,-1);
	\draw (3.5,3) -- (3.5,-1);
	\path[fill=red] (-2,3) rectangle (3.5,-1);
	\draw[very thick, red] (.75,3) .. controls (.75,2) and (-.75,0) .. (-.75,-1);
	\draw[very thick, red] (-.75, 3) .. controls (-.75,2) and (.75,0) .. (.75,-1);
	\path[fill=blue] (1.75,2) .. controls (1.75,1) and (.25,-1) .. (.25,-2) -- (-.75,-1) ..
		controls (-.75,0) and (.75,2) .. (.75,3);
	\draw (-1,2) -- (-1,-2);
	\draw (4.5,2) -- (4.5,-2);
	\path[fill=red] (-1,2) rectangle (4.5,-2);
	\draw[very thick, red] (1.75, 2) .. controls (1.75,1) and (.25,-1) .. (.25,-2);
	\draw[very thick] (-4,5) -- (1.5,5);
	\draw[very thick] (-3,4) -- (2.5,4);
	\draw[very thick] (-2,3) -- (3.5,3);
	\draw[very thick] (-1,2) -- (4.5,2);
	\draw[very thick, postaction={decorate}] (1.75,2) -- (.75,3);
	\draw[very thick, postaction={decorate}] (-.75,3) -- (-1.75,4);
	\draw[very thick, postaction={decorate}] (1.25,4) -- (.25,5);
\end{tikzpicture}};
(0,0)*{
\begin{tikzpicture} [fill opacity=0.2,  decoration={markings,
                        mark=at position 0.6 with {\arrow{>}};    }, scale=.5]
	\draw[very thick] (-4,1) -- (1.5,1);
	\draw[very thick] (-3,0) -- (2.5,0);
	\draw[very thick] (-2,-1) -- (3.5,-1);
	\draw[very thick] (-1,-2) -- (4.5,-2);
	\draw[very thick, postaction={decorate}] (1.75,-2) -- (.75,-1);
	\draw[very thick, postaction={decorate}] (-.75,-1) -- (-1.75,0);
	\draw[very thick, postaction={decorate}] (1.25,0) -- (.25,1);
	\draw (-4,5) -- (-4,1);
	\draw (1.5,5) -- (1.5,1);
	\path[fill=red] (-4,5) rectangle (1.5,1);
	\draw[very thick, red] (-1.25,5) .. controls (-1.25,4) and (.25,2) .. (.25,1);
	\path[fill=blue] (-.25,4) .. controls (-.25,3) and (1.25,1) .. (1.25,0) -- (.25,1) ..
		controls (.25,2) and (-1.25,4) .. (-1.25,5);
	\draw (-3,4) -- (-3,0);
	\draw (2.5,4) -- (2.5,0);
	\path[fill=red] (-3,4) rectangle (2.5,0);
	\draw[very thick, red] (-.25,4) .. controls (-.25,3) and (1.25,1) .. (1.25,0);
	\draw[very thick, red] (-1.75, 4) -- (-1.75,0);
	\path[fill=blue] (-.75,3) -- (-.75,-1) -- (-1.75,0) -- (-1.75,4);
	\draw (-2,3) -- (-2,-1);
	\draw (3.5,3) -- (3.5,-1);
	\path[fill=red] (-2,3) rectangle (3.5,-1);
	\draw[very thick, red] (.75,-1) .. controls (.75,0) and (2.25,2) .. (2.25,3);
	\draw[very thick, red] (-.75, 3) -- (-.75,-1);
	\path[fill=blue] (3.25,2) .. controls (3.25,1) and (1.75,-1) .. (1.75,-2) -- (.75,-1) ..
		controls (.75,0) and (2.25,2) .. (2.25,3);
	\draw (-1,2) -- (-1,-2);
	\draw (4.5,2) -- (4.5,-2);
	\path[fill=red] (-1,2) rectangle (4.5,-2);
	\draw[very thick, red] (3.25, 2) .. controls (3.25,1) and (1.75,-1) .. (1.75,-2);
	\draw[very thick] (-4,5) -- (1.5,5);
	\draw[very thick] (-3,4) -- (2.5,4);
	\draw[very thick] (-2,3) -- (3.5,3);
	\draw[very thick] (-1,2) -- (4.5,2);
	\draw[very thick, postaction={decorate}] (3.25,2) -- (2.25,3);
	\draw[very thick, postaction={decorate}] (-.75,3) -- (-1.75,4);
	\draw[very thick, postaction={decorate}] (-.25,4) -- (-1.25,5);
\end{tikzpicture}};
(0,20)*{
\begin{tikzpicture} [fill opacity=0.2,  decoration={markings,
                        mark=at position 0.6 with {\arrow{>}};    }, scale=.5]
	\draw[very thick] (-4,1) -- (1.5,1);
	\draw[very thick] (-3,0) -- (2.5,0);
	\draw[very thick] (-2,-1) -- (3.5,-1);
	\draw[very thick] (-1,-2) -- (4.5,-2);
	\draw[very thick, postaction={decorate}] (3.25,-2) -- (2.25,-1);
	\draw[very thick, postaction={decorate}] (-.75,-1) -- (-1.75,0);
	\draw[very thick, postaction={decorate}] (-.25,0) -- (-1.25,1);
	\draw (-4,5) -- (-4,1);
	\draw (1.5,5) -- (1.5,1);
	\path[fill=red] (-4,5) rectangle (1.5,1);
	\draw[very thick, red] (-2.75,5) .. controls (-2.75,4) and (-1.25,2) .. (-1.25,1);
	\path[fill=blue] (-1.75,4) .. controls (-1.75,3) and (-.25,1) .. (-.25,0) -- (-1.25,1) ..
		controls (-1.25,2) and (-2.75,4) .. (-2.75,5);
	\draw (-3,4) -- (-3,0);
	\draw (2.5,4) -- (2.5,0);
	\path[fill=red] (-3,4) rectangle (2.5,0);
	\draw[very thick, red] (-1.75,4) .. controls (-1.75,3) and (-.25,1) .. (-.25,0);
	\draw[very thick, red] (-.25, 4) .. controls (-.25,3) and (-1.75,1) .. (-1.75,0);
	\path[fill=blue] (-.25, 4) .. controls (-.25,3) and (-1.75,1) .. (-1.75,0) -- (-.75,-1) ..
		controls (-.75,0) and (.75,2) .. (.75,3);
	\draw (-2,3) -- (-2,-1);
	\draw (3.5,3) -- (3.5,-1);
	\path[fill=red] (-2,3) rectangle (3.5,-1);
	\draw[very thick, red] (2.25,3) -- (2.25,-1);
	\draw[very thick, red] (-.75,-1) .. controls (-.75,0) and (.75,2) .. (.75,3);
	\path[fill=blue] (3.25,2) -- (3.25,-2) -- (2.25,-1) -- (2.25,3);
	\draw (-1,2) -- (-1,-2);
	\draw (4.5,2) -- (4.5,-2);
	\path[fill=red] (-1,2) rectangle (4.5,-2);
	\draw[very thick, red] (3.25, 2) -- (3.25,-2);
	\draw[very thick] (-4,5) -- (1.5,5);
	\draw[very thick] (-3,4) -- (2.5,4);
	\draw[very thick] (-2,3) -- (3.5,3);
	\draw[very thick] (-1,2) -- (4.5,2);
	\draw[very thick, postaction={decorate}] (3.25,2) -- (2.25,3);
	\draw[very thick, postaction={decorate}] (.75,3) -- (-.25,4);
	\draw[very thick, postaction={decorate}] (-1.75,4) -- (-2.75,5);
\end{tikzpicture}};
\endxy \qquad \text{and} \qquad
\xy
(0,-20)*{
\begin{tikzpicture} [fill opacity=0.2,  decoration={markings,
                        mark=at position 0.6 with {\arrow{>}};    }, scale=.5]
	\draw[very thick] (-4,1) -- (1.5,1);
	\draw[very thick] (-3,0) -- (2.5,0);
	\draw[very thick] (-2,-1) -- (3.5,-1);
	\draw[very thick] (-1,-2) -- (4.5,-2);
	\draw[very thick, postaction={decorate}] (.25,-2) -- (-.75,-1);
	\draw[very thick, postaction={decorate}] (.75,-1) -- (-.25,0);
	\draw[very thick, postaction={decorate}] (1.25,0) -- (.25,1);
	\draw (-4,5) -- (-4,1);
	\draw (1.5,5) -- (1.5,1);
	\path[fill=red] (-4,5) rectangle (1.5,1);
	\draw[very thick, red] (-1.25,5) .. controls (-1.25,4) and (.25,2) .. (.25,1);
	\path[fill=blue] (-1.25,5) .. controls (-1.25,4) and (.25,2) .. (.25,1) -- (1.25,0) ..
		controls (1.25,1) and (-.25,3) .. (-.25,4);
	\draw (-3,4) -- (-3,0);
	\draw (2.5,4) -- (2.5,0);
	\path[fill=red] (-3,4) rectangle (2.5,0);
	\draw[very thick, red] (-.25,4) .. controls (-.25,3) and (1.25,1) .. (1.25,0);
	\draw[very thick, red] (1.25, 4) .. controls (1.25,3) and (-.25,1) .. (-.25,0);
	\path[fill=blue] (2.25,3) .. controls (2.25,2) and (.75,0) .. (.75,-1) -- (-.25,0) ..
		controls (-.25,1) and (1.25,3) .. (1.25, 4);
	\draw (-2,3) -- (-2,-1);
	\draw (3.5,3) -- (3.5,-1);
	\path[fill=red] (-2,3) rectangle (3.5,-1);
	\draw[very thick, red] (2.25,3) .. controls (2.25,2) and (.75,0) .. (.75,-1);
	\draw[very thick, red] (-.75, 3) -- (-.75,-1);
	\path[fill=blue] (.25,2) -- (.25,-2) -- (-.75,-1) -- (-.75,3);
	\draw (-1,2) -- (-1,-2);
	\draw (4.5,2) -- (4.5,-2);
	\path[fill=red] (-1,2) rectangle (4.5,-2);
	\draw[very thick, red] (.25,2) -- (.25,-2);
	\draw[very thick] (-4,5) -- (1.5,5);
	\draw[very thick] (-3,4) -- (2.5,4);
	\draw[very thick] (-2,3) -- (3.5,3);
	\draw[very thick] (-1,2) -- (4.5,2);
	\draw[very thick, postaction={decorate}] (.25,2) -- (-.75,3);
	\draw[very thick, postaction={decorate}] (2.25,3) -- (1.25,4);
	\draw[very thick, postaction={decorate}] (-.25,4) -- (-1.25,5);
\end{tikzpicture}};
(0,0)*{
\begin{tikzpicture} [fill opacity=0.2,  decoration={markings,
                        mark=at position 0.6 with {\arrow{>}};    }, scale=.5]
	\draw[very thick] (-4,1) -- (1.5,1);
	\draw[very thick] (-3,0) -- (2.5,0);
	\draw[very thick] (-2,-1) -- (3.5,-1);
	\draw[very thick] (-1,-2) -- (4.5,-2);
	\draw[very thick, postaction={decorate}] (.25,-2) -- (-.75,-1);
	\draw[very thick, postaction={decorate}] (2.25,-1) -- (1.25,0);
	\draw[very thick, postaction={decorate}] (-.25,0) -- (-1.25,1);
	\draw (-4,5) -- (-4,1);
	\draw (1.5,5) -- (1.5,1);
	\path[fill=red] (-4,5) rectangle (1.5,1);
	\draw[very thick, red] (-2.75,5) .. controls (-2.75,4) and (-1.25,2) .. (-1.25,1);
	\path[fill=blue] (-2.75,5) .. controls (-2.75,4) and (-1.25,2) .. (-1.25,1) -- (-.25,0) ..
		controls (-.25,1) and (-1.75,3) .. (-1.75,4);
	\draw (-3,4) -- (-3,0);
	\draw (2.5,4) -- (2.5,0);
	\path[fill=red] (-3,4) rectangle (2.5,0);
	\draw[very thick, red] (-1.75,4) .. controls (-1.75,3) and (-.25,1) .. (-.25,0);
	\draw[very thick, red] (1.25, 4) -- (1.25,0);
	\path[fill=blue] (2.25,3) -- (2.25,-1) -- (1.25,0) -- (1.25, 4);
	\draw (-2,3) -- (-2,-1);
	\draw (3.5,3) -- (3.5,-1);
	\path[fill=red] (-2,3) rectangle (3.5,-1);
	\draw[very thick, red] (.75,3) .. controls (.75,2) and (-.75,0) .. (-.75,-1);
	\draw[very thick, red] (2.25, 3) -- (2.25,-1);
	\path[fill=blue] (.75,3) .. controls (.75,2) and (-.75,0) .. (-.75,-1) -- (.25,-2) ..
		controls (.25,-1) and (1.75,1) .. (1.75,2);
	\draw (-1,2) -- (-1,-2);
	\draw (4.5,2) -- (4.5,-2);
	\path[fill=red] (-1,2) rectangle (4.5,-2);
	\draw[very thick, red] (.25,-2) .. controls (.25,-1) and (1.75,1) .. (1.75,2);
	\draw[very thick] (-4,5) -- (1.5,5);
	\draw[very thick] (-3,4) -- (2.5,4);
	\draw[very thick] (-2,3) -- (3.5,3);
	\draw[very thick] (-1,2) -- (4.5,2);
	\draw[very thick, postaction={decorate}] (1.75,2) -- (.75,3);
	\draw[very thick, postaction={decorate}] (2.25,3) -- (1.25,4);
	\draw[very thick, postaction={decorate}] (-1.75,4) -- (-2.75,5);
\end{tikzpicture}};
(0,20)*{
\begin{tikzpicture} [fill opacity=0.2,  decoration={markings,
                        mark=at position 0.6 with {\arrow{>}};    }, scale=.5]
	\draw[very thick] (-4,1) -- (1.5,1);
	\draw[very thick] (-3,0) -- (2.5,0);
	\draw[very thick] (-2,-1) -- (3.5,-1);
	\draw[very thick] (-1,-2) -- (4.5,-2);
	\draw[very thick, postaction={decorate}] (1.75,-2) -- (.75,-1);
	\draw[very thick, postaction={decorate}] (2.25,-1) -- (1.25,0);
	\draw[very thick, postaction={decorate}] (-1.75,0) -- (-2.75,1);
	\draw (-4,5) -- (-4,1);
	\draw (1.5,5) -- (1.5,1);
	\path[fill=red] (-4,5) rectangle (1.5,1);
	\draw[very thick, red] (-2.75,5) -- (-2.75,1);
	\path[fill=blue] (-2.75,5) -- (-2.75,1) -- (-1.75,0) -- (-1.75,4);
	\draw (-3,4) -- (-3,0);
	\draw (2.5,4) -- (2.5,0);
	\path[fill=red] (-3,4) rectangle (2.5,0);
	\draw[very thick, red] (-1.75,4) -- (-1.75,0);
	\draw[very thick, red] (-.25,4) .. controls (-.25,3) and (1.25,1) .. (1.25,0);
	\path[fill=blue] (-.25,4) .. controls (-.25,3) and (1.25,1) .. (1.25,0) -- (2.25,-1) ..
		controls (2.25,0) and (.75,2) .. (.75,3);
	\draw (-2,3) -- (-2,-1);
	\draw (3.5,3) -- (3.5,-1);
	\path[fill=red] (-2,3) rectangle (3.5,-1);
	\draw[very thick, red] (2.25,3) .. controls (2.25,2) and (.75,0) .. (.75,-1);
	\draw[very thick, red] (2.25,-1) .. controls (2.25,0) and (.75,2) .. (.75,3);
	\path[fill=blue] (2.25,3) .. controls (2.25,2) and (.75,0) .. (.75,-1) -- (1.75,-2) ..
		controls (1.75,-1) and (3.25,1) .. (3.25,2);
	\draw (-1,2) -- (-1,-2);
	\draw (4.5,2) -- (4.5,-2);
	\path[fill=red] (-1,2) rectangle (4.5,-2);
	\draw[very thick, red] (1.75,-2) .. controls (1.75,-1) and (3.25,1) .. (3.25,2);
	\draw[very thick] (-4,5) -- (1.5,5);
	\draw[very thick] (-3,4) -- (2.5,4);
	\draw[very thick] (-2,3) -- (3.5,3);
	\draw[very thick] (-1,2) -- (4.5,2);
	\draw[very thick, postaction={decorate}] (3.25,2) -- (2.25,3);
	\draw[very thick, postaction={decorate}] (.75,3) -- (-.25,4);
	\draw[very thick, postaction={decorate}] (-1.75,4) -- (-2.75,5);
\end{tikzpicture}};
\endxy
\end{equation}
which are equal up to isotopy for any value of the $a_l$'s. The other cases follow similarly. \\

\item Relation \eqref{eq_r3_hard-gen}: We must have $j=i\pm1$ and we'll only compute for $j=i+1$ since
the other case is analogous. Note that all $2$-morphisms involved automatically map to zero if
$\lambda$ is sent to a sequence with $a_{i+1}=2$ or with $a_i=0,1$, so we'll compute for the remaining values.

When $a_i=2$ and $a_{i+1}=0$ we have
\[
 \vcenter{
 \xy 0;/r.15pc/:
    (-4,-4)*{};(4,4)*{} **\crv{(-4,-1) & (4,1)}?(1)*\dir{};
    (4,-4)*{};(-4,4)*{} **\crv{(4,-1) & (-4,1)}?(1)*\dir{};
    (4,4)*{};(12,12)*{} **\crv{(4,7) & (12,9)}?(1)*\dir{};
    (12,4)*{};(4,12)*{} **\crv{(12,7) & (4,9)}?(1)*\dir{};
    (-4,12)*{};(4,20)*{} **\crv{(-4,15) & (4,17)}?(1)*\dir{};
    (4,12)*{};(-4,20)*{} **\crv{(4,15) & (-4,17)}?(1)*\dir{};
    (-4,4)*{}; (-4,12) **\dir{-};
    (12,-4)*{}; (12,4) **\dir{-};
    (12,12)*{}; (12,20) **\dir{-};
  (-4,20); (-4,21) **\dir{-}?(1)*\dir{>};
  (4,20); (4,21) **\dir{-}?(1)*\dir{>};
  (12,20); (12,21) **\dir{-}?(1)*\dir{>};
  (-4,-6)*{\scs i};
  (4,-6)*{\scs i+1};
  (12,-6)*{\scs i};
\endxy} \mapsto \quad
\xy
(0,-19.5)*{
};
\endxy
\qquad
\qquad
 \vcenter{
 \xy 0;/r.15pc/:
    (4,-4)*{};(-4,4)*{} **\crv{(4,-1) & (-4,1)}?(1)*\dir{};
    (-4,-4)*{};(4,4)*{} **\crv{(-4,-1) & (4,1)}?(1)*\dir{};
    (-4,4)*{};(-12,12)*{} **\crv{(-4,7) & (-12,9)}?(1)*\dir{};
    (-12,4)*{};(-4,12)*{} **\crv{(-12,7) & (-4,9)}?(1)*\dir{};
    (4,12)*{};(-4,20)*{} **\crv{(4,15) & (-4,17)}?(1)*\dir{};
    (-4,12)*{};(4,20)*{} **\crv{(-4,15) & (4,17)}?(1)*\dir{};
    (4,4)*{}; (4,12) **\dir{-};
    (-12,-4)*{}; (-12,4) **\dir{-};
    (-12,12)*{}; (-12,20) **\dir{-};
  (-4,20); (-4,21) **\dir{-}?(1)*\dir{>};
  (4,20); (4,21) **\dir{-}?(1)*\dir{>};
  (-12,20); (-12,21) **\dir{-}?(1)*\dir{>};
  (-12,-6)*{\scs i};
  (-4,-6)*{\scs i+1};
  (4,-6)*{\scs i};
\endxy} \mapsto 0
\]
which gives the relation since the former is isotopic to the identity.

Finally, when $a_i=2$ and $a_{i+1} = 1$ we compute that
\[
 \vcenter{
 \xy 0;/r.15pc/:
    (-4,-4)*{};(4,4)*{} **\crv{(-4,-1) & (4,1)}?(1)*\dir{};
    (4,-4)*{};(-4,4)*{} **\crv{(4,-1) & (-4,1)}?(1)*\dir{};
    (4,4)*{};(12,12)*{} **\crv{(4,7) & (12,9)}?(1)*\dir{};
    (12,4)*{};(4,12)*{} **\crv{(12,7) & (4,9)}?(1)*\dir{};
    (-4,12)*{};(4,20)*{} **\crv{(-4,15) & (4,17)}?(1)*\dir{};
    (4,12)*{};(-4,20)*{} **\crv{(4,15) & (-4,17)}?(1)*\dir{};
    (-4,4)*{}; (-4,12) **\dir{-};
    (12,-4)*{}; (12,4) **\dir{-};
    (12,12)*{}; (12,20) **\dir{-};
  (-4,20); (-4,21) **\dir{-}?(1)*\dir{>};
  (4,20); (4,21) **\dir{-}?(1)*\dir{>};
  (12,20); (12,21) **\dir{-}?(1)*\dir{>};
  (-4,-6)*{\scs i};
  (4,-6)*{\scs i+1};
  (12,-6)*{\scs i};
\endxy} \mapsto 0
, \qquad \qquad
 \vcenter{
 \xy 0;/r.15pc/:
    (4,-4)*{};(-4,4)*{} **\crv{(4,-1) & (-4,1)}?(1)*\dir{};
    (-4,-4)*{};(4,4)*{} **\crv{(-4,-1) & (4,1)}?(1)*\dir{};
    (-4,4)*{};(-12,12)*{} **\crv{(-4,7) & (-12,9)}?(1)*\dir{};
    (-12,4)*{};(-4,12)*{} **\crv{(-12,7) & (-4,9)}?(1)*\dir{};
    (4,12)*{};(-4,20)*{} **\crv{(4,15) & (-4,17)}?(1)*\dir{};
    (-4,12)*{};(4,20)*{} **\crv{(-4,15) & (4,17)}?(1)*\dir{};
    (4,4)*{}; (4,12) **\dir{-};
    (-12,-4)*{}; (-12,4) **\dir{-};
    (-12,12)*{}; (-12,20) **\dir{-};
  (-4,20); (-4,21) **\dir{-}?(1)*\dir{>};
  (4,20); (4,21) **\dir{-}?(1)*\dir{>};
  (-12,20); (-12,21) **\dir{-}?(1)*\dir{>};
  (-12,-6)*{\scs i};
  (-4,-6)*{\scs i+1};
  (4,-6)*{\scs i};
\endxy} \mapsto \quad
\xy
(0,-19.5)*{
};
\endxy \quad ;
\]
applying \eqref{sl2sq2} to the above gives the foam
\[
- \quad
\xy
(0,0)*{
%
};
\endxy
\]
which confirms the relation.

\end{itemize}

\end{proof}

Note that the scalings of the
images of the cap and cup $2$-morphisms play no role in the proof of the proposition. They are
determined by the proof of Theorem \ref{thm_CL}.

%
\subsubsection{Clark-Morrison-Walker foams} \label{CMWsec}
%

In the original construction of functorial Khovanov homology \cite{CMW},
Clark-Morrison-Walker use a variation of
Bar-Natan's $2$-category involving disoriented surfaces defined over the Gaussian integers. We can define
foamation $2$-functors to a family of $2$-categories related to their construction. We will assume some familiarity
with their work.

We fix once and for all $\omega$ to be a primitive fourth root of the unity.

\begin{defn}
$\cat{CMWFoam}_m(N)$ is the $2$-category defined as follows:
\begin{itemize}
\item Objects are sequences $(a_1,\ldots,a_m)$ labeling points in the interval $[0,1]$
with $a_i \in \{0,1,2\}$ and $N = \sum_{i=1}^m a_i$, together with a zero object.
\item $1$-morphisms are formal direct sums of $\Z$-graded disoriented planar tangles directed out from
$1$-labeled points in the domain and into such points in the codomain.
\item $2$-morphisms are formal matrices of $\Bbbk[\omega]$-linear combinations of degree-zero dotted disoriented
cobordisms between such disoriented planar tangles, modulo isotopy and local relations.
\end{itemize}
\end{defn}
The disorientations are represented by fringed seams;
the local relations are given by \eqref{sl2closedfoam} and \eqref{sl2neckcutting} in regions where no seams are
present and the following local seam relations:
\begin{align} \label{CMWseam}
 \xy
(0,0)*{
};
\endxy \quad . \notag
\end{align}

By adjusting some coefficients in the formulation of Proposition \ref{prop_sl2} and appropriately interpreting the
image foams as $2$-morphisms in $\cat{CMWFoam}_m(N)$, we obtain the desired $2$-functor. The
interpretation is as follows:
\[
\xy
(0,0)*{
%
};
\endxy
\quad
\]
should be read as the disoriented tangles
\[
\xy
(0,0)*{
%
};
\endxy
\quad
\]
and the $2$-labeled sheets should be deleted from the image foams, retaining the seams and adding
fringes aligned with the disorientation ``tags'' on the tangles.


\begin{prop}
For each $N>0$ there is a $2$-representation
$\Phi_{CMW} \maps \Ucat_Q(\mf{sl}_m)  \to  \cat{CMWFoam}_m(N)$
defined on single strand $2$-morphisms by:

\[
\Phi_{CMW} \left(
 \xy 0;/r.17pc/:
 (0,7);(0,-7); **\dir{-} ?(.75)*\dir{>};
 (7,3)*{ \scs \lambda};
 (-2.5,-6)*{\scs i};
 (-10,0)*{};(10,0)*{};
 \endxy
\right) = \quad \Efoam[.5] \quad , \quad
\Phi_{CMW} \left(
 \xy 0;/r.17pc/:
 (0,7);(0,-7); **\dir{-} ?(.75)*\dir{>};
 (0,0)*{\bullet};
 (7,3)*{ \scs \lambda};
 (-2.5,-6)*{\scs i};
 (-10,0)*{};(10,0)*{};
 \endxy
\right) = \quad \dotEfoam[.5]
\]
on crossings by:
\[
\qquad
\Phi_{CMW} \left(
\xy 0;/r.20pc/:
	(0,0)*{\xybox{
    	(-4,-4)*{};(4,4)*{} **\crv{(-4,-1) & (4,1)}?(1)*\dir{>} ;
    	(4,-4)*{};(-4,4)*{} **\crv{(4,-1) & (-4,1)}?(1)*\dir{>};
    	(-5.5,-3)*{\scs i};
     	(5.5,-3)*{\scs i};
     	(9,1)*{\scs  \lambda};
     	(-10,0)*{};(10,0)*{};}};
\endxy
\right) = \quad (-\omega)\quad \crossingEEfoam[.5]
\]
\[
\Phi_{CMW} \left(
\xy 0;/r.20pc/:
	(0,0)*{\xybox{
    	(-4,-4)*{};(4,4)*{} **\crv{(-4,-1) & (4,1)}?(1)*\dir{>} ;
    	(4,-4)*{};(-4,4)*{} **\crv{(4,-1) & (-4,1)}?(1)*\dir{>};
    	(-5.5,-3)*{\scs i};
     	(5.5,-3)*{\scs \ \ \ i+1};
     	(9,1)*{\scs  \lambda};
     	(-10,0)*{};(10,0)*{};}};
\endxy
\right) =  \quad \crossingEoneEtwofoam[.5] \quad , \quad
\Phi_{CMW} \left(
\xy 0;/r.20pc/:
	(0,0)*{\xybox{
    	(-4,-4)*{};(4,4)*{} **\crv{(-4,-1) & (4,1)}?(1)*\dir{>} ;
    	(4,-4)*{};(-4,4)*{} **\crv{(4,-1) & (-4,1)}?(1)*\dir{>};
    	(-5.5,-3)*{\scs i+1 \ \ \ };
     	(5.5,-3)*{\scs i};
     	(9,1)*{\scs  \lambda};
     	(-10,0)*{};(10,0)*{};}};
\endxy
\right) = \quad \omega \quad \crossingEtwoEonefoam[.5]
\]

\[
\Phi_{CMW} \left(
\xy 0;/r.20pc/:
	(0,0)*{\xybox{
    	(-4,-4)*{};(4,4)*{} **\crv{(-4,-1) & (4,1)}?(1)*\dir{>} ;
    	(4,-4)*{};(-4,4)*{} **\crv{(4,-1) & (-4,1)}?(1)*\dir{>};
    	(-5.5,-3)*{\scs j};
     	(5.5,-3)*{\scs i};
     	(9,1)*{\scs  \lambda};
     	(-10,0)*{};(10,0)*{};}};
\endxy
\right) = \quad \crossingEthreeEonefoam[.5] \quad , \quad
\Phi_{CMW} \left(
\xy 0;/r.20pc/:
	(0,0)*{\xybox{
    	(-4,-4)*{};(4,4)*{} **\crv{(-4,-1) & (4,1)}?(1)*\dir{>} ;
    	(4,-4)*{};(-4,4)*{} **\crv{(4,-1) & (-4,1)}?(1)*\dir{>};
    	(-5.5,-3)*{\scs i};
     	(5.5,-3)*{\scs j};
     	(9,1)*{\scs  \lambda};
     	(-10,0)*{};(10,0)*{};}};
\endxy
\right) = \quad \crossingEoneEthreefoam[.5]
\]
where $j-i > 1$, and on caps and cups by:
\[
\Phi_{CMW} \left(
\xy 0;/r.20pc/:
	(0,0)*{\bbcef{i}};
	(8,4)*{\scs  \lambda};
	(-10,0)*{};(10,0)*{};
\endxy
\right) =  \quad \capFEfoam[.5]
\quad , \quad
\Phi_{CMW} \left(
\xy 0;/r.20pc/:
	(0,0)*{\bbcfe{i}};
	(8,4)*{\scs  \lambda};
	(-10,0)*{};(10,0)*{};
\endxy
\right) = (-1)^{a_i}(-\omega)^{\delta} \quad \capEFfoam[.5]
\]

\[
\Phi_{CMW} \left(
\xy 0;/r.20pc/:
	(0,0)*{\bbpef{i}};
	(8,-4)*{\scs \lambda};
	(-10,0)*{};(10,0)*{};
\endxy
\right) = (-1)^{a_i + 1}(-\omega)^{\delta} \quad \cupFEfoam[.5]
\quad , \quad
\Phi_{CMW} \left(
\xy 0;/r.20pc/:
	(0,0)*{\bbpfe{i}};
	(8,-4)*{\scs \lambda};
	(-10,0)*{};(10,0)*{};
\endxy
\right) = \quad \cupEFfoam[.5]
\]
where in the above diagrams the $i^{th}$ sheet is always in the front, and $\delta =1$ if $\lambda_i$ is
even and $\delta=0$ otherwise.
\end{prop}

The proof is the same as for Proposition \ref{prop_sl2}: we apply Theorem \ref{thm_CL}. Most of the work involves
checking the KLR relations and is straightforward, so we omit almost all of the details.
The following calculation confirms the NilHecke relation \eqref{eq_nil_dotslide},
which we include to show the importance of the disorientation seams:
\begin{align*}
\Phi_{CMW} \left(
\xy 0;/r.18pc/:
  (0,0)*{\xybox{
    (-4,-4)*{};(4,6)*{} **\crv{(-4,-1) & (4,1)}?(1)*\dir{>}?(.25)*{\bullet};
    (4,-4)*{};(-4,6)*{} **\crv{(4,-1) & (-4,1)}?(1)*\dir{>};
     (-8,0)*{};(8,0)*{};
     }};
  \endxy
 -
\xy 0;/r.18pc/:
  (0,0)*{\xybox{
    (-4,-4)*{};(4,6)*{} **\crv{(-4,-1) & (4,1)}?(1)*\dir{>}?(.75)*{\bullet};
    (4,-4)*{};(-4,6)*{} **\crv{(4,-1) & (-4,1)}?(1)*\dir{>};
     (-8,0)*{};(8,0)*{};
     }};
  \endxy
\right)
&= \;\; -\omega \;\;\left(
\xy
(0,7)*{
};
\endxy
\;\;=\;\; \Phi_{CMW} \left( \; \;
 \xy 0;/r.18pc/:
  (4,6);(4,-4) **\dir{-}?(0)*\dir{<}+(2.3,0)*{};
  (-4,6);(-4,-4) **\dir{-}?(0)*\dir{<}+(2.3,0)*{};
 \endxy
\right) \quad .
\end{align*}

In the above pictures we have applied isotopies to the disoriented tangles and cobordisms for clarity.
This relation determines the scaling of the $(i,i)$ crossing in the definition above. The KLR relations also
fix the scaling on the composition of an $(i,i+1)$ and an $(i+1,i)$ crossing; we choose to rescale the $(i+1,i)$
crossing. The scalings of the other $2$-morphisms again follow from the proof of Theorem \ref{thm_CL}.

At present time, we don't have a good explanation for the rescalings in the above $2$-functor. For this reason,
we believe that $\Bfoam{2}{m}$ is a more natural setting for the foamation $2$-functors; in particular, we'll see in
Proposition \ref{prop_sl3_enhanced} that the definition of the foamation functor from
Proposition \ref{prop_sl2} generalizes mutatis mutandis to give a foamation $2$-functor to an enhanced version
of $\mf{sl}_3$ foams.

%
\subsection{$\mathfrak{sl}_3$ foam 2-categories}
%


In this section, we recall the definition of the $\mf{sl}_3$ foam $2$-category and prove the existence of the $\mf{sl}_3$ foamation 2-functor. We then define an enhanced $\mf{sl}_3$ foam
$2$-category similar to Blanchet's $\mf{sl}_2$ foam category which appears naturally
in the categorical skew Howe context. Finally, we give a functor from enhanced
foams to standard foams, contrasting the $\mf{sl}_2$ case.
%
\subsubsection{Standard $\mathfrak{sl}_3$ foams}
%

In \cite{Kh5}, Khovanov gives a foam based categorification of the $\mf{sl}_3$ link invariant.
This construction was generalized by Mackaay-Vaz in \cite{MV2} and
Morrison-Nieh in \cite{MorrisonNieh} in the spirit of Bar-Natan's $\mf{sl}_2$ cobordism $2$-category \cite{BN2}, giving a categorification of Kuperberg's $\mf{sl}_3$ spider \cite{Kup}.  Mackaay and Vaz showed~\cite{MV} that these foam based constructions of $\mf{sl}_3$ link homologies coincide with the $n=3$ case of Khovanov and Rozansky's $\sln$ link homologies defined via matrix factorizations~\cite{KhR}.

We now recall the definition of this $2$-category, which we denote $3\cat{Foam}$, using a hybrid of the above approaches:
\begin{itemize}
\item Objects are sequences of points in the interval $[0,1]$ labeled by $1$ or $2$, together with a
zero object.
\item $1$-morphisms are formal direct sums of $\Z$-graded $\mathfrak{sl}_3$ webs -
directed, trivalent planar graphs with
boundary in which each (interior) vertex is a sink or a source - where an edge is directed out from a point
labeled by $1$ and into a point labeled by $2$ in the domain and vice-versa in the codomain.
\item $2$-morphisms are formal matrices of $\Bbbk$-linear combinations of degree-zero $\mf{sl}_3$ foams -
dotted surfaces with oriented singular seams which locally look like the product of the letter $Y$ with an
interval - considered up to isotopy (relative to the boundary) and local relations.
\end{itemize}
Denoting the $\Z$-grading of  a web by the monomial $q^t$ for $t \in \Z$, the degree of a foam
$F: q^{t_1} W_1 \to q^{t_2} W_2$ is given by the formula
\[
\deg(F) = 2\chi(F) - \# \partial + \frac{\# V}{2} + t_2 - t_1
\]
where $\chi(F)$ is the Euler characteristic of the foam $F$,
$\# \partial$ is the number of boundary points in either $W_1$ or $W_2$ (they agree!), and $\# V$ is the
total number of trivalent vertices in $W_1 \coprod W_2$. A dot should be viewed as a puncture for the sake
of computing an $\mf{sl}_3$ foam's Euler characteristic. We shall depict $\mf{sl}_3$ foams using the colors
red and blue, for clarity; unlike the $\mf{sl}_2$ case, these colors have no meaning as all facets are
treated equally.

The local $\mf{sl}_3$ foam relations are as follows (where a number next to a dot denotes the
number of dots present):
\begin{equation}\label{sl3closedfoam}
 \sphere[.65] \quad = \quad 0 \quad = \quad
\dottedsphere[.65]{} \qquad , \qquad
\twodottedsphere[.65] \quad = \quad -1
\end{equation}
\begin{multline}\label{sl3neckcutting}
\cylinder[.6] \quad = \quad - \quad \slthncone[.6] \quad - \quad \slthnctwo[.6]  \quad - \quad \slthncthree[.6]
\quad - \quad \dottedsphere[.6]{3} \slthncfour[.6] \\
\qquad - \quad \dottedsphere[.6]{3} \slthncfive[.6] \quad - \quad
\left( \dottedsphere[.6]{4} \quad + \quad \dottedsphere[.6]{3}^2\right) \capcup[.6]
\end{multline}
\begin{equation}\label{sl3thetafoam}
\thetafoam[.7]{\alpha}{\gamma}{\beta} \quad = \quad
\left\{
	\begin{array}{rl}
	1 & \text{ if } (\alpha, \beta, \gamma) = (0,1,2) \text{ or a cyclic permutation} \\
	-1 & \text{ if } (\alpha, \beta, \gamma) = (0,2,1) \text{ or a cyclic permutation} \\
	0 & \text{ all other triples with $\alpha, \beta, \gamma \leq 2$ }
	\end{array}
\right.
\end{equation}
\begin{equation}\label{sl3tube}
\tubeLHS[.6] \quad = \quad \tubeRHStopdot[.6] \quad - \quad \tubeRHSbottomdot[.6]
\end{equation}
\begin{equation}\label{sl3rocket}
\Rfoamone[.65] \quad = \quad - \quad \Rfoamtwo[.65] \quad - \quad \Rfoamthree[.65] \quad .
\end{equation}
Note that the local foam relations are all degree-homogeneous.
The direction of the singular seams keeps track of a cyclic ordering of the incident facets; by convention,
we take this ordering to be given by the right-hand rule. This convention is opposite to that
used in \cite{Kh5} and \cite{MV2}, hence one would expect to see opposite foam relations above;
however,  we also reverse the relation between singular seams and trivalent vertices (seams are directed
up through source vertices and down through sink vertices) which corresponds to taking different
generating morphisms for the $2$-category. It is easy to see that the above $2$-category is isomorphic to the standard $\mf{sl}_3$ $2$-category. Our conventions are chosen to align with those in the definition of the 2-category $\Bfoam{2}{m}$.

Using the neck cutting relation \eqref{sl3neckcutting} and the theta-foam relation \eqref{sl3thetafoam},
we can derive the following local relations:
\begin{equation} \label{sl3bamboo}
\bamboo[.6] \quad = \quad \bambooRHSone[.6] \quad - \quad \bambooRHStwo[.6]
\end{equation}
\begin{equation}\label{sl3airlock}
\airlock[.6] \quad = \quad - \quad \capcup[.6]
\end{equation}
\begin{multline}\label{sl33dot}
\dotsheet[.65]{3} \quad =
\quad - \quad \dottedsphere[.6]{3} \quad \dotsheet[.65]{2} \quad
- \left( \dottedsphere[.6]{4} \quad + \quad \dottedsphere[.6]{3}^2 \right) \dotsheet[.65]{} \\
\qquad - \left(\dottedsphere[.6]{5} \quad + 2 \quad \dottedsphere[.6]{3} \quad \dottedsphere[.6]{4} \quad
+ \quad \dottedsphere[.6]{3}^3 \right)  \diagsheet[.65] \quad .
\end{multline}

The values of the $3$-, $4$-, and $5$-dotted spheres should be viewed as (graded) parameters which
are typically set equal to zero in the literature, e.g. in \cite{Kh5} and \cite{MorrisonNieh}, although
this is not required for our considerations. In the case that the $3$-dotted sphere is zero, Morrison-Nieh
show the relation
\[
3 \quad \dotsheet[.65]{} \quad =
\quad
\chsheet[.65] \quad .
\]
which allows for a completely topological description of this $2$-category when $3$ is invertible
in $\Bbbk$.

Note that the set of relations above does not explicitly correspond with that
from either \cite{MV2} or  \cite{MorrisonNieh}.
The relations \eqref{sl3closedfoam}, \eqref{sl3neckcutting}, \eqref{sl3thetafoam}, together with a non-local
relation constitute the relations from \cite{MV2} (although in that work the authors introduce parameters
$a$, $b$, and $c$ in place of the dotted-sphere parameters above).
In \cite{MorrisonNieh}, Morrison and Nieh show that the
relations \eqref{sl3tube} and \eqref{sl3rocket} imply the non-local relation. Our chosen set of relations above
almost agree with those of Morrison-Nieh (when the dotted surface parameters equal zero):
they impose the relation that reversing the orientation of a
singular seam negates the foam instead of specifying the values of the theta-foams;
this seam reversal relation follows from \eqref{sl3neckcutting} and \eqref{sl3thetafoam}.

As in the $\mf{sl}_2$ case, we are interested in a related family of $2$-categories which is natural to
consider from the
perspective of categorical skew Howe duality.
\begin{defn}
$\foam{3}{m}$ is the $2$-category defined as follows:
\begin{itemize}
\item Objects are sequences $(a_1,\ldots,a_m)$ labeling points in the interval $[0,1]$
with $a_i \in \{0,1,2,3\}$ and $N = \sum_{i=1}^m a_i$ together with a zero object.
\item $1$-morphisms are formal direct sums of $\Z$-graded $\mathfrak{sl}_3$ webs mapping
between the points labeled by $1$ and $2$ as in $3\cat{Foam}$.
\item $2$-morphisms are formal matrices of $\Bbbk$-linear combinations of degree-zero
$\mf{sl}_3$ foams mapping between such webs.
\end{itemize}
\end{defn}
Note that the objects in $\foam{3}{m}$ correspond with the direct summands appearing in the decomposition of
$\bigwedge_q^N \left(\C^3 \otimes \C^m\right)$ into $\mf{sl}_m$ weight spaces and
$1$-morphisms correspond
to $\mf{sl}_3$ intertwiners between such summands.
For each $m$ and $N$, there is an obvious $2$-functor $\foam{3}{m} \to 3\cat{Foam}$ which forgets the
$0$'s and $3$'s.

%
\subsubsection{Foamation}
%

We now define $\mf{sl}_3$ foamation $2$-functors $\mathcal{U}_Q(\mf{sl}_m) \to \foam{3}{m}$.
As in the $\mf{sl}_2$ case, the $2$-functor is defined on objects by sending an $\mf{sl}_m$ weight
$\lambda = (\lambda_1,\ldots,\lambda_{m-1})$ to the sequence $(a_1,\ldots,a_m)$ with
$a_i \in \{0,1,2,3\}$, $\lambda_i = a_{i+1}-a_i$, and $\sum_{i=1}^{m}a_i = N$ provided
such a sequence exists and to the zero object otherwise.

The map on $1$-morphisms is again given by:
\[
\onel \{t\} \mapsto q^t
\xy
(0,0)*{
\begin{tikzpicture} [decoration={markings, mark=at position 0.6 with {\arrow{>}}; },scale=.75]
\draw [very thick] (0,0) -- (3,0);
\draw [very thick] (0,1) -- (3,1);
\node at (3.6,0) {$a_1$};
\node at (3.6,1) {$a_m$};
\node at (-.5,0) {$a_1$};
\node at (-.5,1) {$a_m$};
\node at (1.5,.7) {$\vdots$};
\end{tikzpicture}};
\endxy
\]
\[
\cal{E}_i \onel \{t\} \mapsto q^t
\xy
(0,0)*{
\begin{tikzpicture} [decoration={markings, mark=at position 0.6 with {\arrow{>}}; },scale=.75]
\draw [very thick] (0,0) -- (3,0);
\draw [very thick] (0,1) -- (3,1);
\draw [very thick, postaction={decorate}] (2,0) -- (1,1);
\node at (3.6,0) {$a_i$};
\node at (3.6,1) {$a_{i+1}$};
\node at (-1,0) {$a_{i} - 1$};
\node at (-1,1) {$a_{i+1}+1$};
\end{tikzpicture}};
\endxy
\]
and
\[
\cal{F}_i \onel \{t\} \mapsto q^t
\xy
(0,0)*{
\begin{tikzpicture} [decoration={markings, mark=at position 0.6 with {\arrow{>}}; },scale=.75]
\draw [very thick] (0,0) -- (3,0);
\draw [very thick] (0,1) -- (3,1);
\draw [very thick, postaction={decorate}] (2,1) -- (1,0);
\node at (3.6,0) {$a_i$};
\node at (3.6,1) {$a_{i+1}$};
\node at (-1,0) {$a_{i} + 1$};
\node at (-1,1) {$a_{i+1} - 1$};
\end{tikzpicture}};
\endxy
\]
when the boundary values lie in $\{0,1,2,3\}$ and to the zero $1$-morphism otherwise.
Note that the
orientation of the undirected strands (and whether they become deleted) is determined by these
boundary values and that we have not depicted $m-2$ horizontal strands in each of the latter two formulae.

We will make use of a preparatory lemma to deduce the existence of the foamation 2-functors. Let the
images of $\onel$, $\cal{E}_i \onel$,  and $\cal{F}_i \onel$
given above be denoted $\1_{\lambda}$, $\E_i \1_{\lambda} $, and $\F_i \1_{\lambda} $.

\begin{lem}\label{lem_abstract3}
There are  isomorphisms
$$\F_i \E_i \1_{\l} \cong \E_i \F_i \1_{\l} \oplus_{[- \la i,\l \ra]} \1_\l \text{ if } \la i,\l \ra \le 0$$
$$\E_i \F_i \1_{\l} \cong \F_i \E_i \1_{\l} \oplus_{[\la i,\l \ra]} \1_\l \text{ if } \la i,\l \ra \ge 0,$$
and $\F_j \E_i \1_{\l} \cong \E_i \F_j \1_{\l}$ for $i \ne j \in I$ in $\foam{3}{m}$.
\end{lem}

\begin{proof}
The proof is similar to the proof of Lemma~\ref{lem_abstract2}.
\end{proof}

\begin{prop} \label{prop_sl3}
For each $N>0$ there is a $2$-representation
$\Phi_3  \maps \Ucat_Q(\mf{sl}_m)  \to  \foam{3}{m}$ defined on single strand morphisms by:

\[
\Phi_3 \left(
 \xy 0;/r.17pc/:
 (0,7);(0,-7); **\dir{-} ?(.75)*\dir{>};
 (7,3)*{ \scs \lambda};
 (-2.5,-6)*{\scs i};
 (-10,0)*{};(10,0)*{};
 \endxy
\right) = \quad \Efoam[.5] \quad , \quad
\Phi_3 \left(
 \xy 0;/r.17pc/:
 (0,7);(0,-7); **\dir{-} ?(.75)*\dir{>};
 (0,0)*{\bullet};
 (7,3)*{ \scs \lambda};
 (-2.5,-6)*{\scs i};
 (-10,0)*{};(10,0)*{};
 \endxy
\right) = \quad \dotEfoam[.5]
\]
on crossings by:
\[
\qquad
\Phi_3 \left(
\xy 0;/r.20pc/:
	(0,0)*{\xybox{
    	(-4,-4)*{};(4,4)*{} **\crv{(-4,-1) & (4,1)}?(1)*\dir{>} ;
    	(4,-4)*{};(-4,4)*{} **\crv{(4,-1) & (-4,1)}?(1)*\dir{>};
    	(-5.5,-3)*{\scs i};
     	(5.5,-3)*{\scs i};
     	(9,1)*{\scs  \lambda};
     	(-10,0)*{};(10,0)*{};}};
\endxy
\right) = \quad \crossingEEfoam[.5]
\]
\[
\Phi_3 \left(
\xy 0;/r.20pc/:
	(0,0)*{\xybox{
    	(-4,-4)*{};(4,4)*{} **\crv{(-4,-1) & (4,1)}?(1)*\dir{>} ;
    	(4,-4)*{};(-4,4)*{} **\crv{(4,-1) & (-4,1)}?(1)*\dir{>};
    	(-5.5,-3)*{\scs i};
     	(5.5,-3)*{\scs \ \ \ i+1};
     	(9,1)*{\scs  \lambda};
     	(-10,0)*{};(10,0)*{};}};
\endxy
\right) = \quad \crossingEoneEtwofoam[.5] \quad , \qquad
\Phi_3 \left(
\xy 0;/r.20pc/:
	(0,0)*{\xybox{
    	(-4,-4)*{};(4,4)*{} **\crv{(-4,-1) & (4,1)}?(1)*\dir{>} ;
    	(4,-4)*{};(-4,4)*{} **\crv{(4,-1) & (-4,1)}?(1)*\dir{>};
    	(-5.5,-3)*{\scs i+1 \ \ \ };
     	(5.5,-3)*{\scs i};
     	(9,1)*{\scs  \lambda};
     	(-10,0)*{};(10,0)*{};}};
\endxy
\right) = (-1)^{a_{i+1}+1} \quad \crossingEtwoEonefoam[.5]
\]

\[
\Phi_3 \left(
\xy 0;/r.20pc/:
	(0,0)*{\xybox{
    	(-4,-4)*{};(4,4)*{} **\crv{(-4,-1) & (4,1)}?(1)*\dir{>} ;
    	(4,-4)*{};(-4,4)*{} **\crv{(4,-1) & (-4,1)}?(1)*\dir{>};
    	(-5.5,-3)*{\scs j};
     	(5.5,-3)*{\scs i};
     	(9,1)*{\scs  \lambda};
     	(-10,0)*{};(10,0)*{};}};
\endxy
\right) = \quad \crossingEthreeEonefoam[.5] \quad , \qquad
\Phi_3 \left(
\xy 0;/r.20pc/:
	(0,0)*{\xybox{
    	(-4,-4)*{};(4,4)*{} **\crv{(-4,-1) & (4,1)}?(1)*\dir{>} ;
    	(4,-4)*{};(-4,4)*{} **\crv{(4,-1) & (-4,1)}?(1)*\dir{>};
    	(-5.5,-3)*{\scs i};
     	(5.5,-3)*{\scs j};
     	(9,1)*{\scs  \lambda};
     	(-10,0)*{};(10,0)*{};}};
\endxy
\right) = \quad \crossingEoneEthreefoam[.5]
\]
where $j-i>1$, and on caps and cups by:
\[
\Phi_3 \left(
\xy 0;/r.20pc/:
	(0,0)*{\bbcef{i}};
	(8,4)*{\scs  \lambda};
	(-10,0)*{};(10,0)*{};
\endxy
\right) = \pm \quad \capFEfoam[.5] \quad , \qquad
\Phi_3 \left(
\xy 0;/r.20pc/:
	(0,0)*{\bbcfe{i}};
	(8,4)*{\scs  \lambda};
	(-10,0)*{};(10,0)*{};
\endxy
\right) = \pm \quad \capEFfoam[.5]
\]

\[
\Phi_3 \left(
\xy 0;/r.20pc/:
	(0,0)*{\bbpef{i}};
	(8,-4)*{\scs \lambda};
	(-10,0)*{};(10,0)*{};
\endxy
\right) = \pm \quad \cupFEfoam[.5] \quad , \qquad
\Phi_3 \left(
\xy 0;/r.20pc/:
	(0,0)*{\bbpfe{i}};
	(8,-4)*{\scs \lambda};
	(-10,0)*{};(10,0)*{};
\endxy
\right) = \pm \quad \cupEFfoam[.5]
\]
where the $\pm$ signs above depend on (the image of) the weight $\lambda$ and are given by\footnote{In fact, we will see that the $\pm$ signs involved in the definition of the $2$-functor on caps and cups
play no role in the proof of this proposition; they are determined by the proof of
Theorem \ref{thm_CL}. In the next section, we will give a topological interpretation of this system of signs.}:
\begin{equation}\label{sl3signs}

\end{equation}
where $N_i = a_i + a_{i+1}$.
\end{prop}

Again, the foams drawn above are general depictions of the image. To obtain the specific image of a generating
$2$-morphism, we delete any facets incident upon web strands which are deleted. The
singular seams in such pictures may degenerate in such situations, e.g. in the case that $\lambda$ maps to a
sequence with $a_i = 1, a_{i+1} =2$ we have
\[
\Phi_3 \left(
\xy 0;/r.20pc/:
	(0,0)*{\bbcef{i}};
	(8,4)*{\scs  \lambda};
	(-10,0)*{};(10,0)*{};
\endxy
\right) = \quad - \quad
\xy
(0,0)*{
%
};
\endxy
\]
which is a saddle cobordism.
\begin{proof}
As with the proof of Proposition~\ref{prop_sl2}, we apply Theorem~\ref{thm_CL}. Conditions
\eqref{co:int} and \eqref{co:E} follow as before and Lemma~\ref{lem_abstract3} gives conditions
\eqref{co:EF} and \eqref{co:EiFj}.

Working in the setting where the $3$-, $4$-, and $5$-dotted
spheres are set equal to zero, it is shown in \cite{MorrisonNieh} that
the vector space $\Hom(\1_{\lambda}, q^t \1_{\lambda})$ is zero
for $t<0$ and one-dimensional for $t=0$ (provided $\1_{\lambda}$ is non-zero) and that
for any webs $W_1$ and $W_2$, the
vector space $\Hom(q^{t_1}W_1, q^{t_2}W_2)$ is finite dimensional.
The same arguments show that this holds when these dotted spheres are not set equal to zero
(since they have negative Euler characteristic). This confirms condition~\eqref{co:hom}.

We thus conclude by checking condition~\eqref{co:KLR}, the KLR relations:
\begin{itemize}
\item Relation \eqref{eq_nil_rels}: the $2$-morphism
$
\vcenter{
\xy 0;/r.17pc/:
	(-4,-4)*{};(4,4)*{} **\crv{(-4,-1) & (4,1)}?(1)*\dir{};
	(4,-4)*{};(-4,4)*{} **\crv{(4,-1) & (-4,1)}?(1)*\dir{};
	(-4,4)*{};(4,12)*{} **\crv{(-4,7) & (4,9)}?(1)*\dir{};
	(4,4)*{};(-4,12)*{} **\crv{(4,7) & (-4,9)}?(1)*\dir{};
	(-4,12); (-4,13) **\dir{-}?(1)*\dir{>};
	(4,12); (4,13) **\dir{-}?(1)*\dir{>};
	(6,8)*{\lambda};
\endxy}
$
maps to a foam which can only possibly be non-zero for $\lambda$ mapping to sequences with
$a_i=2,3$ and $a_{i+1}=0,1$. When $(a_i,a_{i+1}) = (2,0)$ we compute the image:
\[
\xy
(0,-9)*{
};
\endxy \quad
= 0
\]
which follows from neck-cutting in a neighborhood of the
center singular seam. The computation for the remainder of the values of
$(a_i,a_{i+1})$ follows similarly.

We next compute the images of the $2$-morphisms
$
\xy
(0,0)*{
\xy 0;/r.11pc/:
    (-4,-4)*{};(4,4)*{} **\crv{(-4,-1) & (4,1)}?(1)*\dir{};
    (4,-4)*{};(-4,4)*{} **\crv{(4,-1) & (-4,1)}?(1)*\dir{};
    (4,4)*{};(12,12)*{} **\crv{(4,7) & (12,9)}?(1)*\dir{};
    (12,4)*{};(4,12)*{} **\crv{(12,7) & (4,9)}?(1)*\dir{};
    (-4,12)*{};(4,20)*{} **\crv{(-4,15) & (4,17)}?(1)*\dir{};
    (4,12)*{};(-4,20)*{} **\crv{(4,15) & (-4,17)}?(1)*\dir{};
    (-4,4)*{}; (-4,12) **\dir{-};
    (12,-4)*{}; (12,4) **\dir{-};
    (12,12)*{}; (12,20) **\dir{-};
    (4,20); (4,21) **\dir{-}?(1)*\dir{>};
    (-4,20); (-4,21) **\dir{-}?(1)*\dir{>};
    (12,20); (12,21) **\dir{-}?(1)*\dir{>};
   (18,8)*{\scs \lambda};
\endxy}
\endxy
$ and
$
\xy
(0,0)*{
\xy 0;/r.11pc/:
    (4,-4)*{};(-4,4)*{} **\crv{(4,-1) & (-4,1)}?(1)*\dir{};
    (-4,-4)*{};(4,4)*{} **\crv{(-4,-1) & (4,1)}?(1)*\dir{};
    (-4,4)*{};(-12,12)*{} **\crv{(-4,7) & (-12,9)}?(1)*\dir{};
    (-12,4)*{};(-4,12)*{} **\crv{(-12,7) & (-4,9)}?(1)*\dir{};
    (4,12)*{};(-4,20)*{} **\crv{(4,15) & (-4,17)}?(1)*\dir{};
    (-4,12)*{};(4,20)*{} **\crv{(-4,15) & (4,17)}?(1)*\dir{};
    (4,4)*{}; (4,12) **\dir{-};
    (-12,-4)*{}; (-12,4) **\dir{-};
    (-12,12)*{}; (-12,20) **\dir{-};
    (4,20); (4,21) **\dir{-}?(1)*\dir{>};
    (-4,20); (-4,21) **\dir{-}?(1)*\dir{>};
    (-12,20); (-12,21) **\dir{-}?(1)*\dir{>};
  (10,8)*{\scs \lambda};
\endxy}
\endxy
$, noting that the images can only be non-zero for $\lambda$ mapping to sequences with $a_i=3$ and $a_{i+1}=0$.
The above $2$-morphisms map to the foams
\[
\xy
(0,-22.5)*{
};
\endxy
\]
respectively; these are equal by equation \eqref{sl3airlock}. \\

\item Relation \eqref{eq_nil_dotslide}: the images of this relation are simply a restatement of equation \eqref{sl3tube}. \\

\item Relation \eqref{eq_r2_ij-gen}: The equality $(\alpha_i,\alpha_j)=0$ corresponds to $|i-j|\geq 2$ in which case the
image of the relation is realized via isotopy. For example, when $i<j$ we see that
\[
\xy
(0,-10)*{\crossingEoneEthreefoam[.5]};
(0,10)*{\crossingEthreeEonefoam[.5]};
\endxy
\]
is isotopic to the identity foam.

When $i\neq j$ and $(\alpha_i,\alpha_j)\neq 0$ we must have $j=i\pm1$. We'll compute
the image of this relationship in the case $j=i+1$ (the other case is similar). The image of the left-hand side is zero
when $a_{i+1} = 0,3$ since the intermediate objects in the relation map to the zero object in the image; the same is true
on the right-hand side when $a_{i+1}=0$. When $a_{i+1} = 3$, both $2$-morphisms involved  in the expression
on the right-hand side map to
the same foam, so the relation is satisfied since $t_{i,i+1} = - t_{i+1,i}$. When $a_{i+1}=1$ we have
\[
\vcenter{\xy 0;/r.17pc/:
    (-4,-4)*{};(4,4)*{} **\crv{(-4,-1) & (4,1)}?(1)*\dir{};
    (4,-4)*{};(-4,4)*{} **\crv{(4,-1) & (-4,1)}?(1)*\dir{};
    (-4,4)*{};(4,12)*{} **\crv{(-4,7) & (4,9)}?(1)*\dir{};
    (4,4)*{};(-4,12)*{} **\crv{(4,7) & (-4,9)}?(1)*\dir{};
    (8,8)*{\lambda};
    (4,12); (4,13) **\dir{-}?(1)*\dir{>};
    (-4,12); (-4,13) **\dir{-}?(1)*\dir{>};
  (-5,-3)*{\scs i};
     (5.1,-3)*{\scs \ \ \ i+1};
 \endxy} \quad
\mapsto \quad
\xy
(0,-8)*{
\begin{tikzpicture}  [fill opacity=0.2,  decoration={markings, mark=at position 0.6 with {\arrow{>}};}, scale=.6]
	\filldraw [fill=red] (-3,1) rectangle (1,4);
	\draw[very thick] (-3,1) -- (1,1);
	\draw[very thick] (-3,4) -- (1,4);
	\draw [very thick, red] (-1.75,4) .. controls (-1.75,3) and (-.25,2) .. (-.25,1);
	\draw[very thick, postaction={decorate}] (-.75,3) -- (-1.75,4);
	\draw[very thick, postaction={decorate}] (.75,0) -- (-.25,1);
	\path [fill=blue] (-1.75,4) .. controls (-1.75,3) and (-.25,2) .. (-.25,1) --
			(.75,0) .. controls (.75,1) and (-.75,2) .. (-.75,3);
	\path[fill=blue] (2,0) -- (2,3) -- (.75,3) .. controls (.75,2.35) and (0,1.5) .. (0,1.5) ..
		controls (0,1.5) and (.75,.65) .. (.75,0) -- cycle;
	\draw (2,0) -- (2,3);
	\path[fill=blue] (-2,0) -- (-2,3) -- (-.75,3) .. controls (-.75,2.35) and (0,1.5) .. (0,1.5) ..
		controls (0,1.5) and (-.75,.65) .. (-.75,0) -- cycle;
	\draw (-2,0) -- (-2,3);
	\path[fill=red] (-.75,3) .. controls (-.75,2.35) and (0,1.5) .. (0,1.5) .. controls (0,1.5) and (.75,2.35) .. (.75,3);
	\draw (.75,0) .. controls (.75,.65) and (0,1.5) .. (0,1.5);
	\draw (-.75,0) .. controls (-.75,.65) and (0,1.5) .. (0,1.5);
	\draw[very thick, red, postaction={decorate}] (.75,3) .. controls (.75,2.35) and (0,1.5) .. (0,1.5);
	\draw[very thick, red] (-.75,3) .. controls (-.75,2.35) and (0,1.5) .. (0,1.5);
	\draw[very thick, postaction={decorate}] (2,3) -- (.75,3);
	\draw[very thick, postaction={decorate}] (-.75,3) -- (-2,3);
	\draw[very thick, postaction={decorate}] (2,0) -- (.75,0);
	\draw[very thick, postaction={decorate}] (-.75,0) -- (-2,0);
	\draw[very thick, postaction={decorate}] (-.75,3) -- (.75,3);
	\draw[very thick, postaction={decorate}] (1.75,2) -- (.75,3);
	\draw[very thick, postaction={decorate}] (.25,-1) -- (-.75,0);
	\path [fill=blue] (.75,3) .. controls (.75,2) and (-.75,1) .. (-.75,0) --
			(.25,-1) .. controls (.25,0) and (1.75,1) .. (1.75,2);
	\filldraw [fill=red] (-1,-1) rectangle (3,2);
	\draw[very thick] (-1,-1) -- (3,-1);
	\draw[very thick] (-1,2) -- (3,2);
	\draw[very thick, red] (.25,-1) .. controls (.25,0) and (1.75,1) .. (1.75,2);
\end{tikzpicture}
};
(0,10)*{
\begin{tikzpicture}  [fill opacity=0.2,  decoration={markings, mark=at position 0.6 with {\arrow{>}};}, scale=.6]
	\filldraw [fill=red] (-3,1) rectangle (1,4);
	\draw[very thick] (-3,1) -- (1,1);
	\draw[very thick] (-3,4) -- (1,4);
	\draw [very thick, red] (-.25,4) .. controls (-.25,3) and (-1.75,2) .. (-1.75,1);
	\draw[very thick, postaction={decorate}] (.75,3) -- (-.25,4);
	\draw[very thick, postaction={decorate}] (-.75,0) -- (-1.75,1);
	\path [fill=blue] (-.25,4) .. controls (-.25,3) and (-1.75,2) .. (-1.75,1) --
			(-.75,0) .. controls (-.75,1) and (.75,2) .. (.75,3);
	\path[fill=blue] (2,0) -- (2,3) -- (.75,3) .. controls (.75,2.35) and (0,1.5) .. (0,1.5) ..
		controls (0,1.5) and (.75,.65) .. (.75,0) -- cycle;
	\draw (2,0) -- (2,3);
	\path[fill=blue] (-2,0) -- (-2,3) -- (-.75,3) .. controls (-.75,2.35) and (0,1.5) .. (0,1.5) ..
		controls (0,1.5) and (-.75,.65) .. (-.75,0) -- cycle;
	\draw (-2,0) -- (-2,3);
	\path[fill=red] (-.75,0) .. controls (-.75,.65) and (0,1.5) .. (0,1.5) .. controls (0,1.5) and (.75,.65) .. (.75,0);
	\draw (.75,3) .. controls (.75,2.35) and (0,1.5) .. (0,1.5);
	\draw (-.75,3) .. controls (-.75,2.35) and (0,1.5) .. (0,1.5);
	\draw[very thick, red] (.75,0) .. controls (.75,.65) and (0,1.5) .. (0,1.5);
	\draw[very thick, red, postaction={decorate}] (-.75,0) .. controls (-.75,.65) and (0,1.5) .. (0,1.5);
	\draw[very thick, postaction={decorate}] (2,3) -- (.75,3);
	\draw[very thick, postaction={decorate}] (-.75,3) -- (-2,3);
	\draw[very thick, postaction={decorate}] (2,0) -- (.75,0);
	\draw[very thick, postaction={decorate}] (-.75,0) -- (-2,0);
	\draw[very thick, postaction={decorate}] (-.75,0) -- (.75,0);
	\draw[very thick, postaction={decorate}] (.25,2) -- (-.75,3);
	\draw[very thick, postaction={decorate}] (1.75,-1) -- (.75,0);
	\path [fill=blue] (-.75,3) .. controls (-.75,2) and (.75,1) .. (.75,0) --
			(1.75,-1) .. controls (1.75,0) and (.25,1) .. (.25,2);
	\filldraw [fill=red] (-1,-1) rectangle (3,2);
	\draw[very thick] (-1,-1) -- (3,-1);
	\draw[very thick] (-1,2) -- (3,2);
	\draw[very thick, red] (1.75,-1) .. controls (1.75,0) and (.25,1) .. (.25,2);
\end{tikzpicture}
};
\endxy \quad ;
\]
after applying \eqref{sl3bamboo} to the center singular seam this gives the foam
\[
\xy
(0,0)*{
\begin{tikzpicture}  [fill opacity=0.2,  decoration={markings, mark=at position 0.6 with {\arrow{>}};}, scale=.6]
	\filldraw [fill=red] (-3,1) rectangle (1,4);
	\draw[very thick] (-3,1) -- (1,1);
	\draw[very thick] (-3,4) -- (1,4);
	\draw [very thick, red] (-.25,4) -- (-.25,1);
	\draw[very thick, postaction={decorate}] (.75,3) -- (-.25,4);
	\draw[very thick, postaction={decorate}] (.75,0) -- (-.25,1);
	\draw[very thick, postaction={decorate}] (2,0) -- (.75,0);
	\draw[very thick, postaction={decorate}] (2,3) -- (.75,3);
	\draw[very thick, postaction={decorate}] (-.75,0) -- (-2,0);
	\draw[very thick, postaction={decorate}] (-.75,3) -- (-2,3);
	\draw (.75,3) -- (.75,0);
	\draw (-.75,3) -- (-.75,0);
	\draw (2,3) -- (2,0);
	\draw (-2,3) -- (-2,0);
	\path [fill=blue] (-2,3) rectangle (-.75,0);
	\path [fill=blue] (2,3) rectangle (.75,0);
	\path [fill=blue] (-.75,3) -- (.25,2) -- (.25,-1) -- (-.75,0) -- cycle;
	\path [fill=blue] (.75,3) -- (-.25,4) -- (-.25,1) -- (.75,0) -- cycle;
	\node[opacity=1] at (-.25, 0) {$\bullet$};
	\draw[very thick, postaction={decorate}] (.25,2) -- (-.75,3);
	\draw[very thick, postaction={decorate}] (.25,-1) -- (-.75,0);
	\filldraw [fill=red] (-1,-1) rectangle (3,2);
	\draw[very thick] (-1,-1) -- (3,-1);
	\draw[very thick] (-1,2) -- (3,2);
	\draw[very thick, red] (.25,-1) -- (.25,2);
\end{tikzpicture}
};
\endxy \quad - \quad
\xy
(0,0)*{
\begin{tikzpicture}  [fill opacity=0.2,  decoration={markings, mark=at position 0.6 with {\arrow{>}};}, scale=.6]
	\filldraw [fill=red] (-3,1) rectangle (1,4);
	\draw[very thick] (-3,1) -- (1,1);
	\draw[very thick] (-3,4) -- (1,4);
	\draw [very thick, red] (-.25,4) -- (-.25,1);
	\draw[very thick, postaction={decorate}] (.75,3) -- (-.25,4);
	\draw[very thick, postaction={decorate}] (.75,0) -- (-.25,1);
	\draw[very thick, postaction={decorate}] (2,0) -- (.75,0);
	\draw[very thick, postaction={decorate}] (2,3) -- (.75,3);
	\draw[very thick, postaction={decorate}] (-.75,0) -- (-2,0);
	\draw[very thick, postaction={decorate}] (-.75,3) -- (-2,3);
	\draw (.75,3) -- (.75,0);
	\draw (-.75,3) -- (-.75,0);
	\draw (2,3) -- (2,0);
	\draw (-2,3) -- (-2,0);
	\path [fill=blue] (-2,3) rectangle (-.75,0);
	\path [fill=blue] (2,3) rectangle (.75,0);
	\path [fill=blue] (-.75,3) -- (.25,2) -- (.25,-1) -- (-.75,0) -- cycle;
	\path [fill=blue] (.75,3) -- (-.25,4) -- (-.25,1) -- (.75,0) -- cycle;
	\node[opacity=1] at (.25, 3) {$\bullet$};
	\draw[very thick, postaction={decorate}] (.25,2) -- (-.75,3);
	\draw[very thick, postaction={decorate}] (.25,-1) -- (-.75,0);
	\filldraw [fill=red] (-1,-1) rectangle (3,2);
	\draw[very thick] (-1,-1) -- (3,-1);
	\draw[very thick] (-1,2) -- (3,2);
	\draw[very thick, red] (.25,-1) -- (.25,2);
\end{tikzpicture}
};
\endxy
\]
which is the image of
$\vcenter{\xy 0;/r.11pc/:
  (3,-9);(3,9) **\dir{-}?(1)*\dir{>}+(2.3,0)*{};
  (-3,-9);(-3,9) **\dir{-}?(1)*\dir{>}+(2.3,0)*{};
  (-3,0)*{\bullet};(-6.5,5)*{};
  (-5,-6)*{\scs i};     (5.1,-6)*{\scs \ \ \ i+1};
 \endxy} \;\; - \;\;
  \vcenter{\xy 0;/r.11pc/:
  (3,-9);(3,9) **\dir{-}?(1)*\dir{>}+(2.3,0)*{};
  (-3,-9);(-3,9) **\dir{-}?(1)*\dir{>}+(2.3,0)*{};
  (3,0)*{\bullet};(7,5)*{};
  (-5,-6)*{\scs i};     (5.1,-6)*{\scs \ \ \ i+1};
 \endxy}
 $. Finally, when $a_{i+1}=2$
\[
\vcenter{\xy 0;/r.17pc/:
    (-4,-4)*{};(4,4)*{} **\crv{(-4,-1) & (4,1)}?(1)*\dir{};
    (4,-4)*{};(-4,4)*{} **\crv{(4,-1) & (-4,1)}?(1)*\dir{};
    (-4,4)*{};(4,12)*{} **\crv{(-4,7) & (4,9)}?(1)*\dir{};
    (4,4)*{};(-4,12)*{} **\crv{(4,7) & (-4,9)}?(1)*\dir{};
    (8,8)*{\lambda};
    (4,12); (4,13) **\dir{-}?(1)*\dir{>};
    (-4,12); (-4,13) **\dir{-}?(1)*\dir{>};
  (-5,-3)*{\scs i};
     (5.1,-3)*{\scs \ \ \ i+1};
 \endxy} \quad
\mapsto  - \quad
\xy
(0,-8)*{
\begin{tikzpicture}  [fill opacity=0.2,  decoration={markings, mark=at position 0.6 with {\arrow{>}};}, scale=.6]
	\filldraw [fill=red] (-3,1) rectangle (1,4);
	\draw[very thick] (-3,1) -- (1,1);
	\draw[very thick] (-3,4) -- (1,4);
	\draw [very thick, red] (-1.75,4) .. controls (-1.75,3) and (-.25,2) .. (-.25,1);
	\draw[very thick, postaction={decorate}] (-.75,3) -- (-1.75,4);
	\draw[very thick, postaction={decorate}] (.75,0) -- (-.25,1);
	\path [fill=blue] (-1.75,4) .. controls (-1.75,3) and (-.25,2) .. (-.25,1) --
			(.75,0) .. controls (.75,1) and (-.75,2) .. (-.75,3);
	\path[fill=blue] (2,0) -- (2,3) -- (.75,3) .. controls (.75,2.35) and (0,1.5) .. (0,1.5) ..
		controls (0,1.5) and (.75,.65) .. (.75,0) -- cycle;
	\draw (2,0) -- (2,3);
	\path[fill=blue] (-2,0) -- (-2,3) -- (-.75,3) .. controls (-.75,2.35) and (0,1.5) .. (0,1.5) ..
		controls (0,1.5) and (-.75,.65) .. (-.75,0) -- cycle;
	\draw (-2,0) -- (-2,3);
	\path[fill=red] (-.75,0) .. controls (-.75,.65) and (0,1.5) .. (0,1.5) .. controls (0,1.5) and (.75,.65) .. (.75,0);
	\draw (.75,3) .. controls (.75,2.35) and (0,1.5) .. (0,1.5);
	\draw (-.75,3) .. controls (-.75,2.35) and (0,1.5) .. (0,1.5);
	\draw[very thick, red, postaction={decorate}] (.75,0) .. controls (.75,.65) and (0,1.5) .. (0,1.5);
	\draw[very thick, red] (-.75,0) .. controls (-.75,.65) and (0,1.5) .. (0,1.5);
	\draw[very thick, postaction={decorate}] (.75,3) -- (2,3);
	\draw[very thick, postaction={decorate}] (-2,3) -- (-.75,3);
	\draw[very thick, postaction={decorate}] (.75,0) -- (2,0);
	\draw[very thick, postaction={decorate}] (-2,0) -- (-.75,0);
	\draw[very thick, postaction={decorate}] (.75,0) -- (-.75,0);
	\draw[very thick, postaction={decorate}] (1.75,2) -- (.75,3);
	\draw[very thick, postaction={decorate}] (.25,-1) -- (-.75,0);
	\path [fill=blue] (.75,3) .. controls (.75,2) and (-.75,1) .. (-.75,0) --
			(.25,-1) .. controls (.25,0) and (1.75,1) .. (1.75,2);
	\filldraw [fill=red] (-1,-1) rectangle (3,2);
	\draw[very thick] (-1,-1) -- (3,-1);
	\draw[very thick] (-1,2) -- (3,2);
	\draw[very thick, red] (.25,-1) .. controls (.25,0) and (1.75,1) .. (1.75,2);
\end{tikzpicture}
};
(0,10)*{
\begin{tikzpicture}  [fill opacity=0.2,  decoration={markings, mark=at position 0.6 with {\arrow{>}};}, scale=.6]
	\filldraw [fill=red] (-3,1) rectangle (1,4);
	\draw[very thick] (-3,1) -- (1,1);
	\draw[very thick] (-3,4) -- (1,4);
	\draw [very thick, red] (-.25,4) .. controls (-.25,3) and (-1.75,2) .. (-1.75,1);
	\draw[very thick, postaction={decorate}] (.75,3) -- (-.25,4);
	\draw[very thick, postaction={decorate}] (-.75,0) -- (-1.75,1);
	\path [fill=blue] (-.25,4) .. controls (-.25,3) and (-1.75,2) .. (-1.75,1) --
			(-.75,0) .. controls (-.75,1) and (.75,2) .. (.75,3);
	\path[fill=blue] (2,0) -- (2,3) -- (.75,3) .. controls (.75,2.35) and (0,1.5) .. (0,1.5) ..
		controls (0,1.5) and (.75,.65) .. (.75,0) -- cycle;
	\draw (2,0) -- (2,3);
	\path[fill=blue] (-2,0) -- (-2,3) -- (-.75,3) .. controls (-.75,2.35) and (0,1.5) .. (0,1.5) ..
		controls (0,1.5) and (-.75,.65) .. (-.75,0) -- cycle;
	\draw (-2,0) -- (-2,3);
	\path[fill=red] (-.75,3) .. controls (-.75,2.35) and (0,1.5) .. (0,1.5) .. controls (0,1.5) and (.75,2.35) .. (.75,3);
	\draw (.75,0) .. controls (.75,.65) and (0,1.5) .. (0,1.5);
	\draw (-.75,0) .. controls (-.75,.65) and (0,1.5) .. (0,1.5);
	\draw[very thick, red] (.75,3) .. controls (.75,2.35) and (0,1.5) .. (0,1.5);
	\draw[very thick, red, postaction={decorate}] (-.75,3) .. controls (-.75,2.35) and (0,1.5) .. (0,1.5);
	\draw[very thick, postaction={decorate}] (.75,3) -- (2,3);
	\draw[very thick, postaction={decorate}] (-2,3) -- (-.75,3);
	\draw[very thick, postaction={decorate}] (.75,0) -- (2,0);
	\draw[very thick, postaction={decorate}] (-2,0) -- (-.75,0);
	\draw[very thick, postaction={decorate}] (.75,3) -- (-.75,3);
	\draw[very thick, postaction={decorate}] (.25,2) -- (-.75,3);
	\draw[very thick, postaction={decorate}] (1.75,-1) -- (.75,0);
	\path [fill=blue] (-.75,3) .. controls (-.75,2) and (.75,1) .. (.75,0) --
			(1.75,-1) .. controls (1.75,0) and (.25,1) .. (.25,2);
	\filldraw [fill=red] (-1,-1) rectangle (3,2);
	\draw[very thick] (-1,-1) -- (3,-1);
	\draw[very thick] (-1,2) -- (3,2);
	\draw[very thick, red] (1.75,-1) .. controls (1.75,0) and (.25,1) .. (.25,2);
\end{tikzpicture}
};
\endxy \quad ;
\]
and applying \eqref{sl3tube} (turned sideways!) to the region between the semi-circular seams gives the
foam
\[
\xy
(0,0)*{
\begin{tikzpicture}  [fill opacity=0.2,  decoration={markings, mark=at position 0.6 with {\arrow{>}};}, scale=.6]
	\filldraw [fill=red] (-3,1) rectangle (1,4);
	\draw[very thick] (-3,1) -- (1,1);
	\draw[very thick] (-3,4) -- (1,4);
	\draw [very thick, red] (-.25,4) -- (-.25,1);
	\draw[very thick, postaction={decorate}] (.75,3) -- (-.25,4);
	\draw[very thick, postaction={decorate}] (.75,0) -- (-.25,1);
	\path[fill=blue] (-.25,4) -- (.75,3) -- (.75,0) -- (-.25,1) -- cycle;
	\draw (2,0) -- (2,3);
	\draw (-2,0) -- (-2,3);
	\draw[very thick, postaction={decorate}] (.75,3) -- (2,3);
	\draw[very thick, postaction={decorate}] (-2,3) -- (-.75,3);
	\draw[very thick, postaction={decorate}] (.75,3) -- (-.75,3);
	\draw[very thick, postaction={decorate}] (.75,0) -- (2,0);
	\draw[very thick, postaction={decorate}] (-2,0) -- (-.75,0);
	\draw[very thick, postaction={decorate}] (.75,0) -- (-.75,0);
	\draw[very thick, red, postaction={decorate}] (.75,0) -- (.75,3);
	\draw[very thick, red, postaction={decorate}] (-.75,3) -- (-.75,0);
	\path[fill=red] (-.75,3) rectangle (.75,0);
	\path[fill=blue] (-2,3) rectangle (-.75,0);
	\path[fill=blue] (2,3) rectangle (.75,0);
	\draw[very thick, postaction={decorate}] (.25,2) -- (-.75,3);
	\draw[very thick, postaction={decorate}] (.25,-1) -- (-.75,0);
	\path[fill=blue] (-.75,3) -- (.25,2) -- (.25,-1) -- (-.75,0) -- cycle;
	\node[opacity=1] at (-.15,-.2) {$\bullet$};
	\filldraw [fill=red] (-1,-1) rectangle (3,2);
	\draw[very thick] (-1,-1) -- (3,-1);
	\draw[very thick] (-1,2) -- (3,2);
	\draw[very thick, red] (.25,-1) -- (.25,2);
\end{tikzpicture}
};
\endxy \quad - \quad
\xy
(0,0)*{
\begin{tikzpicture}  [fill opacity=0.2,  decoration={markings, mark=at position 0.6 with {\arrow{>}};}, scale=.6]
	\filldraw [fill=red] (-3,1) rectangle (1,4);
	\draw[very thick] (-3,1) -- (1,1);
	\draw[very thick] (-3,4) -- (1,4);
	\draw [very thick, red] (-.25,4) -- (-.25,1);
	\draw[very thick, postaction={decorate}] (.75,3) -- (-.25,4);
	\draw[very thick, postaction={decorate}] (.75,0) -- (-.25,1);
	\path[fill=blue] (-.25,4) -- (.75,3) -- (.75,0) -- (-.25,1) -- cycle;
	\node[opacity=1] at (.15,3.2) {$\bullet$};
	\draw (2,0) -- (2,3);
	\draw (-2,0) -- (-2,3);
	\draw[very thick, postaction={decorate}] (.75,3) -- (2,3);
	\draw[very thick, postaction={decorate}] (-2,3) -- (-.75,3);
	\draw[very thick, postaction={decorate}] (.75,3) -- (-.75,3);
	\draw[very thick, postaction={decorate}] (.75,0) -- (2,0);
	\draw[very thick, postaction={decorate}] (-2,0) -- (-.75,0);
	\draw[very thick, postaction={decorate}] (.75,0) -- (-.75,0);
	\draw[very thick, red, postaction={decorate}] (.75,0) -- (.75,3);
	\draw[very thick, red, postaction={decorate}] (-.75,3) -- (-.75,0);
	\path[fill=red] (-.75,3) rectangle (.75,0);
	\path[fill=blue] (-2,3) rectangle (-.75,0);
	\path[fill=blue] (2,3) rectangle (.75,0);
	\draw[very thick, postaction={decorate}] (.25,2) -- (-.75,3);
	\draw[very thick, postaction={decorate}] (.25,-1) -- (-.75,0);
	\path[fill=blue] (-.75,3) -- (.25,2) -- (.25,-1) -- (-.75,0) -- cycle;
	\filldraw [fill=red] (-1,-1) rectangle (3,2);
	\draw[very thick] (-1,-1) -- (3,-1);
	\draw[very thick] (-1,2) -- (3,2);
	\draw[very thick, red] (.25,-1) -- (.25,2);
\end{tikzpicture}
};
\endxy
\]
which again is the image of
$\vcenter{\xy 0;/r.11pc/:
  (3,-9);(3,9) **\dir{-}?(1)*\dir{>}+(2.3,0)*{};
  (-3,-9);(-3,9) **\dir{-}?(1)*\dir{>}+(2.3,0)*{};
  (-3,0)*{\bullet};(-6.5,5)*{};
  (-5,-6)*{\scs i};     (5.1,-6)*{\scs \ \ \ i+1};
 \endxy} \;\; - \;\;
  \vcenter{\xy 0;/r.11pc/:
  (3,-9);(3,9) **\dir{-}?(1)*\dir{>}+(2.3,0)*{};
  (-3,-9);(-3,9) **\dir{-}?(1)*\dir{>}+(2.3,0)*{};
  (3,0)*{\bullet};(7,5)*{};
  (-5,-6)*{\scs i};     (5.1,-6)*{\scs \ \ \ i+1};
 \endxy}
$. \\

\item Relation \eqref{eq_dot_slide_ij-gen}: These hold by sliding a dot along a facet, i.e. via isotopy. \\

\item Relation \eqref{eq_r3_easy-gen}: For all choices of $i$, $j$, and $k$ this relation holds via isotopy. This
is obvious whenever one of the strands carries a label which is at least $2$ bigger or smaller than both
other labels. In the remaining cases a computation is necessary; we'll exhibit this only for two cases, since
the remaining cases follow similarly.

First, suppose that $j=i$ and $k=i+1$, then both sides of the relation automatically map to zero unless
$\lambda$ maps to a sequence with $a_{i+1}=1$.
We hence compute the image (of both sides of the relation) in this case, finding them to be
\[
\xy
(0,-22.5)*{
};
\endxy
\]
which are related via isotopy (no matter which values of $a_i$ and $a_{i+2}$ we choose).

Next, suppose that $j=i+1$ and $k=i+2$, then we compute both sides of the relation to be
\[
\xy
(0,-20)*{
%
};
\endxy
\]
which are equal up to isotopy. \\

\item Relation \eqref{eq_r3_hard-gen}: We must have $j=i\pm1$ and we'll only compute for $j=i+1$ since
the other case is analogous. Note that all $2$-morphisms involved automatically map to zero if
$\lambda$ is sent to a sequence with $a_{i+1}=3$ or with $a_i=0,1$, so we'll compute for the remaining values.

First, let $a_{i+1}=0$, then
\[
 \vcenter{
 \xy 0;/r.15pc/:
    (-4,-4)*{};(4,4)*{} **\crv{(-4,-1) & (4,1)}?(1)*\dir{};
    (4,-4)*{};(-4,4)*{} **\crv{(4,-1) & (-4,1)}?(1)*\dir{};
    (4,4)*{};(12,12)*{} **\crv{(4,7) & (12,9)}?(1)*\dir{};
    (12,4)*{};(4,12)*{} **\crv{(12,7) & (4,9)}?(1)*\dir{};
    (-4,12)*{};(4,20)*{} **\crv{(-4,15) & (4,17)}?(1)*\dir{};
    (4,12)*{};(-4,20)*{} **\crv{(4,15) & (-4,17)}?(1)*\dir{};
    (-4,4)*{}; (-4,12) **\dir{-};
    (12,-4)*{}; (12,4) **\dir{-};
    (12,12)*{}; (12,20) **\dir{-};
  (-4,20); (-4,21) **\dir{-}?(1)*\dir{>};
  (4,20); (4,21) **\dir{-}?(1)*\dir{>};
  (12,20); (12,21) **\dir{-}?(1)*\dir{>};
  (-4,-6)*{\scs i};
  (4,-6)*{\scs i+1};
  (12,-6)*{\scs i};
\endxy} \mapsto \quad
\xy
(0,-19.5)*{
};
\endxy
\qquad
\qquad
\quad
 \vcenter{
 \xy 0;/r.15pc/:
    (4,-4)*{};(-4,4)*{} **\crv{(4,-1) & (-4,1)}?(1)*\dir{};
    (-4,-4)*{};(4,4)*{} **\crv{(-4,-1) & (4,1)}?(1)*\dir{};
    (-4,4)*{};(-12,12)*{} **\crv{(-4,7) & (-12,9)}?(1)*\dir{};
    (-12,4)*{};(-4,12)*{} **\crv{(-12,7) & (-4,9)}?(1)*\dir{};
    (4,12)*{};(-4,20)*{} **\crv{(4,15) & (-4,17)}?(1)*\dir{};
    (-4,12)*{};(4,20)*{} **\crv{(-4,15) & (4,17)}?(1)*\dir{};
    (4,4)*{}; (4,12) **\dir{-};
    (-12,-4)*{}; (-12,4) **\dir{-};
    (-12,12)*{}; (-12,20) **\dir{-};
  (-4,20); (-4,21) **\dir{-}?(1)*\dir{>};
  (4,20); (4,21) **\dir{-}?(1)*\dir{>};
  (-12,20); (-12,21) **\dir{-}?(1)*\dir{>};
  (-12,-6)*{\scs i};
  (-4,-6)*{\scs i+1};
  (4,-6)*{\scs i};
\endxy} \mapsto  0.
\]
This confirms the relation since the former is isotopic to the identity foam when $a_i = 2,3$,
noting that in these cases either the leftmost or rightmost front facet is deleted.

Next, let $a_{i+1}=1$, then we compute
\[
 \vcenter{
 \xy 0;/r.15pc/:
    (-4,-4)*{};(4,4)*{} **\crv{(-4,-1) & (4,1)}?(1)*\dir{};
    (4,-4)*{};(-4,4)*{} **\crv{(4,-1) & (-4,1)}?(1)*\dir{};
    (4,4)*{};(12,12)*{} **\crv{(4,7) & (12,9)}?(1)*\dir{};
    (12,4)*{};(4,12)*{} **\crv{(12,7) & (4,9)}?(1)*\dir{};
    (-4,12)*{};(4,20)*{} **\crv{(-4,15) & (4,17)}?(1)*\dir{};
    (4,12)*{};(-4,20)*{} **\crv{(4,15) & (-4,17)}?(1)*\dir{};
    (-4,4)*{}; (-4,12) **\dir{-};
    (12,-4)*{}; (12,4) **\dir{-};
    (12,12)*{}; (12,20) **\dir{-};
  (-4,20); (-4,21) **\dir{-}?(1)*\dir{>};
  (4,20); (4,21) **\dir{-}?(1)*\dir{>};
  (12,20); (12,21) **\dir{-}?(1)*\dir{>};
  (-4,-6)*{\scs i};
  (4,-6)*{\scs i+1};
  (12,-6)*{\scs i};
\endxy} \mapsto - \quad
\xy
(0,-19.5)*{
};
\endxy
\qquad
\quad
 \vcenter{
 \xy 0;/r.15pc/:
    (4,-4)*{};(-4,4)*{} **\crv{(4,-1) & (-4,1)}?(1)*\dir{};
    (-4,-4)*{};(4,4)*{} **\crv{(-4,-1) & (4,1)}?(1)*\dir{};
    (-4,4)*{};(-12,12)*{} **\crv{(-4,7) & (-12,9)}?(1)*\dir{};
    (-12,4)*{};(-4,12)*{} **\crv{(-12,7) & (-4,9)}?(1)*\dir{};
    (4,12)*{};(-4,20)*{} **\crv{(4,15) & (-4,17)}?(1)*\dir{};
    (-4,12)*{};(4,20)*{} **\crv{(-4,15) & (4,17)}?(1)*\dir{};
    (4,4)*{}; (4,12) **\dir{-};
    (-12,-4)*{}; (-12,4) **\dir{-};
    (-12,12)*{}; (-12,20) **\dir{-};
  (-4,20); (-4,21) **\dir{-}?(1)*\dir{>};
  (4,20); (4,21) **\dir{-}?(1)*\dir{>};
  (-12,20); (-12,21) **\dir{-}?(1)*\dir{>};
  (-12,-6)*{\scs i};
  (-4,-6)*{\scs i+1};
  (4,-6)*{\scs i};
\endxy} \mapsto \quad
\xy
(0,-19.5)*{
};
\endxy \quad ;
\]
the relation then follows from equation \eqref{sl3rocket} when $a_i= 2,3$, again noting that either the leftmost or
rightmost front facet is deleted for both foams.

Finally, if $a_{i+1}=2$, then
\[
 \vcenter{
 \xy 0;/r.15pc/:
    (-4,-4)*{};(4,4)*{} **\crv{(-4,-1) & (4,1)}?(1)*\dir{};
    (4,-4)*{};(-4,4)*{} **\crv{(4,-1) & (-4,1)}?(1)*\dir{};
    (4,4)*{};(12,12)*{} **\crv{(4,7) & (12,9)}?(1)*\dir{};
    (12,4)*{};(4,12)*{} **\crv{(12,7) & (4,9)}?(1)*\dir{};
    (-4,12)*{};(4,20)*{} **\crv{(-4,15) & (4,17)}?(1)*\dir{};
    (4,12)*{};(-4,20)*{} **\crv{(4,15) & (-4,17)}?(1)*\dir{};
    (-4,4)*{}; (-4,12) **\dir{-};
    (12,-4)*{}; (12,4) **\dir{-};
    (12,12)*{}; (12,20) **\dir{-};
  (-4,20); (-4,21) **\dir{-}?(1)*\dir{>};
  (4,20); (4,21) **\dir{-}?(1)*\dir{>};
  (12,20); (12,21) **\dir{-}?(1)*\dir{>};
  (-4,-6)*{\scs i};
  (4,-6)*{\scs i+1};
  (12,-6)*{\scs i};
\endxy} \mapsto 0
\qquad
\qquad
 \vcenter{
 \xy 0;/r.15pc/:
    (4,-4)*{};(-4,4)*{} **\crv{(4,-1) & (-4,1)}?(1)*\dir{};
    (-4,-4)*{};(4,4)*{} **\crv{(-4,-1) & (4,1)}?(1)*\dir{};
    (-4,4)*{};(-12,12)*{} **\crv{(-4,7) & (-12,9)}?(1)*\dir{};
    (-12,4)*{};(-4,12)*{} **\crv{(-12,7) & (-4,9)}?(1)*\dir{};
    (4,12)*{};(-4,20)*{} **\crv{(4,15) & (-4,17)}?(1)*\dir{};
    (-4,12)*{};(4,20)*{} **\crv{(-4,15) & (4,17)}?(1)*\dir{};
    (4,4)*{}; (4,12) **\dir{-};
    (-12,-4)*{}; (-12,4) **\dir{-};
    (-12,12)*{}; (-12,20) **\dir{-};
  (-4,20); (-4,21) **\dir{-}?(1)*\dir{>};
  (4,20); (4,21) **\dir{-}?(1)*\dir{>};
  (-12,20); (-12,21) **\dir{-}?(1)*\dir{>};
  (-12,-6)*{\scs i};
  (-4,-6)*{\scs i+1};
  (4,-6)*{\scs i};
\endxy} \mapsto - \quad
\xy
(0,-19.5)*{
};
\endxy \quad ;
\]
again, this confirms the relation since the latter is the identity foam when $a_i=2,3$.
\end{itemize}
\end{proof}

%
\subsubsection{Enhanced $\mathfrak{sl}_3$ foams}
%


We now aim to better explain the signs in \eqref{sl3signs} which give the scalings of the cap and
cup $2$-morphisms.
We take inspiration from Blanchet's $\mathfrak{sl}_2$ foam construction in which the edges of
webs labeled by $2$ and the corresponding facets of foams play a special role (and in particular are
not deleted).

We hence define an $\mathfrak{sl}_3$ foam $2$-category in which we retain $3$-labeled edges and the
corresponding $3$-labeled facets. Although such a construction is not suggested at the decategorified level
(as it was in the $\mathfrak{sl}_2$ case)
we will see that the foamation functor is much more natural to define in this context and
that an appropriately defined functor which forgets the $3$-labeled data gives a topological interpretation of the
scalings. We believe that such $n$-labeled facets will play a role in extending the work in this paper to the
$n\geq 4$ case; this will be the subject of a follow-up paper \cite{LQR2}.

%

\begin{defn}
$\Bfoam{3}{m}$ is the 2-category defined as follows:
\begin{itemize}
\item Objects are sequences $(a_1,\ldots,a_m)$ labeling points in the interval $[0,1]$ with
$a_i \in \{0,1,2,3\}$ and $N = \sum_{i=1}^m a_i$ together with a zero object.
\item $1$-morphisms are formal direct sums of $\Z$-graded enhanced $\mathfrak{sl}_3$ webs -
trivalent planar graphs with boundary with edges of two types: directed edges
$ \; \;
\xy
(0,0)*{
\begin{tikzpicture} [fill opacity=0.2,  decoration={markings,
                        mark=at position 0.6 with {\arrow{>}};    }]
\draw[very thick, postaction={decorate}] (0,-.5) -- (0,.5);
\end{tikzpicture}};
\endxy
\; \;
$
and undirected ``$3$-labeled" edges
$ \; \;
\xy
(0,0)*{
\begin{tikzpicture} [fill opacity=0.2,  decoration={markings,
                        mark=at position 0.6 with {\arrow{>}};    }]
\draw[double] (0,-.5) -- (0,.5);
\end{tikzpicture}};
\endxy
\; \;
$
where vertices involving only the directed edges are as in $3\cat{Foam}$ and vertices involving the
$3$-labeled edges take the form
\[
\xy
(0,0)*{
\begin{tikzpicture} [fill opacity=0.2,  decoration={markings,
                        mark=at position 0.6 with {\arrow{>}};    }]
\draw[double] (0,0) -- (0,1);
\draw[very thick, postaction={decorate}] (-.875,-.5) -- (0,0);
\draw[very thick, postaction={decorate}] (0,0) -- (.875,-.5);
\end{tikzpicture}};
\endxy
\qquad
\text{ or }
\qquad
\xy
(0,0)*{
\begin{tikzpicture} [fill opacity=0.2,  decoration={markings,
                        mark=at position 0.6 with {\arrow{>}};    }]
\draw[double] (0,0) -- (0,1);
\draw[very thick, postaction={decorate}] (.875,-.5) -- (0,0);
\draw[very thick, postaction={decorate}] (0,0) -- (-.875,-.5);
\end{tikzpicture}};
\endxy
\quad .
\]
Oriented edges are directed out from points labeled by $1$ and into points labeled
by $2$ in the domain and vice-versa in the codomain and $3$-labeled edges are attached to
points labeled by $3$ in both the domain and codomain. As in $\foam{3}{m}$, no edges are attached to
points labeled by $0$.


\item $2$-morphisms are $\mathfrak{sl}_3$ foams between such webs where we allow additional $3$-labeled
facets incident upon the $3$-labeled strands of the webs and attached to the remainder of the foam
along singular seams which are allowed to intersect the traditional singular seams; the $3$-labeled facets are not
allowed to carry dots. We impose local relations on these foams.
\end{itemize}
\end{defn}


We shall refer to the $2$-morphisms in this category as ``Blanchet" $\mathfrak{sl}_3$ foams and depict the
$3$-labeled facets in yellow. The traditional facets of these foams will continue to be depicted in both
red and blue.

The local relations are given by the relations in $3\cat{Foam}$ in regions where $3$-labeled facets are
not present with additional relations for the $3$-labeled facets.
The seams where a $3$-labeled facet meet the traditional facets are allowed to move freely on the
foam (relative to the points where such seams meet the web vertices depicted above).
We impose a strong neck-cutting relation for these facets:


\[
};
\endxy
\]
and the condition that we may delete any $3$-labeled facet $F$ not incident upon the boundary
at the cost of multiplying by $(-1)^{\chi(F)}$. Finally, we have the relation
\begin{equation}\label{3Btube}
\xy
(0,0)*{
%
};
\endxy
\quad .
\end{equation}
An Euler characteristic argument shows that these relations are consistent. Using such foams, we have the
following result.

\begin{prop} \label{prop_sl3_enhanced}
The definition in Proposition \ref{prop_sl2} describes a family of $2$-functors
$\Ucat_Q(\mathfrak{sl}_m) \to \Bfoam{3}{m}$.
\end{prop}

As before, we view the definition as showing the general image of each generating $2$-morphism; edges
connected to points labeled by $0$ and facets incident upon them are understood to be deleted.
The proof of this proposition follows as in the proof of Propositions \ref{prop_sl2} and \ref{prop_sl3}.

The relations for the $3$-labeled facets allow us to delete any facet which does not meet the boundary;
however, this is not enough to define a forgetful functor $\Bfoam{3}{m} \to \foam{3}{m}$.
We can give such a rule by taking into account the boundary data.


Given a Blanchet $\mf{sl}_3$ foam $F$, define $\chi_3(F)$ to be the Euler
characteristic of its $3$-labeled facets and let
$n_u(F)$ denote the number of $3$-labeled edges in the codomain of $F$. Let
$\Psi(F) = (-1)^{\chi_3(F)-n_u(F)} \bar{F}$ where $\bar{F}$ is the $\mathfrak{sl}_3$ foam
obtained from $F$ by deleting the $3$-labeled facets (and the $3$-labeled edges from the boundary webs).
Similarly, define $n_b(F)$ to be the number of $3$-labeled edges in the domain of $F$
and $n_l(F)$ and $n_r(F)$ to be the number of points labeled by $3$ in the codomain and
domain (respectively) of the webs between which $F$ maps; of course, these later two denote the number of
$3$-labeled vertical intervals on the left and right boundary of $F$.


\begin{prop}\label{sl3_forgetful}
The assignment $F \mapsto \Psi(F)$ defines a $2$-functor $\Bfoam{3}{m} \to \foam{3}{m}$ where
objects are sent to themselves and enhanced webs are sent to the webs obtained by deleting the
$3$-labeled edges.
\end{prop}
\begin{proof}
It suffices to show that $\Psi$ is compatible with horizontal and vertical composition of foams.
To this end,
consider foams $F_1, F_2,$ and $F_3$ which can be composed as indicated by the following schematic:
\[
 \xy
(0,0)*{
\begin{tikzpicture} [scale=0.8]
  \draw (0,0) rectangle (1,1);
  \draw (0,1) rectangle (1,2);
  \draw (-1,0) rectangle (0,1);
  \node at (0.5,0.5) {$F_1$};
   \node at (0.5,1.5) {$F_2$};
  \node at (-0.5,0.5) {$F_3$};
\end{tikzpicture}};
\endxy \quad .
\]

We have
\begin{align*}
\chi_3(F_1 \cup F_2) - n_u(F_1 \cup F_2) &= \chi_3(F_1 \cup F_2) - n_u(F_2) \\
	&= \chi_3(F_1) - n_u(F_1) + \chi_3(F_2) - n_u(F_2)
\end{align*}
and
\begin{align*}
\chi_3(F_1 \cup F_3) - n_u(F_1 \cup F_3) & = \chi_3(F_1) - n_l(F_1) + \chi_3(F_3)
		- n_u(F_1) + n_l(F_1) - n_u(F_3) \\
	&= \chi_3(F_1) - n_u(F_1) + \chi_3(F_3) - n_u(F_3)
\end{align*}
which gives the result.
\end{proof}

One can now consider the composition of the $2$-functors defined in Propositions
\ref{prop_sl3_enhanced} and \ref{sl3_forgetful}.
\begin{prop}
The composition $\Ucat_Q(\mathfrak{sl}_m) \to \Bfoam{3}{m} \to \foam{3}{m}$ gives the $2$-functor
from Proposition \ref{prop_sl3}.
\end{prop}
\begin{proof}
The proof follows from a routine, yet tedious, calculation. We'll exhibit a few of the more interesting cases:
\begin{itemize}
\item Let $\lambda$ map to a sequence where $a_{i+1}=2$, then
\[
\xy 0;/r.20pc/:
	(0,0)*{\xybox{
    	(-4,-4)*{};(4,4)*{} **\crv{(-4,-1) & (4,1)}?(1)*\dir{>} ;
    	(4,-4)*{};(-4,4)*{} **\crv{(4,-1) & (-4,1)}?(1)*\dir{>};
    	(-5.5,-3)*{\scs i};
     	(5.5,-3)*{\scs \ \ \ i+1};
     	(9,1)*{\scs  \lambda};
     	(-10,0)*{};(10,0)*{};}};
\endxy \quad \mapsto \quad
\xy
(0,0)*{
};
\endxy \quad ;
\]
note that this guarantees that the relevant degree zero bubble
$
\xy 0;/r.18pc/:
 (0,0)*{\icbub{2}{i}};
  (4,8)*{\lambda};
 \endxy
$
is sent to the identity $2$-morphism in $\foam{3}{m}$.
\end{itemize}
\end{proof}

\subsubsection{Clark-Morrison-Walker $\mathfrak{sl}_3$ foams}
%

One may notice that the topology of a $3$-labeled facet is relatively unimportant; the signs obtained by
removing any $3$-labeled facet (not incident upon the boundary webs) depend only on the facet's boundary
seams. One may then ask why one needs these facets at all: couldn't we instead introduce a
Clark-Morrison-Walker (CMW) version of $\mathfrak{sl}_3$ foams?


Indeed, we can define such a theory by removing all $3$-labeled facets and edges of webs
from the definitions in the previous subsection, keeping only the new ``CMW seams''
and imposing the relation that one may remove a closed seam at the cost of multiplying by $-1$.
It is easy to see that we obtain a family of $2$-functors from $\Ucat_Q(\mathfrak{sl}_m)$ to the
$2$-category of CMW $\mathfrak{sl}_3$ foams. Note that the CMW seams in such a theory do not need
fringes, unlike the $\mathfrak{sl}_2$ case.

However, when one tries to define a forgetful $2$-functor to the category of (traditional) $\mathfrak{sl}_3$ foams
as before, it surprisingly appears that the rigidity
obtained from the interaction of the $3$-labeled facets with the $3$-labeled edges plays a non-trivial role.
Indeed, there is no hope to define such a $2$-functor as we now demonstrate.

Assume a forgetful $2$-functor exists. We need maps
\[
 \xy
(0,0)*{
};
\endxy
\]
with $\gamma \delta = -1$. Since the foam
\[
 \xy
(0,0)*{
%
};
\endxy
\]
is isotopic to the identity, we must have $\alpha \gamma =1$ and since
\[
 \xy
(0,0)*{
%
};
\endxy
\]
we must have $\delta = \alpha$. This then gives
\[
1=\alpha \gamma = \delta \gamma = -1,
\]
a contradiction.

Note that this argument remains valid even if we enhance the CMW seams with fringes (repeat the above
argument with all seam directed into the page). The above argument no longer remains valid if the seams
are given an orientation; however, similar arguments prevent a definition of CMW $\mathfrak{sl}_3$ foams
possessing a forgetful functor to $\mathfrak{sl}_3$ foams.

%
\subsection{Extended foamation $2$-functors}
%


In order to construct categorified link invariants, we will need to consider the images of the Rickard complexes
under the foamation $2$-functors and hence the images of divided powers. Recall that the later are
$1$-morphisms which lie in $\UcatD_Q(\mf{sl}_m)$ but not in $\Ucat_Q(\mf{sl}_m)$.

The foamation $2$-functors $\cal{U}_Q(\mf{sl}_m) \to n\cat{(B)Foam}_m(N)$ extend to $2$-functors
$$\UcatD_Q(\mf{sl}_m) \to Kar(n\cat{(B)Foam}_m(N)).$$  We would like to consider these extended
$2$-functors; however, the Karoubi envelope of the foam $2$-categories is not easy to work with. Indeed,
the indecomposable $2$-morphisms in these $2$-categories are closely related to (dual) canonical bases.

It turns out, however, that the foam $2$-categories are ``closer'' to their Karoubi envelopes than the $2$-categories
$\Ucat_Q(\mf{sl}_m)$ are to $\UcatD_Q(\mf{sl}_m)$. Indeed we shall see that the images of the divided powers
already exist in the foam $2$-categories.

\begin{prop}
The $2$-functors defined in Propositions \ref{prop_sl2}, \ref{prop_sl3}, and \ref{prop_sl3_enhanced} extend
to $2$-functors from $\Ucatc_Q(\mf{sl}_m)$ to the relevant foam $2$-categories.
\end{prop}
\begin{proof}
%
Recall that any category $\cal{C}$ embeds fully faithfully into $Kar(\cal{C})$
by sending an object $c$ to the pair $(c,\id_c)$ and a morphism $f:c\to d$ to the triple $(\id_c,f,\id_d)$, see ~\cite[Chapter 6.5]{Bor,Bor}. It thus suffices to show that
when we restrict the $2$-functor $\UcatD_Q(\mf{sl}_m) \to Kar(n\cat{(B)Foam}_m(N))$ to $\Ucatc_Q(\mf{sl}_m)$, the
images of all $1$-morphisms lie
in the subcategory $n\cat{(B)Foam}_m(N) \subset Kar(n\cat{(B)Foam}_m(N))$ (up to isomorphism).

We begin with the $\mf{sl}_2$ case. Since $\cal{E}_i^k \onel \mapsto 0$ and
$\cal{F}_i^k \onel \mapsto 0$ for all $k\geq 3$, we need only to consider
the $1$-morphisms $\cal{E}_i^{(2)} \onel$ and $\cal{F}_i^{(2)} \onel$. Note that the $1$-morphism
\[
\cal{E}_i^{(2)}\onel = (\cal{E}_i^2 \onel \{1\},
\xy 0;/r.18pc/:
  (0,0)*{\xybox{
    (-4,-4)*{};(4,6)*{} **\crv{(-4,-1) & (4,1)}?(1)*\dir{>};
    (4,-4)*{};(-4,6)*{} **\crv{(4,-1) & (-4,1)}?(1)*\dir{>}?(.75)*{\bullet};
     (8,1)*{ \lambda};
     (-10,0)*{};(10,0)*{};
     }};
  \endxy
)
\]
is mapped to zero unless $a_i=2$ and $a_{i+1}=0$. In this case, the above is mapped to
\[
\left( q \; \;
\xy
(0,0)*{
};
\endxy \quad
\right)
\]
which is isomorphic to
\[
\left(
\xy
(0,0)*{
%
};
\endxy \quad
\right)
\]
in $Kar(2\cat{BFoam}_m(N))$ using equation \eqref{sl2NH_enh}.
Similarly, we find that the image of $\cal{F}_i^{(2)} \onel$ is isomorphic to
\[
\left(
\xy
(0,0)*{
%
};
\endxy \quad
\right)
\]
when $a_i = 0$ and $a_{i+1}=2$ (the only case when $\cal{F}_i^2 \onel$ is not mapped to zero).

In the $\mathfrak{sl}_3$ case, it suffices to consider the $2$-functor $\Ucat_Q(\mf{sl}_m) \to \Bfoam{3}{m}$, since
the $2$-functor $\Ucat_Q(\mf{sl}_m) \to \foam{3}{m}$ is obtained via composition with the forgetful functor.
We find that $\cal{E}_i^{k} \onel$ and $\cal{F}_i^k \onel$ are both sent to zero for
$k\geq4$, so it suffices to consider
the $1$-morphisms $\cal{E}_i^{(2)} \onel$, $\cal{E}_i^{(3)} \onel$, $\cal{F}_i^{(2)} \onel$, and
$\cal{F}_i^{(3)} \onel$.

We find that $\cal{E}_i^{(2)} \onel$ is (again) mapped to the $1$-morphism
\[
\left(
q \left(
\xy
(0,0)*{
};
\endxy \quad
\right)
\]
which is isomorphic in $Kar(\Bfoam{3}{m})$ to
\[
\left(
\xy
(0,0)*{
%
};
\endxy \quad
\right)
\]
for any value of the $a_l$'s using the $\mf{sl}_3$ foam relations.
Similarly, the image of $\cal{F}_i^{(2)} \onel$ is isomorphic to
\[
\left(
\xy
(0,0)*{
%
};
\endxy \quad
\right) \quad .
\]
The directed lines in the latter two $1$-morphisms above can be viewed
as ``$2$-labeled'' edges directed in the opposite direction; from this perspective, the boundary labels appear more appropriate.

The only case where $\cal{E}_i^{3}\onel$ is not sent to zero is for $a_i=3$ and $a_{i+1}= 0$.
and in this case $\cal{E}_i^{(3)}\onel$ is mapped to
\begin{equation}\label{E(3)}
\left(
q^3\quad
\xy
(0,0)*{
%
};
\endxy
\quad
\right)
\end{equation}
which is isomorphic to
\[
\left(
\xy
(0,0)*{
%
};
\endxy \quad
\right)
\]
in $Kar(\Bfoam{3}{m})$.
The isomorphism above is evident after noticing that the foam in \eqref{E(3)}
is equal to the following foam:
\[
\xy
(0,0)*{
%
};
\endxy
\]
using \eqref{sl3airlock} and \eqref{3Btube}. Finally, the image of $\cal{F}_i^{(3)} \onel$ is isomorphic
to
\[
\left(
\xy
(0,0)*{
%
};
\endxy \quad
\right)
\]
when $a_i=0$ and $a_{i+1}=3$ (i.e. the only non-zero case).
\end{proof}

We can summarize the $2$-functors $\Ucatc_Q(\mf{sl}_m) \to n\cat{(B)Foam}_m(N)$ using ladders with labelings on the
diagonal edges:
\[
\cal{E}_i^{(k)} \onel \mapsto
\xy
(0,0)*{
%
};
\endxy
\]
where $2$-labeled edges in $\foam{3}{m}$ and $\Bfoam{3}{m}$ should be viewed as
$1$-labeled edges oriented in the
opposite direction and $3$-labeled edges should be deleted in $\foam{3}{m}$ and un-oriented in
$\Bfoam{3}{m}$. For example, the image of $\cal{E}^{(3)} \mathbf{1}_{-3}$ in
$\foam{3}{2}$ is the ``empty web'' between the sequences $(3,0)$ and $(0,3)$.

%
\section{Applications}
%


Many known constructions in link homology follow from the $2$-functors defined in the previous section.
Indeed, we will re-construct Khovanov homology, $\mf{sl}_3$ link homology, and
categorified highest weight projectors in these theories using the categorified quantum Weyl group action.
The skew Howe perspective also provides a framework for showing that Cautis and Kamnitzer's algebro-geometric
formulation of $\mf{sl}_3$ link homology is isomorphic to Khovanov's $\mf{sl}_3$ link homology
(the latter is  known to be the same as $\mf{sl}_3$ Khovanov-Rozanksy homology and the category $\cal{O}$
$\mf{sl}_3$ link invariant).
We also explain how foam relations follow as consequences of the relations in the categorified
quantum group.

%
\subsection{Link homology via skew Howe duality} \label{SectionKnotHom}
%
In this section, we show that all of the ingredients needed to define $\mf{sl}_2$ and $\mf{sl}_3$ link
homology theories can be recovered from the foamation functors. We also show how the invariant of
any link can be given as the image of a complex in $\Ucatc_Q(\mf{sl}_m)$. This suggests that the graphical
calculus in the categorified quantum group can be used to explore properties of categorified link
invariants.


%
\subsubsection{Categorified braidings}
%

In \cite[Theorem 4.3]{CKL}, Cautis-Kamnitzer-Licata show that the action of the quantum Weyl group elements
$T_i 1_\l$ on the skew Howe representation $\Alt{q}{N}(\C^n_q \otimes \C^m_q)$ gives the
braiding on the category of finite-dimensional $\U(\mf{sl}_n)$ representations, up to a factor of $\pm q^r$.
In the language of webs, this says that the value of a crossing is given by the image of the corresponding
quantum Weyl group element.

The same holds true at the categorified level (after extending the foamation $2$-functors to
$2$-categories of complexes),
i.e. the complexes assigned to crossings in $\mf{sl}_2$ and
$\mf{sl}_3$ link homology can be recovered, up to shifts in quantum and homological degree, as the
images of the Rickard complexes. Recall these are given by
\[
\cal{T}_i \onel =
\xymatrix{ \cal{E}_i^{(-\l_i)} \onel \ar[r]^-{d_1} & \cal{E}_i^{(-\l_i+1)} \cal{F}_i \onel \{1\} \ar[r]^-{d_2} & \cdots
\ar[r]^-{d_s} & \cal{E}_i^{(-\l_i+s)} \cal{F}_i^{(s)} \onel \{s\} \ar[r]^-{d_{s+1}} &\cdots}
\]
when $\l_i \leq 0$ and
\[
\cal{T}_i \onel =
\xymatrix{ \cal{F}_i^{(\l_i)} \onel \ar[r]^-{d_1} & \cal{F}_i^{(\l_i+1)} \cal{E}_i \onel \{1\} \ar[r]^-{d_2} & \cdots
\ar[r]^-{d_s} & \cal{F}_i^{(\l_i+s)} \cal{E}_i^{(s)} \onel \{s\} \ar[r]^-{d_{s+1}} &\cdots}
\]
when $\l_i \geq 0$, where in the above formulae the leftmost term is in homological degree zero. The above
complexes are isomorphic when $\l_i = 0$.

The complexes $\cal{T}_i \onel$ are invertible, up to homotopy, with inverses given by
\[
\onel \cal{T}_i^{-1} =
\xymatrix{ \cdots \ar[r]^-{d_{s+1}^*} & \onel \cal{E}_i^{(s)} \cal{F}_i^{(-\l_i+s)} \{-s\} \ar[r]^-{d_s^*}
& \cdots \ar[r]^-{d_2^*}
& \onel \cal{E}_i \cal{F}_i^{(-\l_i+1)} \{-1\} \ar[r]^-{d_1^*} & \onel \cal{F}_i^{(-\l_i)}}
\]
when $\l_i \leq 0$ and
\[
\onel \cal{T}_i^{-1} =
\xymatrix{ \cdots \ar[r]^-{d_{s+1}^*} & \onel \cal{F}_i^{(s)} \cal{E}_i^{(\l_i+s)} \{-s\} \ar[r]^-{d_s^*}
& \cdots \ar[r]^-{d_2^*}
& \onel \cal{F}_i \cal{E}_i^{(\l_i+1)} \{-1\} \ar[r]^-{d_1^*} & \onel \cal{E}_i^{(\l_i)}}
\]
when $\l_i \geq 0$, where in these formulae the rightmost term is in homological degree zero.
Note that these complexes are obtained by taking the adjoints of the above (in the category of complexes).

We begin with the $\mf{sl}_2$ case.
When $\lambda$ maps to a sequence with $a_i=1=a_{i+1}$,
\begin{equation}\label{sl2crossing+11}
\Phi_2(\cal{T}_i \onel) = \left(
\xymatrix{
\xy
(0,0)*{
};
\endxy
} \right)
\end{equation}
which gives the value of the positive $(1,1)$ crossing
$ \:\:
\xy
(0,0)*{
%
};
\endxy \:\:
$. This complex is the Blanchet foam analog of the formula for the crossing given in
\cite{Kh1} and \cite{BN2}. The negative crossing
$ \:\:
\xy
(0,0)*{
%
};
\endxy \:\:
$
is given by
\begin{equation}\label{sl2crossing-11}
\Phi_2(\onel \cal{T}_i^{-1}) = \left(
\xymatrix{
q^{-1} \quad
\xy
(0,0)*{
%
};
\endxy
} \right) \quad .
\end{equation}
In both \eqref{sl2crossing+11} and \eqref{sl2crossing-11}, the identity web appears in homological degree zero.

In order to give a construction of the link invariant via the foamation $2$-functors,
we will also need the formulae for the braidings involving $0$'s and $2$'s. Defining the positive crossings
to be the images of the $\cal{T}_i\onel$ in the appropriate weights and the negative crossings to be the images
of the $\onel \cal{T}_i^{-1}$, this gives the formulae
\begin{align}
\xy
(0,0)*{
%
};
\endxy \quad  \notag
\end{align}
where the dotted strands are meant to indicate a ``$0$-labeled'' edge, i.e. an edge that is not actually present.
The braiding on two such $0$-labeled edges is simply the empty web mapping between the appropriate labels.

In the $\mf{sl}_3$ case, we give the formulae for the braidings in $\Bfoam{3}{m}$, since those in
$\foam{3}{m}$ can be recovered from these via the forgetful $2$-functor. We'll first compute the braidings for
the traditional $\mf{sl}_3$ edges. The $(1,1)$ crossings are again given as
\begin{equation}\label{sl3crossing+11}
\xy
(0,0)*{
};
\endxy
} \right) \quad
\end{equation}
where the identity webs are in homological degree zero in both of the above formulae. Similarly, the
$(1,2)$ braidings are given by
\begin{equation}\label{sl3crossing+12}
\xy
(0,0)*{
%
};
\endxy
} \right) \quad ;
\end{equation}
note that \eqref{sl3crossing+12} is a positive $(1,2)$ braiding and \eqref{sl3crossing-12} is a negative
$(1,2)$ braiding, although topologically the former is a left-handed crossing and the latter is right-handed.
The $(2,1)$ braidings are given by
\begin{equation}\label{sl3crossing+21}
\xy
(0,0)*{
%
};
\endxy
}
\right)
\end{equation}
and the $(2,2)$ braidings are given by
\begin{equation}\label{sl3crossing+22}
\xy
(0,0)*{
%
};
\endxy
}
\right)
\end{equation}

As in the $\mathfrak{sl}_2$ case, we will also need the formulae for the braidings between sequences involving
$0$'s and $3$'s in order to construct the link invariant.
Defining the positive crossings to be the images of the $\cal{T}_i\onel$ (for appropriate $\l$)
and the negative crossings to be the images of the $\onel \cal{T}_i^{-1}$, this gives the formulae
{\allowdisplaybreaks
\begin{align}
\xy
(0,0)*{
%
};
\endxy \quad  \notag
\end{align}} where again the dotted strands are meant to indicate ``$0$-labeled''  (non-)edges.
The braiding on two such edges is the empty web.

Let $n=2,3$; it will be useful to note that any object $\mathbf{a} = (a_1,\cdots,a_m)$ in
$n\cat{(B)Foam}_m(N)$ can be identified with a canonical object which corresponds to the same
$\mf{sl}_n$ representation as $\mathbf{a}$
(up to isomorphism).
Given an object $\mathbf{a}$ in $n\cat{(B)Foam}_m(N)$, denote by $\bar{\mathbf{a}}$ the associated
``reduced sequence'' defined to be the same sequence as $\mathbf{a}$ with all values
$0$ and $n$ deleted. For example, if $\mathbf{a}=(1,3,0,2,0,1)$ and $n=3$, then $\bar{\mathbf{a}}=(1,2,1)$.

\begin{defn}
Given an object $\mathbf{a}$ of $n\cat{(B)Foam}_m(N)$, the associated \emph{canonical sequence} is
the unique object $\mathbf{a}'$ in $n\cat{(B)Foam}_m(N)$
such that $\bar{\mathbf{a}'}=\bar{\mathbf{a}}$ and
\[\mathbf{a}'=(0,\ldots,0,a'_k,a'_{k+1},\ldots,a'_{k+r},n,\ldots,n)\] with $0<a'_{k+s}<n$ for $0\leq s\leq r$.
\end{defn}

The trivial braidings \eqref{sl2crossings} and \eqref{sl3crossings} can be used to give an
equivalence
between an object $\mathbf{a}$ in $n\cat{(B)Foam}_m(N)$ and its canonical sequence $\mathbf{a}'$
(this is the analog of \cite[Corollaries 7.3 and 7.8]{Cautis} in the web and foam setting).
Let the web
$\xymatrix{\mathbf{a} \ar[r]^{\beta_{\mathbf{a}}} & \mathbf{a}'}$ be given by the (composition of)
braidings involving $0$- and $n$-labeled edges and let the web
$\xymatrix{\mathbf{a}' \ar[r]^{\beta_{\mathbf{a}}^{-1}} & \mathbf{a}}$ be given
using the inverses of the above braidings.
Since the images of the Rickard complexes braid in any (integrable) $2$-representation \cite{CK},
the above maps are uniquely defined up to coherent isomorphism. Fix once and for all choices of
$\beta_{\mathbf{a}}$ and $\beta_{\mathbf{a}}^{-1}$ for each object $\mathbf{a}$ in each of the foam
$2$-categories.



%
\subsubsection{The $\mf{sl}_2$ tangle invariant}
%

The webs that appear in the image of  the foamation $2$-functors are all in ladder form; hence,
we require a method for assigning a complex of ladders in the foam $2$-categories to each tangle.
A process which transforms any web to a ladder is detailed in \cite{CKM}; however, an adaptation of a
construction from \cite{Cautis} is more useful for our purposes.

Let $\tau$ be an oriented $(r,t)$-tangle diagram, i.e. a tangle diagram with $r$ endpoints on the right and
$t$ endpoints on the left, which we assume to be in Morse position with respect to the horizontal axis.
We now describe a method for
assigning to this diagram a complex $\llbracket \tau \rrbracket_2$
in $\Bfoam[r+2l]{2}{r+2l}$, for $l$ sufficiently large.

We assign to each basic tangle a complex of $1$-morphisms mapping between canonical sequences; the
complex assigned to a tangle will then be the horizontal composition (in the $2$-category of complexes) of
the basic complexes.

A tangle involving no crossings, cups, or caps is mapped to the identity web of the sequence
$(0,\dots,0,1,\ldots,1,2,\ldots,2)$ where the number of $1$'s is equal to the number of strands in the tangle.
For example, we have
\[
\xy
(0,0)*{
};
\endxy
\]
where the bottom dotted web edges are zero-labeled, i.e. not actually present.

We'd like to map the cup as follows
\[
\xy
(0,0)*{%
};
\endxy
\]
however, the domain on this web will not be a canonical sequence (especially when other strands of the
tangle are present). We will hence pre-compose with the relevant web
$\beta_{\mathbf{a}}^{-1}$. For example,
\[
\xy
(0,0)*{%
};
\endxy
\]
where again we have depicted the $0$-labeled edges.

We similarly define the map on the remaining cup and caps to be given by
\[
\xy
(0,0)*{%
};
\endxy
\]
where we pre- or post-compose with the appropriate braiding maps
$\beta^{-1}$ and $\beta$
as necessary to ensure that the webs map between canonical sequences.

We assign the complexes \eqref{sl2crossing+11} and \eqref{sl2crossing-11} to the positive and negative
left-oriented crossings. This assignment determines the value of the invariant on the remainder of the crossings
(up to isomorphism) since they can be obtained from the left-oriented crossings by composing with
caps and cups, e.g. we have
\[
\xy
(0,0)*{%
};
\endxy
\]
where the latter is understood to represent the complex assigned to a left-oriented crossing horizontally
composed in the category of complexes with the indicated webs (and the necessary braiding maps
$\beta^{-1}$ and $\beta$ so that the webs in the complex map between canonical sequences). Formulae
for the other crossings can be obtained similarly.

\begin{prop} The complex $\llbracket \tau \rrbracket_2$ assigned to a tangle diagram $\tau$,
viewed in the homotopy category of complexes of $\Bfoam{2}{m}$, gives an invariant of framed tangles.
\end{prop}
\begin{proof}
It suffices to check the tangle Reidemeister moves (see \cite{Kassel} or \cite{CK01}); this is a standard
computation following the argument detailed in \cite{BN2} adapted to the Blanchet foam setting.
Alternatively, one can simplify the computation using the proof of \cite[Proposition 7.9]{Cautis},
where it is shown that (most of) the desired relations hold already in the categorified quantum group.
\end{proof}

One can check that (locally)
\[
\left \llbracket \quad
\xy
(0,0)*{\begin{tikzpicture} [decoration={markings, mark=at position 1 with {\arrow{>}}; },scale=.3]
\draw[very thick] (0,-2) -- (0,0) -- (1,1) to[out=45,in=90] (2,.5) to[out=270,in=315] (1,0) -- (.7,.3);
\draw[very thick, postaction={decorate}] (.3,.7) -- (0,1) -- (0,3);
\end{tikzpicture}};
\endxy \quad
\right \rrbracket_2 \simeq
q^{-1}
\left \llbracket \quad
\xy
(0,0)*{\begin{tikzpicture} [decoration={markings, mark=at position 1 with {\arrow{>}}; },scale=.3]
\draw[very thick, postaction={decorate}] (0,-2) -- (0,3);
\end{tikzpicture}};
\endxy \quad
\right \rrbracket_2
\quad \text{and} \quad
\left \llbracket \quad
\xy
(0,0)*{\begin{tikzpicture} [decoration={markings, mark=at position 1 with {\arrow{>}}; },scale=.3]
\draw[very thick, postaction={decorate}] (.7,.7) -- (1,1) to[out=45,in=90] (2,.5) to[out=270,in=315]
	(1,0) --  (0,1) -- (0,3);
\draw[very thick] (0,-2) -- (0,0) -- (.3,.3);
\end{tikzpicture}};
\endxy \quad
\right \rrbracket_2 \simeq
q
\left \llbracket \quad
\xy
(0,0)*{\begin{tikzpicture} [decoration={markings, mark=at position 1 with {\arrow{>}}; },scale=.3]
\draw[very thick, postaction={decorate}] (0,-2) -- (0,3);
\end{tikzpicture}};
\endxy \quad
\right \rrbracket_2
\]
so renormalizing the invariant using the writhe $w(\tau)$ of the tangle:
\[
\llbracket \tau \rrbracket_2^{r} := q^{w(\tau)} \llbracket \tau \rrbracket_2
\]
gives an invariant independent of framing.

Given a link $L$, the invariant $\llbracket L \rrbracket_2^r$ is a complex of webs mapping between the
sequence $(\underline{0},\underline{2}) := (0,\ldots,0,2,\ldots,2)$ and itself. Applying the functor
\[
\HOM(\id_{(\underline{0},\underline{2})},-) := \bigoplus_{t\in \Z} \Hom(q^{-t}\id_{(\underline{0},\underline{2})},-)
\]
to this complex (where $\id_{(\underline{0},\underline{2})}$ is the identity web)
and setting the parameter $\dottedsphere[.5]{3}=0$ gives a complex of finite-dimensional
graded vector spaces, which we denote $Kh_2(L)$. As the notation indicates, we have the following result.

\begin{prop}
The (co)homology of the complex $Kh_2(L)$ is the Khovanov homology of the link $L$.
\end{prop}
\begin{proof}
Let $D$ be a diagram of the link $L$.
The complex $\llbracket D \rrbracket_2^r$ consists of $\mf{sl}_2$ webs with no $1$-labeled boundary, i.e.
these webs consist of $1$-labeled circles joined to each other (and to the boundary) by $2$-labeled
edges. Such a web in $\llbracket D \rrbracket_2^r$ contributes a direct summand of dimension
$2^{\# \text{ of circles}}$ to the complex $Kh_2(D)$. Indeed, if $W$ is such a web then
$\HOM(\1_{(\underline{0},\underline{2})},W)$ is a free $\Bbbk[\; \dottedsphere[.5]{3} \; ]$-module
with basis given by $1$-labeled cups with one or no dots, intersecting $2$-labeled sheets transversely.

The complex $Kh_2(D)$ is hence obtained from a cube of resolutions in which the nodes of the cube are
exactly those appearing in the construction of Khovanov homology. One can check that the maps labeling
the edges of this cube are, up to a $\pm$-sign, the maps $m$ and $\Delta$ from \cite{Kh1}. Since the
squares in this cube of resolutions anti-commute (by construction), an argument from \cite{ORS}
shows that this complex is isomorphic to the complex assigned to $D$ in \cite{Kh1}.
\end{proof}

%
\subsubsection{An explicit example}
%

The invariant of the Hopf link:
\[
L =
\xy
(0,0)*{
};
\endxy
\]
can be constructed in $\Bfoam[2l]{2}{2l}$ for any $l\geq2$; we'll take the minimal case $l=2$.
By the procedure detailed above, we have that $\llbracket L \rrbracket_2$ is given by the complex:
\begin{equation}\label{Hopfcomplex1}
\xy
(0,0)*{
%
};
\endxy
\end{equation}
which is shorthand for the complex obtained from the following cube
of resolutions (after applying some web isomorphisms):
\begin{equation}\label{Hopfcomplex2}
\xymatrix{ & & q \quad \xy
(0,0)*{%
}; \endxy \ar@<1ex>[urr]^-{-\alpha} & &
}
\end{equation}
where the foams $\alpha$ in the complex are those depicted in \eqref{sl2crossing+11} horizontally composed
with the relevant identity foams.

Note that the complex \eqref{Hopfcomplex1} is the image under $\Phi_2$ of the complex:
\[
\cal{E}_2^{(2)} \cal{E}_1^{(2)} \cal{F}_1\cal{F}_2 \cal{F}_1\cal{E}_1\cal{E}_2\cal{E}_3 \cal{T}_2 \cal{T}_2 \cal{E}_1 \cal{F}_1^{(2)} \cal{T}_2^{-1}
\cal{T}_3^{-1} \cal{T}_3 \cal{T}_2 \cal{F}_3 \cal{F}_2^{(2)}\onell{(0,2,0)}
\]
in $Kom(\Ucatc_Q(\mf{sl}_4))$. Indeed, for any tangle $\tau$, we can realize the complex
$\llbracket \tau \rrbracket_2$ as the image of a complex in the categorified quantum group by
pulling back the various pieces assigned to elementary tangles to
$Kom(\Ucatc_Q(\slm))$. One may then use the graphical
calculus of the categorified quantum group to perform calculations in link homology, see e.g.
\cite[Section 10]{Cautis}.

%
\subsubsection{The $\mf{sl}_3$ tangle invariant}
%

We define the $\mf{sl}_3$ tangle invariant $\llbracket \tau \rrbracket_3$ in $\foam{3}{m}$ in a similar
manner as above\footnote{We could define this invariant in $\Bfoam{3}{m}$ as well; however, the invariant
in $\foam{3}{m}$ is (essentially) the $\mf{sl}_3$ invariant found in the literature.}. An oriented tangle
(diagram) with no caps, cups, or crossings determines a sequence
$\mathbf{s}$ of $1$'s and $2$'s (corresponding to the
strands directed to the left and right respectively) and we map such a tangle to the identity web of the
sequence $(\underline{0},\mathbf{s},\underline{3})$, e.g.
\[
\xy
(0,0)*{
};
\endxy
\]
where the dotted and dashed lines denote web edges which are not actually present, i.e.
$0$- and $3$-labeled edges.

The invariant is defined on cups by
\[
\xy
(0,0)*{%
};
\endxy
\]
and on caps by
\[
\xy
(0,0)*{%
};
\endxy
\]
where, as in the $\mf{sl}_2$ case, we pre- and post-compose with the relevant braidings so that the webs map
between canonical sequences. These braidings are given by deleting the $3$-labeled edges from those
given in \eqref{sl3crossings}.

We define the invariant on left-oriented crossings by equations
\eqref{sl3crossing+11} and \eqref{sl3crossing-11}. We'd like
to define the image of the remainder of the crossings using the images of the braidings \eqref{sl3crossing+12} -
\eqref{sl3crossing-22} under the forgetful functor $\Bfoam{3}{m} \to \foam{3}{m}$; however, this assignment
would not be invariant under planar isotopy as the complexes differ by factors of $q^{\pm1}$. It is possible to
rescale the Rickard complexes $\cal{T}_i \onel$ depending on the weight $\lambda$ to fix this issue, but this
introduces unwanted scalings on the trivial braidings \eqref{sl3crossings}.
We instead follow our $\mf{sl}_2$ approach and define the remainder of the crossings in terms of the
left-oriented crossings and caps and cups.

\begin{prop}
The complex $\llbracket \tau \rrbracket_3$ assigned to a tangle diagram $\tau$, viewed in the homotopy
category of complexes of $\foam{3}{m}$, is an invariant of framed tangles.
\end{prop}

Renormalizing this invariant via $\llbracket \tau \rrbracket_3^r = q^{2w(\tau)}\llbracket \tau \rrbracket_3$
gives an invariant independent of framing which is (essentially) the same as Morrison-Nieh's
\cite{MorrisonNieh} extension of Khovanov's $\mf{sl}_3$ link homology \cite{Kh5} to tangles,
after setting the $3$-, $4$-, and $5$-dotted spheres equal to zero.

%
\subsubsection{Categorified clasps}
%

In \cite{Cautis}, Cautis showed that given any categorification of the skew Howe representations
$\Alt{q}{N}(\C^n_q\otimes \C^m_q)$, one obtains a categorification of $\sln$ clasps, the $\sln$ analogs of the
Jones-Wenzl projectors, using the higher representation theory of the categorified quantum group.
He conjectured that these representations could be categorified in the foam setting and that this construction
would give the categorified Jones-Wenzl projectors from \cite{CoopKrush} and \cite{Roz} and the
categorified $\mf{sl}_3$ projectors from \cite{Rose}.
Although the foam categories only categorify the intertwiners between such representations (and not the
representations themselves), Cautis' methods indeed give a uniform construction of categorified clasps
in the $\mf{sl}_2$ and $\mf{sl}_3$ foam $2$-categories. We now recall the details of this construction.

Fix a reduced expression $w=s_{i_1} \dots s_{i_k}$ for the longest word $w$ in the Weyl group for $\slm$ and
consider the complex $\cal{T}_{w} \onel:= \cal{T}_{i_1} \dots \cal{T}_{i_k}\onel$ in $Kom(\Ucatc_Q(\slm))$;
this complex gives the invariant assigned to a ``half-twist'' tangle. Cautis shows that the images of the complexes
$\cal{T}_{w}^{2k} \onel$ in any integrable $2$-representation stabilize as $k\to \infty$.
Denote the image of $\cal{T}_w \onel$ in such a $2$-representation by $\T_w \1_{\l}$ and let
$\T_w^{\infty} \1_{\l} := \lim_{k\to \infty} \T_{w}^{2k} \onel$.The complexes $\T_w^{\infty} \1_{\l}$
are idempotents (with respect to horizontal composition of complexes) and give categorified
clasps in any $2$-representation categorifying $\Alt{q}{N}(\C^n_q\otimes \C^m_q)$.

We first consider the $\mf{sl}_2$ case. Let $\sf{P}_m^+ := \T_w^{\infty} \1_{(0,\ldots,0)}$ in $\Bfoam[m]{2}{m}$
(which is a complex of webs mapping from the sequence $(1,\ldots,1)$ to itself), then
we have the following result, which should be viewed as the Blanchet foam analog of
\cite[Theorem 3.2]{CoopKrush}.
\begin{prop}
The complex $\sf{P}_m^+$ satisfies the following properties:
\begin{enumerate}
\item $\sf{P}_m^+$ is supported only in positive homological degree. \label{sl2proj1}
\item The identify web $\id_{(1,\ldots,1)}$ appears only in homological degree zero. \label{sl2proj2}
\item $\sf{P}_m^+$ annihilates the webs
\[
\xy
(0,0)*{\begin{tikzpicture} [decoration={markings, mark=at position 0.6 with {\arrow{>}}; },scale=.5]
 \draw [double, postaction={decorate}] (3,1) -- (1.75,1);
 \draw [very thick, postaction={decorate}] (1.25,2) -- (0,2);
\draw [very thick, postaction={decorate}] (1.75,1) -- (0,1);
 \draw [very thick, postaction={decorate}] (1.75,1) -- (1.25,2);
  \end{tikzpicture}};
\endxy
\quad , \quad
\xy
(0,0)*{\begin{tikzpicture} [decoration={markings, mark=at position 0.6 with {\arrow{>}}; },scale=.5]
 \draw [double, postaction={decorate}] (3,2) -- (1.75,2);
 \draw [very thick, postaction={decorate}] (1.75,2) -- (0,2);
\draw [very thick, postaction={decorate}] (1.25,1) -- (0,1);
 \draw [very thick, postaction={decorate}] (1.75,2) -- (1.25,1);
\end{tikzpicture}};
\endxy
\quad , \quad
\xy
(0,0)*{\begin{tikzpicture} [decoration={markings, mark=at position 0.6 with {\arrow{>}}; },scale=.5]
 \draw [very thick, postaction={decorate}] (3,2) -- (1.75,2);
 \draw [very thick, postaction={decorate}] (3,1) -- (1.25,1);
\draw [double, postaction={decorate}] (1.25,1) -- (0,1);
 \draw [very thick, postaction={decorate}] (1.75,2) -- (1.25,1);
\end{tikzpicture}};
\endxy
\quad , \quad
\xy
(0,0)*{\begin{tikzpicture} [decoration={markings, mark=at position 0.6 with {\arrow{>}}; },scale=.5]
 \draw [very thick, postaction={decorate}] (3,2) -- (1.25,2);
 \draw [very thick, postaction={decorate}] (3,1) -- (1.75,1);
\draw [double, postaction={decorate}] (1.25,2) -- (0,2);
 \draw [very thick, postaction={decorate}] (1.75,1) -- (1.25,2);
\end{tikzpicture}};
\endxy
\]
in $\Bfoam[m]{2}{m}$, up to homotopy. \label{sl2proj3}
\end{enumerate}
\end{prop}
\begin{proof}
Properties \eqref{sl2proj1} and \eqref{sl2proj2} follow via inspection. Property \eqref{sl2proj3} follows from
arguments in \cite[Section 5]{Cautis} or by adapting arguments from \cite{Rose} to the $\mf{sl}_2$ foam setting.
\end{proof}
It follows that $\sf{P}_m^+$ categorifies the analog of the Jones-Wenzl projector $p_m$ in the category of
Blanchet webs. Using the foamation $2$-functor $\Phi_{CMW}$,
the above procedure also gives a construction of the categorified Jones-Wenzl projectors from
\cite{CoopKrush} and \cite{Roz} in the CMW foam setting.

In the $\mf{sl}_3$ case, let $\mathbf{s}$ denote a sequence of $1$'s and $2$'s of length $m$; let
$\#\mathbf{s}_1$ denote the number of $1$'s and $\#\mathbf{s}_2$ the number of $2$'s in
$\mathbf{s}$. Define $\sf{P}_{\mathbf{s}}^+ := \T_w^{\infty} \1_{\l}$ in $\foam[m+\#\mathbf{s}_2]{3}{m}$ where
$\lambda$ maps to $\mathbf{s}$ under $\Phi_3$.
\begin{prop}
The complex $\sf{P}_{\mathbf{s}}^+$ is the categorified clasp $\tilde{P}_{\mathbf{s}}$ constructed
in \cite{Rose}.
\end{prop}
There is nothing to prove here; the categorified $\mf{sl}_3$ clasps in \cite{Rose} are constructed precisely as
the limit of the complexes $\T_w^{2k}\1_{\l}$ as $k\to \infty$. Note that the $+$'s and $-$'s in the sequences in that
work correspond to our $1$'s and $2$'s, respectively.

Having constructed categorified clasps, we can extend our $\mf{sl}_2$ and $\mf{sl}_3$ tangle invariants to give categorified
invariants of framed tangles in which each component is labeled by an irreducible representation. This construction is detailed in many
places in the literature, in particular in \cite{CoopKrush}, \cite{Rose}, and \cite{Cautis}, so we will be brief. Given a framed tangle $\tau$
with components labeled by irreducible representations, choose for each component a tensor product of fundamental
representations having the corresponding irreducible as a highest weight subrepresentation. Assign to the tangle the complex assigned
to a cabling of the tangle (we use here the fact that $\tau$ is framed) with the
categorified projector inserted along the cabling. The number of strands in the cabling of each component
(and the direction of such strands in the $\mf{sl}_3$ case) as well as which projector $\sf{P}^+$ is inserted is given by the relevant
tensor product of fundamental representations; this corresponds to a sequence of $1$'s in the $\mf{sl}_2$ case and a sequence of
$1$'s and $2$'s in the $\mf{sl}_3$ case.

One can show (see \cite{CoopKrush}, \cite{Rose}, or \cite{Cautis}) that the above invariant doesn't depend on
the choice of where the projector is inserted or which tensor product of fundamentals is used (up to equivalence in the case that the
tangle is not a link) and gives a categorification of the Reshetikhin-Turaev invariant of framed tangles.


%
\subsection{Comparing knot homologies} \label{sec_compare}
%

Let  $\Phi \maps \UcatD_Q(\slm) \to \cal{K}$ be any 2-representation giving a categorification of
$\bigwedge^N_q(\C^n_q\otimes \C^m_q)$. Cautis shows that for $N$ and $m$
sufficiently large, this 2-representation
assigns to any framed, oriented link $K$, a complex of 1-morphisms $\Psi(K) \in \End(Kom(\Phi(\l)))$ where
$\l$ is the highest weight in $\bigwedge^N_q(\C^n_q\otimes \C^m_q)$~\cite[Section 7.5]{Cautis}.  His framework does not require the full structure of a 2-representation of $\UcatD_Q(\slm)$, but rather the weaker data encoded in what he calls a categorical 2-representation.  This weaker action is more like the data described in Theorem~\ref{thm_CL} without requiring the KLR action.
The KLR relations greatly simplify the resulting complexes; in particular, they imply analogs of the higher Serre relation~\cite{Stosic} and commutativity relations for divided powers~\cite{KLMS}.  Using these relations, the complex $\Psi(K)$ associated to a link $K$ can be reduced to a complex that only involves direct sums of the identity 1-morphism $\1_{\l}$ of $\Phi(\l)$, with various grading shifts. One does not actually
need to know that the 2-representation $\cal{K}$ categorifies $\bigwedge^N_q(\C^n_q\otimes \C^m_q)$; it suffices that the nonzero weight spaces of
$\bigwedge^N_q(\C^n_q\otimes \C^m_q)$ parameterize the nonzero objects in $\cal{K}$.

Applying the functor $\HOM(\1_{\l}, -)$ to the complex $\Psi(K)$ maps it to a complex of
graded vector spaces. The number of $\1_{\l}$ summands and their grading shifts are formally determined
by the categorified quantum group, hence so are the vector spaces appearing in the complex.
The differentials depend only on the map $\HOM(\onel, \onel) \to \HOM(\1_{\l}, \1_{\l})$, so
it follows that this map completely determines the link homology theory.


When the graded algebra $A:= \HOM(\1_{\l}, \1_{\l})$ is $1$-dimensional in degree zero and
$0$ in all other degrees, only one such map exists, hence all constructions of $\sln$ link
homology satisfying this condition are equivalent.
After quotienting by the $3$-dotted sphere in the $\mf{sl}_2$ case and the $3$-, $4$-, and $5$-dotted spheres in
the $\mf{sl}_3$ case, the foam $2$-categories satisfy this condition.

This observation gives a method for showing that Cautis-Kamnitzer link homology is equivalent to
Khovanov-Rozansky homology.
Using constructions from previous work~\cite{CKL3,CKL4}, Cautis describes (weak) categorical 2-representations on derived categories $\cal{K}_{Gr,m}$ of coherent sheaves on varieties arising as orbits in the affine Grassmannian, as well as on coherent sheaves on Nakajima quiver varieties $\cal{K}_{Q,m}$.
Both of these categorical 2-representations are conjectured by Cautis, Kamnitzer, and Licata to extend to 2-representations of $\UcatD_Q(\slm)$.  By the results of \cite{CLau} it suffices to prove that the KLR algebras act; this was done in the $m=2$ case in \cite{CKL2} and will be generalized to symmetric Kac-Moody algebras
(in particular for arbitrary $m$)
in~\cite{Ca}. Moreover, in this setting the algebra $A$ satisfies the $1$-dimensionality condition, so this will
show that the link homology theory from \cite{CK02} fits into the framework described above.

The results from this paper will hence show that the foam based constructions of $\sln$ link homology agree
with the Cautis-Kamnitzer construction for
$n=2,3$. This re-proves Theorem~$8.2$ from \cite{CK01} and pairs with the
results from \cite{MV} to give the $n=3$ case of Conjecture~$6.4$ from \cite{CK02} equating $\mf{sl}_3$
Cautis-Kamnitzer and Khovanov-Rozansky link homology. In the sequel to this paper, will we establish
the analogous results for general $n$.

%
\subsection{Deriving foam relations from categorified quantum groups} \label{sec_deriving_rels}
%

In \cite{CKM}, Cautis, Kamnitzer, and Morrison showed that the relations on $\mf{sl}_n$ webs could be derived via skew Howe duality
from the relations in $\U_q(\slm)$.  Here we categorify this result in the $n=2,3$ case to show that many foam relations can be deduced
from the assignments defining the foamation 2-functors $\Phi_2$ and $\Phi_3$.  The main result of this section is that all $\mf{sl}_3$ foam relations, all CMW $\mf{sl}_2$ foam relations (assuming a strong form of
locality), and many Blanchet $\mf{sl}_2$ foam relations
follow from relations in the categorified quantum group $\UcatD_Q(\slm)$.
In a follow-up paper, we study foam categories for arbitrary $n$ using this framework~\cite{LQR2}.

%
\subsubsection{Blanchet $\mf{sl}_2$ foam relations}
%

Since the Blanchet foams arising as images under our $2$-functors must contain both $1$- and $2$-labeled
facets (unless they are identity foams) and always bound webs whose edges are
oriented leftward, we cannot expect to recover all defining relations from the relations in $\Ucat_Q(\slm)$.
For example, we have no hope of recovering the $1$- and $2$-labeled neck-cutting relations or closed
foam relations.

There are nevertheless numerous foam relations arising from the quantum group relations, which we list below. Note that some of the relations we obtain actually slightly generalize Blanchet's original relations, using $2$- and $3$-dotted enhanced spheres as graded parameters.

\begin{itemize}
 \item The nilHecke relation \eqref{eq_nil_dotslide} implies the enhanced neck-cutting relation:


\begin{align*}
\xy
(0,0)*{
};
\endxy \notag
\end{align*}
The first is relation \eqref{sl2NH_enh} (note the orientation of the seams) and the second is equivalent
to this using an isotopy of the $1$-labeled tube.

\item  Degree-zero bubbles in weight $\pm 2$ imply the ``blister'' relation \eqref{sl2bubble11dot}. The LHS of
the blister relation \eqref{sl2bubble11nodot} follows from the non-dotted bubble in weight $\pm2$.

\item Composing the second enhanced neck-cutting relation with
 \[
\xy
(0,0)*{
%
};
\endxy
\]
and using the previous blister relations, we obtain the following generalization of the dot-migration relation \eqref{sl2DotSliding}:

\begin{equation} \label{sl2DotSlidingGen}
\xy
(0,0)*{
%
};
\endxy
\end{equation}
Relation \eqref{eq_r2_ij-gen} with $j=i+1$ and $a_{i+1}=2$ implies that twice-dotted blisters can migrate between
$2$-labeled facets; this allows us to view them as graded parameters.

\item Composing the enhanced neck-cutting relation with
 \[
\xy
(0,0)*{
%
gives the RHS of \eqref{sl2bubble11nodot}, using \eqref{sl2bubble11dot}.
We do not obtain the analog of relation \eqref{sl2bubble11nodot} with two dots on the same facet,
since the dot-sliding relation has an additional term.

\item Relation \eqref{eq_r2_ij-gen} with $j=i+1$ and $a_{i+1}=1$ implies the relation:
\[
\xy
(0,0)*{
%
};
\endxy \quad = \quad \bambooRHSone[.6] \quad - \quad \bambooRHStwo[.6]
\]
which can be viewed as another enhanced version of neck-cutting.

\item Degree zero bubbles in weight $\pm 1$ with $m=2$ and $N=3$ imply relation \eqref{sl2Tube_2}.

\item Relation \eqref{eq_r3_hard-gen} implies the foam relation \eqref{sl2sq2}.
\end{itemize}

Finally, we comment on the behavior of a twice-dotted foam facet. When $\lambda=-2$, we compute
\[
\xy 0;/r.17pc/:
    (0,-10);(0,10) **\dir{-} *\dir{>};
    (6,2)*{\lambda};
    (0,0)*{\bullet}+(-3,1)*{\scs 2};
\endxy
\;\; =\;\;
 \xy 0;/r.17pc/:
    (-8,5)*{}="1";
    (0,5)*{}="2";
    (0,-5)*{}="2'";
    (8,-5)*{}="3";
    (-8,-10);"1" **\dir{-}?(.7)*\dir{>};;
    "2";"2'" **\dir{-} ?(.5)*\dir{>};
    "1";"2" **\crv{(-8,12) & (0,12)} ;
    "2'";"3" **\crv{(0,-12) & (8,-12)};
    "3"; (8,10) **\dir{-} ?(.31)*\dir{>};
    (13,-9)*{\lambda};
    (8,3)*{\bullet}+(3,1)*{\scs 2};
    (10,8)*{\scs };
    (-10,-8)*{\scs };
    \endxy
    \;\; = \;\; -\;\;
 \vcenter{   \xy 0;/r.17pc/:
    (4,6)*{};(4,4)*{} *\dir{-} ?(0)*\dir{>};
    (4,4)*{};(-4,-4)*{} **\crv{(4,1) & (-4,-1)}?(1)*\dir{<};
    (-4,4)*{};(4,-4)*{} **\crv{(-4,-) & (4,-1)}?(1)*\dir{>};?(0)*\dir{>};
    (4,-4)*{};(-4,-12)*{} **\crv{(4,-7) & (-4,-9)};
    (-4,-4)*{};(4,-12)*{} **\crv{(-4,-7) & (4,-9)}?(1)*\dir{<};
    (-12,4)*{};(-12,-14)*{} **\dir{-} ?(.3)*\dir{<};
    (4,-12)*{};(-4,-12)*{} **\crv{(4,-16) & (-4,-16)};
    (-4,4)*{};(-12,4)*{} **\crv{(-4,8) & (-12,8)};
    (-4,-12)*{\bullet}+(-3,1)*{\scs 2};
    (10,-9)*{\lambda};
 \endxy}
 \;\; + \;\;
 \sum_{g_1+g_2+g_3=1}
 \xy 0;/r.17pc/:
  (0,-12)*{}; (0,12)*{} **\dir{-} ?(1)*\dir{>};
  (0,0)*{\bullet}+(-3,1)*{\scs g_1};
    (12,6)*{\icbub{\scs  -3 +g_2}{i}};
    (10,-8)*{\iccbub{\scs 2+g_3}{i}};
    (18,-9)*{\lambda};
 \endxy
\]
The term on the left in the last part is sent to zero under the foamation functor. We obtain the following:
\[
\Phi_2\left(\;\;
\xy 0;/r.17pc/:
    (0,-12);(0,12) **\dir{-} *\dir{>};
    (6,2)*{\lambda};
    (0,0)*{\bullet}+(-3,1)*{\scs 2};
\endxy \;\;
\right)
\quad =\quad
\Phi_2\left( \;\;
 \xy 0;/r.17pc/:
  (0,-12)*{}; (0,12)*{} **\dir{-} ?(1)*\dir{>};
    (10,0)*{\iccbub{\scs 3}{}};
    (16,-9)*{\lambda};
 \endxy\;\;
+
\;\;
 \xy 0;/r.17pc/:
  (0,-12)*{}; (0,12)*{} **\dir{-} ?(1)*\dir{>};
   (10,7)*{\iccbub{\scs 2}{}};
   (10,-7)*{\iccbub{\scs 2}{}};
    (16,-9)*{\lambda};
 \endxy\;\;
 +
  \;\;
 \xy 0;/r.17pc/:
  (0,-12)*{}; (0,12)*{} **\dir{-} ?(1)*\dir{>};
      (0,0)*{\bullet};
    (10,0)*{\iccbub{\scs 2}{}};
    (16,-9)*{\lambda};
 \endxy\;\;
 \right)
\]
providing a way to decompose twice-dotted foam facets using the image of $2$- and $3$-dotted bubbles.

%
\subsubsection{CMW $\mf{sl}_2$ foam relations}
%

The local relations for disorientation seams, together with the neck-cutting relation and evaluations of (dotted) spheres
constitute the only CMW foam relations.
Since the CMW seam relations \eqref{CMWseam} involve complex coefficients, we cannot expect to derive
them from
the categorical skew Howe action of $\Ucat_Q(\slm)$.
We hence must impose the additional requirement that some relations can
be performed completely locally (which in practice says that some relations have a ``square root'').
We will show that with this additional assumption we can derive a slightly more general CMW foam
$2$-category in which both the $2$- and $3$-dotted spheres are (graded) parameters.

\begin{itemize}
 \item Seam relations: Considering \eqref{eq_r3_hard-gen} with $a_i=2$ and $a_{i+1}=1$, we find that the first term of the relation
 is mapped to zero and the remaining foams give
\begin{equation}
 \xy
(0,0)*{
};
\endxy \nn
\end{equation}
(up to isotopy).
Assuming this relation can be expressed locally, this requires the relation:
\begin{equation}
 \xy
(0,0)*{
%
};
\endxy \nn
\end{equation}
with $\omega'$ a primitive fourth root of the unity (a priori, this is not required to equal the fourth root
$\omega$ from the section \ref{CMWsec}).

Having fixed a value for $\omega'$, the values of degree-zero bubbles in weight $\pm 1$ with $m=2$ and $N=3$
give that a circular seam squares to give $-1$ (in both cases).
Again assuming complete locality, this gives that a circle can be removed from a foam at the cost of multiplying by a primitive
fourth root of unity. Using the above, we determine:
\begin{equation}
 \xy
(0,0)*{
\begin{tikzpicture} [fill opacity=0.2,  decoration={markings,
                        mark=at position 0.5 with {\arrow{>}};    }, scale=.6]
	\draw[fill=red] (0,0) rectangle (2,2);
	\draw [ultra thick,red] (.5,1) arc (180:-180:.5);
	\draw [dash pattern= on 1 pt off 3 pt, ultra thick,red] (.6,1) arc (180:-180:.4);
\end{tikzpicture}};
\endxy
\;\; =\;\;\omega'\;\;
\xy
(0,0)*{
\begin{tikzpicture} [fill opacity=0.2,  decoration={markings,
                        mark=at position 0.5 with {\arrow{>}};    }, scale=.6]
	\draw[fill=red] (0,0) rectangle (2,2);
\end{tikzpicture}};
\endxy
 \qquad , \qquad
\xy
(0,0)*{
\begin{tikzpicture} [fill opacity=0.2,  decoration={markings,
                        mark=at position 0.5 with {\arrow{>}};    }, scale=.6]
	\draw[fill=red] (0,0) rectangle (2,2);
	\draw [ultra thick,red] (.5,1) arc (180:-180:.5);
	\draw [dash pattern= on 1 pt off 3 pt, ultra thick,red] (.4,1) arc (180:-180:.6);
\end{tikzpicture}};
\endxy
 \;\; =\;\;-\omega' \;\;
\xy
(0,0)*{
\begin{tikzpicture} [fill opacity=0.2,  decoration={markings,
                        mark=at position 0.5 with {\arrow{>}};    }, scale=.6]
	\draw[fill=red] (0,0) rectangle (2,2);
\end{tikzpicture}};
\endxy \nn \quad .
\end{equation}

\item Closed foam relations:
Since negative degree bubbles are zero, we deduce that a non-dotted sphere is zero by considering the image of (non-dotted) bubbles
in weight $\pm 2$ with $N=2$ and $m=2$.
The values of once-dotted bubbles in weight $\pm 2$ give the value of a once-dotted sphere,
depending on the value of $\omega$. Choosing $\omega'=\omega$ (which we do for the duration),
we obtain that a once-dotted sphere has value $1$.
After we deduce a neck-cutting relation, we will be able to evaluate $n$-dotted spheres with $n\geq 4$ in terms of
spheres with fewer dots.

\item Neck-cutting: The NilHecke relation \eqref{eq_nil_dotslide} gives us two neck-cutting relations:
\begin{align*}
\xy
(0,0)*{
};
\endxy
\;\right) \quad .
\end{align*}
Using the seam and closed foam relations, we can recover a deformed version of the neck-cutting relation
from equation \eqref{sl2neckcutting}.
Caping with a dotted disk containing a disorientation seam, we have:
\begin{align*}
\xy
(0,0)*{
%
};
(0,-7)*{\dottedsphere[.6]{2}}
\endxy
\end{align*}
which gives a relation for sliding a dot through a seam. We then compute
\[
\xy
(0,0)*{
%
};
\endxy
\;\;\right)
\]
i.e. the following deformation of the neck-cutting relation:
\[
\cylinder[.6] \quad = \quad \slthncfour[.6] \quad + \quad \slthncfive[.6] \quad - \quad
\dottedsphere[.6]{2} \capcup[.6] \quad .
\]
Specializing $\dottedsphere[.6]{2}=0$, we recover the foam $2$-category from \cite{CMW}.
\end{itemize}

%
\subsubsection{$\mathfrak{sl}_3$ foam relations}
%

In the $\mathfrak{sl}_3$ setting, all foam relations are consequences of the relations in $\Ucat_Q(\slm)$:
\begin{itemize}
\item Dotted spheres: The values in relation \eqref{sl3closedfoam} are recovered by the value of
$\xy 0;/r.15pc/:
(0,0)*{\icbub{\alpha}{}};
 \endxy$
in weight $3$ with $m=2$ and $N=3$ for $\alpha=0,1,2$.
\item Neck-cutting: The image of equation \eqref{eq_ident_decomp-ngeqz} in weight $3$ and
with $m=2$ and $N=3$
gives the neck-cutting relation \eqref{sl3neckcutting} (note that the cup gives a $-1$ coefficient and the
cap gives $+1$). One can obtain the simpler neck-cutting relations found in the literature by quotienting
the categorified quantum group by the relevant bubbles (or equivalently passing to the quotient of the foam
category where we set the $3$- and $4$-dotted spheres equal to zero).
\item Equation \eqref{sl3tube} is a consequence of the NilHecke relation \eqref{eq_nil_dotslide}.
\item Equation \eqref{sl3rocket} is a consequence of equation \eqref{eq_r3_hard-gen}.
\item $\Theta$-foams:
For $\alpha+\beta \leq3$, the values of:
\begin{equation}\label{nestedtheta}
\xy 0;/r.16pc/:
 (-6,0)*{};
  (6,0)*{};
  (14,-10)*{\lambda};
  (-4,0)*{}="t1";  (4,0)*{}="t2";
 "t2";"t1" **\crv{(4,6) & (-4,6)}; ?(0)*\dir{<}
   ?(.96)*\dir{<} ?(.3)*\dir{}+(2,2)*{\scs i};
  "t2";"t1" **\crv{(4,-6) & (-4,-6)};
  ?(.7)*\dir{}+(0,0)*{\bullet}+(2,-3)*{\scs  \beta};
  (-12,0)*{}="t1";  (12,0)*{}="t2";
  "t2";"t1" **\crv{(12,15) & (-12,15)}; ?(0)*\dir{>} ?(1)*\dir{>} ?(.3)*\dir{}+(2,2)*{\scs \;\; i+1};
  "t2";"t1" **\crv{(12,-15) & (-12,-15)};
  ?(.7)*\dir{}+(0,0)*{\bullet}+(2,-4)*{\scs \alpha};
  \endxy
\qquad \text{and} \qquad
\xy  0;/r.16pc/:
 (-6,0)*{};
  (6,0)*{};
  (14,-10)*{\mu};
  (-4,0)*{}="t1";  (4,0)*{}="t2";
 "t2";"t1" **\crv{(4,6) & (-4,6)}; ?(.02)*\dir{>}
   ?(1)*\dir{>} ?(.3)*\dir{}+(2,2)*{\scs i};
  "t2";"t1" **\crv{(4,-6) & (-4,-6)};
  ?(.7)*\dir{}+(0,0)*{\bullet}+(2,-3)*{\scs \alpha};
  (-12,0)*{}="t1";  (12,0)*{}="t2";
  "t2";"t1" **\crv{(12,15) & (-12,15)}; ?(0)*\dir{<} ?(1)*\dir{<} ?(.3)*\dir{}+(2,2)*{\scs \;\; i+1};
  "t2";"t1" **\crv{(12,-15) & (-12,-15)};
  ?(.7)*\dir{}+(0,0)*{\bullet}+(2,-4)*{\scs \beta};
  \endxy
\end{equation}
when $\lambda$ maps to a sequence with $a_i=0$, $a_{i+1}=3$ and $a_{i+2}=0$ and
$\mu$ maps to a sequence with $a_i=3$, $a_{i+1}=0$ and $a_{i+2}=3$
give the values in relation \eqref{sl3thetafoam} when $\alpha+\beta \leq 3$ and $\gamma=0$.
In fact, these values, together with the remainder of the foam relations, determine the values of
all theta-foams.

Using the values of theta-foams we have already determined, we can deduce the blister relations:
\[
\xy
(0,0)*{

 };
 \endxy
\]
from the neck-cutting and dotted sphere relations. The equality
\[
\xy 0;/r.16pc/:
 (-6,0)*{};
  (6,0)*{};
  (14,-10)*{\lambda};
  (-4,0)*{}="t1";  (4,0)*{}="t2";
 "t2";"t1" **\crv{(4,6) & (-4,6)}; ?(0)*\dir{<}
   ?(.96)*\dir{<} ?(.3)*\dir{}+(2,2)*{\scs i};
  "t2";"t1" **\crv{(4,-6) & (-4,-6)};
  ?(.7)*\dir{}+(0,0)*{\bullet}+(2,-3)*{\scs  2};
  (-12,0)*{}="t1";  (12,0)*{}="t2";
  "t2";"t1" **\crv{(12,15) & (-12,15)}; ?(0)*\dir{>} ?(1)*\dir{>} ?(.3)*\dir{}+(2,2)*{\scs \;\; i+1};
  "t2";"t1" **\crv{(12,-15) & (-12,-15)};
  ?(.7)*\dir{}+(0,0)*{\bullet}+(2,-4)*{\scs 2};
  \endxy
  \;\; = \;\;
  \xy 0;/r.16pc/:
  (-8,0)*{\iccbub{2}{i+1 \;\;\;\;}};
  (8,0)*{\icbub{3}{i}};
  (0,10)*{\lambda};
  \endxy
  \;\; - \;\;
    \xy 0;/r.16pc/:
  (-8,0)*{\iccbub{3}{i+1 \;\;\;\;}};
  (8,0)*{\icbub{2}{i}};
  (0,10)*{\lambda};
  \endxy
\]
implies that
$
\xy
(0,0)*{
};
\endxy = 0;
$
the neck-cutting relation then gives the additional blister relation:
\[
\xy
(0,0)*{
%
};
(0,15)*{\dottedsphere[.5]{3}};
\endxy
\]
Composing \eqref{sl3tube} with the foam
\[
\xy
(0,0)*{
%
};
\endxy
\]
then gives the dot migration relation:
\begin{equation}\label{sl3dotmigration}
\xy
(0,0)*{
%
};
(20,0)*{\dottedsphere[.5]{3}};
\endxy
\quad = 0
\end{equation}
(compare to \cite[Figure 17]{Kh5}). Using this relation, in conjunction with
\eqref{sl33dot}, we can evaluate the remaining theta-foams from equation \eqref{sl3thetafoam}.
\end{itemize}
Note that we may also recover many of the relations which follow as consequences of the defining relations:
\begin{itemize}
 \item Equation \eqref{sl3bamboo} is a consequence of equation \eqref{eq_r2_ij-gen}.
 \item Using \eqref{eq_ident_decomp-ngeqz} with $\lambda_i=1$ and $N_i=3$, we compute
\[
 \xy 0;/r.17pc/:
    (-8,5)*{}="1";
    (0,5)*{}="2";
    (0,-5)*{}="2'";
    (8,-5)*{}="3";
    (-8,-10);"1" **\dir{-}?(.7)*\dir{>};;
    "2";"2'" **\dir{-} ?(.5)*\dir{>};
    "1";"2" **\crv{(-8,12) & (0,12)} ;
    "2'";"3" **\crv{(0,-12) & (8,-12)};
    "3"; (8,10) **\dir{-} ?(.31)*\dir{>};
    (13,-9)*{\lambda};
    (0,4)*{\bullet}+(3,1)*{\scs 3};
    (10,8)*{\scs };
    (-10,-8)*{\scs };
    \endxy
\;\; = \;\; -\;\;
 \vcenter{   \xy 0;/r.17pc/:
    (-4,-6)*{};(-4,-4)*{} **\dir{-};
    (-4,-4)*{};(4,4)*{} **\crv{(-4,-1) & (4,1)}?(1)*\dir{>};
    (4,-4)*{};(-4,4)*{} **\crv{(4,-1) & (-4,1)}?(1)*\dir{<};?(0)*\dir{<};
    (-4,4)*{};(4,12)*{} **\crv{(-4,7) & (4,9)};
    (4,4)*{};(-4,12)*{} **\crv{(4,7) & (-4,9)}?(1)*\dir{>};
    (12,-4)*{};(12,14)*{} **\dir{-} *\dir{>};
    (-4,12)*{};(4,12)*{} **\crv{(-4,16) & (4,16)};
    (4,-4)*{};(12,-4)*{} **\crv{(4,-8) & (12,-8)}?(1)*\dir{>};
  (15,8)*{\l};
     (-6,-3)*{\scs i};
     (6.5,-3)*{\scs i};
    (4,12)*{\bullet}+(3,1)*{\scs 3};
 \endxy}
  \;\; + \;\;
   \sum_{ \xy  (0,3)*{\scs f_1+f_2+f_3}; (0,0)*{\scs =\lambda_i-1};\endxy}
    \vcenter{\xy 0;/r.17pc/:
    (15,10)*{\l};
    (-8,0)*{};
  (8,0)*{};
  (-4,-17)*{}; (-4,-15)*{} **\dir{-};
  (-4,-15)*{}="b1";
  (4,-15)*{}="b2";
  "b2";"b1" **\crv{(5,-8) & (-5,-8)}; ?(.05)*\dir{<} ?(.93)*\dir{<}
  ?(.8)*\dir{}+(0,-.1)*{\bullet}+(-3,2)*{\scs f_3};
  (-4,15)*{}="t1";
  (4,15)*{}="t2";
  "t2";"t1" **\crv{(5,8) & (-5,8)}; ?(.15)*\dir{>} ?(.95)*\dir{>}
  ?(.4)*\dir{}+(0,-.2)*{\bullet}+(3,-2)*{\scs \; f_1};
  (0,0)*{\iccbub{\scs \quad -\l_i-1+f_2}{i}};
  (7,-13)*{\scs i};
  (-7,13)*{\scs i};
    (12,-15)*{};(12,16)*{} **\dir{-} *\dir{>};
    (-4,15)*{};(4,15)*{} **\crv{(-4,20) & (4,20)}?(0);
    (4,-15)*{};(12,-15)*{} **\crv{(4,-18) & (12,-18)};
  (3.5,16)*{\bullet}+(3,1)*{\scs 3};
  \endxy}
\]
\[
= \;\; \left(
\xy 0;/r.15pc/:
 (0,0)*{\icbub{5}{i}};
 \endxy
-2
\xy 0;/r.15pc/:
 (0,8)*{\icbub{4}{i}};
 (0,-8)*{\icbub{3}{i}};
 \endxy
+\left(
\xy 0;/r.15pc/:
 (0,0)*{\icbub{3}{i}};
 \endxy
\right)^3 \right)
\xy 0;/r.17pc/:
    (0,-10);(0,10) **\dir{-} *\dir{>};
\endxy
\;\;+\;\; \left(
\xy 0;/r.15pc/:
 (0,0)*{\icbub{4}{i}};
 \endxy
-\left(
\xy 0;/r.15pc/:
 (0,0)*{\icbub{3}{i}};
 \endxy
\right)^2 \right)
\xy 0;/r.17pc/:
    (0,-10);(0,10) **\dir{-} *\dir{>};
\endxy
\;\;+\;\;
\xy 0;/r.15pc/:
 (0,0)*{\icbub{3}{i}};
 \endxy
\;
\xy 0;/r.17pc/:
    (0,-10);(0,10) **\dir{-} *\dir{>};
\endxy
\]

which gives equation \eqref{sl33dot}.
\item Equation \eqref{sl3airlock} follows from the degree-zero bubble
$
\xy 0;/r.15pc/:
 (-6,0)*{};
  (6,0)*{};
  (-4,0)*{}="t1";
  (4,0)*{}="t2";
  "t2";"t1" **\crv{(4,6) & (-4,6)}?(.7)*\dir{}+(-2,0)*{\scs i};
   ?(0)*\dir{<} ?(.95)*\dir{<};
  "t2";"t1" **\crv{(4,-6) & (-4,-6)};
  (6,6)*{\lambda}
\endxy = \id
$ when $a_i = 1$ and $a_{i+1}=2$.
\end{itemize}



%

%
\end{document}